\documentclass[11pt,reqno]{amsart}

\usepackage{amssymb,amsmath,graphicx,amsfonts,euscript}
\usepackage{color}
\newtheorem{thm}{Theorem}[section]

\newtheorem{prop}[thm]{Proposition}

\newtheorem{rem}[thm]{Remark}

\newtheorem{lemma}[thm]{Lemma}

\def\bl{\big\|}

\def\pp{\partial}
\def\na{\nabla}

\def\la{\langle}
\def\ra{\rangle}

\def\R{\mathbb{R}}
\def\Z{\mathbb{Z}}
\def\lr{\langle\na\rangle}

\allowdisplaybreaks

\numberwithin{equation}{section}

\begin{document}
\title[The Compressible 2D MHD SYSTEM]{Global small solutions to the compressible 2D magnetohydrodynamic system without magnetic diffusion}

\author[J. Wu and Y. Wu]{Jiahong Wu$^{1}$ and Yifei Wu$^{2}$}

\address{$^1$ Department of Mathematics,
Oklahoma State University,
401 Mathematical Sciences,
Stillwater, OK 74078, USA}

\email{jiahong.wu@okstate.edu}

\address{$^2$Center for Applied Mathematics,
Tianjin University,
Tianjin 300072,
China}

\email{yerfmath@gmail.com}

\date{\today}

\begin{abstract}
This paper establishes the global existence and uniqueness
of smooth solutions to the two-dimensional compressible magnetohydrodynamic system when the initial data is close to an equilibrium state. In addition, explicit large-time decay rates for various Sobolev norms of the solutions are also obtained.
These results are achieved through a new approach of diagonalizing a system of coupled linearized equations.
The standard method of diagonalization via the eigenvalues and eigenvectors of the matrix symbol
is very difficult to implement here. This new process allows us to obtain an integral representation of the full system through explicit kernels. In addition, in order to overcome various difficulties such as the anisotropicity and criticality, we fully exploit the structure of the integral representation and employ extremely delicate Fourier analysis and associated estimates.
\end{abstract}

\maketitle

\tableofcontents

\section{Introduction}\label{intro}

\vskip .1in
This paper examines the global (in time) existence and uniqueness
of solutions
to the two-dimensional (2D) compressible magnetohydrodynamic (MHD) system, namely
\begin{equation} \label{e1.1}
\begin{cases}
\partial_t \rho +\nabla\cdot (\rho \vec u)= 0,\quad (t,\,x,\,y)\in \mathbb{R}_+\times\mathbb{R}\times\mathbb{R},\\
\partial_t (\rho\vec u) +\nabla\cdot(\rho\vec u \otimes \vec u) -\Delta \vec u-\lambda\nabla(\nabla\cdot \vec u) + \nabla P= -\frac12\nabla\big(|\vec b|^2\big)+\vec b\cdot \nabla \vec b,\\
\partial_t \vec b +\vec u\cdot \nabla \vec b=\vec b\cdot \nabla \vec u, \\
\nabla \cdot \vec b = 0\\
\end{cases}
\end{equation}
with the initial data
$$
\rho|_{t=0}=\rho_0(x,y),\quad \vec u|_{t=0}=\vec u_0(x,\,y),\quad\vec b|_{t=0}=\vec b_0(x,\,y).
$$
Here $\rho\in \R^+$ denotes the density, $\vec u=(u,\,v)\in \R^2$ represents the velocity field, $P=P(\rho)$ the pressure,  $\vec b\in \R^2$ the magnetic field and $\lambda$ is a constant with $|\lambda|<1$. Moreover, the pressure term $P(\rho)$ is assumed to obey the following polytropic law,
$$
P(\rho) = A\rho^{\gamma},
$$
where $A$ is the entropy constant and $\gamma\ge1$ is called the adiabatic index. We remark that there is a large
literature on the compressible MHD equations (with both velocity dissipation and magnetic diffusion) due to their physical importance and mathematical challenges
(see, e.g., \cite{ChenWang,HuWang,HuangLi,Kawa,Kawa2}).

\vskip .1in
This paper aims to achieve three goals. The first is to establish the
global well-posedness of smooth solutions of \eqref{e1.1} when the initial data
$(\rho_0, \vec u_0, \vec b_0)$ is smooth and close to the equilibrium state
$(1, \vec 0, \vec e_1)$, where we denote $\vec 0=(0,0)$ and $\vec e_1=(1,0)$.
To this end, we may assume that $\gamma=3$ and $A=\frac13$, since the
other cases can be essentially reduced to this case, after omitting
some high-order terms. The second is to explore the hidden structure
in the system of the linearized equations and offer a new way of
diagonalizing a complex system of linearized equations.
The third is to obtain explicit and sharp large-time decay rates for
the solutions in various Sobolev spaces. As we shall see in the subsequent sections,
we need to overcome several major difficulties including the anisotropicity and the criticality
associated with the 2D compressible MHD equations with only velocity dissipation and no magnetic diffusion.

\vskip .1in
In the 2D case, $\nabla \cdot \vec b=0$ implies that for a scalar function $\phi$,
$$
\vec b=\nabla^\perp \phi\equiv \big(\partial_y\phi,-\partial_x\phi\big).
$$
With this substitution, \eqref{e1.1} becomes
\begin{equation} \label{e1.2}
\begin{cases}
\partial_t \rho +\nabla\cdot (\rho \vec u)= 0,\quad (t,\,x,\,y)\in \mathbb{R}_+\times\mathbb{R}\times\mathbb{R},\\
\partial_t (\rho\vec u) +\nabla\cdot(\rho\vec u \otimes \vec u) -\Delta \vec u-\lambda\nabla(\nabla\cdot \vec u) + \rho^2\nabla \rho = -\nabla\phi \Delta\phi,\\
\partial_t \phi +\vec u\cdot \nabla \phi=0.
\end{cases}
\end{equation}
We will mainly work with this form of the equations, which is more convenient
for the estimates. However, this form is not essential for our main result.

\vskip .1in
To precisely state our main result, we introduce the functional settings.
The operator notation $\la\na\ra$ is standard. Let
\begin{align*}
U=(n,u,v,\psi)  \quad \mbox{and}\quad U_0=(n_0,u_0,v_0,\psi_0).
\end{align*}
Let $M$ be a big integer ($M\ge 8$ is sufficient, as a careful computation would show). Let $\epsilon>0$ be a small parameter, $\gamma\in (\frac12,1]$, and $\frac\gamma2<\bar \gamma<1+\frac\gamma2$. We define $X_n,X_{\textbf{u}}, X_\psi$ with their norms given by
\begin{equation*}
\begin{aligned}
\|n\|_{X_n} &= \sup_{t\geq 0}\Big\{\langle t\rangle^{-\epsilon}\|\la\na\ra^M n\|_{L^2_{xy}}+\langle t\rangle^{\frac14}\|\la\na\ra^3n\|_{L^2_{xy}}+\langle t\rangle^{\frac12}\|\la\na\ra^{\frac32} n\|_{L^\infty_{xy}}+\langle t\rangle^{\frac34}\|\la\na\ra\pp_x n\|_{L^2_{xy}}\Big\};\\
\|\vec u\|_{X_{\textbf{u}}} &= \sup_{t\geq 0}\Big\{\langle t\rangle^{-\epsilon}\|\la\na\ra^M \vec u\|_{L^2_{xy}}+\langle t\rangle^\frac12\|\vec u\|_{L^2_{xy}}
+\langle t\rangle \|\la\na\ra\vec u\|_{L^\infty_{xy}}+\langle t\rangle^{\frac34} \|v\|_{L^2_{x}L^\infty_y}\\
&\qquad+\langle t\rangle^{\frac34} \||\na|^\gamma \vec u\|_{L^2_{xy}}+\langle t\rangle \|\la\na\ra\pp_x u\|_{L^2_{xy}}
+\langle t\rangle \|\la\na\ra\nabla v\|_{L^2_{xy}}\Big\};\\
\|\psi\|_{X_\psi} &= \sup_{t\geq 0}\Big\{\langle t\rangle^{-\epsilon}\|\la\na\ra^M \nabla \psi\|_{L^2_{xy}}+\langle t\rangle^{\frac14}\big\|\la\na\ra^4|\na|^\gamma \psi\big\|_{L^2_{xy}}+\langle t\rangle^{\frac12}\big\|\pp_x \psi\big\|_{L^2_{xy}}\\
&\qquad+\langle t\rangle^{\frac34}\big\|\pp_x \na\psi\big\|_{L^2_{xy}}+\langle t\rangle^\frac12\big\||\na|^{\bar \gamma}\la\na\ra\psi\big\|_{L^\infty_{xy}}\Big\}.
\end{aligned}
\end{equation*}
Now we define the working space $X=X_n\times X_{\textbf{u}}\times X_\psi$, whose norm is given by
\begin{equation}\label{X}
\begin{aligned}
\|U\|_{X} = \|n\|_{X_n}+\|\vec u\|_{X_{\textbf{u}}} +\|\psi\|_{X_\psi}.
\end{aligned}
\end{equation}
Note that we have incorporated the large-time behaviors of various Sobolev norms
into the definition of $X$ and these rates will become clear in the subsequent
sections. Moreover,
\begin{equation}\label{X0}
\begin{aligned}
\|U_0\|_{X_0} =\|\la\na\ra^M(n_0, \vec u_0,\,\nabla \psi_0)\|_{L^2_{xy}}
+\|\lr^5(n_0, \vec u_0,\,\nabla\psi_0)\|_{L^1_{xy}}.
\end{aligned}
\end{equation}
Our main result can then be stated as follows. We use $A \lesssim B$ or $B  \gtrsim A$ to denote the statement
that $A\le C B$ for some absolute constant $C>0$.

\begin{thm} \label{th1}
Assume $|\lambda|\le c_0$ for some absolute constant $c_0>0$, and let $n=\rho-1, \psi=\phi-y$, $n_0=\rho_0-1, \psi_0=\phi_0-y$. Then there exists a small constant $\delta>0$ such that,
if the initial data $(n_0, \vec u_0,\,\phi_0)$ satisfies $\|(n_0, \vec u_0,\,\psi_0)\|_{X_0}\le\delta $,
then there exists a unique global solution $(\rho, u,\,v,\,\phi)\in X$ to the system \eqref{e1.1}.
Moreover,
$$
\|(n, u,\,v,\,\phi)\|_X \lesssim \delta.
$$
Especially, the following decay estimates hold
$$
\|n(t)\|_{L^\infty_{xy}}\lesssim \delta\, t^{-\frac12};\quad \|\vec u(t)\|_{L^\infty_{xy}}\lesssim \delta\, t^{-1};
\quad
\|\nabla\psi(t)\|_{L^\infty_{xy}}\lesssim \delta\, t^{-\frac12}.
$$
\end{thm}

We remark that the smallness condition on $\lambda$ is not essential for our
main result. The term $\lambda\nabla(\nabla\cdot \vec u)$ is treated as a
nonlinear term in order to simplify the treatment on the linearized equations.
We could have included $\lambda\nabla(\nabla\cdot \vec u)$ as a linear term. That would make the diagonalization process more sophisticated and our main
idea murky. Moreover, we remark that the treatment on the ``nonlinear'' term $\lambda\nabla(\nabla\cdot \vec u)$ is rather non-trivial and has of independence interest, even though $\lambda$ is small.

\vskip .1in
Due to the lack of dissipation or damping in the transport equation for
$\phi$, the global existence of smooth solutions to \eqref{e1.2} is
an extremely difficult problem
and is currently open. Recent strategy has been to seek solutions
near an equilibrium state.
This has been very successful in the incompressible counterpart
of \eqref{e1.2}. The paper of Lin, Xu and
Zhang \cite{LinZhang1} appears to be the very first to establish
the global existence of smooth solutions
(near an equilibrium) for the 2D incompressible MHD equations
without magnetic diffusion. They resort
to Lagrangian coordinates
and anisotropic Besov spaces. \cite{LinZhang1} inspired several different
approaches (\cite{HuLin,Zhifei,WWX,ZhangT}). Little has been done for the
compressible MHD equations. The only work currently available is a very recent
preprint of Hu \cite{HuX}, in which the author constructs a global solution for the
compressible MHD equation without magnetic diffusion
starting from an initial data near a constant
background and the Lagrangian deformation gradient near the identity
matrix. \cite{HuX} involves hybrid Besov spaces. The approach of this paper is completely different from that in \cite{HuX}.

\vskip .1in
Besides the global existence of smooth solutions of \eqref{e1.2},
this paper also aims at thoroughly
understanding the structure of the linearized system and how this
structure leads to the exact large-time
decay rates for various Sobolev norms of the solutions.
Keeping all nonlinear terms on the right-hand side,
the equations for the perturbation $U=(n,u,v,\psi)$ can be written as
\begin{equation} \label{Ueq}
\pp_t\overrightarrow{U}=A\overrightarrow{U}+\overrightarrow{N},
\end{equation}
where
$$
\overrightarrow{U}=\left(\begin{array}{c}
                     n \\
                     u \\
                     v \\
                     \psi
                   \end{array}\right),\,\,
A=\left(\begin{array}{cccc}
               0 & -\pp_x & -\pp_y & 0 \\
               -\pp_x & \Delta+\lambda\partial_{xx} & \lambda\partial_{xy} & 0 \\
               -\pp_y & \lambda\partial_{xy} & \Delta+\lambda\partial_{yy} & -\Delta \\
               0 & 0 & -1 & 0
             \end{array}\right),\,\,
\overrightarrow{N}=\left(\begin{array}{c}
                     N_0 \\
                     N_1 \\
                     N_2 \\
                     N_3
                   \end{array}\right)
$$
with $N_0, N_1, N_2$ and $N_3$ given in (\ref{N0}--\ref{N3}). The terms involving $\lambda$ are treated as nonlinear terms. The natural next step would be to diagonalize the system, but the standard
method via the eigenvalues and eigenvectors of $A$ appears to be impossible to carry out.  We provide a new systematic approach. The diagonalization here is obtained by suitable differentiation of (\ref{Ueq}) and magically all the resulting equations have the same structure,
\begin{align*}
\Big((\pp_{tt}-\Delta\pp_t-\Delta)^2-\Delta \pp_{yy}\Big)n=F_0,
\\
\Big((\pp_{tt}-\Delta\pp_t-\Delta)^2-\Delta \pp_{yy}\Big)u=F_1,
\\
\Big((\pp_{tt}-\Delta\pp_t-\Delta)^2-\Delta \pp_{yy}\Big)v=F_2,
\\
\Big((\pp_{tt}-\Delta\pp_t-\Delta)^2-\Delta \pp_{yy}\Big)\psi=F_3,
\end{align*}
where $F_0, F_1, F_2$ and $F_3$ are given in (\ref{F0}--\ref{F3}).
By factorizing and inverting the differential operator in the equation
$$
\Big[(\pp_{tt}-\Delta\pp_t-\Delta)^2-\Delta \pp_{yy}\Big]\Phi=F,
$$
we obtain the integral representation
$$
\Phi=L(\Phi_{1,0},\Phi_{2,0},\Phi_{3,0},\Phi_{4,0})+\int_0^t K(t-s,\partial_x,\partial_y) F(s)\,ds,
$$
where $\Phi_{1,0},\Phi_{2,0},\Phi_{3,0},\Phi_{4,0}$ and $K$ are given in Lemma \ref{lem:Formula}. This
integral representation and its simplified variants form the foundation for proving Theorem \ref{th1}.
Due to the standard continuity argument,
the proof of Theorem \ref{th1} is then reduced to verifying
the integral representation obeys
\begin{equation} \label{toprove}
\|U\|_X \le  C_0\|U_0\|_{X_0}+ Q(\|U\|_X),
\end{equation}
where $C_0$ is an absolute positive constant and $Q(r)$ represents a polynomial of $r$ with the lowest order at least quadratic. In this paper, we shall prove that
\begin{equation} \label{conneed}
\|U\|_X \le \frac12 C_0\|U_0\|_{X_0}  + C_*|\lambda|\|U\|_X+ \frac12Q(\|U\|_X),
\end{equation}
for some absolute positive constant $C_*$. Indeed,
by choosing $|\lambda|\le c_0$, for $c_0=\frac1{2C_*}$, we obtain \eqref{toprove}.
 To prove (\ref{conneed}), we
first need to understand the exact decay properties of $L$ and $K$. Because $K=K(t,x,y)$ is anisotropic,
the decay rates of $K$ and its various spatial and time derivatives depend on the Fourier frequencies.
These decay rates are explicitly given in Section \ref{KernelProperty}.
In order to prove (\ref{conneed}),
we further recast the integral representation and estimate
each component of $U=(n,u,v,\psi)$ and verify (\ref{conneed}). This is a complicated and lengthy process.

\vskip .1in
We need to deal with several difficulties. The first is the tanglement
in the linear operator. The linear operator under study is
$$
(\pp_{tt}-\Delta\pp_t-\Delta)^2-\Delta \pp_{yy},
$$
which is a four-order (both in time and spatial variables) wave type operator.
In general, the number of the unknown functions gives the order of the
linearized equations. In light of this, the problem  in the incompressible
case is much easier than the compressible one, since, under the
incompressibility condition, the number of the unknown function members reduces
to 2. But in the compressible case, we have four unknown members
 $U=(n,u,v,\psi)$. This marks the first challenge in this paper and it takes
extraordinary efforts to completely understand the properties of the operator
and obtain the decay estimates.
Another difficulty is the anisotropicity.  It is easy to see from the expansion
\begin{align*}
(\pp_{tt}-\Delta\pp_t-\Delta)^2-\Delta \pp_{yy}
=\left(\big(\pp_{tt}-\Delta\pp_t-\Delta\big)-\sqrt{\Delta \pp_{yy}}\right)\left(\big(\pp_{tt}-\Delta\pp_t-\Delta\big)+\sqrt{\Delta \pp_{yy}}\right),
\end{align*}
that the operator has different behaviors along different directions. Indeed, by a delicate analysis, we observe that kernel of the linear operator obeys the estimate that
$$
 |\widehat K(t,\xi,\eta)|
\lesssim\chi_{A\ge 1}\frac1{A^4} e^{-ct}+\chi_{A\le 1}\frac{1}{A|\xi||\eta|}e^{-\frac14 A^2t}
+\frac1{A^{4}}\chi_{|\xi|\le A^2 }e^{-\frac{\xi^2}{2A^2}t},
$$
with $A=\sqrt{\xi^2+\eta^2}$. It is singular and anisotropic.
More precisely, it has a high-order negative regularity in the kernel,
and has weaker behavior in $y$-direction than in $x$-direction.
These properties of the operator are reflected in the definition of the
working space, kernel properties and various estimates.

\vskip .1in
Another difficulty is the criticality, due to the slow dispersion and
quasilinearity. Indeed,  the spatial $L^\infty$-norm of $\vec u$ behaves like $t^{-1}$, which is not
integrable in time. What's more, the spatial $L^\infty$-norm of $n$ and $\psi$ decay like $t^{-1/2}$,
which is even far from being integrable. To overcome this difficulty, we make full use of the structure of the system.
In particular, for $n$ and $\psi$, we obtain the cubic type nonlinearities, thanks to the structure of the velocity diffusion.
In contrast to the incompressible case,
the compressible MHD problem is much harder. Indeed, by using $\nabla\cdot \vec u=0$ in the incompressible case, one may transfer
the derivative $\partial_y$ to $\partial_x$ and then faster decay rates result due to the presence of $\partial_x$. However,
in the compressible case, the fast decay rate associate with $\partial_x$ can not be converted into
that for $\partial_y$ due to the lack of incompressibility condition. This criticality prompts us to use
extremely fine Fourier analysis and associated estimates.

\vskip .1in
The rest of this paper is divided into seven sections and an appendix. The second section derives
the integral representation for the perturbation $U=(n,u,v,\psi)$ through the kernels $K$ and $L$.
The third section provides pointwise as well as $L^p$-estimates for $\widehat K(t,\xi,\eta)$ (defined in
(\ref{Kop})) and the Fourier transforms of various derivatives of $K$.  The fourth section establishs the decaying estimates of the linear flow. The fifth section contains
the local existence and uniqueness theory and an energy estimate serving as a part of the proof
for Theorem \ref{th1}. Section \ref{sect:n} through Section \ref{sec:psi} verify (\ref{conneed}).
The appendix proves Lemma \ref{lem:Elem1} and several inequalities used in the previous sections.


\vskip .5in
\section{Integral representation}
\label{integralrepresentation}

This section derives the integral representation for the perturbation $U=(n,u,v,\psi)$. The derivation consists
of four steps, which are presented in four subsections. The first subsection writes the equations for $U$ with the
nonlinear terms kept on the right-hand side. The second subsection diagonalizes the equations through suitable
differentiation and magically all the equations have the same structure. The third subsection obtains an
preliminary integral representation through factoring the differential operator. This representation involves
some intermediate variables. The last subsection eliminates the intermediate variables to reach the final integral
representation.

\subsection{Linearization}\label{sec:Linearization}
First of all, by a simple computation,  the second equation in \eqref{e1.2} can be rewritten as
\begin{equation}\label{simpleueq}
\partial_t \vec u +\vec u \cdot \nabla \vec u -\frac{\Delta \vec u+\lambda\nabla(\nabla\cdot \vec u)}{\rho} + \rho \nabla \rho=-\frac{\nabla\phi\Delta\phi}{\rho}.
\end{equation}
Setting
$$
\rho=n+1, \quad \vec u=(u,v), \quad \phi=\psi+y
$$
converts \eqref{e1.2} into the following equivalent system of equations for $(n,u,v, \psi)$,
\begin{equation}\label{e2.1}
\begin{cases}
\pp_t n+\pp_x u+\pp_y v=N_0,\\
\pp_t u-\Delta u-\lambda(\partial_{xx} u+\partial_{xy}v)+\partial_x n=N_1,\\
\pp_tv-\Delta v-\lambda(\partial_{xy} u+\partial_{yy}v)+\partial_y n+\Delta \psi=N_2,\\
\pp_t\psi+v=N_3,
\end{cases}
\end{equation}
where
\begin{align}
N_0=&-\pp_x (n u)-\pp_y(n v);\label{N0}\\
N_1=&- (u\pp_x u+v\pp_y u) -\frac{n\Delta u+n\lambda(\partial_{xx} u+\partial_{xy}v)}{\rho} -\frac{\pp_x\psi\Delta\psi}{\rho}- n\pp_x n;\label{N1}\\
N_2=&- (u\pp_x v+v\pp_y v) -\frac{n\Delta v+n\lambda(\partial_{xy} u+\partial_{yy}v)-n\Delta\psi}{\rho}-\frac{\pp_y\psi\Delta\psi}{\rho}- n\pp_y n;\label{N2}\\
N_3=&-u\partial_x\psi-v\pp_y\psi.\label{N3}
\end{align}

\subsection{Diagonalization}\label{sec:Dia}

From \eqref{e2.1}, we can write the system in the form
$$
\pp_t\overrightarrow{U}=A\overrightarrow{U}+\overrightarrow{N},
$$
where
$$
\overrightarrow{U}=\left(\begin{array}{c}
                     n \\
                     u \\
                     v \\
                     \psi
                   \end{array}\right),\,
A=\left(\begin{array}{cccc}
               0 & -\pp_x & -\pp_y & 0 \\
               -\pp_x & \Delta+\lambda\partial_{xx} & \lambda\partial_{xy} & 0 \\
               -\pp_y & \lambda\partial_{xy} & \Delta+\lambda\partial_{yy} & -\Delta \\
               0 & 0 & -1 & 0
             \end{array}\right),\,
\overrightarrow{N}=\left(\begin{array}{c}
                     N_0 \\
                     N_1 \\
                     N_2 \\
                     N_3
                   \end{array}\right).
$$
The aim of this subsection is to diagonalize $A$.  However, the standard approach involving the eigenvalues and
eigenvectors of the symbol of $A$ appears to demand extremely tedious calculations. Our strategy is to diagonalize
through carefully selected differentiations. As mentioned before, we treat the terms involving the parameter
$\lambda$ as nonlinear terms for the brevity of the presentation.
The main  result of this subsection is
\begin{prop} The functions $(n,u,v, \psi)$ obey the following equations,
\begin{align}
\Big((\pp_{tt}-\Delta\pp_t-\Delta)^2-\Delta \pp_{yy}\Big)n=F_0,
\label{e2.10}\\
\Big((\pp_{tt}-\Delta\pp_t-\Delta)^2-\Delta \pp_{yy}\Big)u=F_1,
\label{e2.14}\\
\Big((\pp_{tt}-\Delta\pp_t-\Delta)^2-\Delta \pp_{yy}\Big)v=F_2,
\label{e2.17}\\
\Big((\pp_{tt}-\Delta\pp_t-\Delta)^2-\Delta \pp_{yy}\Big)\psi=F_3.\label{e2.18}
\end{align}
where
\begin{align}
F_0= &(\pp_{tt}-\Delta\pp_t-\Delta)\Pi_0+\Delta \pp_{y}\Pi_3;\label{F0}\\
F_1= &\big(\pp_{tt}-\Delta\pp_t-\Delta-\pp_{yy}\big)\Pi_1+\pp_{xy}\Pi_2;\label{F1}\\
F_2= &\big(\pp_{tt}-\Delta\pp_t-\pp_{xx}\big)\Pi_2+\pp_{xy}\Pi_1;\label{F2}\\
F_3= &(\pp_{tt}-\Delta\pp_t-\Delta)\Pi_3+ \pp_{y}\Pi_0,\label{F3}
\end{align}
and
\begin{align}
\Pi_0=&\pp_t N_0-\Delta N_0-\pp_xN_1-\pp_yN_2-\lambda(\pp_{x}\Delta u+\pp_y \Delta v);\label{Pi0}\\
\Pi_1=&\pp_t N_1-\pp_xN_0+\lambda\pp_t(\pp_{xx} u+\pp_{xy} v);\label{Pi1}\\
\Pi_2=&-\pp_yN_0+\partial_tN_2-\Delta N_3+\lambda\pp_t(\pp_{xy} u+\pp_{yy} v);\label{Pi2}\\
\Pi_3=&-N_2+(\partial_t-\Delta) N_3-\lambda(\pp_{xy} u+\pp_{yy} v).\label{Pi3}
\end{align}
\end{prop}

\begin{proof}
First, we give some reductions.
Taking the time derivative on the first equations of \eqref{e2.1}, then using the second, third and also the first equation of \eqref{e2.1}, we obtain
\begin{align*}
0=&\pp_{tt}n+\pp_x\pp_tu+\pp_y\pp_tv-\pp_tN_0\\
=&\pp_{tt}n+\pp_x\big[\Delta u+\lambda(\partial_{xx} u+\partial_{xy}v)-\partial_x n+N_1\big]\\
&\quad+\pp_y\big[\Delta v+\lambda(\partial_{xy} u+\partial_{yy}v)-\partial_y n-\Delta \psi+N_2\big]-\pp_tN_0\\
=&\pp_{tt}n+\Delta(\pp_x u+\pp_y v)-\Delta n-\pp_y\Delta \psi\\
&\quad+\pp_xN_1+\pp_yN_2-\pp_tN_0+\lambda(\pp_{x}\Delta u+\pp_y \Delta v)\\
=&\pp_{tt}n-\Delta\pp_{t}n-\Delta n-\pp_y\Delta \psi\\
&\quad+(\Delta-\pp_t) N_0+\pp_xN_1+\pp_yN_2+\lambda(\pp_{x}\Delta u+\pp_y \Delta v).
\end{align*}
Therefore, we obtain that
\begin{equation}\label{e2.2}
\pp_{tt} n-\Delta\pp_t n-\Delta n=\Delta\partial_y\psi+(\pp_t-\Delta) N_0-\pp_xN_1-\pp_yN_2-\lambda(\pp_{x}\Delta u+\pp_y \Delta v).
\end{equation}
Now we denote
$$
\Pi_0=(\pp_t-\Delta) N_0-\pp_xN_1-\pp_yN_2-\lambda(\pp_{x}\Delta u+\pp_y \Delta v),
$$
then \eqref{e2.2} becomes
\begin{equation}\label{e2.3}
\pp_{tt} n-\Delta\pp_t n-\Delta n=\Delta\partial_y\psi+\Pi_0.
\end{equation}

Taking the time derivative on the second equation of \eqref{e2.1}, and then using the first equation, we obtain
\begin{align*}
0=&\pp_{tt}u-\Delta \pp_tu+\pp_x \pp_tn-\pp_t N_1-\lambda\pp_t(\pp_{xx}u+\pp_{xy}v)\\
=&\pp_{tt}u-\Delta \pp_tu-\pp_x (\pp_{x}u+\pp_{y}v)+\pp_xN_0-\pp_t N_1-\lambda\pp_t(\pp_{xx}u+\pp_{xy}v).
\end{align*}
Thus, we get
\begin{equation}\label{e2.4}
\pp_{tt} u-\Delta\pp_t u-\pp_{xx} u=\partial_{xy}v-\pp_xN_0+\pp_t N_1+\lambda\pp_t(\pp_{xx} u+\pp_{xy} v).
\end{equation}
Again we denote
$$
\Pi_1=-\pp_xN_0+\pp_t N_1+\lambda\pp_t(\pp_{xx} u+\pp_{xy} v),
$$
then \eqref{e2.4} turns to
\begin{equation}\label{e2.5}
\pp_{tt} u-\Delta\pp_t u-\pp_{xx} u=\partial_{xy}v+\Pi_1.
\end{equation}

Taking the time derivative on the third equation of \eqref{e2.1}, then by the first and the fourth equations we find
\begin{align*}
0=&\pp_{tt}v-\Delta \pp_tv+\pp_y \pp_tn+\Delta\pp_t\psi-\pp_t N_2-\lambda\pp_t(\pp_{xy}u+\pp_{yy}v)\\
=&\pp_{tt}v-\Delta \pp_tv-\pp_y (\pp_{x}u+\pp_{y}v)+\pp_yN_0\\
&\quad-\Delta v+\Delta N_3-\pp_t N_2-\lambda\pp_t(\pp_{xx}u+\pp_{xy}v).
\end{align*}
Therefore, we obtain that
\begin{equation}\label{e2.6}
\pp_{tt} v-\Delta\pp_t v-\Delta v-\pp_{yy} v=\partial_{xy}u-\pp_yN_0+\partial_tN_2-\Delta N_3+\lambda\pp_t(\pp_{xy} u+\pp_{yy} v).
\end{equation}
Again we denote
$$
\Pi_2=-\pp_yN_0+\partial_tN_2-\Delta N_3+\lambda\pp_t(\pp_{xy} u+\pp_{yy} v),
$$
then \eqref{e2.6} turns to
\begin{equation}\label{e2.7}
\pp_{tt} v-\Delta\pp_t v-\Delta v-\pp_{yy} v=\partial_{xy}u+\Pi_2.
\end{equation}

At last,
taking the time derivative on the fourth equation of \eqref{e2.1}, and using the third and also the fourth equations, we find
\begin{align*}
0=&\pp_{tt}\psi+\pp_tv-\pp_t N_3\\
=&\pp_{tt}\psi+\Delta v+\lambda(\pp_{xy}u+\pp_{yy}v)-\pp_yn-\Delta \psi+N_2-\pp_t N_3\\
=&\pp_{tt}\psi-\Delta\pp_t\psi-\Delta\psi-\pp_yn+N_2+\Delta N_3-\pp_t N_3+\lambda(\pp_{xy}u+\pp_{yy}v).
\end{align*}
Thus we have
\begin{equation}\label{e2.8}
\pp_{tt} \psi-\Delta\pp_t \psi-\Delta \psi=\partial_{y}n-N_2+(\partial_t-\Delta) N_3-\lambda(\pp_{xy} u+\pp_{yy} v).
\end{equation}
Again we denote
$$
\Pi_3=-N_2+(\partial_t-\Delta) N_3-\lambda(\pp_{xy} u+\pp_{yy} v),
$$
then \eqref{e2.6} turns to
\begin{equation}\label{e2.9}
\pp_{tt} \psi-\Delta\pp_t \psi-\Delta \psi=\partial_{y}n+\Pi_3.
\end{equation}
In conclusion, we get \eqref{e2.3}, \eqref{e2.5}, \eqref{e2.7}, \eqref{e2.9} at this step. To diagonalize the systems, some further analysis on these equations are needed.

Applying the operator $\pp_{tt}-\Delta\pp_t-\Delta$  on \eqref{e2.3}, and using \eqref{e2.9} we obtain that
\begin{align*}
(\pp_{tt}-\Delta\pp_t-\Delta)^2n=&(\pp_{tt}-\Delta\pp_t-\Delta)\Delta \pp_{y}\psi+(\pp_{tt}-\Delta\pp_t-\Delta)\Pi_0\\
=&\Delta \pp_{yy}n+F_0,
\end{align*}
which is \eqref{e2.10}. Applying the operator $\pp_{tt}-\Delta\pp_t-\Delta-\pp_{yy}$  on \eqref{e2.5}, and then using \eqref{e2.7}, we obtain that
\begin{align*}
&(\pp_{tt}-\Delta\pp_t-\pp_{xx})(\pp_{tt}-\Delta\pp_t-\Delta-\pp_{yy})u\\
& \qquad\qquad =(\pp_{tt}-\Delta\pp_t-\Delta-\pp_{yy})\pp_{xy}v+\big(\pp_{tt}-\Delta\pp_t-\Delta-\pp_{yy}\big)\Pi_1\\
& \qquad\qquad =\pp_{xxyy}u+F_1.
\end{align*}
Using the formula,
\begin{align*}
&(\pp_{tt}-\Delta\pp_t-\pp_{xx})(\pp_{tt}-\Delta\pp_t-\Delta-\pp_{yy})\\
=&(\pp_{tt}-\Delta\pp_t-\Delta+\pp_{yy})(\pp_{tt}-\Delta\pp_t-\Delta-\pp_{yy})\\
=&(\pp_{tt}-\Delta\pp_t-\Delta)^2-\pp_{yyyy},
\end{align*}
we obtain
$$
(\pp_{tt}-\Delta\pp_t-\Delta)^2u-\pp_{yyyy}u=\pp_{xxyy}u+F_1.
$$
So we have \eqref{e2.14}. Similarly, applying the operator $\pp_{tt}-\Delta\pp_t-\pp_{xx}$  on \eqref{e2.7}, and then using \eqref{e2.5}, we obtain that
\begin{align*}
(\pp_{tt}-\Delta\pp_t-\Delta-\pp_{yy})(\pp_{tt}-\Delta\pp_t-\pp_{xx})v
=&(\pp_{tt}-\Delta\pp_t-\pp_{xx})\pp_{xy}u+\big(\pp_{tt}-\Delta\pp_t-\pp_{xx}\big)\Pi_2\\
=&
\pp_{xxyy}v+F_2,
\end{align*}
which gives \eqref{e2.17}. Finally, applying the operator $\pp_{tt}-\Delta\pp_t-\Delta$  on \eqref{e2.9}, and then using \eqref{e2.3}, we obtain that
\begin{align*}
(\pp_{tt}-\Delta\pp_t-\Delta)^2\psi=& (\pp_{tt}-\Delta\pp_t-\Delta)\pp_yn+ (\pp_{tt}-\Delta\pp_t-\Delta)\Pi_3\\
=&
\Delta \pp_{yy}\psi+F_3,
\end{align*}
which is \eqref{e2.18}. This finishes the diagonalization process.
\end{proof}

\subsection{Preliminary integral representation}

Since the linear parts in \eqref{e2.10}, \eqref{e2.14}, \eqref{e2.17} and \eqref{e2.18} all have
the same structure, it suffices to consider the inhomogeneous equation
\begin{align}\label{Inhomoeqs0}
\Big[(\pp_{tt}-\Delta\pp_t-\Delta)^2-\Delta \pp_{yy}\Big]\Phi=F.
\end{align}
and our aim here is to derive an equivalent integral representation of (\ref{Inhomoeqs0}) by inverting the differentiation operator. The main result is presented in Lemma \ref{lem:Formula}. As we shall see, the
integral representation here involves some intermediate variables. Rather than using the eigenvalues and eigenfunctions, we present a basic method to resolve \eqref{Inhomoeqs0}.

\vskip .1in
We now start the derivation process.
First, we set
\begin{align}\label{e2.3.1}
\Big((\pp_{tt}-\Delta\pp_t-\Delta)-\sqrt{\Delta \pp_{yy}}\Big)\Phi=\Psi_1;
\Big((\pp_{tt}-\Delta\pp_t-\Delta)+\sqrt{\Delta \pp_{yy}}\Big)\Phi=\Psi_2.
\end{align}
Then
\begin{equation} \label{e2.3.3}
2\sqrt{\Delta \pp_{yy}}\Phi=\Psi_2-\Psi_1.
\end{equation}
Moreover, one may find from \eqref{Inhomoeqs0} that
\begin{equation} \label{e2.3.2}
\begin{cases}
\Big(\pp_{tt}-\Delta\pp_t-\Delta)+\sqrt{\Delta \pp_{yy}}\Big)\Psi_1=F;\\
\Big(\pp_{tt}-\Delta\pp_t-\Delta)-\sqrt{\Delta \pp_{yy}}\Big)\Psi_2=F.
\end{cases}
\end{equation}
According to the first equation in \eqref{e2.3.2}, we set
\begin{align}
\Big(\pp_{t}-\frac12\Delta-\sqrt{\frac14\Delta^2+\Delta-\sqrt{\Delta \pp_{yy}}}\Big)\Psi_1=\Phi_1;\label{e2.3.4}\\
\Big(\pp_{t}-\frac12\Delta+\sqrt{\frac14\Delta^2+\Delta-\sqrt{\Delta \pp_{yy}}}\Big)\Psi_1=\Phi_2.\label{e2.3.5}
\end{align}
Then
\begin{equation} \label{e2.3.7.1}
2\sqrt{\frac14\Delta^2+\Delta-\sqrt{\Delta \pp_{yy}}}\Psi_1= \Phi_2-\Phi_1.
\end{equation}
Moreover, from the first equation in \eqref{e2.3.2}, we have
\begin{equation}
\begin{cases}
\Big(\pp_{t}-\frac12\Delta+\sqrt{\frac14\Delta^2+\Delta-\sqrt{\Delta \pp_{yy}}}\Big)\Phi_1=F;\\
\Big(\pp_{t}-\frac12\Delta-\sqrt{\frac14\Delta^2+\Delta-\sqrt{\Delta \pp_{yy}}}\Big)\Phi_2=F.\label{e2.3.7}
\end{cases}
\end{equation}
Now according to the second equation in \eqref{e2.3.2}, we set
\begin{align}
\Big(\pp_{t}-\frac12\Delta-\sqrt{\frac14\Delta^2+\Delta+\sqrt{\Delta \pp_{yy}}}\Big)\Psi_2=\Phi_3;\label{e2.3.8}\\
\Big(\pp_{t}-\frac12\Delta+\sqrt{\frac14\Delta^2+\Delta+\sqrt{\Delta \pp_{yy}}}\Big)\Psi_2=\Phi_4.\label{e2.3.9}
\end{align}
Then
\begin{equation} \label{e2.3.11}
2\sqrt{\frac14\Delta^2+\Delta+\sqrt{\Delta \pp_{yy}}}\Psi_2= \Phi_4-\Phi_3.
\end{equation}
Moreover, it holds that
\begin{equation}
\begin{cases}
\Big(\pp_{t}-\frac12\Delta+\sqrt{\frac14\Delta^2+\Delta+\sqrt{\Delta \pp_{yy}}}\Big)\Phi_3=F;\\
\Big(\pp_{t}-\frac12\Delta-\sqrt{\frac14\Delta^2+\Delta+\sqrt{\Delta \pp_{yy}}}\Big)\Phi_4=F.\label{e2.3.10}
\end{cases}
\end{equation}
Now by \eqref{e2.3.7}, \eqref{e2.3.10} and the common Duhamel's Principle, we have
\begin{align*}
\Phi_1=e^{(\frac12\Delta-\sqrt{\frac14\Delta^2+\Delta-\sqrt{\Delta\partial_{yy}}})t}\Phi_{1,0}+\int_0^t e^{(\frac12\Delta-\sqrt{\frac14\Delta^2+\Delta-\sqrt{\Delta\partial_{yy}}})(t-s)} F(s)\,ds;\\
\Phi_2=e^{(\frac12\Delta+\sqrt{\frac14\Delta^2+\Delta-\sqrt{\Delta\partial_{yy}}})t}\Phi_{2,0}+\int_0^t e^{(\frac12\Delta+\sqrt{\frac14\Delta^2+\Delta-\sqrt{\Delta\partial_{yy}}})(t-s)} F(s)\,ds;\\
\Phi_3=e^{(\frac12\Delta-\sqrt{\frac14\Delta^2+\Delta+\sqrt{\Delta\partial_{yy}}})t}\Phi_{3,0}+\int_0^t e^{(\frac12\Delta-\sqrt{\frac14\Delta^2+\Delta+\sqrt{\Delta\partial_{yy}}})(t-s)} F(s)\,ds;\\
\Phi_4=e^{(\frac12\Delta+\sqrt{\frac14\Delta^2+\Delta+\sqrt{\Delta\partial_{yy}}})t}\Phi_{4,0}+\int_0^t e^{(\frac12\Delta+\sqrt{\frac14\Delta^2+\Delta+\sqrt{\Delta\partial_{yy}}})(t-s)} F(s)\,ds.
\end{align*}
Furthermore, by \eqref{e2.3.3}, \eqref{e2.3.7.1} and \eqref{e2.3.11}, we find
\begin{align*}
\Phi=\frac{1}{2\sqrt{\Delta\partial_{yy}}}\Bigg[\frac{\Phi_{4}-\Phi_{3}}{2\sqrt{\frac14\Delta^2+\Delta+\sqrt{\Delta\partial_{yy}}}}
-\frac{\Phi_{2}-\Phi_{1}}{2\sqrt{\frac14\Delta^2+\Delta-\sqrt{\Delta\partial_{yy}}}}\Bigg].
\end{align*}
Therefore, we thus have obtained the following lemma.
\begin{lemma}\label{lem:Formula}
 Let $\Phi$ be the solution of \eqref{Inhomoeqs0} with the initial datum of $\Phi_j: \Phi_j(0)=\Phi_{j,0}, j=1,2,3,4$.
Then it obeys the following formula,
\begin{align}\label{eq:Formula}
\Phi=L(\Phi_{1,0},\Phi_{2,0},\Phi_{3,0},\Phi_{4,0})+\int_0^t K(t-s,\partial_x,\partial_y) F(s)\,ds,
\end{align}
where the linear part
\begin{align}
&L(\Phi_{1,0},\Phi_{2,0},\Phi_{3,0},\Phi_{4,0})\\
=&\frac{1}{2\sqrt{\Delta\partial_{yy}}}\Bigg[\frac{1}{2\sqrt{\frac14\Delta^2+\Delta+\sqrt{\Delta\partial_{yy}}}}
\Big(e^{(\frac12\Delta+\sqrt{\frac14\Delta^2+\Delta+\sqrt{\Delta\partial_{yy}}})t}\Phi_{4,0}
-e^{(\frac12\Delta-\sqrt{\frac14\Delta^2+\Delta+\sqrt{\Delta\partial_{yy}}})t}\Phi_{3,0}\Big)\notag\\
&\,\,-\frac{1}{2\sqrt{\frac14\Delta^2+\Delta-\sqrt{\Delta\partial_{yy}}}}
\Big(e^{(\frac12\Delta+\sqrt{\frac14\Delta^2+\Delta-\sqrt{\Delta\partial_{yy}}})t}\Phi_{2,0}
-e^{(\frac12\Delta-\sqrt{\frac14\Delta^2+\Delta-\sqrt{\Delta\partial_{yy}}})t}\Phi_{1,0}\Big)\Bigg] ,\label{Def:LPhi}
\end{align}
and the operator
\begin{align}
&K(t,\partial_x,\partial_y)\notag\\
=&\frac{1}{2\sqrt{\Delta\partial_{yy}}}\Bigg[\frac{1}{2\sqrt{\frac14\Delta^2+\Delta+\sqrt{\Delta\partial_{yy}}}}
\Big(e^{(\frac12\Delta+\sqrt{\frac14\Delta^2+\Delta+\sqrt{\Delta\partial_{yy}}})t}
-e^{(\frac12\Delta-\sqrt{\frac14\Delta^2+\Delta+\sqrt{\Delta\partial_{yy}}})t}\Big)\notag\\
&\,\,-\frac{1}{2\sqrt{\frac14\Delta^2+\Delta-\sqrt{\Delta\partial_{yy}}}}
\Big(e^{(\frac12\Delta+\sqrt{\frac14\Delta^2+\Delta-\sqrt{\Delta\partial_{yy}}})t}
-e^{(\frac12\Delta-\sqrt{\frac14\Delta^2+\Delta-\sqrt{\Delta\partial_{yy}}})t}\Big)\Bigg]. \label{Kop}
\end{align}
\end{lemma}

\vskip .1in
\subsection{Final integral representation}\label{sec:LF} The previous integral representation involves some
intermediate variables $\Phi_{1,0},\Phi_{2,0},\Phi_{3,0}$ and $\Phi_{4,0}$. This subsection eliminates these
variables and replaces them by the initial data $\Phi(0)$. In addition, we also eliminate the intermediate
variables $\Pi_0, \Pi_1,\Pi_2$ and $\Pi_3$ (defined in (\ref{Pi0}--\ref{Pi3}))  from $F=(F_0, F_1, F_2, F_3)$ and rewrite $F$ directly in terms of $N_0, N_1, N_2$ and $N_3$.

\begin{lemma} \label{fiii}
The operator $L$ defined in (\ref{Def:LPhi}) can be written as
\begin{align}
L(\Phi_{1,0},\Phi_{2,0},&\Phi_{3,0},\Phi_{4,0})
=
K(t)\big[(\pp_{tt}-\Delta\pp_t-\Delta)\pp_{t}\Phi(0)\big]\label{L-1}\\
&+(\partial_{tt}-\Delta\partial_t-\Delta)K(t)\big[\partial_t\Phi(0)\big]\label{L-2}\\
&
+\big(\partial_t-\Delta\big) K(t)\big[(\pp_{tt}-\Delta\pp_t-\Delta)\Phi(0)\big]\label{L-3}\\
&-\frac12\Delta\sqrt{\Delta\partial_{yy}} K(t)\big[\Phi(0)\big]+K_1(t)[\Phi(0)]\label{L-5},
\end{align}
where $K_1$ is given by
\begin{align}
K_1=&\frac{1}{4}\Big(
e^{(\frac12\Delta+\sqrt{\frac14\Delta^2+\Delta+\sqrt{\Delta\partial_{yy}}})t}
+e^{(\frac12\Delta-\sqrt{\frac14\Delta^2+\Delta+\sqrt{\Delta\partial_{yy}}})t}\notag\\
&\qquad
+e^{(\frac12\Delta+\sqrt{\frac14\Delta^2+\Delta-\sqrt{\Delta\partial_{yy}}})t}
+e^{(\frac12\Delta-\sqrt{\frac14\Delta^2+\Delta-\sqrt{\Delta\partial_{yy}}})t}\Big)\label{K1op}
\end{align}
or, with $A=\sqrt{\xi^2+\eta^2}$,
\begin{align}
\widehat{K_1}(t,\xi,\eta)=&\frac14\Big(e^{(-\frac12A^2+\sqrt{\frac14A^4-A^2+A|\eta|})t}
+e^{(-\frac12A^2-\sqrt{\frac14A^4-A^2+A|\eta|})t}\notag\\
&\qquad+e^{(-\frac12A^2+\sqrt{\frac14A^4-A^2-A|\eta|})t}
+e^{(-\frac12A^2-\sqrt{\frac14A^4-A^2-A|\eta|})t}\Big). \label{K1}
\end{align}
\end{lemma}
\begin{proof}
By \eqref{e2.3.9} and \eqref{e2.3.1}, we have
\begin{align}
\Phi_{4,0}=&\Big(\pp_{t}-\frac12\Delta+\sqrt{\frac14\Delta^2+\Delta+\sqrt{\Delta \pp_{yy}}}\Big)\Psi_2(0)\notag\\
=&\pp_{t}\Big((\partial_{tt}-\Delta\partial_t-\Delta)+\sqrt{\Delta\partial_{yy}}\Big)\Phi(0)\notag\\
&-\Big(\frac12\Delta-\sqrt{\frac14\Delta^2+\Delta+\sqrt{\Delta \pp_{yy}}}\Big)\Big((\partial_{tt}-\Delta\partial_t-\Delta)+\sqrt{\Delta\partial_{yy}}\Big)\Phi(0);\label{Phi40}
\end{align}
similarly, by \eqref{e2.3.8}, \eqref{e2.3.5}, \eqref{e2.3.4} and \eqref{e2.3.1},
\begin{align}
\Phi_{3,0}=&\pp_{t}\Big((\partial_{tt}-\Delta\partial_t-\Delta)+\sqrt{\Delta\partial_{yy}}\Big)\Phi(0)\notag\\
&-\Big(\frac12\Delta+\sqrt{\frac14\Delta^2+\Delta+\sqrt{\Delta \pp_{yy}}}\Big)\Big((\partial_{tt}-\Delta\partial_t-\Delta)+\sqrt{\Delta\partial_{yy}}\Big)\Phi(0);\label{Phi30}\\
\Phi_{2,0}=&\pp_{t}\Big((\partial_{tt}-\Delta\partial_t-\Delta)-\sqrt{\Delta\partial_{yy}}\Big)\Phi(0)\notag\\
&-\Big(\frac12\Delta-\sqrt{\frac14\Delta^2+\Delta+\sqrt{\Delta \pp_{yy}}}\Big)\Big((\partial_{tt}-\Delta\partial_t-\Delta)-\sqrt{\Delta\partial_{yy}}\Big)\Phi(0);\label{Phi20}\\
\Phi_{1,0}=&\pp_{t}\Big((\partial_{tt}-\Delta\partial_t-\Delta)-\sqrt{\Delta\partial_{yy}}\Big)\Phi(0)\notag\\
&-\Big(\frac12\Delta+\sqrt{\frac14\Delta^2+\Delta+\sqrt{\Delta \pp_{yy}}}\Big)\Big((\partial_{tt}-\Delta\partial_t-\Delta)-\sqrt{\Delta\partial_{yy}}\Big)\Phi(0).\label{Phi10}
\end{align}
Inserting \eqref{Phi40}--\eqref{Phi10} into \eqref{Def:LPhi}, we have
\begin{align}
&L(\Phi_{1,0},\Phi_{2,0},\Phi_{3,0},\Phi_{4,0})\notag\\
=
&\frac{1}{2\sqrt{\Delta\partial_{yy}}}\Big[\frac{1}{2\sqrt{\frac14\Delta^2+\Delta+\sqrt{\Delta\partial_{yy}}}}
\Big(e^{(\frac12\Delta+\sqrt{\frac14\Delta^2+\Delta+\sqrt{\Delta\partial_{yy}}})t}
-e^{(\frac12\Delta-\sqrt{\frac14\Delta^2+\Delta+\sqrt{\Delta\partial_{yy}}})t}\Big)\notag\\
&\qquad\qquad\quad
\cdot\Big(\partial_t(\partial_{tt}-\Delta\partial_t-\Delta)+\partial_t\sqrt{\Delta\partial_{yy}}\Big)\Phi(0)\notag\\
&\qquad-\frac{1}{2\sqrt{\frac14\Delta^2+\Delta-\sqrt{\Delta\partial_{yy}}}}
\Big(e^{(\frac12\Delta+\sqrt{\frac14\Delta^2+\Delta-\sqrt{\Delta\partial_{yy}}})t}
-e^{(\frac12\Delta-\sqrt{\frac14\Delta^2+\Delta-\sqrt{\Delta\partial_{yy}}})t}\Big)\notag\\
&\qquad\qquad\quad
\cdot\Big(\partial_t(\partial_{tt}-\Delta\partial_t-\Delta)-\partial_t\sqrt{\Delta\partial_{yy}}\Big)\Phi(0)\Big]\label{eq:L-1}\\
&-\frac{\Delta}{4\sqrt{\Delta\partial_{yy}}}\Big[\frac{1}{2\sqrt{\frac14\Delta^2+\Delta+\sqrt{\Delta\partial_{yy}}}}
\Big(e^{(\frac12\Delta+\sqrt{\frac14\Delta^2+\Delta+\sqrt{\Delta\partial_{yy}}})t}
-e^{(\frac12\Delta-\sqrt{\frac14\Delta^2+\Delta+\sqrt{\Delta\partial_{yy}}})t}\Big)\notag\\
&\qquad\qquad\quad
\cdot\Big((\partial_{tt}-\Delta\partial_t-\Delta)+\sqrt{\Delta\partial_{yy}}\Big)\Phi(0)\notag\\
&\qquad-\frac{1}{2\sqrt{\frac14\Delta^2+\Delta-\sqrt{\Delta\partial_{yy}}}}
\Big(e^{(\frac12\Delta+\sqrt{\frac14\Delta^2+\Delta-\sqrt{\Delta\partial_{yy}}})t}
-e^{(\frac12\Delta-\sqrt{\frac14\Delta^2+\Delta-\sqrt{\Delta\partial_{yy}}})t}\Big)\notag\\
&\qquad\qquad\quad
\cdot\Big((\partial_{tt}-\Delta\partial_t-\Delta)-\sqrt{\Delta\partial_{yy}}\Big)\Phi(0)\Big]\label{eq:L-2}\\
&-\frac{1}{4\sqrt{\Delta\partial_{yy}}}\Big[
\Big(e^{(\frac12\Delta+\sqrt{\frac14\Delta^2+\Delta+\sqrt{\Delta\partial_{yy}}})t}
+e^{(\frac12\Delta-\sqrt{\frac14\Delta^2+\Delta+\sqrt{\Delta\partial_{yy}}})t}\Big)\notag\\
&\qquad\qquad\quad
\cdot\Big((\partial_{tt}-\Delta\partial_t-\Delta)+\sqrt{\Delta\partial_{yy}}\Big)\Phi(0)\notag\\
&\qquad-
\Big(e^{(\frac12\Delta+\sqrt{\frac14\Delta^2+\Delta-\sqrt{\Delta\partial_{yy}}})t}
+e^{(\frac12\Delta-\sqrt{\frac14\Delta^2+\Delta-\sqrt{\Delta\partial_{yy}}})t}\Big)\notag\\
&\qquad\qquad\quad
\cdot\Big((\partial_{tt}-\Delta\partial_t-\Delta)-\sqrt{\Delta\partial_{yy}}\Big)\Phi(0)\Big].\label{eq:L-3}
\end{align}
Further, by using the definition of $K$ in \eqref{eq:K} and the formula \eqref{K-tt2t2}, we have
\begin{align*}
&\eqref{eq:L-1}=K(t)\big[(\pp_{tt}-\Delta\pp_t-\Delta)\pp_{t}\Phi(0)\big]\\
&\quad+\Big[\frac{1}{4\sqrt{\frac14\Delta^2+\Delta+\sqrt{\Delta\partial_{yy}}}}
\Big(e^{(\frac12\Delta+\sqrt{\frac14\Delta^2+\Delta+\sqrt{\Delta\partial_{yy}}})t}-e^{(\frac12\Delta-\sqrt{\frac14\Delta^2+\Delta+\sqrt{\Delta\partial_{yy}}})t}\Big)\notag\\
&\quad+\frac{1}{4\sqrt{\frac14\Delta^2+\Delta-\sqrt{\Delta\partial_{yy}}}}
\Big(e^{(\frac12\Delta+\sqrt{\frac14\Delta^2+\Delta-\sqrt{\Delta\partial_{yy}}})t}-e^{(\frac12\Delta-\sqrt{\frac14\Delta^2+\Delta-\sqrt{\Delta\partial_{yy}}})t}\Big)\Big]\partial_t\Phi(0)\\
&\,\,=K(t)\big[(\pp_{tt}-\Delta\pp_t-\Delta)\pp_{t}\Phi(0)\big]+(\partial_{tt}-\Delta\partial_t-\Delta)K(t)\big[\partial_t\Phi(0)\big];
\end{align*}
similarly,  we find
\begin{align*}
\eqref{eq:L-2}
=&
-\frac12\Delta K(t)\big[(\pp_{tt}-\Delta\pp_t-\Delta)\Phi(0)\big]
-\frac12\Delta\sqrt{\Delta\partial_{yy}} K(t)\big[\Phi(0)\big];
\end{align*}
and by the definition of $\partial_t K$ in \eqref{Kt-2} and $K_1$ in \eqref{K1}, we get
\begin{align*}
\eqref{eq:L-3}
=&-\frac{1}{4\sqrt{\Delta\partial_{yy}}}\Big[
\Big(e^{(\frac12\Delta+\sqrt{\frac14\Delta^2+\Delta+\sqrt{\Delta\partial_{yy}}})t}
+e^{(\frac12\Delta-\sqrt{\frac14\Delta^2+\Delta+\sqrt{\Delta\partial_{yy}}})t}\Big)\\
&\qquad-
\Big(e^{(\frac12\Delta+\sqrt{\frac14\Delta^2+\Delta-\sqrt{\Delta\partial_{yy}}})t}
+e^{(\frac12\Delta-\sqrt{\frac14\Delta^2+\Delta-\sqrt{\Delta\partial_{yy}}})t}\Big)\Big](\partial_{tt}-\Delta\partial_t-\Delta)|\Phi(0)\\
&+\frac{1}{4}\Big[
e^{(\frac12\Delta+\sqrt{\frac14\Delta^2+\Delta+\sqrt{\Delta\partial_{yy}}})t}
+e^{(\frac12\Delta-\sqrt{\frac14\Delta^2+\Delta+\sqrt{\Delta\partial_{yy}}})t}\\
&\qquad
+e^{(\frac12\Delta+\sqrt{\frac14\Delta^2+\Delta-\sqrt{\Delta\partial_{yy}}})t}
+e^{(\frac12\Delta-\sqrt{\frac14\Delta^2+\Delta-\sqrt{\Delta\partial_{yy}}})t}\Big]\Phi(0)\\
=&
-\frac12\Delta K(t)\big[(\pp_{tt}-\Delta\pp_t-\Delta)\Phi(0)\big]
+\partial_t K(t)\big[(\pp_{tt}-\Delta\pp_t-\Delta)\Phi(0)\big]+K_1(t)[\Phi(0)].
\end{align*}
Therefore, collecting the estimates above, we get \eqref{L-5}.
\end{proof}

\vskip .1in
We now derive the final form of $F$ by eliminating the intermediate
variables $\Pi_0, \Pi_1,\Pi_2$ and $\Pi_3$  (defined in (\ref{Pi0}--\ref{Pi3})).
First, we consider $F_0$. By \eqref{Pi0}--\eqref{Pi3}, we have
\begin{align}
F_0=& (\pp_{tt}-\Delta\pp_t-\Delta)\Pi_0+\Delta \partial_y \Pi_3\notag\\
=& (\pp_{tt}-\Delta\pp_t-\Delta)(\pp_t N_0-\Delta N_0-\pp_xN_1-\pp_yN_2)+\Delta \partial_y(-N_2+(\partial_t-\Delta) N_3)\notag\\
&\quad +\lambda\Big(\partial_{xxx}\Delta u+\partial_{xxy}\Delta v+\partial_{x}\partial_{t}\Delta^2 u+\partial_{y}\partial_{t}\Delta^2 v-\partial_{x}\partial_{tt}\Delta u-\partial_{y}\partial_{tt}\Delta v\Big)\notag\\
=& \partial_t(\pp_{tt}-\Delta\pp_t-\Delta)N_0-\Delta(\pp_{tt}-\Delta\pp_t-\Delta) N_0-\partial_x(\pp_{tt}-\Delta\pp_t-\Delta)N_1
-\partial_y(\partial_{tt}-\Delta\partial_t)N_2\notag\\
&\quad +\Delta\partial_y (\partial_t-\Delta)N_3-\lambda(\pp_{tt}-\Delta\pp_t-\pp_{xx})(\pp_{x}\Delta u+\pp_y \Delta v).\label{F0-r}
\end{align}

Second, we consider $F_1$.
\begin{align}
F_1= &\big(\pp_{tt}-\Delta\pp_t-\Delta-\pp_{yy}\big)\Pi_1+\pp_{xy}\Pi_2 \notag\\
=& (\pp_{tt}-\Delta\pp_t-\Delta-\partial_{yy})(\pp_t N_1-\pp_xN_0)+ \partial_{xy}(-\partial_yN_0+\partial_t N_2-\Delta N_3)\notag\\
&\quad +\lambda(\pp_{tt}-\Delta\pp_t-\Delta)\partial_t\partial_x(\partial_xu+\partial_yv)\notag\\
=& -\partial_x(\pp_{tt}-\Delta\pp_t-\Delta)N_0+\partial_t(\pp_{tt}-\Delta\pp_t-\Delta-\partial_{yy})N_1+ \partial_{xy}\partial_t N_2 -\partial_{xy}\Delta N_3\notag\\
&\quad +\lambda(\pp_{tt}-\Delta\pp_t-\Delta)\partial_t\partial_x(\partial_xu+\partial_yv).\label{F1-r}
\end{align}

Third,  we consider $F_2$.
\begin{align}
F_2= &\big(\pp_{tt}-\Delta\pp_t-\pp_{xx}\big)\Pi_2+\pp_{xy}\Pi_1 \notag\\
=& (\pp_{tt}-\Delta\pp_t-\partial_{xx})(-\pp_yN_0+\partial_tN_2-\Delta N_3)+ \partial_{xy}(\pp_t N_1-\pp_xN_0)\notag\\
&\quad +\lambda(\pp_{tt}-\Delta\pp_t)\partial_t\partial_y(\partial_xu+\partial_yv)\notag\\
=& -\partial_y(\pp_{tt}-\Delta\pp_t)N_0+\partial_t\partial_{xy}N_1+\partial_t(\pp_{tt}-\Delta\pp_t-\partial_{xx})N_2 -\Delta (\partial_{tt}-\Delta\partial_t-\partial_{xx}) N_3\notag\\
&\quad +\lambda(\pp_{tt}-\Delta\pp_{t})\partial_{t}(\partial_{xy}u+\partial_{yy}v).\label{F2-r}
\end{align}

At last,  we consider $F_3$.
\begin{align}
F_3= &(\pp_{tt}-\Delta\pp_t-\Delta)\Pi_3+ \pp_{y}\Pi_0 \notag\\
=& (\pp_{tt}-\Delta\pp_t-\Delta)(-N_2+(\partial_t-\Delta) N_3)+ \partial_{y}(\pp_t N_0-\Delta N_0-\pp_xN_1-\pp_yN_2)\notag\\
&\quad -\lambda(\pp_{tt}-\Delta\pp_t)\partial_y(\partial_xu+\partial_yv)\notag\\
=& \partial_y\partial_tN_0-\partial_{y}\Delta N_0-\partial_{xy}N_1-(\pp_{tt}-\Delta\pp_t-\partial_{xx})N_2 +(\partial_{t}-\Delta) (\partial_{tt}-\Delta\partial_t-\Delta) N_3\notag\\
&\quad -\lambda(\pp_{tt}-\Delta\pp_t)\partial_y(\partial_xu+\partial_yv).\label{F3-r}
\end{align}

\vskip .4in
\section{Fourier analysis on the linear flow}\label{KernelProperty}

This section provides pointwise as well as $L^p$-estimates for $\widehat K(t,\xi,\eta)$ (defined in
(\ref{Kop})) and the Fourier transforms of various derivatives of $K$. The pointwise estimates are
stated in Lemma \ref{Kpointwise} while the $L^p$-estimates are stated in Lemma \ref{lem:Kn1}
through Lemma \ref{lem:Kn9}. The last lemma of this section provides the estimates for $\widehat K_1$
defined in (\ref{K1}). Before presenting these estimates, we provide several notations and basic tool inequalities.

\vskip .1in
The definition of the Fourier transform is standard, namely
$$
\widehat {f}(\zeta)=\int_{\R^d} e^{{-}ix\cdot\zeta }{f}(x)\,dx,\quad \mbox{ for any }\quad \zeta\in \R^d.
$$
In the 2D case, $\mathcal F_{\xi}f$ and $\mathcal F_\eta f$ are used to denote the
corresponding Fourier transforms with respect to the $x$ and $y$ variables, respectively.
Furthermore,  for each number $N > 0$, we define the Fourier multipliers
\begin{align*}
\widehat{P_{\leq N} f}(\zeta) &:= \chi_{\leq N}(\zeta) \hat f(\zeta),\\
\widehat{P_{> N} f}(\zeta) &:= \chi_{> N}(\zeta) \hat f(\zeta),\\
\widehat{P_N f}(\zeta) &:= (\chi_{\leq N} - \chi_{\leq N/2})(\zeta) \hat
f(\zeta)
\end{align*}
and similarly $P_{<N}$ and $P_{\geq N}$. Here, we use  $\chi$  to denote a smooth bump function such that
\begin{align*}
\begin{cases}
\chi(x)=1, \ & |x|\le 1,\\
\chi(x)=0,    \ &|x|\ge 1+10^{-4},
\end{cases}
\end{align*}
and denote $\chi_R=\chi(\cdot/R)$.
 We also define, for $0<N_1 < N_2$
$$ P_{N_1 < \cdot \leq N_2} := P_{\leq N_2} - P_{\leq N_1}.$$
We will use the following Bernstein's inequality.
\begin{lemma}[Bernstein's inequality]\label{Bernstein}
 For $1 \leq p \leq q \leq \infty$ and $M>0$,
\begin{align*}
\bigl\| |\nabla|^{\pm s} P_M f\bigr\|_{L^p(\R^d)} &\sim M^{\pm s} \| P_M f \|_{L^p(\R^d)},\\
\|P_{\leq M} f\|_{L^q(\R^d)} &\lesssim M^{\frac{d}{p}-\frac{d}{q}} \|P_{\leq M} f\|_{L^p(\R^d)},\\
\|P_M f\|_{L^q(\R^d)} &\lesssim M^{\frac{d}{p}-\frac{d}{q}} \| P_M f\|_{L^p(\R^d)}.
\end{align*}
\end{lemma}
We define differential operator $P(D)$ as
$$
\widehat{P(D)f}(\xi,\eta)=P(\xi,\eta) \widehat{f}(\xi,\eta),\quad \mbox{ for any } (\xi,\eta)\in\R^2.
$$
Then we have the following generalized Young's inequality.
\begin{lemma}[Generalized Young's inequality]\label{Young}
Let $1\le r_1,r_2\le 2 \le p \le\infty$, $q_1,q_2\ge p'$ be the numbers satisfying
$\frac1p=\frac1{r_1}-\frac1{q_1}=\frac1{r_2}-\frac1{q_2}=1-\frac1{p'}$, then for any two-variable function $f(x,y)\in L^{r_1}_{x}L^{r_2}_{y}(\R\times \R)$,
\begin{align}\label{est:Young1}
\bigl\|P(D)f\bigr\|_{L^p_{xy}}\lesssim \|P(\xi,\eta)\|_{L^{q_1}_{\xi}L^{q_2}_{\eta}}\|f\|_{L^{r_1}_{x}L^{r_2}_{y}}.
\end{align}
In particular, let $1\le r\le 2 \le p \le\infty$, $q\ge p'$ be the numbers satisfying $\frac1p+\frac1q=\frac1r$, then for any $f\in L^r(\R^2)$,
\begin{align}\label{est:Young2}
\bigl\|P(D)f\bigr\|_{L^p_{xy}}\lesssim \|P(\xi,\eta)\|_{L^q_{\xi\eta}}\|f\|_{L^r_{xy}}.
\end{align}
\end{lemma}
\begin{proof} Since $p\ge 2$ and $r'\ge 2$, by Young's and H\"older's inequalities,  we have
\begin{align*}
\bigl\|P(D)f\bigr\|_{L^p_{xy}}\lesssim \bigl\|P(\xi,\eta) \hat{f}\bigr\|_{L^{p'}_{\xi\eta}}
\lesssim \bigl\|P(\xi,\eta)\bigr\|_{L^{q_1}_{\xi}L^{q_2}_{\eta}}\bigl\| \hat{f}\bigr\|_{L^{r_1'}_\xi L^{r_2'}_\eta}.
\end{align*}
Since $r_1'\ge 2,r_2'\ge 2$ and $r_1'\ge r_2$, by Fubini's, and Young's inequalities again, we further have
\begin{align*}
\bigl\|P(D)f\bigr\|_{L^p_{xy}}
\lesssim & \bigl\|P(\xi,\eta)\bigr\|_{L^{q_1}_{\xi}L^{q_2}_{\eta}}\Big\|\bigl\| \mathcal F_{\xi}{f}(\xi,y)\bigr\|_{L^{r_2}_y}\Big\|_{L^{r_1'}_\xi}\\
\lesssim &\bigl\|P(\xi,\eta)\bigr\|_{L^{q_1}_{\xi}L^{q_2}_{\eta}}\Big\|\bigl\| \mathcal F_{\xi}{f}(\xi,y)\bigr\|_{L^{r_1'}_\xi}\Big\|_{L^{r_2}_y}\\
\lesssim &\bigl\|P(\xi,\eta)\bigr\|_{L^{q_1}_{\xi}L^{q_2}_{\eta}}\bigl\|f(x,y)\bigr\|_{L^{r_1}_xL^{r_2}_y}.
\end{align*}
This proves \eqref{est:Young1}. In particular, letting $q_1=q_2=q,r_1=r_2=r$ yields \eqref{est:Young2}.
\end{proof}

\vskip .15in
The rest of this section is divided into two subsections with the first devoted
to the pointwise estimates and the second to the $L^p$-estimates. To simplify
the notation, we again write $A=\sqrt{\xi^2+\eta^2}$. Then,
by (\ref{Kop}),
\begin{align}
&\widehat K(t,\xi,\eta)\notag\\
=&\frac{1}{2A|\eta|}\Bigg[\frac{1}{2\sqrt{\frac14A^4-A^2+A|\eta|}}
\Big(e^{(-\frac12A^2+\sqrt{\frac14A^4-A^2+A|\eta|})t}
-e^{(-\frac12A^2-\sqrt{\frac14A^4-A^2+A|\eta|})t}\Big)\notag\\
&\,\,-\frac{1}{2\sqrt{\frac14A^4-A^2-A|\eta|}}
\Big(e^{(-\frac12A^2+\sqrt{\frac14A^4-A^2-A|\eta|})t}
-e^{(-\frac12A^2-\sqrt{\frac14A^4-A^2-A|\eta|})t}\Big)\Bigg].\label{eq:K}
\end{align}

\subsection{Pointwise estimates}
The main results in this subsection are stated as follows.
\begin{prop} \label{Kpointwise}
Let $\widehat K$ be defined in \eqref{eq:K}. Then there exists a constant $c>0$ such that the following estimates hold:
\begin{align}
1.\quad &  |\widehat K(t,\xi,\eta)|
\lesssim\chi_{A\ge 1}\frac1{A^4} e^{-ct}+\chi_{A\le 1}\min\Big\{\frac{1}{A|\xi||\eta|},\frac{1}{A^4}\Big\}e^{-\frac14 A^2t}
+\frac1{A^{4}}\chi_{|\xi|\lesssim A^2 }e^{-\frac{\xi^2}{2A^2}t}.\label{est:K}\\
2.\quad & |\pp_t\widehat K(t,\xi,\eta)|
\lesssim \chi_{A\ge 1} \frac1{A^4} e^{-ct}+\chi_{A\le 1}\min\Big\{\frac{1}{A^3},\frac{1}{A|\eta|}\Big\}e^{-\frac14 A^2t}+\frac{\xi^{2}}{A^{6}}\chi_{|\xi|\lesssim A^2 }e^{-\frac{\xi^2}{2A^2}t}.\label{est:Kt}
\\
3.\quad & |\pp_{tt}\widehat K(t,\xi,\eta)|
\lesssim\chi_{A\ge 1} \frac1{A^4} e^{-ct}+\chi_{A\le 1}\min\Big\{\frac{1}{A^2},\frac{1}{|\eta|}\Big\}e^{-\frac14 A^2t}+\frac{\xi^{4}}{A^{8}}\chi_{|\xi|\lesssim A^2 }e^{-\frac{\xi^2}{2A^2}t}.\label{est:Ktt}
\\
4.\quad & |\big(\pp_{tt}+A^2\pp_t+A^2\big)\widehat K(t,\xi,\eta)|
\lesssim
\chi_{A\ge 1} \frac1{A^2} e^{-ct}+\chi_{A\le 1}\min\Big\{\frac{1}{|\xi|},\frac{1}{A^2}\Big\}e^{-\frac14 A^2t}\notag\\
&\quad+\frac1{A^2}\chi_{|\xi|\lesssim A^2 }e^{-\frac{\xi^2}{2A^2}t}.\label{est:K-tt2t2}
\\
5.\quad &  |\pp_t\big(\pp_{tt}+A^2\pp_t+A^2\big)\widehat K(t,\xi,\eta)|
\lesssim \chi_{A\ge 1} \frac1{A^2} e^{-ct}+\chi_{A\le 1}e^{-\frac14 A^2t}
+\chi_{|\xi|\lesssim A^2 }\frac{\xi^2}{A^4}e^{-\frac{\xi^2}{2A^2}t}.\label{est:K-ttt2t2}
\\
6.\quad &  |\pp_{tt}\big(\pp_{tt}+A^2\pp_t+A^2\big)\widehat K(t,\xi,\eta)|
\lesssim \chi_{A\ge 1} \frac1{A^2} e^{-ct}+\chi_{A\le 1}Ae^{-\frac14 A^2t}
+\chi_{|\xi|\lesssim A^2 }\frac{\xi^4}{A^6}e^{-\frac{\xi^2}{2A^2}t}.\label{est:K-ttt2tt2}
\\
7.\quad &  |\big(\pp_{tt}+A^2\pp_t+\xi^2\big)\widehat K(t,\xi,\eta)|
\lesssim\chi_{A\ge 1} \frac1{A^2} e^{-ct}+\chi_{A\le 1}\frac{1}{A}e^{-\frac14 A^2t}+\frac{\xi^2}{A^{4}}\chi_{|\xi|\lesssim A^2 }e^{-\frac{\xi^2}{2A^2}t}.\label{est:K-tt2t2-x}
\\
8.\quad &  |\big(\pp_{tt}+A^2\pp_t+\xi^2\big)\pp_t\widehat K(t,\xi,\eta)|
\lesssim\chi_{A\ge 1} \frac1{A^2} e^{-ct}+\chi_{A\le 1}e^{-\frac14 A^2t}+\frac{\xi^4}{A^{6}}\chi_{|\xi|\lesssim A^2 }e^{-\frac{\xi^2}{2A^2}t}.\label{est:K-tt2t2-xt}
\end{align}
\end{prop}

We make several remarks. As shown in the proposition, we pointwise estimates are split into three parts contained: $\chi_{A\ge 1}, \chi_{A\le 1}, \chi_{|\xi|\lesssim A^2 }$. They reflect the different behaviors of the operator in different regions. The parts $\chi_{A\ge 1}\cdot$ tells the behavior of the operator in high-frequence; the parts $\chi_{A\ge 1}\cdot$ tells the behavior of the operator in low-frequence, also it tells the strength of the singularity of the operator; the parts $\chi_{|\xi|\lesssim A^2 }\cdot$ will tell us how the $\xi-$direction affects the decaying of the linear flow.

 In contrast to the incompressible MHD equations studied in \cite{WWX},
the estimates for $K$ sensitively depend on the spatial dimension and as we shall see in the later
sections, the 2D compressible MHD equations is critical in the sense that the spatial  $L^\infty$
norms behave like $t^{-1}$ and is barely time integrable. One may expect from the parabolic structure
that $\partial_t\sim \pp_{xx}$ in the large time behavior, but this proposition implies that
$\partial_t\sim \pp_x$. 
This weaker regularity effect of $\partial_t$ is due to the dispersive effect of
$e^{i\sqrt{A^2-\frac14A^4+\pm A|\eta|}t}$, as indicated in the proof of this proposition. We also
point out that the operator $\big(\pp_{tt}-\Delta\pp_t-\Delta\big) K$ behaves like that of $\sqrt{-\Delta}\partial_y K$, which is better than $\Delta K$. That is,
$$
\big(\pp_{tt}+A^2\pp_t+A^2\big)\sim A|\eta|.
$$
We emphasize the anisotropicity  here because of its non-obvious and its improtant role in our analysis.
This anisotropicity will be frequently used in the subsequent sections.

\vskip .1in
\begin{proof}[Proof of Proposition \ref{Kpointwise}]
First, we estimate $\widehat K$. According to the singularity of $\frac{1}{A|\eta|}$ from
the expression \eqref{eq:K}, we split into the following two cases.
$$
\mbox{\textbf{Case 1:}}\quad  A|\eta|\le \frac12 A^2;\quad \mbox{\textbf{Case 2:}}\quad A|\eta|\ge \frac12 A^2.
$$

\noindent \textbf{Case 1:} $A|\eta|\le \frac12 A^2$. It is divided by the following two subcases,
$$
\mbox{\textbf{Subcase 11:}}\quad \frac14A^4-A^2+A|\eta|\ge 0;\quad \mbox{\textbf{Subcase 12:}}\quad \frac14A^4-A^2+A|\eta|< 0.
$$
\textbf{Subcase 11:} $\frac14A^4-A^2+A|\eta|\ge 0$, then $A\ge 1$, and
$$
-\frac12A^2+\sqrt{\frac14A^4-A^2+A|\eta|}\le -\frac12.
$$
Now we need the following  lemma, which will be proved in Appendix A.1.
\begin{lemma}\label{lem:Elem1}
Let $a>0, b\in \R,c>0$. Then there exists a constant $C>0$, such that
\begin{align*}
& \Bigg|\frac1c\Bigg\{\frac1{\sqrt{b+c}}\Big[e^{(-a+\sqrt{b+c})t}-e^{(-a-\sqrt{b+c})t}\Big]
-\frac1{\sqrt{b-c}}\Big[e^{(-a+\sqrt{b-c})t}-e^{(-a-\sqrt{b-c})t}\Big]\Bigg\}\Bigg| \\
&\quad \le \begin{cases}
C\min\left\{t^3,\frac tc\right\}e^{-at}\quad {\rm when}\quad b+c< 0,\\
C\min\left\{t^3,\frac1{c\sqrt{b+c}}, \frac{\langle t\rangle}{b+c}\right\}e^{(-a+\sqrt{b+c})t}\quad {\rm when}\quad b+c\ge 0.
\end{cases}
\end{align*}
\end{lemma}

\vskip .1in
Thus, by Lemma \ref{lem:Elem1}, if $A\sim  1$ we have
\begin{align*}
|\widehat K(t,\xi,\eta)|\lesssim t^3e^{-\frac12 t}\le e^{-ct}, \quad \mbox{ for some small constant } c>0;
\end{align*}
if $A\gg 1$, then $\frac14A^4-A^2+A|\eta|\sim A^4$, and
\begin{align*}
|\widehat K(t,\xi,\eta)|\lesssim \frac{\langle t\rangle}{\frac14A^4-A^2+A|\eta|}e^{-\frac12 t}\lesssim \frac1{A^4}e^{-ct}.
\end{align*}
Therefore, no matter in which case,
\begin{align}
|\widehat K(t,\xi,\eta)|\le \chi_{A\ge 1}\frac1{A^4} e^{-ct}.\label{17.14}
\end{align}
\textbf{Subcase 12:} $\frac14A^4-A^2+A|\eta|< 0$.
If $A\gtrsim 1$, then by Lemma \ref{lem:Elem1}, we have
$$
|\widehat K(t,\xi,\eta)|\lesssim t^3 e^{-\frac12A^2t}\lesssim \frac1{A^4} e^{-\frac14A^2t}.
$$
So we only need to consider $A\ll 1$. Then, we have
\begin{align}\label{13.47}
\frac14A^4-A^2-A|\eta|<0,\quad \mbox{ and }\quad \big|\frac14A^4-A^2+A|\eta|\big|\sim\big|\frac14A^4-A^2-A|\eta|\big|\sim A^2.
\end{align}
So an immediately estimate from \eqref{eq:K} is that
$$
|\widehat K(t,\xi,\eta)|\lesssim \frac{1}{A^2|\eta|}e^{-\frac14A^2t}\lesssim \frac{1}{A|\xi||\eta|}e^{-\frac14A^2t}.
$$
Moreover, using the following elementary inequality (see Appendix \ref{app3} for its proof),
\begin{align}\label{basic-1}
\big|\frac{\sin x}{x}-\frac{\sin y}{y}\big|\lesssim \big||x|-|y|\big|\min\big\{|x|+|y|, \frac1{|x|+|y|}\big\},\quad \mbox{ for any } x,y \in \R,
\end{align}
and \eqref{13.47}, we have
\begin{align*}
|\widehat K(t,\xi,\eta)|=&\Bigg|\frac{1}{2A|\eta|}e^{-\frac12A^2t}\Bigg(\frac{\sin \big(\sqrt{\big|\frac14A^4-A^2+A|\eta|\big|}t\big)}{\sqrt{\big|\frac14A^4-A^2+A|\eta|\big|}}-\frac{\sin \big(\sqrt{\big|\frac14A^4-A^2-A|\eta|\big|}t\big)}{\sqrt{\big|\frac14A^4-A^2-A|\eta|\big|}}\Bigg)\Bigg|\\
\lesssim &
\frac{t}{A|\eta|}e^{-\frac12A^2t}\Bigg(t\sqrt{\big|\frac14A^4-A^2+A|\eta|\big|}-t\sqrt{\big|\frac14A^4-A^2-A|\eta|\big|}\Bigg)\\
& \qquad\qquad \times \Big( \chi_{At\lesssim 1} At
+ \chi_{At\gtrsim 1} \frac{1}{At}\Big)\\
\lesssim &
\Big(t^3\chi_{At\lesssim 1} +t\frac1{A^2}\chi_{At\gtrsim 1}\Big)e^{-\frac12A^2t}\\
\lesssim &
\Big(\frac1{A^3}+\frac1{A^2}t\Big)e^{-\frac12A^2t}
\lesssim
\frac1{A^4}e^{-\frac14A^2t}.
\end{align*}
Therefore, we proved that in this subcase and $A\ll 1$,
$$
|\widehat K(t,\xi,\eta)|\lesssim
\chi_{A\le 1}\min\Big\{\frac{1}{A|\xi||\eta|},\frac{1}{A^4}\Big\}e^{-\frac14 A^2t}.
$$

\noindent \textbf{Case 2:} $A|\eta|\ge \frac12 A^2$. Again, we consider the following two subcases respectively,
$$
\mbox{\textbf{Subcase 21:}}\quad \frac14A^4-A^2+A|\eta|<\frac1{36}A^4;\quad \mbox{\textbf{Subcase 22:}}\quad \frac14A^4-A^2+A|\eta|\ge \frac1{36}A^4.
$$

\textbf{Subcase 21:} $\frac14A^4-A^2+A|\eta|<\frac1{36}A^4$. In this subcase, we have
$$
(\frac14-\frac1{36})A^4\le A^2-A|\eta|=A\frac{\xi^2}{A+|\eta|}\le \xi^2
.
$$
That is, $A^4\lesssim \xi^2$, and thus $A\lesssim 1$. Then by mean value theorem,
\begin{align}
|\widehat K(t,\xi,\eta)|=&\Bigg|\frac{1}{2A|\eta|}e^{-\frac12A^2t}\>\Bigg[
e^{\sqrt{\frac14A^4-A^2+A|\eta|}t}\>\frac{1
-e^{-2\sqrt{\frac14A^4-A^2+A|\eta|}t}}{2\sqrt{\frac14A^4-A^2+A|\eta|}}\notag\\
&\,\,-
e^{\sqrt{\frac14A^4-A^2-A|\eta|}t}\>\frac{1
-e^{-2\sqrt{\frac14A^4-A^2-A|\eta|}t}}{2\sqrt{\frac14A^4-A^2-A|\eta|}}\Bigg]\Bigg|\label{eq:K-2'}\\
\lesssim &
\frac{t}{A^2}e^{-\frac12A^2t}\Big|e^{\sqrt{\frac14A^4-A^2+A|\eta|}t}\Big|+\frac{t}{A^2}e^{-\frac12A^2t}\Big|e^{\sqrt{\frac14A^4-A^2-A|\eta|}t}\Big|\notag\\
\lesssim &
\frac{t}{A^2}e^{-\frac13A^2t}
\lesssim
\frac{1}{A^4}e^{-\frac14A^2t}\label{eq:K-2},
\end{align}
where we have used $\frac14A^4-A^2-A|\eta|\le \frac14A^4-A^2+A|\eta|<\frac1{36}A^4$.
Moreover, if $A^4\sim  \xi^2$, then by \eqref{eq:K-2},
$$
|\widehat K(t,\xi,\eta)|\lesssim
\frac{1}{A|\xi||\eta|}e^{-\frac14A^2t}.
$$
If $A^4\ll  \xi^2$, then
$$
0<-\big(\frac14A^4-A^2+A|\eta|\big)\sim \xi^2,\quad \mbox{ and }0<-\big(\frac14A^4-A^2-A|\eta|\big)\sim A^2.
$$
Hence, by \eqref{eq:K-2'},
\begin{align*}
|\widehat K(t,\xi,\eta)|
\lesssim &
\frac{1}{A|\xi||\eta|}e^{-\frac12A^2t}\Big|e^{\sqrt{\frac14A^4-A^2+A|\eta|}t}\Big|+\frac{1}{A^2|\eta|}e^{-\frac12A^2t}\Big|e^{\sqrt{\frac14A^4-A^2-A|\eta|}t}\Big|\\
\lesssim &
\frac{1}{A|\xi||\eta|}e^{-\frac12A^2t}.
\end{align*}
Therefore, no matter in which case, we also obtain that
$$
|\widehat K(t,\xi,\eta)|\lesssim
\chi_{A\lesssim 1}\min\Big\{\frac{1}{A|\xi||\eta|},\frac{1}{A^4}\Big\}e^{-\frac14 A^2t}.
$$

\textbf{Subcase 22:} $\frac14A^4-A^2+A|\eta|\ge \frac1{36}A^4$, then $\xi^2\lesssim A^4$. Further, since
       \begin{align*}
        -\frac12A^2+\sqrt{\frac14A^4-A^2+A|\eta|}
        =&-\frac{A}{A+|\eta|}\frac{\xi^2}{\frac12A^2+\sqrt{\frac14A^4-A^2+A|\eta|}},
       \end{align*}
       we have
       \begin{align}\label{phi1:1}
       -\frac{2\xi^2}{A^2}\le -\frac12A^2+\sqrt{\frac14A^4-A^2+A|\eta|}\le -\frac{\xi^2}{2A^2}.
       \end{align}
Thus, by \eqref{phi1:1} and Lemma \ref{lem:Elem1}, we have
\begin{align*}
|\widehat K(t,\xi,\eta)|
\lesssim
\chi_{|\xi|\lesssim A^2}\frac{1}{A|\eta|\sqrt{\frac14A^4-A^2+A|\eta|}}e^{-\frac{\xi^2}{2A^2}t}
\lesssim \frac1{A^4}\chi_{|\xi|\lesssim A^2}e^{-\frac{\xi^2}{2A^2}t}.
\end{align*}
To collect the estimates above, we obtain \eqref{est:K}.

\vskip .1in
Now we turn to $\partial_tK(t,\partial_x,\partial_y)$. A direct computation gives us that
\begin{align}
&\partial_t\widehat K(t-s,\xi,\eta)\notag\\
=&\frac{1}{2A|\eta|}\Bigg\{\frac{1}{2\sqrt{\frac14A^4-A^2+A|\eta|}}
\Big[\Big(-\frac12A^2+\sqrt{\frac14A^4-A^2+A|\eta|}\Big)e^{(-\frac12A^2+\sqrt{\frac14A^4-A^2+A|\eta|})t}\notag\\
&\quad-\Big(-\frac12A^2-\sqrt{\frac14A^4-A^2+A|\eta|}\Big)e^{(-\frac12A^2-\sqrt{\frac14A^4-A^2+A|\eta|})t}\Big]\notag\\
&\quad-\frac{1}{2\sqrt{\frac14A^4-A^2-A|\eta|}}
\Big[\Big(-\frac12A^2+\sqrt{\frac14A^4-A^2-A|\eta|}\Big)e^{(-\frac12A^2+\sqrt{\frac14A^4-A^2-A|\eta|})t}\notag\\
&\quad-\Big(-\frac12A^2-\sqrt{\frac14A^4-A^2-A|\eta|}\Big)e^{(-\frac12A^2-\sqrt{\frac14A^4-A^2-A|\eta|})t}\Big]\Bigg\}\label{Kt-1}\\
=& -\frac{A^2}{2}\widehat K(t,\xi,\eta)+\frac{1}{4A|\eta|}\Big[e^{(-\frac12A^2+\sqrt{\frac14A^4-A^2+A|\eta|})t}+e^{(-\frac12A^2-\sqrt{\frac14A^4-A^2+A|\eta|})t}\notag\\
&\quad -e^{(-\frac12A^2+\sqrt{\frac14A^4-A^2-A|\eta|})t}-e^{(-\frac12A^2-\sqrt{\frac14A^4-A^2-A|\eta|})t}\Big].\label{Kt-2}
\end{align}
Therefore, by \eqref{Kt-2} and the similar estimates as those for \eqref{est:K}, we obtain \eqref{est:Kt}.
By a similar way, we also get the following the estimates on $\partial_{tt}K(t,\partial_x,\partial_y)$ in \eqref{est:Ktt}.

\vskip .1in
We now give the estimates on the special forms. First, we consider $\big(\pp_{tt}-\Delta\pp_t-\Delta\big)K(t,\partial_x,\partial_y)$.
To this end, we derive an identity.  Let
$$
\phi_{\mu_1,\mu_2}=e^{\big(\frac12\Delta+\mu_1\sqrt{\frac14\Delta^2+\Delta+\mu_2\sqrt{\Delta\partial_{yy}}}\big)t},
$$
where $\mu_1,\mu_2=\pm1$, then we have
\begin{align*}
\partial_{tt}\phi_{\mu_1,\mu_2}=&\Bigg(\frac12\Delta+\mu_1\sqrt{\frac14\Delta^2+\Delta+\mu_2\sqrt{\Delta\partial_{yy}}}\Bigg)^2\phi_{\mu_1,\mu_2}\\
=&\Big(\frac12\Delta^2+\mu_1\Delta\sqrt{\frac14\Delta^2+\Delta+\mu_2\sqrt{\Delta\partial_{yy}}}+\Delta+\mu_2\sqrt{\Delta\partial_{yy}}\Big)\phi_{\mu_1,\mu_2}\\
=&\Delta\partial_{t}\phi_{\mu_1,\mu_2} +\Delta\phi_{\mu_1,\mu_2}+\mu_2\sqrt{\Delta\partial_{yy}}\phi_{\mu_1,\mu_2},
\end{align*}
that is,
\begin{align}\label{phi-pm}
\partial_{tt}\phi_{\mu_1,\mu_2}-\Delta\partial_t\phi_{\mu_1,\mu_2}-\Delta\phi_{\mu_1,\mu_2}=\mu_2\sqrt{\Delta\partial_{yy}}\phi_{\mu_1,\mu_2}.
\end{align}
Now using \eqref{phi-pm}, we have
\begin{align}
&\big(\pp_{tt}+A^2\pp_t+A^2\big)\widehat K(t,\xi,\eta)\notag\\
=&\frac{1}{4\sqrt{\frac14A^4-A^2+A|\eta|}}\Big(e^{(-\frac12A^2+\sqrt{\frac14A^4-A^2+A|\eta|})t}
-e^{(-\frac12A^2-\sqrt{\frac14A^4-A^2+A|\eta|})t}\Big)\notag\\
&+\frac{1}{4\sqrt{\frac14A^4-A^2-A|\eta|}}\Big(e^{(-\frac12A^2+\sqrt{\frac14A^4-A^2-A|\eta|})t}
-e^{(-\frac12A^2-\sqrt{\frac14A^4-A^2-A|\eta|})t}\Big).\label{K-tt2t2}
\end{align}
Compared with $\pp_{tt}\widehat K(t,\xi,\eta), A^2\widehat K(t,\xi,\eta)$, the form $\big(\pp_{tt}+A^2\pp_t+A^2\big)\widehat K(t,\xi,\eta)$ erases the bad
factor of $\frac{1}{|\eta|}$, and thus gives \eqref{est:K-tt2t2}.
Further,
\begin{align}
&\pp_t\big(\pp_{tt}+A^2\pp_t+A^2\big)\widehat K(t,\xi,\eta)\notag\\
=&-\frac12A^2\big(\pp_{tt}+A^2\pp_t+A^2\big)\widehat K(t,\xi,\eta)
+\frac14\Big(e^{(-\frac12A^2+\sqrt{\frac14A^4-A^2+A|\eta|})t}
+e^{(-\frac12A^2-\sqrt{\frac14A^4-A^2+A|\eta|})t}\notag\\
&\qquad+e^{(-\frac12A^2+\sqrt{\frac14A^4-A^2-A|\eta|})t}
+e^{(-\frac12A^2-\sqrt{\frac14A^4-A^2-A|\eta|})t}\Big).\label{K-ttt2t2}
\end{align}
From \eqref{K-ttt2t2}, and by the similar argument as the proof of \eqref{est:K}, we have \eqref{est:K-ttt2t2}.

Next, we consider another special form $\big(\pp_{tt}-\Delta\pp_t-\pp_{xx}\big)K(t,\partial_x,\partial_y)$.
It obeys the same estimate as the operator $\big(\pp_{tt}-\Delta\pp_t-\Delta\big)K(t,\partial_x,\partial_y)$ when $|\eta|$ is small.
Indeed, we have
$$
\big(\pp_{tt}-\Delta\pp_t-\pp_{xx}\big)K(t,\partial_x,\partial_y)
=\big(\pp_{tt}-\Delta\pp_t-\Delta\big)K(t,\partial_x,\partial_y)+\pp_{yy}K(t,\partial_x,\partial_y).
$$
But on the other hand, from the form itself, we see that it is better when $|\xi|$ is small. Thus similar as above, we have \eqref{est:K-tt2t2-x} and \eqref{est:K-tt2t2-xt}.
\end{proof}

Recalling $\widehat{K_1}(t,\xi,\eta)$ defined in (\ref{K1}), we have
\begin{align}
\pp_t\big(\pp_{tt}+A^2\pp_t+A^2\big)\widehat K(t,\xi,\eta)
=-\frac12A^2\big(\pp_{tt}+A^2\pp_t+A^2\big)\widehat K(t,\xi,\eta)
+\widehat{K_1}(t,\xi,\eta).\label{K-ttt2t2-K1}
\end{align}
This equality is useful in the following sections.

\subsection{Estimates in $L^p$-space} This subsection provides estimates for $\widehat{K}$ and
the Fourier transforms of various derivatives of $K$ in $L^p$-spaces. The pointwise estimates in the
previous subsections will be used here.
We first state and prove an elementary lemma, which will be used repeatedly
in the estimates.

\begin{lemma}\label{lem:basic}
Let $c$ be a positive constant, and let $N\ge 0$ be the dyadic number ($N=2^j$ for some $j\in\Z$), then
\begin{align}\label{est:At}
\int\!\!\!\!\int_{A\le 1}A^\beta e^{-c A^2t}\,d\xi d\eta
&\lesssim \langle t\rangle^{-1-\frac{\beta}{2}},\quad \mbox{ for any } \beta>-2;
\end{align}
and
\begin{align}\label{est:Axit-1}
\int\!\!\!\!\int_{A\sim N} \frac{\xi^\beta}{A^{\alpha}}e^{-c\frac{\xi^2}{A^2}t}\,d\xi d\eta
&\lesssim N^{\beta+2-\alpha}\langle t\rangle^{-\frac{1+\beta}{2}},\quad \mbox{ for any } \alpha\in\R,  \beta>-1.
\end{align}
In particular, for any $\alpha\in\R,\beta'\in\R,   \beta>-1$ with $\beta'\ge \beta, 2\beta'-\beta+2-\alpha>0$, we have
\begin{align}\label{est:Axit-2}
\int\!\!\!\!\int_{A\le 1} \chi_{|\xi|\lesssim A^2 }\frac{\xi^{\beta'}}{A^{\alpha}}e^{-c\frac{\xi^2}{A^2}t}\,d\xi d\eta
&\lesssim \langle t\rangle^{-\frac{1+\beta}{2}}.
\end{align}
\end{lemma}
\begin{proof}
First, we prove \eqref{est:At}. If $t\ge 1$, we set $\xi_t=\xi\sqrt{t}, \eta_t= \eta\sqrt{t}$. Then by changing variable,
we have
\begin{align*}
\int\!\!\!\!\int_{A\le 1}A^\beta e^{-c A^2t}\,d\xi d\eta
&= t^{-1-\frac\beta2}\int\!\!\!\!\int_{\xi_t^2+\eta_t^2\le t}(\xi_t^2+\eta_t^2)^\frac\beta2 e^{-c(\xi_t^2+\eta_t^2)}\,d\xi_t d\eta_t\\
&\le t^{-1-\frac\beta2}\int\!\!\!\!\int_{\R^2}(\xi^2+\eta^2)^\frac\beta2 e^{-c(\xi^2+\eta^2)}\,d\xi d\eta \lesssim t^{-1-\frac\beta2}.
\end{align*}
If $0\le t\le 1$, then
\begin{align*}
\int\!\!\!\!\int_{A\le 1}A^\beta e^{-\frac14 A^2t}\,d\xi d\eta\le \int\!\!\!\!\int_{A\le 1}A^\beta \,d\xi d\eta\lesssim 1.
\end{align*}

Now we prove \eqref{est:Axit-1}. Similarly, by changing variable, we obtain that for some positive constant $\tilde{c}$,
\begin{align*}
\int\!\!\!\!\int_{A\sim N} \frac{\xi^\beta}{A^{\alpha}}e^{-c\frac{\xi^2t}{A^2}}\,d\xi d\eta
&\sim \int\!\!\!\!\int_{A\sim N} \frac{\xi^\beta}{N^{\alpha}}e^{-\tilde{c}\frac{\xi^2t}{N^2}}\,d\xi d\eta\\
&\le t^{-\frac{1+\beta}2}N^{\beta+1-\alpha}\int_{|\eta|\lesssim N}\!\!\int_{\xi\in\R} \Big(\frac{\xi \sqrt{t}}{N}\Big)^\beta e^{-\tilde{c}\frac{\xi^2t}{N^2}}\,d(\frac{\xi\sqrt{t}}{N}) d\eta\\
&\lesssim t^{-\frac{1+\beta}2}N^{\beta+2-\alpha}.
\end{align*}
If $0\le t\le 1$, then
\begin{align*}
\int\!\!\!\!\int_{A\sim N} \frac{\xi^\beta}{A^{\alpha}}e^{-\frac{\xi^2}{2A^2}t}\,d\xi d\eta
\le \int\!\!\!\!\int_{A\sim N} \frac{\xi^\beta}{A^{\alpha}}\,d\xi d\eta \lesssim  N^{\beta+2-\alpha}.
\end{align*}
Further, since
$
\chi_{|\xi|\lesssim A^2 }\frac{\xi^{\beta'}}{A^{\alpha}}\lesssim \frac{\xi^{\beta}}{A^{\alpha-2(\beta'-\beta)}},
$
by using \eqref{est:Axit-1} and the dyadic decomposition, we also have \eqref{est:Axit-2}.
\end{proof}

\begin{rem} The $L^q$ estimates can also be easily obtained from the conclusions in Lemma \ref{lem:basic}. For example, from \eqref{est:At},
let $\tilde c=cq$, we obtain
$$
\big\|A^\beta e^{-c A^2t}\big\|_{L^q_{\xi\eta}}=  \big\|A^{\beta q} e^{-\tilde{c} A^2t}\big\|_{L^1_{\xi\eta}}^{\frac1q}
\lesssim \langle t\rangle^{-(1+\frac{\beta q}{2})\frac1q}.
$$
\end{rem}

\subsubsection{Estimates on $\widehat{K(t)}$}
\begin{lemma}\label{lem:Kn1} Let $1\le q\le \infty$, and $N\gtrsim 1$. Then
\begin{align}\label{est:kn1-1}
\big\|A^4\widehat K(t,\xi,\eta)\big\|_{L^q_{\xi\eta}(A\le 1)}\lesssim \langle t\rangle^{-\frac{1}{2q}};\qquad
\big\|\widehat K(t,\xi,\eta)\big\|_{L^q_{\xi\eta}(A\sim N)}\lesssim N^{\frac2q-4}\langle t\rangle^{-\frac{1}{2q}}.
\end{align}
Moreover, for any $\beta>\frac12$,
\begin{align}\label{est:kn1-2}
\big\|A^{3+\beta}\widehat K(t,\xi,\eta)\big\|_{L^2_{\xi\eta}(A\le 1)}\lesssim \langle t\rangle^{-\frac{1}{4}}.
\end{align}
\end{lemma}
\begin{proof}
By \eqref{est:K}, we have
\begin{align}\label{est:Kn1}
\big| A^4\widehat K(t,\xi,\eta)\big|\lesssim \chi_{A\ge 1}e^{-ct}+\chi_{A\le 1} e^{-\frac14 A^2t}+\chi_{|\xi|\lesssim A^2 }e^{-\frac{\xi^2}{2A^2}t}.
\end{align}
Note that from \eqref{est:Kn1},
\begin{align}
\big\|A^4\widehat K(t,\xi,\eta)\big\|_{L^\infty_{\xi\eta}}\lesssim 1.\label{14.50}
\end{align}
Further, when $A\ge 1$,
\begin{align*}
\big| A^4\widehat K(t,\xi,\eta)\big|\lesssim \chi_{A\ge 1}e^{-ct}+\chi_{|\xi|\lesssim A^2 }e^{-\frac{\xi^2}{2A^2}t}.
\end{align*}
Thus, by \eqref{est:Axit-1},
\begin{align*}
\big\|A^4\widehat K(t,\xi,\eta)\big\|_{L^1_{\xi\eta}(A\sim N)}
\lesssim &  \big\|\chi_{A\ge 1} e^{-ct}\big\|_{L^1_{\xi\eta}(A\sim N)}+ \big\|e^{-\frac{\xi^2}{2A^2}t}\big\|_{L^1_{\xi\eta}(A\sim N)}\\
\lesssim & N^2e^{-ct}+N^2\langle t\rangle^{-\frac12}\lesssim N^2\langle t\rangle^{-\frac12}.
\end{align*}
That is,
\begin{align}\label{est:n2}
\big\|\widehat K(t,\xi,\eta)\big\|_{L^1_{\xi\eta}(A\sim N)}
\lesssim & N^{-2}\langle t\rangle^{-\frac12}.
\end{align}
Now, when $A\le 1$,
\begin{align}\label{1351}
\big| A^4\widehat K(t,\xi,\eta)\big|\lesssim \chi_{A\le 1} e^{-\frac14 A^2t}+\chi_{|\xi|\lesssim A^2 }e^{-\frac{\xi^2}{2A^2}t}.
\end{align}
Similarly, by \eqref{est:At} and \eqref{est:Axit-2},
\begin{align*}
\int\!\!\!\!\int_{A\le 1} e^{-\frac14 (\xi^2+\eta^2)t}\,d\xi d\eta
\lesssim \langle t\rangle^{-1};\quad
\int\!\!\!\!\int_{A\le 1} e^{-\frac{\xi^2}{2A^2}t}\,d\xi d\eta
\lesssim \langle t\rangle^{-\frac12},
\end{align*}
thus,
\begin{align}\label{est:n3}
\big\|A^4\widehat K(t,\xi,\eta)\big\|_{L^1_{\xi\eta}(A\le 1)}
\lesssim & \langle t\rangle^{-\frac12}.
\end{align}
Now the conclusion \eqref{est:kn1-1} follows from interpolation between \eqref{14.50} and \eqref{est:n3}, \eqref{est:n2} respectively. For \eqref{est:kn1-2},  by \eqref{1351}, we have
\begin{align*}
\big| A^{3+\beta}\widehat K(t,\xi,\eta)\big|\lesssim \chi_{A\le 1} A^{\beta-1}e^{-\frac14 A^2t}+A^{\beta-1}\chi_{|\xi|\lesssim A^2 }e^{-\frac{\xi^2}{2A^2}t},
\end{align*}
which is square integrable. Thus direct integration and using \eqref{est:At}, \eqref{est:Axit-2} give \eqref{est:kn1-2}.
\end{proof}
%

We give the estimates on the operator $\partial_{xy}\nabla K(t)$, which read as
\begin{lemma}\label{lem:Ku1} Let $1\le q\le \infty$ and $N\gtrsim 1$. Then
\begin{align}
\big\|\xi\eta A\widehat K(t,\xi,\eta)\big\|_{L^q_{\xi\eta}(A\le 1)}\lesssim & \langle t\rangle^{-\frac1{q}};\label{est:ku1-1}\\
\big\|\xi\eta \widehat K(t,\xi,\eta)\big\|_{L^q_{\xi\eta}(A\sim N)}\lesssim &\langle N\rangle^{-2+\frac2q}\langle t\rangle^{-\frac12-\frac1{2q}}.\label{est:ku1}
\end{align}
Moreover, for any $\beta\in[ \frac12,\frac32]$,
\begin{align}\label{est:ku1-2}
\big\|A^{\beta+1}\xi\eta \widehat K(t,\xi,\eta)\big\|_{L^\infty_{\xi\eta}(A\le 1)}\lesssim &\langle t\rangle^{-\frac12};
\quad
\big\|A^{\beta}\xi\eta \widehat K(t,\xi,\eta)\big\|_{L^2_{\xi\eta}(A\le 1)}\lesssim \langle t\rangle^{-\frac\beta2};
\end{align}
\end{lemma}
\begin{proof} By \eqref{est:K}, we have
\begin{align}
\big|\xi\eta \widehat K(t,\xi,\eta)\big|\lesssim &\chi_{A\ge 1}\frac1{A^2} e^{-ct}
+\chi_{A\le 1} \frac1{A}e^{-\frac14 A^2t}+\chi_{|\xi|\lesssim A^2}\frac{\xi}{A^3} e^{-\frac{\xi^2}{2A^2}t}.\label{e:u3}
\end{align}
Thus, we further have, for any $\alpha$,
\begin{align}
\big|A^{\alpha}\xi\eta \widehat K(t,\xi,\eta)\big|\lesssim &\chi_{A\ge 1}\frac1{A^{2-\alpha}} e^{-ct}+\chi_{A\le 1}A^{\alpha-1} e^{-\frac14 A^2t}+\chi_{|\xi|\lesssim A^2}\frac{\xi}{A^{3-\alpha}} e^{-\frac{\xi^2}{2A^2}t}.\label{e:u4}
\end{align}
Similar as the proof of Lemma \ref{est:Kn1},  by \eqref{e:u4}, \eqref{est:At} and \eqref{est:Axit-2},  we have
\begin{align*}
\big\|\xi\eta A\widehat K(t,\xi,\eta)\big\|_{L^\infty_{\xi\eta}(A\le 1)}\lesssim 1;\qquad
\big\|\xi\eta A\widehat K(t,\xi,\eta)\big\|_{L^1_{\xi\eta}(A\le 1)}\lesssim \langle t\rangle^{-1},
\end{align*}
and thus \eqref{est:ku1-1} follows from interpolation.
For $N\gtrsim 1$, by \eqref{e:u3} and \eqref{est:Axit-1},
\begin{align*}
\big\|\xi\eta \widehat K(t,\xi,\eta)\big\|_{L^\infty_{\xi\eta}(A\sim N)}\lesssim N^{-2}\langle t\rangle^{-\frac12};\qquad
\big\|\xi\eta \widehat K(t,\xi,\eta)\big\|_{L^1_{\xi\eta}(A\sim N)}\lesssim   \langle t\rangle^{-1},
\end{align*}
and thus \eqref{est:ku1} follows from interpolation again.  Using \eqref{e:u4} and Lemma \ref{lem:basic},  \eqref{est:ku1-2} is also easily followed.
\end{proof}

\begin{lemma}\label{lem:Ku1''}  Let $N\gtrsim 1$ and $1\le q\le \infty$, then
\begin{align}
\big\|A\xi^2\eta \widehat K(t,\xi,\eta)\big\|_{L^q_{\xi\eta}(A\le 1)}\lesssim &\langle t\rangle^{-\frac12-\frac1q};\label{est:ku1-5}\\
\big\|\xi^2\eta \widehat K(t,\xi,\eta)\big\|_{L^q_{\xi\eta}(A\sim N)}\lesssim&  N^{\frac2q-1}\langle t\rangle^{-1-\frac1{2q}}.\label{est:ku1-6}
\end{align}
Moreover,
\begin{align}
\big\|A^2\xi^2\eta \widehat K(t,\xi,\eta)\big\|_{L^\infty_{\xi\eta}(A\le 1)}\lesssim &\langle t\rangle^{-1};\label{est:ku1-7}\\
\big\|A\xi^2\eta \widehat K(t,\xi,\eta)\big\|_{L^\infty_{\xi\eta}(A\ge 1)}\lesssim &\langle t\rangle^{-1}.\label{est:ku1-8}
\end{align}
\end{lemma}
\begin{proof} By \eqref{e:u3}, we have
\begin{align}
\big|\xi^2\eta A\widehat K(t,\xi,\eta)\big|\lesssim &\chi_{A\ge 1} e^{-ct}
+\chi_{A\le 1}A e^{-\frac14 A^2t}+\chi_{|\xi|\lesssim A^2}\frac{\xi^2}{A^{2}} e^{-\frac{\xi^2}{2A^2}t}.\label{e:u5}
\end{align}
Then by \eqref{est:At} and \eqref{est:Axit-2}, we have
\begin{align*}
\big\|\xi^2\eta A\widehat K(t,\xi,\eta)\big\|_{L^\infty_{\xi\eta}(A\le 1)}\lesssim \langle t\rangle^{-\frac12};\\
\big\|\xi^2\eta A\widehat K(t,\xi,\eta)\big\|_{L^1_{\xi\eta}(A\le 1)}\lesssim \langle t\rangle^{-\frac32},
\end{align*}
and by \eqref{est:Axit-1},
\begin{align*}
\big\|\xi^2\eta A\widehat K(t,\xi,\eta)\big\|_{L^\infty_{\xi\eta}(A\sim N)}\lesssim \langle t\rangle^{-1};\\
\big\|\xi^2\eta A\widehat K(t,\xi,\eta)\big\|_{L^1_{\xi\eta}(A\sim N)}\lesssim N^2 \langle t\rangle^{-\frac32},
\end{align*}
thus the conclusions \eqref{est:ku1-5} and \eqref{est:ku1-6} follow from interpolation. Further, \eqref{est:ku1-7} easily follows from \eqref{e:u5}.
\end{proof}

Since
$$
A^2-A|\eta|\sim \xi^2,
$$
we have
\begin{align}\label{eqs:28-15}
\big|A^2(A^2-A|\eta|)\widehat K(t,\xi,\eta) \big|\lesssim
\chi_{A\ge 1}e^{-ct}+\chi_{A\le 1}e^{-\frac14 A^2t}+\chi_{|\xi|\lesssim A^2}\frac{\xi^2}{A^2} e^{-\frac{\xi^2}{2A^2}t}.
\end{align}
Thus similar as above, we have
\begin{lemma}\label{lem:Ku3}
Let $1\le q\le \infty$. Then
\begin{align}\label{est:ku3}
\big\|A^2(A^2-A|\eta|)\widehat{K}(t,\xi,\eta)\big\|_{L^q_{\xi\eta}(A\sim N)}\lesssim \langle N\rangle^{\frac2q}\langle t\rangle^{-1-\frac1{2q}}.
\end{align}
Moreover, for any $\beta\in [0,2-\frac1q]$,
\begin{align}
\big\|A^\beta A^2(A^2-A|\eta|)\widehat{K}(t,\xi,\eta)\big\|_{L^q_{\xi\eta}(A\le 1)}\lesssim &\langle t\rangle^{-\frac1q-\frac\beta2}.\label{est:ku3-1}
\end{align}
\end{lemma}

\subsubsection{Estimates on $\pp_t\widehat{K(t)}$ and $\pp_{tt}\widehat{K(t)}$}
\begin{lemma}\label{lem:Kn2} Let $1\le q\le \infty$ and $N\gtrsim 1$. Then if $\beta\in (\frac12, \frac52]$,
\begin{align}
\big\|\eta\partial_t\widehat{K}(t,\xi,\eta)\big\|_{L^q_{\xi\eta}(A\sim N)}
\lesssim & N^{-2(1-\frac1q)-1}\langle t\rangle^{-1-\frac1{2q}};\label{est:kn2-1}\\
\big\|A\eta\partial_t\widehat{K}(t,\xi,\eta)\big\|_{L^q_{\xi\eta}(A\lesssim 1)}
\lesssim& \langle t\rangle^{-\frac1q};\label{est:kn2-2}\\
\big\|A^{\beta} \eta\pp_t\widehat K(t,\xi,\eta)\big\|_{L^2_{\xi\eta}(A\le 1)}\lesssim &\langle t\rangle^{-\frac\beta2}.\label{est:kn2'-1}
\end{align}
\end{lemma}
\begin{proof}
By \eqref{est:Kt}, we have
\begin{align}\label{est:Kt-1}
\big|\eta\widehat{\partial_tK}(t,\xi,\eta)\big|\lesssim \chi_{A\gtrsim1}\frac1{A^3} e^{-ct}+\chi_{A\lesssim 1}\frac1A e^{-\frac{1}{4}A^2t}+\frac{\xi^2}{A^5}\chi_{|\xi|\lesssim A^2}e^{-\frac{\xi^2}{2A^2}t}.
\end{align}
Then by \eqref{est:Kt-1},
\begin{align}\label{cx2}
\big|\chi_{A\sim N}\eta\widehat{\partial_tK}(t,\xi,\eta)\big|\lesssim \frac1{A^3} e^{-ct}+\frac{\xi^2}{A^5}e^{-\frac{\xi^2}{2A^2}t}.
\end{align}
Thus, by \eqref{cx2} and \eqref{est:Axit-1},
\begin{align*}
\big|\chi_{A\sim N}\eta\widehat{\partial_tK}(t,\xi,\eta)\big|
\lesssim &
\frac1{N^3} e^{-ct}+\frac{1}{N^3t}\frac{\xi^2t}{2A^2}e^{-\frac{\xi^2}{2A^2}t}
\lesssim N^{-3}\langle t\rangle^{-1};\\
\big\|\chi_{A\sim N}\eta\widehat{\partial_tK}(t,\xi,\eta)\big\|_{L^1_{\xi\eta}(A\sim N)}
\lesssim &
\int\!\!\!\!\int_{A\sim N}\!\!\frac1{A^3} e^{-ct}\,d\xi d\eta+\int\!\!\!\!\int_{A\sim N}\frac{\xi^2}{A^5}e^{-\frac{\xi^2}{2A^2}t}\,d\xi d\eta\\
\lesssim & N^{-1}\langle t\rangle^{-\frac32}.
\end{align*}
This proves \eqref{est:kn2-1} by interpolation. When  $A\le 1$,  by \eqref{est:Kt-1} again,
\begin{align*}
\big|\chi_{A\le 1}A\eta\widehat{\partial_tK}(t,\xi,\eta)\big|\lesssim e^{-\frac14A^2t}+\frac{|\xi|}{A^2}\chi_{|\xi|\lesssim A^2}e^{-\frac{\xi^2}{2A^2}t}\lesssim 1.
\end{align*}
That is,
\begin{align*}
\big\|A\eta\widehat{\partial_tK}(t,\xi,\eta)\big\|_{L^\infty_{\xi\eta}(A\le1)}\lesssim 1.
\end{align*}
Further, by \eqref{est:At} and \eqref{est:Axit-2},
\begin{align*}
\big\|A\eta\widehat{\partial_tK}(t,\xi,\eta)\big\|_{L^1_{\xi\eta}(A\le 1)}
\lesssim &
\int\!\!\!\!\int_{A\le 1}e^{-\frac14A^2t}\,d\xi d\eta
+\int\!\!\!\!\int_{A\le 1}\frac{|\xi|}{A^2}e^{-\frac{\xi^2}{2A^2}t}\,d\xi d\eta\\
\lesssim &
 \langle t\rangle^{-1}.
\end{align*}
Then \eqref{est:kn2-2} follows from interpolation. Similarly,
since
\begin{align}\label{est:Kt-1'}
\big|A^{\beta}\eta\widehat{\partial_tK}(t,\xi,\eta)\big|\lesssim \chi_{A\gtrsim1}\frac1{A^{3-\beta}} e^{-ct}+\chi_{A\lesssim 1}A^{\beta-1}e^{-\frac{1}{4}A^2t}+\frac{\xi^2}{A^{5-\beta}}\chi_{|\xi|\lesssim A^2}e^{-\frac{\xi^2}{2A^2}t}.
\end{align}
Note that it is integrable when $\beta>\frac12$, then integration and Lemma \ref{lem:basic} give \eqref{est:kn2'-1}.
\end{proof}
Moreover, from \eqref{est:Kt-1'}, we have for any $\beta\in [0,2]$,
\begin{align}\label{est:kn2'-3}
\big\| A^{\beta}A\eta\widehat{\partial_tK}(t,\xi,\eta)\big\|_{L^\infty_{\xi\eta}}\lesssim \langle t\rangle^{-\frac\beta2}.
\end{align}

We also need the following estimates.
\begin{lemma}\label{lem:Kn10}
Let $\beta\in [0,3]$, then
\begin{align*}
\big\|A^\beta\eta\pp_{tt}\widehat K(t-s,\xi,\eta)\big\|_{L^\infty_{\xi\eta}}
\lesssim & \langle t\rangle^{-\frac\beta2};\\
\big\|A\eta\pp_{tt}\widehat K(t-s,\xi,\eta)\big\|_{L^2_{\xi\eta}}
\lesssim & \langle t\rangle^{-1}.
\end{align*}
\end{lemma}
\begin{proof}
It follows from  \eqref{est:Ktt} and Lemma \ref{lem:basic} directly.
\end{proof}

Now we consider some related estimates about $(\partial_{tt}+A^2\partial_t)\widehat{K}$.
\begin{lemma}\label{lem:Kn4} Let $1\le q\le \infty$, $\beta\in[0,1]$ and $N\gtrsim 1$. Then
\begin{align}
\big\|\eta(\partial_{tt}+A^2\partial_t)\widehat{K}(t,\xi,\eta)\big\|_{L^q_{\xi\eta}(A\sim N)}\lesssim & N^{-1+\frac2q}\langle t\rangle^{-1-\frac1{2q}};\label{est:kn4-1}\\
\big\|\eta(\partial_{tt}+A^2\partial_t)\widehat{K}(t,\xi,\eta)\big\|_{L^q_{\xi\eta}(A\le 1)}\lesssim & \langle t\rangle^{-\frac1q};\label{est:kn4-2}\\
\big\|A^{\beta}\eta(\partial_{tt}+A^2\partial_t)\widehat{K}(t,\xi,\eta)\big\|_{L^2_{\xi\eta}(A\le 1)}\lesssim & \langle t\rangle^{-\frac{\beta+1}{2}};\label{est:kn4-3}\\
\big\|A\eta(\partial_{tt}+A^2\partial_t)\widehat{K}(t,\xi,\eta)\big\|_{L^\infty_{\xi\eta}(A\ge 1)}\lesssim &\langle t\rangle^{-1};\label{est:kn4-5}\\
\big\|A^{1+\beta}\eta(\partial_{tt}+A^2\partial_t)\widehat{K}(t,\xi,\eta)\big\|_{L^\infty_{\xi\eta}(A\le 1)}\lesssim & \langle t\rangle^{-\frac{\beta+1}2}.
\label{est:kn4-4}
\end{align}
\end{lemma}
\begin{proof}
By \eqref{est:Kt} and \eqref{est:Ktt} we have
\begin{align}\label{e:v2}
\big|\eta (\partial_{tt}+A^2\partial_t)\widehat{K}(t,\xi,\eta)\big|
\lesssim
\chi_{A\ge 1}\frac1Ae^{-ct}+\chi_{A\le 1}e^{-\frac14 A^2t}+\chi_{|\xi|\lesssim A^2}\frac{\xi^2}{A^3} e^{-\frac{\xi^2}{2A^2}t}.
\end{align}
Thus, for $N\gtrsim 1$, by \eqref{est:Axit-1},
\begin{align*}
\big\|\eta (\partial_{tt}+A^2\partial_t)\widehat{K}(t,\xi,\eta)\big\|_{L^\infty_{\xi\eta}(A\sim N)}\lesssim N^{-1}\langle t\rangle^{-1};\\
\big\|\eta (\partial_{tt}+A^2\partial_t)\widehat{K}(t,\xi,\eta)\big\|_{L^1_{\xi\eta}(A\sim N)}\lesssim  N \langle t\rangle^{-\frac32},
\end{align*}
and by \eqref{est:At} and \eqref{est:Axit-2},
\begin{align*}
\big\|\eta (\partial_{tt}+A^2\partial_t)\widehat{K}(t,\xi,\eta)\big\|_{L^\infty_{\xi\eta}(A\le 1)}\lesssim  1;\\
\big\|\eta (\partial_{tt}+A^2\partial_t)\widehat{K}(t,\xi,\eta)\big\|_{L^1_{\xi\eta}(A\le 1)}\lesssim  \langle t\rangle^{-1}.
\end{align*}
Then the conclusions \eqref{est:kn4-1} and \eqref{est:kn4-2} follow from interpolation. Moreover,
\begin{align*}
\big|A^\beta\eta (\partial_{tt}+A^2\partial_t)\widehat{K}(t,\xi,\eta)\big|
\lesssim
A^{\beta-1}\chi_{A\ge 1}e^{-ct}+\chi_{A\le 1}A^\beta e^{-\frac14 A^2t}+\chi_{|\xi|\lesssim A^2}\frac{\xi^2}{A^{3-\beta}} e^{-\frac{\xi^2}{2A^2}t}.
\end{align*}
Then \eqref{est:kn4-3}, \eqref{est:kn4-5} and \eqref{est:kn4-4} follow from integration directly.
\end{proof}

Furthermore, since
\begin{align}\label{e:u2}
\big|A^2 (\partial_{tt}+A^2\partial_t)\widehat{K}(t,\xi,\eta)\big|
\lesssim
\chi_{A\ge 1}e^{-ct}+\chi_{A\le 1}e^{-\frac14 A^2t}+\chi_{|\xi|\lesssim A^2}\frac{\xi^2}{A^2} e^{-\frac{\xi^2}{2A^2}t}.
\end{align}
Thus one may find that it has the same bound with $A^2(A^2-A|\eta|)\widehat{K}(t,\xi,\eta)$. Thus the same as Lemma \ref{lem:Ku3}, we have
\begin{lemma}\label{lem:Ku2} Let $1\le q\le \infty$ and $N\gtrsim 1$. Then
\begin{align}\label{est:ku2}
\big\|A^2 (\partial_{tt}+A^2\partial_t)\widehat{K}(t,\xi,\eta)\big\|_{L^q_{\xi\eta}(A\sim N)}\lesssim \langle N\rangle^{\frac2q}\langle t\rangle^{-1-\frac1{2q}}.
\end{align}
Moreover, for $\beta\in [0,2-\frac1{q}]$,
\begin{align}
\big\|A^{\beta}A^2 (\partial_{tt}+A^2\partial_t)\widehat{K}(t,\xi,\eta)\big\|_{L^q_{\xi\eta}(A\le 1)}\lesssim &\langle t\rangle^{-\frac1q-\frac\beta2}.\label{est:ku2-1}
\end{align}
\end{lemma}
\begin{proof} By \eqref{e:u2}, \eqref{est:At} and \eqref{est:Axit-2}, we have
\begin{align*}
\big\|A^{\beta}A^2 (\partial_{tt}+A^2\partial_t)\widehat{K}(t,\xi,\eta)\big\|_{L^\infty_{\xi\eta}(A\le 1)}\lesssim & \langle t\rangle^{-\frac\beta2},\quad \mbox{for any } \beta \in [0,2];\\
\big\|A^{\beta}A^2 (\partial_{tt}+A^2\partial_t)\widehat{K}(t,\xi,\eta)\big\|_{L^1_{\xi\eta}(A\le 1)}\lesssim & \langle t\rangle^{-1-\frac\beta2}, \quad \mbox{for any } \beta \in [0,1],
\end{align*}
then \eqref{est:ku2-1} immediately follows  by interpolation.
For $N\gtrsim 1$, by \eqref{est:Axit-1},
\begin{align*}
\big\|A^2 (\partial_{tt}+A^2\partial_t)\widehat{K}(t,\xi,\eta)\big\|_{L^\infty_{\xi\eta}(A\sim N)}\lesssim & \langle t\rangle^{-1};\\
\big\|A^2 (\partial_{tt}+A^2\partial_t)\widehat{K}(t,\xi,\eta)\big\|_{L^1_{\xi\eta}(A\sim N)}\lesssim & N^2 \langle t\rangle^{-\frac32}.
\end{align*}
Then the conclusion \eqref{est:ku2} follows from interpolation again.
\end{proof}

\subsubsection{Estimates on $\big(\pp_{tt}+A^2\pp_t+A^2\big)\widehat K(t,\xi,\eta)$}
To this end, we also need some special basic estimates.
\begin{lemma}\label{lem:basic-2}
Let $c$ be a positive constant, $\beta\ge 0$ and let $N\ge 0$ be the dyadic number ($N=2^j$ for some $j\in\Z$), then
\begin{align}\label{est:At-2}
\big\|A^\beta e^{-cA^2t}\big\|_{L^\infty_\xi L^1_{\eta}(A\le 1)}+\big\|A^\beta e^{-cA^2t}\big\|_{L^1_\xi L^\infty_{\eta}(A\le 1)}
&\lesssim \langle t\rangle^{-\frac{1+\beta}{2}},
\end{align}
and
\begin{align}
\Big\|\frac{\xi^\beta}{A^{\alpha}}e^{-c\frac{\xi^2}{A^2}t}\Big\|_{L^1_\xi L^\infty_{\eta}(A\sim N)}
&\lesssim N^{\beta+1-\alpha}\langle t\rangle^{-\frac{1+\beta}{2}};\label{est:Axit-3}\\
\Big\|\frac{\xi^\beta}{A^{\alpha}}e^{-c\frac{\xi^2}{A^2}t}\Big\|_{L^\infty_\xi L^1_{\eta}(A\sim N)}
&\lesssim N^{\beta+1-\alpha}\langle t\rangle^{-\frac{\beta}{2}}.\label{est:Axit-4}
\end{align}
Moreover, for any $\alpha\in\R,\beta'\in\R,   \beta\ge 0$ with $\beta'\ge \beta, 2\beta'-\beta+1-\alpha>0$,
\begin{align}
\Big\|\chi_{|\xi|\lesssim A^2}\frac{\xi^{\beta'}}{A^{\alpha}}e^{-c\frac{\xi^2}{A^2}t}\Big\|_{L^1_\xi L^\infty_{\eta}(A\le 1)}
\lesssim \langle t\rangle^{-\frac{1+\beta}{2}};\quad
\Big\|\chi_{|\xi|\lesssim A^2}\frac{\xi^{\beta'}}{A^{\alpha}}e^{-c\frac{\xi^2}{A^2}t}\Big\|_{L^\infty_\xi L^1_{\eta}(A\le 1)}
\lesssim \langle t\rangle^{-\frac{\beta}{2}}.\label{est:Axit-5}
\end{align}
\end{lemma}
\begin{proof}
First, we prove \eqref{est:At-2}. Since
$$
A^\beta e^{-cA^2t}\lesssim \langle t\rangle^{-\frac\beta2}e^{-\frac c2A^2t}\le \langle t\rangle^{-\frac\beta2}e^{-\frac c2\xi^2t},
$$
we have
$$
\big\|A^\beta e^{-cA^2t}\big\|_{L^1_\xi L^\infty_{\eta}(A\le 1)}
\lesssim \langle t\rangle^{-\frac\beta2} \big\|e^{-\frac c2\xi^2t}\big\|_{L^1_\xi}\lesssim \langle t\rangle^{-\frac{1+\beta}2}.
$$
By the symmetry, we also have $\big\|A^\beta e^{-cA^2t}\big\|_{L^\infty_\xi L^1_{\eta}(A\le 1)}\lesssim \langle t\rangle^{-\frac{1+\beta}2}$.
This gives \eqref{est:At-2}.

Now we turn to prove \eqref{est:Axit-3} and \eqref{est:Axit-4}. For $A\sim N$, there exists a constant $\tilde c>0$, such that
$$
\frac{\xi^\beta}{A^{\alpha}}e^{-c\frac{\xi^2}{A^2}t}\sim \frac{\xi^\beta}{N^{\alpha}}e^{-\tilde{c}\frac{\xi^2}{N^2}t}
\lesssim N^{\beta-\alpha}\langle t\rangle^{-\frac\beta2}.
$$
Therefore,
\begin{align*}
\Big\|\frac{\xi^\beta}{A^{\alpha}}e^{-c\frac{\xi^2}{A^2}t}\Big\|_{L^\infty_\xi L^1_{\eta}(A\sim N)}
&\lesssim N^{\beta+1-\alpha}\langle t\rangle^{-\frac{\beta}{2}}.
\end{align*}
Moreover,
\begin{align*}
\Big\|\frac{\xi^\beta}{A^{\alpha}}e^{-c\frac{\xi^2}{A^2}t}\Big\|_{L^1_\xi L^\infty_{\eta}(A\sim N)}
\lesssim \Big\|\frac{\xi^\beta}{N^{\alpha}}e^{-\tilde{c}\frac{\xi^2}{N^2}t}\Big\|_{L^1_\xi(|\xi|\lesssim N)}
\lesssim N^{\beta+1-\alpha}\langle t\rangle^{-\frac{1+\beta}{2}}.
\end{align*}
Furthermore, using \eqref{est:Axit-3}, \eqref{est:Axit-4} and the dyadic decomposition, we also have \eqref{est:Axit-5}.
\end{proof}

\begin{lemma}\label{lem:Kn6} Let $1\le q,r\le \infty$ and $N\gtrsim 1$. Then
\begin{align*}
\big\|A^2(\partial_{tt}+A^2\partial_t+A^2)\widehat K(t,\xi,\eta)\big\|_{L^q_{\xi}L^r_{\eta}(A\le 1)}
\lesssim & \langle t\rangle^{-\frac1{2q}};\\
\big\|A^2 (\partial_{tt}+A^2\partial_t+A^2)\widehat K(t,\xi,\eta)\big\|_{L^q_{\xi}L^r_{\eta}(A\sim N)}
\lesssim & N^{\frac1q+\frac1r}\langle t\rangle^{-\frac1{2q}}.
\end{align*}
\end{lemma}
\begin{proof}
By \eqref{est:K-tt2t2}, we have
\begin{align*}
|A^2(\partial_{tt}+A^2\partial_t+A^2)\widehat K(t,\xi,\eta)|\lesssim
\chi_{A\ge 1} e^{-ct}+\chi_{A\le 1}e^{-\frac14 A^2t}+\chi_{|\xi|\lesssim A^2 }e^{-\frac{\xi^2}{2A^2}t}.
\end{align*}
Then by
\eqref{est:At-2} and \eqref{est:Axit-5}, we have
\begin{align*}
\big\|A^2 (\partial_{tt}+A^2\partial_t+A^2)\widehat K(t,\xi,\eta)\big\|_{L^\infty_{\xi}L^r_{\eta}(A\le 1)}\lesssim 1;\\
\big\|A^2 (\partial_{tt}+A^2\partial_t+A^2)\widehat K(t,\xi,\eta)\big\|_{L^1_{\xi}L^r_{\eta}(A\le 1)}\lesssim  t^{-\frac12},
\end{align*}
and by \eqref{est:Axit-3} and \eqref{est:Axit-4},
\begin{align*}
\big\|A^2 (\partial_{tt}+A^2\partial_t+A^2)\widehat K(t,\xi,\eta)\big\|_{L^\infty_{\xi}L^r_{\eta}(A\sim N)}\lesssim N^{\frac1r};\\
\big\|A^2 (\partial_{tt}+A^2\partial_t+A^2)\widehat K(t,\xi,\eta)\big\|_{L^1_{\xi}L^r_{\eta}(A\sim N)}\lesssim  N^{1+\frac1r}t^{-\frac12}.
\end{align*}
Then the conclusion follows from interpolation.
\end{proof}

\begin{lemma}\label{lem:Kn3} Let $1\le q\le 2$, and $N\gtrsim 1$. Then
\begin{align}
\big\|\xi\big(\pp_{tt}+A^2\pp_t+A^2\big)\widehat K(t,\xi,\eta)\big\|_{L^q_{\xi\eta}(A\sim N)}\lesssim& N^{\frac2q-1}\langle t\rangle^{-\frac12-\frac1{2q}};\label{est:kn3-1}\\
\big\|\xi\big(\pp_{tt}+A^2\pp_t+A^2\big)\widehat K(t,\xi,\eta)\big\|_{L^q_{\xi\eta}(A\le 1)}\lesssim& \langle t\rangle^{-\frac1{q}}.\label{est:kn3-2}
\end{align}
Moreover,
\begin{align}
&\big\|A\xi\big(\pp_{tt}+A^2\pp_t+A^2\big)\widehat K(t,\xi,\eta)\big\|_{L^2_{\xi}L^\infty_\eta(A\sim N)}\label{est:kn3-3}\\
&\big\|A\xi\big(\pp_{tt}+A^2\pp_t+A^2\big)\widehat K(t,\xi,\eta)\big\|_{L^2_{\xi}L^\infty_\eta(A\le 1)}.\label{est:kn3-4}
\end{align}
Further, for any $\beta\ge \frac12$,
\begin{align}
\big\|A^\beta\xi\big(\pp_{tt}+A^2\pp_t+A^2\big)\widehat K(t,\xi,\eta)\big\|_{L^q_{\xi\eta}(A\le 1)}
\lesssim& \langle t\rangle^{-\frac12-\frac1{2q}};\label{est:kn3-7}\\
\big\|A^\beta\xi\big(\pp_{tt}+A^2\pp_t+A^2\big)\widehat K(t,\xi,\eta)\big\|_{L^1_{\xi}L^2_\eta(A\sim N)}
\lesssim & N^{\frac12+\beta}\langle t\rangle^{-1};\label{est:kn3-5}\\
\big\|A^\beta\xi\big(\pp_{tt}+A^2\pp_t+A^2\big)\widehat K(t,\xi,\eta)\big\|_{L^1_{\xi}L^2_\eta(A\le 1)}
\lesssim & \langle t\rangle^{-1}.\label{est:kn3-6}
\end{align}
\end{lemma}
\begin{proof} By \eqref{est:K-tt2t2}, we find that the function $\xi\big(\pp_{tt}+A^2\pp_t+A^2\big)\widehat K(t,\xi,\eta)$
has the same estimates as $\xi\eta A\widehat K(t,\xi,\eta)$ in \eqref{e:u4} ($\alpha=1$). Thus,
we get the estimates \eqref{est:kn3-1} and \eqref{est:kn3-2} as Lemma \ref{lem:Ku1}.
Moreover, by \eqref{est:K-tt2t2} again, we have
\begin{align*}
|A^\beta\xi\big(\pp_{tt}+A^2\pp_t+A^2\big)\widehat K(t,\xi,\eta)|\lesssim \chi_{A\ge 1}A^{-1+\beta} e^{-ct}
+ \chi_{A\le  1}A^\beta e^{-\frac14 A^2t}+\frac{\xi}{A^{2-\beta}}\chi_{|\xi|\lesssim A^2}e^{-\frac{\xi^2}{2A^2}t}.
\end{align*}
Then \eqref{est:kn3-3}--\eqref{est:kn3-6} follow by integration directly and using Lemma \ref{lem:basic-2}.
\end{proof}

\begin{lemma}\label{lem:Kn8} Let $1\le q\le \infty$, $N\gtrsim 1$. Then for any $\beta\in [\frac12,1],$
\begin{align}
\big\|\xi^2\big(\pp_{tt}+A^2\pp_t+A^2\big)\widehat K(t,\xi,\eta)\big\|_{L^q_{\xi\eta}(A\sim N)}
&\lesssim N^{\frac2q}\langle t\rangle^{-1-\frac1{2q}};\label{est:Kn9-3}\\
\big\|\xi^2\big(\pp_{tt}+A^2\pp_t+A^2\big)\widehat K(t,\xi,\eta)\big\|_{L^q_{\xi\eta}(A\le 1)}
&\lesssim \langle t\rangle^{-\frac12-\frac1{q}};\label{est:Kn9-4}\\
\big\|A^{\beta}\xi^2\big(\pp_{tt}+A^2\pp_t+A^2\big)\widehat K(t,\xi,\eta)\big\|_{L^\infty_{\xi\eta}(A\le 1)}\lesssim & \langle t\rangle^{-\frac{1+\beta}2}.\label{est:ku5-4}
\end{align}
Moreover,
\begin{align}
\big\|\xi^2\big(\pp_{tt}+A^2\pp_t+A^2\big)\widehat K(t,\xi,\eta)\big\|_{L^2_{\xi}L^\infty_\eta(A\le 1)}\lesssim & \langle t\rangle^{-\frac34};\label{est:ku5-1}\\
\big\|\xi^2\big(\pp_{tt}+A^2\pp_t+A^2\big)\widehat K(t,\xi,\eta)\big\|_{L^2_{\xi}L^\infty_\eta(A\sim N)}\lesssim & N^{\frac12}\langle t\rangle^{-1}.\label{est:ku5-2}
\end{align}
%
\end{lemma}
\begin{proof}
By \eqref{est:K-tt2t2},
\begin{align}
|\xi^2\big(\pp_{tt}+A^2\pp_t+A^2\big)\widehat K(t,\xi,\eta)|\lesssim & \chi_{A\ge 1} e^{-ct}
+ \chi_{A\le  1}A e^{-\frac14 A^2t}+\frac{\xi^2}{A^2}\chi_{|\xi|\lesssim A^2}e^{-\frac{\xi^2}{2A^2}t};\label{15.40}\\
|A^\beta\xi^2\big(\pp_{tt}+A^2\pp_t+A^2\big)\widehat K(t,\xi,\eta)|\lesssim & \chi_{A\ge 1} A^\beta e^{-ct}
+ \chi_{A\le  1}A^{1+\beta} e^{-\frac14 A^2t}+\frac{\xi^2}{A^{2-\beta}}\chi_{|\xi|\lesssim A^2}e^{-\frac{\xi^2}{2A^2}t}.\label{15.41}
\end{align}
Then similar as the proof in Lemma \ref{lem:Ku1''}, we obtain \eqref{est:Kn9-3} and \eqref{est:Kn9-4}. \eqref{est:ku5-1} and \eqref{est:ku5-2} follow from \eqref{15.40} and Lemma \ref{lem:basic-2} directly. Also, from  \eqref{15.41}  and Lemma \ref{lem:basic-2},  \eqref{est:ku5-4} are followed.
\end{proof}

Now we give the estimates on $(\partial_{tt}-\Delta\partial_t-\pp_{xx})$.
\begin{lemma}\label{lem:Kn11} Let $1\le q\le \infty$ and $N\gtrsim 1$. Then
\begin{align}\label{est:kn11-1}
\big\|(\partial_{tt}+A^2\partial_t+\xi^2)\widehat{K}(t,\xi,\eta)\big\|_{L^q_{\xi\eta}(A\sim N)}\lesssim & N^{-2+\frac2q}\langle t\rangle^{-1-\frac1{2q}};\\
\big\|A(\partial_{tt}+A^2\partial_t+\xi^2)\widehat{K}(t,\xi,\eta)\big\|_{L^q_{\xi\eta}(A\le 1)}\lesssim &\langle t\rangle^{-\frac1q};\label{est:kn11-2}\\
\big\|A^2(\partial_{tt}+A^2\partial_t+\xi^2)\widehat{K}(t,\xi,\eta)\big\|_{L^\infty_{\xi\eta}(A\ge 1)}\lesssim &\langle t\rangle^{-1};\label{est:kn11-5}
\end{align}%
and for any $\beta_1\in[\frac12,2]$, $\beta_2\in [0,2]$,
\begin{align}
\big\|A^{\beta_1}(\partial_{tt}+A^2\partial_t+\xi^2)\widehat{K}(t,\xi,\eta)\big\|_{L^2_{\xi\eta}(A\le 1)}\lesssim& \langle t\rangle^{-\frac{\beta_1}2};\label{est:kn11-3}\\
\big\|A^{1+\beta_2}(\partial_{tt}+A^2\partial_t+\xi^2)\widehat{K}(t,\xi,\eta)\big\|_{L^\infty_{\xi\eta}(A\le 1)}\lesssim& \langle t\rangle^{-\frac{\beta_2}2}.\label{est:kn11-4}
\end{align}
\end{lemma}
\begin{proof}
From \eqref{est:K-tt2t2-x} and \eqref{e:v2}, the function $A(\partial_{tt}+A^2\partial_t+\xi^2)\widehat{K}(t,\xi,\eta)$ obeys the similar estimates on $\eta(\partial_{tt}+A^2\partial_t)\widehat{K}(t,\xi,\eta)$. So the estimates can be obtained by the same way in the proof of Lemma \ref{lem:Kn4}.
\end{proof}

Much
similar as $A^2 (\partial_{tt}+A^2\partial_t)\widehat{K}$, we have the estimate on the operator $\partial_t(\partial_{tt}-\Delta\partial_t-\Delta)$ as following.
\begin{lemma}\label{lem:Kn5} Let $1\le q\le \infty$ and $N\gtrsim 1$. Then
\begin{align}
\big\|\pp_t (\partial_{tt}+A^2\partial_t+A^2)\widehat K(t,\xi,\eta)\big\|_{L^q_{\xi\eta}(A\le 1)}
\lesssim& \langle t\rangle^{-\frac1{q}};\label{est:Kn6}\\
\big\|A\pp_t (\partial_{tt}+A^2\partial_t+A^2)\widehat K(t,\xi,\eta)\big\|_{L^q_{\xi\eta}(A\le 1)}
\lesssim& \langle t\rangle^{-\frac12-\frac1{q}};\label{est:Kn6-1}\\
\big\|\pp_t (\partial_{tt}+A^2\partial_t+A^2)\widehat K(t,\xi,\eta)\big\|_{L^q_{\xi\eta}(A\sim N)}
\lesssim&  N^{-2(1-\frac1q)}\langle t\rangle^{-1-\frac1{2q}}.\label{est:Kn6-2}
\end{align}
Moreover, for $\beta\in [0,\frac32]$,
\begin{align}
\big\|A^{\beta}\pp_t (\partial_{tt}+A^2\partial_t+A^2)\widehat{K}(t,\xi,\eta)\big\|_{L^2_{\xi\eta}(A\le 1)}\lesssim &\langle t\rangle^{-\frac12-\frac\beta2};\label{est:ku4-1}\\
\big\|A^2\pp_t (\partial_{tt}+A^2\partial_t+A^2)\widehat K(t,\xi,\eta)\big\|_{L^q_{\xi\eta}(A\le 1)}
\lesssim & \langle t\rangle^{-1-\frac1{2q}}.\label{est:Kn8-1}
\end{align}
\end{lemma}
\begin{proof}
By \eqref{est:K-ttt2t2} and \eqref{est:Kt-1}, the function $\pp_t (\partial_{tt}+A^2\partial_t+A^2)\widehat K(t,\xi,\eta)$ obeys the similar estimates on $A\eta\partial_t\widehat{K}(t,\xi,\eta)$. So \eqref{est:Kn6} and \eqref{est:Kn6-2} are followed similarly as \eqref{est:kn2-1} and \eqref{est:kn2-2}.
Again, by \eqref{est:K-ttt2t2},
\begin{align*}
|A\pp_t\big(\pp_{tt}+A^2\pp_t+A^2\big)\widehat K(t,\xi,\eta)|
\lesssim &  \chi_{A\ge 1} \frac1{A} e^{-ct}+\chi_{A\le 1}Ae^{-\frac14 A^2t}
+\chi_{|\xi|\lesssim A^2 }\frac{\xi^2}{A^3}e^{-\frac{\xi^2}{2A^2}t}.
\end{align*}
Therefore, by \eqref{est:At} and \eqref{est:Axit-2},
\begin{align*}
\big\|A\pp_t (\partial_{tt}+A^2\partial_t+A^2)\widehat K(t,\xi,\eta)\big\|_{L^\infty_{\xi\eta}(A\le 1)}\lesssim t^{-\frac12};\\
\big\|A\pp_t (\partial_{tt}+A^2\partial_t+A^2)\widehat K(t,\xi,\eta)\big\|_{L^1_{\xi\eta}(A\le 1)}\lesssim  t^{-\frac32}.
\end{align*}
Then the conclusion \eqref{est:Kn6-1}  follows from interpolation again. By \eqref{est:K-ttt2t2}, when $A\le1$, we  have
\begin{align*}
|A^\beta\pp_t\big(\pp_{tt}+A^2\pp_t+A^2\big)\widehat K(t,\xi,\eta)|
\lesssim &  \chi_{A\le 1}A^\beta e^{-\frac14 A^2t}
+\chi_{|\xi|\lesssim A^2 }\frac{\xi^2}{A^{4-\beta}}e^{-\frac{\xi^2}{2A^2}t}.
\end{align*}
Then the conclusions \eqref{est:ku4-1} and \eqref{est:Kn8-1} follow from Lemma \ref{lem:basic}.
\end{proof}

In particular, one may find from \eqref{est:K-ttt2t2} that for any $\beta\in [0,2]$,
\begin{align}
\big\|\langle A \rangle^{2-\beta}A^\beta\pp_t (\partial_{tt}+A^2\partial_t+A^2)\widehat K(t,\xi,\eta)\big\|_{L^\infty_{\xi\eta}}
\lesssim & \langle t\rangle^{-\frac\beta2}.\label{est:Kn8-2}
\end{align}

Furthermore,
\begin{lemma}\label{lem:Ku6} Let $1\le q\le \infty$. Then
\begin{align}
\big\|A\xi\pp_t (\partial_{tt}+A^2\partial_t+A^2)\widehat{K}(t,\xi,\eta)\big\|_{L^\infty_{\xi\eta}(A\ge 1)}\lesssim &\langle t\rangle^{-\frac32};\label{est:Ku6-1}\\
\big\|A^2\xi\pp_t (\partial_{tt}+A^2\partial_t+A^2)\widehat K(t,\xi,\eta)\big\|_{L^\infty_{\xi\eta}(A\le 1)}
\lesssim & \langle t\rangle^{-\frac32}.\label{est:Ku6-2}
\end{align}
\end{lemma}
\begin{proof} By \eqref{est:K-ttt2t2}, we have when $A\ge 1$,
\begin{align*}
|A\xi\pp_t\big(\pp_{tt}+A^2\pp_t+A^2\big)\widehat K(t,\xi,\eta)|
\lesssim &  \chi_{A\ge 1} e^{-ct}
+\chi_{|\xi|\lesssim A^2 }\frac{\xi^3}{A^3}e^{-\frac{\xi^2}{2A^2}t}\lesssim \langle t\rangle^{-\frac32}.
\end{align*}
while when $A\le1$,
\begin{align*}
|A^2\xi\pp_t\big(\pp_{tt}+A^2\pp_t+A^2\big)\widehat K(t,\xi,\eta)|
\lesssim & \chi_{A\le 1}A^3e^{-\frac14 A^2t}
+\chi_{|\xi|\lesssim A^2 }\frac{\xi^3}{A^2}e^{-\frac{\xi^2}{2A^2}t}\lesssim \langle t\rangle^{-\frac32}.
\end{align*}
This proves the conclusions.
\end{proof}

We also need the following estimates.
\begin{lemma}\label{lem:Kn9}
Let $\beta\in [0,2]$, then
\begin{align*}
\big\|\langle A \rangle^{2-\beta}A^\beta\pp_{tt}\big(\pp_{tt}+A^2\pp_t+A^2+\eta^2\big)\widehat K(t,\xi,\eta)\big\|_{L^\infty_{\xi\eta}}
\lesssim & \langle t\rangle^{-\frac12-\frac\beta2};\\
\big\|\pp_{tt}\big(\pp_{tt}+A^2\pp_t+A^2+\eta^2\big)\widehat K(t,\xi,\eta)\big\|_{L^2_{\xi\eta}}
\lesssim & \langle t\rangle^{-1}.
\end{align*}
\end{lemma}
\begin{proof} Since
$$
\pp_{tt}\big(\pp_{tt}+A^2\pp_t+A^2+\eta^2\big)\widehat K(t,\xi,\eta)
=\pp_{tt}\big(\pp_{tt}+A^2\pp_t+A^2\big)\widehat K(t,\xi,\eta)+\eta^2\pp_{tt}\widehat K(t,\xi,\eta),
$$
the conclusions follow directly from \eqref{est:K-ttt2tt2}, \eqref{est:Ktt} and Lemma \ref{lem:basic}.
\end{proof}

\subsubsection{Estimates on $\widehat K_1(t,\xi,\eta)$}
We also need the estimates on the operator $K_1(t)$. Note that by the definition of $\widehat K_1(t,\xi,\eta)$, we have
\begin{align*}
|\widehat{K_1}(t,\xi,\eta)|\lesssim \chi_{A\ge 1}e^{-ct}+\chi_{A\le 1}e^{-\frac14 A^2t}+\chi_{|\xi|\lesssim A^2}e^{-\frac{\xi^2}{2A^2}t}.
\end{align*}
Moreover,
\begin{align*}
|\xi\widehat{K_1}(t,\xi,\eta)|\lesssim \chi_{A\ge 1}Ae^{-ct}+\chi_{A\le 1}Ae^{-\frac14 A^2t}+\chi_{|\xi|\lesssim A^2}\xi e^{-\frac{\xi^2}{2A^2}t}.
\end{align*}
Thus by integration and Lemma \ref{lem:basic} we have
\begin{lemma}\label{lem:K1}
Let $N\gtrsim 1$. Then
\begin{align*}
\big\|\widehat{K_1}(t,\xi,\eta)\big\|_{L^2_{\xi}L^\infty_\eta(A\le 1)}+N^{-\frac12}\big\|\widehat{K_1}(t,\xi,\eta)&\big\|_{L^2_{\xi}L^\infty_\eta(A\sim N)}\lesssim  \langle t\rangle^{-\frac1{4}};\\
\big\|A^{-1}\xi\widehat{K_1}(t,\xi,\eta)\big\|_{L^2_{\xi\eta}(A\le 1)}&\lesssim \langle t\rangle^{-\frac12};\\
\big\|\xi\widehat{K_1}(t,\xi,\eta)\big\|_{L^2_{\xi\eta}(A\le 1)}+N^{-2}\big\|\xi\widehat{K_1}(t,\xi,\eta)&\big\|_{L^2_{\xi\eta}(A\sim N)}\lesssim \langle t\rangle^{-\frac34}.
\end{align*}
\end{lemma}

\vskip .5in

\section{Linear estimates on the opertor $K(t)$}
In this section, we establish the decaying estimates of the linear flow, by using the conclusions obtained in the previous section.

\begin{prop}\label{lem:Kn1-L}
For any smooth function $f$,  $\beta>\frac12$,
\begin{equation*}
\begin{array}{ll}
  \big\||\nabla|^{3+\beta}K(t)f\big\|_{L^2_{xy}}\lesssim \langle t\rangle^{-\frac14}\big\|\langle \nabla \rangle^{\beta+}f\big\|_{L^1_{xy}}.
\end{array}
\end{equation*}
\end{prop}
\begin{proof}
By Lemma \ref{lem:Kn1} and Lemma \ref{Young},
\begin{align*}
\big\||\nabla|^{3+\beta}K(t)f\big\|_{L^2_{xy}}
&\lesssim
\bl P_{\le 1}|\nabla|^{3+\beta}K(t)f\bl_{L^2_{xy}}+\sum\limits_{N\ge 1}\big\|P_{N}|\nabla|^{3+\beta}K(t)f\big\|_{L^2_{xy}}\\
&\lesssim
\bl A^{3+\beta}\widehat K(t,\xi,\eta)\bl_{L^2_{\xi\eta}(A\le 1)}\|P_{\le 1}f\|_{L^1_{xy}}+\sum\limits_{N\ge 1}N^{3+\beta}\big\|\widehat K(t,\xi,\eta)\bl_{L^2_{\xi\eta}(A\sim N)}\big\|P_Nf\big\|_{L^1_{xy}}\\
&\lesssim
\langle t\rangle^{-\frac14}\|P_{\le 1}f\|_{L^1_{xy}}+\langle t\rangle^{-\frac14}\sum\limits_{N\ge 1}N^{\beta}\big\|P_Nf\big\|_{L^1_{xy}}\\
&\lesssim
\langle t\rangle^{-\frac14}\big\|\langle \nabla \rangle^{\beta+}f\big\|_{L^1_{xy}}.
\end{align*}
\end{proof}

\begin{prop}\label{lem:Ku1-L}
For any smooth function $f$,  $\beta'\ge \beta$ and $\beta\in[\frac12,\frac32]$,
\begin{equation*}
\begin{array}{ll}
&{\rm (i),}\quad\big\||\nabla|^{\beta'}\partial_{xy}K(t)f\big\|_{L^2_{xy}}
\lesssim \langle t\rangle^{-\frac\beta2}\big\|\langle \nabla \rangle^{\beta'-1+}f\big\|_{L^1_{xy}};\\
&{\rm (ii),}\quad\big\||\nabla|^{\beta'+1}\partial_{xy}K(t)f\big\|_{L^2_{xy}}
\lesssim \langle t\rangle^{-\frac12}\big\|\langle \nabla \rangle^{\beta'-1+}f\big\|_{L^2_{xy}};\\
&{\rm (iii),}\quad\big\|\nabla\partial_{xy}K(t)f\big\|_{L^\infty_{xy}}
\lesssim \langle t\rangle^{-1}\big\|\langle \nabla \rangle^{1+}f\big\|_{L^1_{xy}};
\end{array}
\end{equation*}
\end{prop}
\begin{proof}
By Lemma \ref{lem:Ku1} and Lemma \ref{Young},
\begin{align*}
\big\||\nabla|^{\beta'}\partial_{xy}K(t)f\big\|_{L^2_{xy}}
&\lesssim
\bl P_{\le 1}|\nabla|^{\beta'}\partial_{xy}K(t)f\big\|_{L^2_{xy}}+\sum\limits_{N\ge 1}\big\|P_{N}|\nabla|^{\beta'}\partial_{xy}K(t)f\big\|_{L^2_{xy}}\\
&\lesssim
\bl A^{\beta}\xi\eta\widehat K(t,\xi,\eta)\bl_{L^2_{\xi\eta}(A\le 1)}\|P_{\le 1}f\|_{L^1_{xy}}+\sum\limits_{N\ge 1}N^{\beta'}\big\|\xi\eta\widehat K(t,\xi,\eta)\bl_{L^2_{\xi\eta}(A\sim N)}\big\|P_Nf\big\|_{L^1_{xy}}\\
&\lesssim
\langle t\rangle^{-\frac\beta 2}\|P_{\le 1}f\|_{L^1_{xy}}+\langle t\rangle^{-\frac34}\sum\limits_{N\ge 1}N^{\beta'-1}\big\|P_Nf\big\|_{L^1_{xy}}\\
&\lesssim
\langle t\rangle^{-\frac\beta 2}\big\|\langle \nabla \rangle^{\beta'-1+}f\big\|_{L^1_{xy}}.
\end{align*}
Similarly,
\begin{align*}
\big\||\nabla|^{\beta'+1}\partial_{xy}K(t)f\big\|_{L^2_{xy}}
&\lesssim
\bl A^{\beta+1}\xi\eta\widehat K(t,\xi,\eta)\bl_{L^\infty_{\xi\eta}(A\le 1)}\|P_{\le 1}f\|_{L^2_{xy}}+\sum\limits_{N\ge 1}N^{\beta'+1}\big\|\xi\eta\widehat K(t,\xi,\eta)\bl_{L^\infty_{\xi\eta}(A\sim N)}\big\|P_Nf\big\|_{L^2_{xy}}\\
&\lesssim
\langle t\rangle^{-\frac1 2}\|P_{\le 1}f\|_{L^2_{xy}}+\langle t\rangle^{-\frac1 2}\sum\limits_{N\ge 1}N^{\beta'-1}\big\|P_Nf\big\|_{L^2_{xy}}\\
&\lesssim
\langle t\rangle^{-\frac\beta 2}\big\|\langle \nabla \rangle^{\beta'-1+}f\big\|_{L^1_{xy}};
\end{align*}
and
\begin{align*}
\big\|\nabla\partial_{xy}K(t)f\big\|_{L^\infty_{xy}}
&\lesssim
\bl A\xi\eta\widehat K(t,\xi,\eta)\bl_{L^1_{\xi\eta}(A\le 1)}\|P_{\le 1}f\|_{L^1_{xy}}+\sum\limits_{N\ge 1}N\big\|\xi\eta\widehat K(t,\xi,\eta)\bl_{L^1_{\xi\eta}(A\sim N)}\big\|P_Nf\big\|_{L^1_{xy}}\\
&\lesssim
\langle t\rangle^{-1}\|P_{\le 1}f\|_{L^1_{xy}}+\langle t\rangle^{-1}\sum\limits_{N\ge 1}N\big\|P_Nf\big\|_{L^1_{xy}}\\
&\lesssim
\langle t\rangle^{-1}\big\|\langle \nabla \rangle^{1+}f\big\|_{L^1_{xy}}.
\end{align*}
\end{proof}

\begin{prop}\label{lem:Ku1''-L}
For any smooth function $f$,
\begin{align*}
\mbox{{\rm (i),}}\quad\bl \nabla \partial_{xx}\partial_y K(t)f\bl_{L^2_{xy}}
\lesssim & \min\Big\{\langle t\rangle^{-\frac34}\|\langle \nabla \rangle^{\frac12+}f\|_{L^{\frac43}}, \langle t\rangle^{-1}\|\langle\nabla\rangle^{1+}f\|_{L^1_{xy}}\Big\};\\
{\rm (ii),}\quad\bl \Delta \partial_{xx}\partial_y K(t)f\bl_{L^2_{xy}}
\lesssim & \langle t\rangle^{-1}\big\|\langle \nabla \rangle f\big\|_{L^2_{xy}}.
\end{align*}
\end{prop}
\begin{proof}
The estimates are followed from Lemma \ref{lem:Ku1''}, Lemma \ref{Young} and Sobolev's inequality. First,
\begin{align*}
\bl \nabla \partial_{xx}\partial_y K(t)f\bl_{L^2_{xy}}
&\lesssim
\bl P_{\le 1}\nabla \partial_{xx}\partial_y K(t)f\bl_{L^2_{xy}}+\big\|P_{\ge 1}\nabla \partial_{xx}\partial_y K(t)f\bl_{L^2_{xy}}\\
&\lesssim
\bl A\xi^2\eta\widehat K(t,\xi,\eta)\bl_{L^2_{\xi\eta}(A\le 1)}\|P_{\le 1}f\|_{L^1_{xy}}+\big\|A\xi^2\eta\widehat K(t,\xi,\eta)\bl_{L^\infty_{\xi\eta}(A\ge 1)}\big\|P_{\ge 1}f\big\|_{L^2_{xy}}\\
&\lesssim
\langle t\rangle^{-1}\|\langle\nabla\rangle^{1+}f\|_{L^1_{xy}}.
\end{align*}

On the other hand, by \eqref{est:ku1-5} and \eqref{est:ku1-6} ($q=4$) instead,
\begin{align*}
\bl \nabla \partial_{xx}\partial_y K(t)f\bl_{L^2_{xy}}
&\lesssim
\bl P_{\le 1}\nabla \partial_{xx}\partial_y K(t)f\bl_{L^2_{xy}}+\sum\limits_{N\ge 1}\big\|P_{N}\nabla \partial_{xx}\partial_y K(t)f\bl_{L^2_{xy}}\\
&\lesssim
\bl A\xi^2\eta\widehat K(t,\xi,\eta)\bl_{L^4_{\xi\eta}(A\le 1)}\|P_{\le 1}f\|_{L^\frac43_{xy}}+\sum\limits_{N\ge 1}N^{\frac12}\big\|\xi^2\eta\widehat K(t,\xi,\eta)\bl_{L^4_{\xi\eta}(A\sim N)}\big\|P_{N}f\big\|_{L^\frac43_{xy}}\\
&\lesssim
\langle t\rangle^{-\frac34}\|\langle \nabla \rangle^{\frac12+}f\|_{L^{\frac43}}.
\end{align*}

Furthermore,  by \eqref{est:ku1-7} and \eqref{est:ku1-8},
\begin{align*}
\bl \Delta \partial_{xx}\partial_y K(t)f\bl_{L^2_{xy}}
&\lesssim
\bl P_{\le 1}\Delta \partial_{xx}\partial_y K(t)f\bl_{L^2_{xy}}+\big\|P_{\ge 1}\nabla\cdot  \partial_{xx}\partial_y K(t)\nabla f\bl_{L^2_{xy}}\\
&\lesssim
\bl A^2\xi^2\eta\widehat K(t,\xi,\eta)\bl_{L^\infty_{\xi\eta}(A\le 1)}\|P_{\le 1}f\|_{L^2_{xy}}+\big\|A\xi\eta\widehat K(t,\xi,\eta)\bl_{L^\infty_{\xi\eta}(A\ge 1)}\big\|P_{\ge 1}\nabla f\big\|_{L^2_{xy}}\\
&\lesssim
\langle t\rangle^{-1}\big\|\langle \nabla \rangle f\big\|_{L^2_{xy}}.
\end{align*}
\end{proof}

\begin{prop}\label{lem:Ku3-L}
For any smooth function $f$, $p\in [2,\infty]$, $\beta\in [0,2-\frac 1q]$,
\begin{align*}
\bl |\nabla|^\beta \Delta \big(\Delta+\sqrt{\Delta \partial_{yy}}\big) K(t)f\bl_{L^p_{xy}}
\lesssim  \langle t\rangle^{-\frac1{p'}-\frac\beta2}\big\|\langle \nabla \rangle^{\frac2{p'}+\beta+} f\big\|_{L^1_{xy}}.
\end{align*}
\end{prop}
\begin{proof}
From Lemma \ref{lem:Ku3} and Lemma \ref{Young},
\begin{align*}
&\quad \bl |\nabla|^\beta \Delta \big(\Delta+\sqrt{\Delta \partial_{yy}}\big) K(t)f\bl_{L^p_{xy}}\\
&\lesssim
\bl P_{\le 1}|\nabla|^\beta \Delta \big(\Delta+\sqrt{\Delta \partial_{yy}}\big) K(t)f\bl_{L^p_{xy}}+\sum\limits_{N\ge 1}N^\beta\big\|P_{N}\Delta \big(\Delta+\sqrt{\Delta \partial_{yy}}\big) K(t)f\bl_{L^p_{xy}}\\
&\lesssim
\big\|A^\beta A^2(A^2-A|\eta|)\widehat{K}(t,\xi,\eta)\big\|_{L^{p'}_{\xi\eta}(A\le 1)}\|P_{\le 1}f\|_{L^1_{xy}}+\sum\limits_{N\ge 1}N^{\beta}\big\|A^2(A^2-A|\eta|)\widehat{K}(t,\xi,\eta)\big\|_{L^{p'}_{\xi\eta}(A\sim N)}\big\|P_{N}f\big\|_{L^1_{xy}}\\
&\lesssim
\langle t\rangle^{-\frac1{p'}-\frac\beta2}\big\|\langle \nabla \rangle^{\frac2{p'}+\beta+} f\big\|_{L^1_{xy}}.
\end{align*}
\end{proof}

\begin{prop}\label{lem:Kn2-L}
For any smooth function $f$,  $\beta\in (\frac12,\frac 52]$,
\begin{equation*}
\begin{array}{ll}
&{\rm (i),}\quad\bl |\nabla|^\beta  \partial_{y}\partial_t K(t)f\bl_{L^2_{xy}}
\lesssim
 \langle t\rangle^{-\frac\beta2}\big\|\langle \nabla \rangle^{\beta-2+} f\big\|_{L^1_{xy}};\\
&{\rm (ii),}\qquad\bl \nabla\partial_{y}\partial_t K(t)f\bl_{L^\infty_{xy}}
\lesssim
 \langle t\rangle^{-1}\big\|\langle \nabla \rangle^{0+} f\big\|_{L^1_{xy}}.
\end{array}
\end{equation*}
Moreover, for any $\beta'\in [0,2]$,
\begin{align}\label{est:Kn2-L-3}
{\rm (iii),}\qquad\bl |\nabla|^{\beta'}\nabla  \partial_{y}\partial_t K(t)f\bl_{L^2_{xy}}
\lesssim & \langle t\rangle^{-\frac\beta2}\big\|f\big\|_{L^2_{xy}}.\quad
\end{align}
\end{prop}
\begin{proof}
From Lemma \ref{lem:Kn2} and Lemma \ref{Young},
\begin{align*}
\bl |\nabla|^\beta  \partial_{y}\partial_t K(t)f\bl_{L^2_{xy}}
&\lesssim
\bl P_{\le 1}|\nabla|^\beta  \partial_{y}\partial_t K(t)f\bl_{L^2_{xy}}+\sum\limits_{N\ge 1}N^\beta\big\|P_{N} \partial_{y}\partial_t K(t)f\bl_{L^2_{xy}}\\
&\lesssim
\big\|A^\beta\eta\partial_t\widehat{K}(t,\xi,\eta)\big\|_{L^2_{\xi\eta}(A\lesssim 1)}\|P_{\le 1}f\|_{L^1_{xy}}+\sum\limits_{N\ge 1}N^{\beta}\big\|\eta\partial_t\widehat{K}(t,\xi,\eta)\big\|_{L^2_{\xi\eta}(A\sim N)}\big\|P_{N}f\big\|_{L^1_{xy}}\\
&\lesssim
\langle t\rangle^{-\frac\beta2}\|P_{\le 1}f\|_{L^1_{xy}}+\langle t\rangle^{-1-\frac{1}{2q}}\sum\limits_{N\ge 1}N^{\beta-2}\big\|P_{N}f\big\|_{L^1_{xy}}\\
&\lesssim
\langle t\rangle^{-\frac\beta2}\big\|\langle \nabla \rangle^{\beta-2+} f\big\|_{L^1_{xy}}.
\end{align*}
By the similar way,  using \eqref{est:kn2-1} and \eqref{est:kn2-2} ($q=1$) instead, we get (ii).
For \eqref{est:Kn2-L-3}, we use the special estimate \eqref{est:kn2'-3}, and get
\begin{align*}
\bl |\nabla|^{\beta'}\nabla  \partial_{y}\partial_t K(t)f\bl_{L^2_{xy}}
&\lesssim
\big\| A^{\beta}A\eta\widehat{\partial_tK}(t,\xi,\eta)\big\|_{L^\infty_{\xi\eta}}\big\|f\big\|_{L^2_{xy}}\\
&\lesssim
\langle t\rangle^{-\frac\beta2}\big\|f\big\|_{L^2_{xy}}.
\end{align*}
\end{proof}

\begin{prop}\label{lem:Kn4-L}
For any smooth function $f$,  $\beta\in [0,1]$,
\begin{equation*}
\begin{array}{ll}
&{\rm (i),}\quad\bl |\nabla|^\beta  \partial_{y}\big(\partial_{tt}-\Delta\partial_t\big) K(t)f\bl_{L^2_{xy}}
\lesssim
 \langle t\rangle^{-\frac{\beta+1}2}\big\|\langle \nabla \rangle^{\beta-1+} f\big\|_{L^1_{xy}};\\
&{\rm (ii),}\quad\bl |\nabla|^{\beta+1}  \partial_{y}\big(\partial_{tt}-\Delta\partial_t\big) K(t)f\bl_{L^2_{xy}}
\lesssim
\langle t\rangle^{-\frac{\beta+1}2}\big\|\langle \nabla \rangle^{\beta} f\big\|_{L^2_{xy}};\\
&{\rm (iii),}\qquad\bl \nabla  \partial_{y}(\partial_{tt}-\Delta\partial_t\big) K(t)f\bl_{L^\infty_{xy}}
\lesssim
 \langle t\rangle^{-1}\big\|\langle \nabla \rangle^{1+} f\big\|_{L^2_{xy}}.
\end{array}
\end{equation*}
\end{prop}
\begin{proof}
From Lemma \ref{lem:Kn4} and Lemma \ref{Young},
\begin{align*}
&\quad \bl |\nabla|^\beta  \partial_{y}\big(\partial_{tt}-\Delta\partial_t\big) K(t)f\bl_{L^2_{xy}}\\
&\lesssim
\bl P_{\le 1}|\nabla|^\beta   \partial_{y}\big(\partial_{tt}-\Delta\partial_t\big) K(t)f\bl_{L^2_{xy}}+\sum\limits_{N\ge 1}N^\beta\big\|P_{N} \partial_{y}\big(\partial_{tt}-\Delta\partial_t\big) K(t)f\bl_{L^2_{xy}}\\
&\lesssim
\big\|A^\beta\eta(\partial_{tt}+A^2\partial_t)\widehat{K}(t,\xi,\eta)\big\|_{L^2_{\xi\eta}(A\lesssim 1)}\|P_{\le 1}f\|_{L^1_{xy}}+\sum\limits_{N\ge 1}N^{\beta}\big\|\eta(\partial_{tt}+A^2\partial_t)\widehat{K}(t,\xi,\eta)\big\|_{L^2_{\xi\eta}(A\sim N)}\big\|P_{N}f\big\|_{L^1_{xy}}\\
&\lesssim
\langle t\rangle^{-\frac{\beta+1}2}\|P_{\le 1}f\|_{L^1_{xy}}+\langle t\rangle^{-1-\frac{1}{2q}}\sum\limits_{N\ge 1}N^{\beta-1}\big\|P_{N}f\big\|_{L^1_{xy}}\\
&\lesssim
\langle t\rangle^{-\frac{\beta+1}2}\big\|\langle \nabla \rangle^{\beta-1+} f\big\|_{L^1_{xy}}.
\end{align*}
Moreover, by \eqref{est:kn4-1} ($q=\infty$) and \eqref{est:kn4-4} instead,
\begin{align*}
&\quad \bl |\nabla|^{\beta+1}  \partial_{y}\big(\partial_{tt}-\Delta\partial_t\big) K(t)f\bl_{L^2_{xy}}\\
&\lesssim
\bl P_{\le 1}|\nabla|^{\beta+1}   \partial_{y}\big(\partial_{tt}-\Delta\partial_t\big) K(t)f\bl_{L^2_{xy}}+\big\|P_{\ge 1} \nabla \partial_{y}\big(\partial_{tt}-\Delta\partial_t\big) K(t) |\nabla|^\beta f\bl_{L^2_{xy}}\\
&\lesssim
\big\|A^{\beta+1}\eta(\partial_{tt}+A^2\partial_t)\widehat{K}(t,\xi,\eta)\big\|_{L^\infty_{\xi\eta}(A\lesssim 1)}\|P_{\le 1}f\|_{L^2_{xy}}+\big\|A\eta(\partial_{tt}+A^2\partial_t)\widehat{K}(t,\xi,\eta)\big\|_{L^\infty_{\xi\eta}(A\ge 1)}\big\|\langle \nabla \rangle^\beta f\big\|_{L^2_{xy}}\\
&\lesssim
\langle t\rangle^{-\frac{\beta+1}2}\|f\|_{L^2_{xy}}+\langle t\rangle^{-1}\big\|\langle \nabla \rangle^\beta f\big\|_{L^2_{xy}}\\
&\lesssim
\langle t\rangle^{-\frac{\beta+1}2}\big\|\langle \nabla \rangle^{\beta} f\big\|_{L^2_{xy}}.
\end{align*}
Similarly, by \eqref{est:kn4-1} ($q=2$), \eqref{est:kn4-4},  and Sobolev's inequality,
\begin{align*}
&\quad \bl \nabla  \partial_{y}(\partial_{tt}-\Delta\partial_t\big) K(t)f\bl_{L^\infty_{xy}}\\
&\lesssim
\bl P_{\le 1}\nabla   \partial_{y}\big(\partial_{tt}-\Delta\partial_t\big) K(t)f\bl_{L^\infty_{xy}}+\big\|P_{\ge 1} \nabla \partial_{y}\big(\partial_{tt}-\Delta\partial_t\big) K(t)  f\bl_{L^\infty_{xy}}\\
&\lesssim
\big\|A\eta(\partial_{tt}+A^2\partial_t)\widehat{K}(t,\xi,\eta)\big\|_{L^2_{\xi\eta}(A\lesssim 1)}\|P_{\le 1}f\|_{L^2_{xy}}+\big\|A\eta(\partial_{tt}+A^2\partial_t)\widehat{K}(t,\xi,\eta)\big\|_{L^\infty_{\xi\eta}(A\ge 1)}\big\| f\big\|_{L^\infty_{xy}}\\
&\lesssim
\langle t\rangle^{-1}\big\|\langle \nabla \rangle^{1+} f\big\|_{L^2_{xy}}.
\end{align*}
\end{proof}

\begin{prop}\label{lem:Ku2-L}
For any smooth function $f$, $p\in [2,\infty]$, $\beta\in [0,2-\frac 1q]$,
\begin{align*}
\bl |\nabla|^\beta\Delta\big(\partial_{tt}-\Delta\partial_t\big)K(t)f\bl_{L^p_{xy}}
\lesssim  \langle t\rangle^{-\frac1{p'}-\frac\beta2}\big\|\langle \nabla \rangle^{\frac2{p'}+\beta+} f\big\|_{L^1_{xy}}.
\end{align*}
\end{prop}
\begin{proof}
From Lemma \ref{lem:Ku3} and Lemma \ref{lem:Ku2}, the operator $\Delta\big(\partial_{tt}-\Delta\partial_t\big)K(t)$ obeys the same estimates as $ \Delta \big(\Delta+\sqrt{\Delta \partial_{yy}}\big)K(t)$. Hence the lemma follows from the same way as in the proof of Proposition  \ref{lem:Ku3-L}.
\end{proof}

\begin{prop}\label{lem:Kn6-L}
For any smooth function $f$,
\begin{align*}
\bl\Delta(\partial_{tt}-\Delta\partial_t-\Delta)K(t)f\bl_{L^2_{xy}}\lesssim \langle t\rangle^{-\frac14}\min\Big\{\|\langle \nabla \rangle^{\frac12+}f\|_{L^1_xL^2_y}, \|\langle \nabla \rangle^{1+}f\|_{L^1_{xy}}\Big\}.
\end{align*}
\end{prop}
\begin{proof}
From Lemma \ref{lem:Kn6} and Lemma \ref{Young},
\begin{align*}
&\quad \bl \Delta(\partial_{tt}-\Delta\partial_t-\Delta)K(t)f\bl_{L^2_{xy}}\\
&\lesssim
\bl P_{\le 1}\Delta(\partial_{tt}-\Delta\partial_t-\Delta)K(t)f\bl_{L^2_{xy}}+\sum\limits_{N\ge 1}\big\|P_{N} \Delta(\partial_{tt}-\Delta\partial_t-\Delta)K(t)f\bl_{L^2_{xy}}\\
&\lesssim
\big\|A^2(\partial_{tt}+A^2\partial_t+A^2)\widehat K(t,\xi,\eta)\big\|_{L^2_{\xi}L^\infty_{\eta}(A\le 1)}\|P_{\le 1}f\|_{L^1_{x}L^2_y}\\
&\qquad +\sum\limits_{N\ge 1}\big\|A^2(\partial_{tt}+A^2\partial_t+A^2)\widehat K(t,\xi,\eta)\big\|_{L^2_{\xi}L^\infty_{\eta}(A\sim N)}\big\|P_{N}f\big\|_{L^1_{x}L^2_y}\\
&\lesssim
\langle t\rangle^{-\frac14}\big\|\langle \nabla \rangle^{\frac12+} f\big\|_{L^1_{x}L^2_y}.
\end{align*}
Further, by the Sobolev inequlity, $\big\|\langle \nabla \rangle^{\frac12+} f\big\|_{L^1_{x}L^2_y}\lesssim \|\langle \nabla \rangle^{1+}f\|_{L^1_{xy}}$, we prove the lemma.
\end{proof}

\begin{prop}\label{lem:Kn3-L}
For any smooth function $f$,  $2\le p\le \infty$, and $\beta\ge \frac12$,
\begin{equation*}
\begin{array}{ll}
&{\rm (i),}\quad\bl\pp_x(\partial_{tt}-\Delta\partial_t-\Delta)K(t)f\bl_{L^p_{xy}}
\lesssim
 \langle t\rangle^{-\frac1{p'}}\|\langle \nabla \rangle^{1-\frac2p+}f\|_{L^1_{xy}};\\
&{\rm (ii),}\quad\bl\nabla \pp_x(\partial_{tt}-\Delta\partial_t-\Delta)K(t)f\bl_{L^2_{xy}}
\lesssim
 \langle t\rangle^{-\frac34}\|\langle \nabla \rangle^{\frac12+}f\|_{L^1_{x}L^2_y};\\
&{\rm (iii),}\quad\bl|\nabla|^\beta \pp_x(\partial_{tt}-\Delta\partial_t-\Delta)K(t)f\bl_{L^2_{xy}}
\lesssim
 \langle t\rangle^{-\frac34}\|\langle \nabla \rangle^{1+}f\|_{L^1_{xy}};\\
&{\rm (iv),}\quad\bl|\nabla|^\beta\pp_x(\partial_{tt}-\Delta\partial_t-\Delta)K(t)f\bl_{L^\infty_{xy}}
\lesssim
 \langle t\rangle^{-1}\min\Big\{\|\langle \nabla \rangle^{1+\beta+}f\|_{L^1_{xy}}+\|\langle \nabla \rangle^{\frac12+\beta+}f\|_{L^1_{x}L^2_y}\Big\}.
\end{array}
\end{equation*}
\end{prop}
\begin{proof}
From Lemma \ref{lem:Kn3} and Lemma \ref{lem:Ku1}, we find that $\xi(\partial_{tt}+A^2\partial_t+A^2)$ and $A\xi\eta$ obey the same $L^p$ estimates. Hence, the parts (i) and (iii) in the present proposition can be obtained by the same way as Proposition \ref{lem:Ku1-L} (i) and (iii).
%
For part (ii), by \eqref{est:kn3-3}, \eqref{est:kn3-4} and Lemma \ref{Young},
\begin{align*}
&\quad \bl\nabla \pp_x(\partial_{tt}-\Delta\partial_t-\Delta)K(t)f\bl_{L^2_{xy}}\\
&\lesssim
\bl P_{\le 1}\nabla\pp_x(\partial_{tt}-\Delta\partial_t-\Delta)K(t)f\bl_{L^2_{xy}}+\sum\limits_{N\ge 1}\big\|P_{N} \nabla\pp_x(\partial_{tt}-\Delta\partial_t-\Delta)K(t)f\bl_{L^2_{xy}}\\
&\lesssim
\big\|A\xi(\partial_{tt}+A^2\partial_t+A^2)\widehat K(t,\xi,\eta)\big\|_{L^2_{\xi}L^\infty_\eta(A\le 1)}\|P_{\le 1}f\|_{L^1_{x}L^2_y}\\
&\qquad +\sum\limits_{N\ge 1}\big\|A\xi(\partial_{tt}+A^2\partial_t+A^2)\widehat K(t,\xi,\eta)\big\|_{L^2_{\xi}L^\infty_\eta(A\sim N)}\big\|P_{N}f\big\|_{L^1_{x}L^2_y}\\
&\lesssim
 \langle t\rangle^{-\frac34}\|\langle \nabla \rangle^{\frac12+}f\|_{L^1_{x}L^2_y}.
\end{align*}
Similarly,
\begin{align*}
&\quad \bl|\nabla|^\beta \pp_x(\partial_{tt}-\Delta\partial_t-\Delta)K(t)f\bl_{L^\infty_{xy}}\\
&\lesssim
\bl P_{\le 1}|\nabla|^\beta\pp_x(\partial_{tt}-\Delta\partial_t-\Delta)K(t)f\bl_{L^\infty_{xy}}+\sum\limits_{N\ge 1}\big\|P_{N} |\nabla|^\beta\pp_x(\partial_{tt}-\Delta\partial_t-\Delta)K(t)f\bl_{L^\infty_{xy}}\\
&\lesssim
\big\|A^\beta\xi(\partial_{tt}+A^2\partial_t+A^2)\widehat K(t,\xi,\eta)\big\|_{L^1_{\xi}L^2_\eta(A\le 1)}\|P_{\le 1}f\|_{L^1_{x}L^2_y}\\
&\qquad +\sum\limits_{N\ge 1}\big\|A^\beta\xi(\partial_{tt}+A^2\partial_t+A^2)\widehat K(t,\xi,\eta)\big\|_{L^1_{\xi}L^2_\eta(A\sim N)}\big\|P_{N}f\big\|_{L^1_{x}L^2_y}\\
&\lesssim
 \langle t\rangle^{-1}\|\langle \nabla \rangle^{\frac12+\beta+}f\|_{L^1_{x}L^2_y}.
\end{align*}
At last, we use the Sobolev inequality, $\|g\|_{L^1_{x}L^2_y}\lesssim \|\langle \nabla\rangle^{\frac12+}g\|_{L^1_{xy}}$, to complete the proof of the lemma.
\end{proof}

\begin{prop}\label{lem:Kn8-L}
For any smooth function $f$,  $\beta\in [\frac12,1]$,
\begin{equation*}
\begin{array}{ll}
&{\rm (i),}\quad\bl\pp_{xx}(\partial_{tt}-\Delta\partial_t-\Delta)K(t)f\bl_{L^2_{xy}}
\lesssim
 \min\Big\{\langle t\rangle^{-\frac34}\|\langle \nabla \rangle^{\frac12+}f\|_{L^1_{x}L^2_y}, \langle t\rangle^{-1}\|\langle \nabla \rangle^{1+}f\|_{L^1_{xy}}\Big\};\\
&{\rm (ii),}\quad\bl|\nabla|^\beta \pp_{xx}(\partial_{tt}-\Delta\partial_t-\Delta)K(t)f\bl_{L^2_{xy}}
\lesssim
 \langle t\rangle^{-\frac{\beta+1}{2}}\|\langle \nabla \rangle^{\beta+}f\|_{L^2_{xy}};\\
&{\rm (iii),}\quad\bl\pp_{xx}(\partial_{tt}-\Delta\partial_t-\Delta)K(t)f\bl_{L^\infty_{xy}}
\lesssim
  \langle t\rangle^{-1}\|\langle \nabla \rangle^{1+}f\|_{L^2_{xy}}.
\end{array}
\end{equation*}
\end{prop}
\begin{proof}
From Lemma \ref{lem:Kn8} and Lemma \ref{lem:Ku1''}, we note that $\xi^2(\partial_{tt}+A^2\partial_t+A^2)$ and $A\xi^2\eta$ obey the same $L^p$ estimates. Thus, the same way in the proof of   Proposition \ref{lem:Ku1''-L} (i) gives part (i) in this proposition.
%
Moreover, for part (ii) from \eqref{est:Kn9-3} ($q=\infty$),  \eqref{est:ku5-4} and Lemma \ref{Young},
\begin{align*}
&\quad \bl|\nabla|^\beta \pp_{xx}(\partial_{tt}-\Delta\partial_t-\Delta)K(t)f\bl_{L^2_{xy}}\\
&\lesssim
\bl P_{\le 1}|\nabla|^\beta\pp_{xx}(\partial_{tt}-\Delta\partial_t-\Delta)K(t)f\bl_{L^2_{xy}}+\sum\limits_{N\ge 1}\big\|P_{N} |\nabla|^\beta\pp_{xx}(\partial_{tt}-\Delta\partial_t-\Delta)K(t)f\bl_{L^2_{xy}}\\
&\lesssim
\big\|A^\beta\xi^2(\partial_{tt}+A^2\partial_t+A^2)\widehat K(t,\xi,\eta)\big\|_{L^\infty_{\xi\eta}(A\le 1)}\|P_{\le 1}f\|_{L^2_{xy}}\\
&\qquad +\sum\limits_{N\ge 1}N^\beta\big\|\xi^2(\partial_{tt}+A^2\partial_t+A^2)\widehat K(t,\xi,\eta)\big\|_{L^\infty_{\xi\eta}(A\sim N)}\big\|P_{N}f\big\|_{L^2_{xy}}\\
&\lesssim
 \langle t\rangle^{-\frac{1+\beta}{2}}\|\langle \nabla \rangle^{\beta+}f\|_{L^2_{xy}}.
\end{align*}
Similarly as (i), and using \eqref{est:Kn9-3} and \eqref{est:Kn9-3} ($q=2$)  instead,   we get (iii).
\end{proof}

\begin{prop}\label{lem:Kn11-L}
For any smooth function $f$,  $\beta_1\in [\frac12,2], \beta_2\in [0,2]$,
\begin{equation*}
\begin{array}{ll}
&{\rm (i),}\quad\bl|\nabla|^{\beta_1}(\partial_{tt}-\Delta\partial_t-\pp_{xx})K(t)f\bl_{L^2_{xy}}
\lesssim
 \langle t\rangle^{-\frac{\beta_1}{2}}\|\langle \nabla \rangle^{\beta_1-1+}f\|_{L^1_{xy}};\\
&{\rm (ii),}\quad\bl|\nabla|^{\beta_2}\nabla (\partial_{tt}-\Delta\partial_t-\pp_{xx})K(t)f\bl_{L^2_{xy}}
\lesssim
 \langle t\rangle^{-\frac{\beta_2}{2}}\|\langle \nabla \rangle^{\beta_2-1}f\|_{L^2_{xy}}.;\\
&{\rm (iii),}\quad\bl\nabla (\partial_{tt}-\Delta\partial_t-\pp_{xx})K(t)f\bl_{L^\infty_{xy}}
\lesssim
  \langle t\rangle^{-1}\|\langle \nabla \rangle^{1+}f\|_{L^1_{xy}};\\
&{\rm (iv),}\quad\bl\Delta (\partial_{tt}-\Delta\partial_t-\pp_{xx})K(t)f\bl_{L^\infty_{xy}}
\lesssim
  \langle t\rangle^{-1}\|\langle \nabla \rangle^{1+}f\|_{L^2_{xy}}.
\end{array}
\end{equation*}
\end{prop}
\begin{proof}
From Lemma \ref{lem:Kn11} and Lemma \ref{Young},
\begin{align*}
&\quad \bl|\nabla|^{\beta_1}(\partial_{tt}-\Delta\partial_t-\pp_{xx})K(t)f\bl_{L^2_{xy}}\\
&\lesssim
\bl P_{\le 1}|\nabla|^{\beta_1}(\partial_{tt}-\Delta\partial_t-\pp_{xx})K(t)f\bl_{L^2_{xy}}+\sum\limits_{N\ge 1}\big\|P_{N}|\nabla|^{\beta_1}(\partial_{tt}-\Delta\partial_t-\pp_{xx})K(t)f\bl_{L^2_{xy}}\\
&\lesssim
\big\|A^{\beta_1}(\partial_{tt}+A^2\partial_t+\xi^2)\widehat K(t,\xi,\eta)\big\|_{L^2_{\xi\eta}(A\le 1)}\|P_{\le 1}f\|_{L^1_{xy}}\\
&\qquad +\sum\limits_{N\ge 1}N^{\beta_1}\big\|(\partial_{tt}+A^2\partial_t+\xi^2)\widehat K(t,\xi,\eta)\big\|_{L^2_{\xi\eta}(A\sim N)}\big\|P_{N}f\big\|_{L^1_{xy}}\\
&\lesssim
 \langle t\rangle^{-\frac{\beta_1}{2}}\|\langle \nabla \rangle^{\beta_1-1+}f\|_{L^1_{xy}}.
\end{align*}
By \eqref{est:kn11-5} and \eqref{est:kn11-4} instead, we have
\begin{align*}
&\quad \bl|\nabla|^{\beta_2}\nabla (\partial_{tt}-\Delta\partial_t-\pp_{xx})K(t)f\bl_{L^2_{xy}}\\
&\lesssim
\bl P_{\le 1}|\nabla|^{\beta_2}(\partial_{tt}-\Delta\partial_t-\pp_{xx})K(t)f\bl_{L^2_{xy}}+\big\|P_{\ge1}|\nabla|^{\beta_2}(\partial_{tt}-\Delta\partial_t-\pp_{xx})K(t)f\bl_{L^2_{xy}}\\
&\lesssim
\big\|A^{\beta_2+1}(\partial_{tt}+A^2\partial_t+\xi^2)\widehat K(t,\xi,\eta)\big\|_{L^\infty_{\xi\eta}(A\le 1)}\|P_{\le 1}f\|_{L^2_{xy}}\\
&\qquad +\big\|A^2(\partial_{tt}+A^2\partial_t+\xi^2)\widehat K(t,\xi,\eta)\big\|_{L^\infty_{\xi\eta}(A\ge 1)}\big\|P_{\ge1}\langle \nabla\rangle^{\beta_2-1}f\big\|_{L^2_{xy}}\\
&\lesssim
 \langle t\rangle^{-\frac{\beta_2}{2}}\|\langle \nabla \rangle^{\beta_2-1}f\|_{L^2_{xy}}.
\end{align*}
Simlarly,
\begin{align*}
&\quad \bl\nabla(\partial_{tt}-\Delta\partial_t-\pp_{xx})K(t)f\bl_{L^\infty_{xy}}\\
&\lesssim
\big\|A(\partial_{tt}+A^2\partial_t+\xi^2)\widehat K(t,\xi,\eta)\big\|_{L^1_{\xi\eta}(A\le 1)}\|P_{\le 1}f\|_{L^1_{xy}}\\
&\qquad +\sum\limits_{N\ge 1}N\big\|(\partial_{tt}+A^2\partial_t+\xi^2)\widehat K(t,\xi,\eta)\big\|_{L^1_{\xi\eta}(A\sim N)}\big\|P_{N}f\big\|_{L^1_{xy}}\\
&\lesssim
 \langle t\rangle^{-1}\|\langle \nabla \rangle^{1+}f\|_{L^2_{xy}};
\end{align*}
and
\begin{align*}
&\quad \bl\Delta(\partial_{tt}-\Delta\partial_t-\pp_{xx})K(t)f\bl_{L^\infty_{xy}}\\
&\lesssim
\big\|A^2(\partial_{tt}+A^2\partial_t+\xi^2)\widehat K(t,\xi,\eta)\big\|_{L^2_{\xi\eta}(A\le 1)}\|P_{\le 1}f\|_{L^2_{xy}}\\
&\qquad +\sum\limits_{N\ge 1}N^2\big\|(\partial_{tt}+A^2\partial_t+\xi^2)\widehat K(t,\xi,\eta)\big\|_{L^2_{\xi\eta}(A\sim N)}\big\|P_{N}f\big\|_{L^2_{xy}}\\
&\lesssim
 \langle t\rangle^{-1}\|\langle \nabla \rangle^{1+}f\|_{L^2_{xy}};
\end{align*}
\end{proof}

\begin{prop}\label{lem:Kn5-L}
For any smooth function $f$,  $\beta\in [0,2]$,
\begin{equation*}
\begin{array}{ll}
&{\rm (i),}\quad\bl|\nabla|^{\beta}\partial_t(\partial_{tt}-\Delta\partial_t-\Delta)K(t)f\bl_{L^2_{xy}}
\lesssim
 \min\Big\{\langle t\rangle^{-\min\{\frac54,\frac{\beta+1}{2}\}}\|\langle \nabla \rangle^{\beta-1+}f\|_{L^1_{xy}}, \langle t\rangle^{-\frac{\beta}{2}}\|\langle \nabla \rangle^{\beta-2}f\|_{L^2_{xy}}\Big\};\\
&{\rm (ii),}\quad\bl\pp_t (\partial_{tt}-\Delta\partial_t-\Delta)K(t)f\bl_{L^\infty_{xy}}
\lesssim
 \langle t\rangle^{-1}\|\langle \nabla \rangle^{0+}f\|_{L^1_{xy}}.;\\
&{\rm (iii),}\quad\bl\nabla \pp_t (\partial_{tt}-\Delta\partial_t-\Delta)K(t)f\bl_{L^\infty_{xy}}
\lesssim
  \langle t\rangle^{-1}\|\langle \nabla \rangle^{0+}f\|_{L^2_{xy}};\\
&{\rm (iv),}\quad\bl\Delta\pp_t  (\partial_{tt}-\Delta\partial_t-\Delta)K(t)f\bl_{L^\infty_{xy}}
\lesssim
  \min\Big\{\langle t\rangle^{-\frac54}\|\langle \nabla \rangle^{1+}f\|_{L^2_{xy}}, \langle t\rangle^{-1}\|f\|_{L^\infty_{xy}}\Big\}
\end{array}
\end{equation*}
\end{prop}
\begin{proof}
From Lemma \ref{lem:Kn5} and Lemma \ref{Young},
\begin{align*}
&\quad \bl|\nabla|^{\beta}\partial_t(\partial_{tt}-\Delta\partial_t-\Delta)K(t)f\bl_{L^2_{xy}}\\
&\lesssim
\bl P_{\le 1}|\nabla|^{\beta}\partial_t(\partial_{tt}-\Delta\partial_t-\Delta)K(t)f\bl_{L^2_{xy}}+\sum\limits_{N\ge 1}\big\|P_{N}|\nabla|^{\beta}\partial_t(\partial_{tt}-\Delta\partial_t-\Delta)K(t)f\bl_{L^2_{xy}}\\
&\lesssim
\big\|A^{\beta}\pp_t(\partial_{tt}+A^2\partial_t+A^2)\widehat K(t,\xi,\eta)\big\|_{L^2_{\xi\eta}(A\le 1)}\|P_{\le 1}f\|_{L^1_{xy}}\\
&\qquad +\sum\limits_{N\ge 1}N^{\beta}\big\|\pp_t(\partial_{tt}+A^2\partial_t+A^2)\widehat K(t,\xi,\eta)\big\|_{L^2_{\xi\eta}(A\sim N)}\big\|P_{N}f\big\|_{L^1_{xy}}\\
&\lesssim
 \langle t\rangle^{-\min\{\frac54,\frac{\beta+1}{2}\}}\|P_{\le 1}f\|_{L^1_{xy}} +\langle t\rangle^{-\frac54}\sum\limits_{N\ge 1}N^{\beta-1}\big\|P_{N}f\big\|_{L^1_{xy}}\\
&\lesssim
 \langle t\rangle^{-\min\{\frac54,\frac{\beta+1}{2}\}}\|\langle \nabla \rangle^{\beta-1+}f\|_{L^1_{xy}}.
\end{align*}
Similarly, using \eqref{est:Kn8-2},
\begin{align*}
&\quad \bl|\nabla|^{\beta}\partial_t(\partial_{tt}-\Delta\partial_t-\Delta)K(t)f\bl_{L^2_{xy}}\\
&\lesssim
\big\|\langle A\rangle^{2-\beta}A^{\beta}\pp_t(\partial_{tt}+A^2\partial_t+A^2)\widehat K(t,\xi,\eta)\big\|_{L^\infty_{\xi\eta}}\|\langle \nabla\rangle^{2-\beta}f\|_{L^2_{xy}}\\
&\lesssim
 \langle t\rangle^{-\frac{\beta}{2}}\|\langle \nabla \rangle^{\beta-2}f\|_{L^2_{xy}}.
\end{align*}
Moreover, by \eqref{est:Kn6} and  \eqref{est:Kn6-2} ($q=1$),
\begin{align*}
&\quad \bl\pp_t (\partial_{tt}-\Delta\partial_t-\Delta)K(t)f\bl_{L^\infty_{xy}}\\
&\lesssim
\big\|\pp_t(\partial_{tt}+A^2\partial_t+A^2)\widehat K(t,\xi,\eta)\big\|_{L^1_{\xi\eta}(A\le 1)}\|P_{\le 1}f\|_{L^1_{xy}}\\
&\qquad +\sum\limits_{N\ge 1}\big\|\pp_t(\partial_{tt}+A^2\partial_t+A^2)\widehat K(t,\xi,\eta)\big\|_{L^1_{\xi\eta}(A\sim N)}\big\|P_{N}f\big\|_{L^1_{xy}}\\
&\lesssim
 \langle t\rangle^{-1}\|\langle \nabla \rangle^{0+}f\|_{L^1_{xy}}.
\end{align*}
Using \eqref{est:Kn6-1} and  \eqref{est:Kn6-2} ($q=2$) instead, we have (iii).
While using  \eqref{est:Kn6-2} and \eqref{est:Kn8-1} ($q=2$), we have $ \bl\Delta\pp_t(\partial_{tt}-\Delta\partial_t-\Delta)K(t)f\bl_{L^\infty_{xy}}\lesssim \langle t\rangle^{-\frac54}\|\langle \nabla \rangle^{1+}f\|_{L^2_{xy}}.$
At last, by the special estimate \eqref{est:Kn8-2}, we have
\begin{align*}
\bl\Delta\pp_t(\partial_{tt}-\Delta\partial_t-\Delta)K(t)f\bl_{L^\infty_{xy}}
&\lesssim
\big\|A^2\pp_t(\partial_{tt}+A^2\partial_t+A^2)\widehat K(t,\xi,\eta)\big\|_{L^\infty_{\xi\eta}}\|f\|_{L^\infty_{xy}}\\
&\lesssim
 \langle t\rangle^{-1}\|f\|_{L^\infty_{xy}}.
\end{align*}
This finishes the proof of the lemma.
\end{proof}

\begin{prop}\label{lem:Ku6-L}
For any smooth function $f$,
\begin{align*}
\bl\Delta\pp_x\partial_t(\partial_{tt}-\Delta\partial_t-\Delta)K(t)f\bl_{L^2_{xy}}
\lesssim &
\langle t\rangle^{-\frac32}\|\langle \nabla\rangle f\|_{L^2_{xy}}.
\end{align*}
\end{prop}
\begin{proof}
From Lemma \ref{lem:Ku6} and Lemma \ref{Young},
\begin{align*}
\bl\Delta\pp_x\partial_t(\partial_{tt}-\Delta\partial_t-\Delta)K(t)f\bl_{L^2_{xy}}
&=
\bl \langle \nabla\rangle^{-1}\Delta\pp_x\partial_t(\partial_{tt}-\Delta\partial_t-\Delta)K(t) \langle \nabla\rangle f\bl_{L^2_{xy}}\\
&\lesssim
\big\|\langle A\rangle^{-1}A^2\pp_t(\partial_{tt}+A^2\partial_t+A^2)\widehat K(t,\xi,\eta)\big\|_{L^\infty_{\xi\eta}}\|\langle \nabla\rangle f\|_{L^2_{xy}}\\
&\lesssim
 \langle t\rangle^{-\frac32}\|\langle \nabla \rangle f\|_{L^2_{xy}}.
\end{align*}
\end{proof}

\begin{prop}\label{lem:Kn9-L}
For any smooth function $f$,  $\beta\in [0,2]$,
\begin{align*}
&{\rm (i),}\quad\bl\langle\nabla\rangle^{2-\beta}|\nabla|^\beta\partial_{tt}(\partial_{tt}-\Delta\partial_t-\Delta-\pp_{yy})K(t)f\bl_{L^2_{xy}}
\lesssim
\langle t\rangle^{-\frac{\beta+1}2}\|  f\|_{L^2_{xy}};\\
&{\rm (ii),}\quad\bl\partial_{tt}(\partial_{tt}-\Delta\partial_t-\Delta-\pp_{yy})K(t)f\bl_{L^\infty_{xy}}
\lesssim
\langle t\rangle^{-1}\| f\|_{L^2_{xy}}.
\end{align*}
\end{prop}
\begin{proof} From Lemma \ref{lem:Kn9} and Lemma \ref{Young},
\begin{align*}
&\qquad \bl\langle\nabla\rangle^{2-\beta}|\nabla|^\beta\partial_{tt}(\partial_{tt}-\Delta\partial_t-\Delta-\pp_{yy})K(t)f\bl_{L^2_{xy}}\\
&\lesssim
\big\|\langle A\rangle^{2-\beta}A^\beta\pp_{tt}(\partial_{tt}+A^2\partial_t+A^2+\eta^2)\widehat K(t,\xi,\eta)\big\|_{L^\infty_{\xi\eta}}\| f\|_{L^2_{xy}}\\
&\lesssim
\langle t\rangle^{-\frac{\beta+1}2}\| f\|_{L^2_{xy}};
\end{align*}
and
\begin{align*}
\bl\partial_{tt}(\partial_{tt}-\Delta\partial_t-\Delta-\pp_{yy})K(t)f\bl_{L^\infty_{xy}}
&\lesssim
\big\|\pp_{tt}(\partial_{tt}+A^2\partial_t+A^2+\eta^2)\widehat K(t,\xi,\eta)\big\|_{L^2_{\xi\eta}}\| f\|_{L^2_{xy}}\\
&\lesssim
\langle t\rangle^{-1}\| f\|_{L^2_{xy}}.
\end{align*}
This proves (i) and (ii).
\end{proof}

\begin{prop}\label{lem:K1-L}
For any smooth function $f$,
\begin{align*}
&{\rm (i),}\quad\bl K_1(t)f\bl_{L^2_{xy}}
\lesssim
\langle t\rangle^{-\frac14}\big\| \langle \nabla\rangle^{\frac12+}|\nabla|^{\frac12-} f\big\|_{L^1_{xy}};\\
&{\rm (ii),}\quad\bl\pp_x K_1(t)f\bl_{L^2_{xy}}
\lesssim
\min\Big\{\langle t\rangle^{-\frac12}\| \langle \nabla\rangle^{1+}\nabla  f\|_{L^1_{xy}}, \langle t\rangle^{-\frac34}\| \langle \nabla\rangle^{2+}f\|_{L^1_{xy}}\Big\}.
\end{align*}
\end{prop}
\begin{proof}
From Lemma \ref{lem:K1} and Lemma \ref{Young},
\begin{align*}
\bl K_1(t)f\bl_{L^2_{xy}}
&\lesssim
\bl P_{\le 1}K_1(t)f\bl_{L^2_{xy}}+\sum\limits_{N\ge1}\bl P_N K_1(t)f\bl_{L^2_{xy}}\\
&\lesssim
\big\|\widehat{K_1}(t,\xi,\eta)\big\|_{L^2_{\xi}L^\infty_\eta(A\le 1)}\|P_{\le 1}f\|_{L^1_{x}L^2_y}
+\sum\limits_{N\ge1}\big\|\widehat{K_1}(t,\xi,\eta)\big\|_{L^2_{\xi}L^\infty_\eta(A\sim N)}\|P_Nf\|_{L^1_{x}L^2_y}\\
&\lesssim
\langle t\rangle^{-\frac14}\| \langle \nabla\rangle^{\frac12+}|\nabla|^{\frac12-} f\|_{L^1_{xy}},
\end{align*}
and
\begin{align*}
\bl \pp_x K_1(t)f\bl_{L^2_{xy}}
&\lesssim
\bl P_{\le 1}\frac{\nabla}{-\Delta}\pp_x K_1(t)\cdot \nabla f\bl_{L^2_{xy}}+\sum\limits_{N\ge1}\bl P_N \pp_x K_1(t)f\bl_{L^2_{xy}}\\
&\lesssim
\big\|A^{-1}\xi \widehat{K_1}(t,\xi,\eta)\big\|_{L^2_{\xi\eta}(A\le 1)}\|P_{\le 1}\nabla f\|_{L^1_{xy}}
+\sum\limits_{N\ge1}\big\|\xi\widehat{K_1}(t,\xi,\eta)\big\|_{L^2_{\xi\eta}(A\sim N)}\|P_Nf\|_{L^1_{xy}}\\
&\lesssim
\langle t\rangle^{-\frac12}\| \langle \nabla\rangle^{1+}\nabla  f\|_{L^1_{xy}}.
\end{align*}
Similarly, using the third estimates in Lemma \ref{lem:K1}, we also have
\begin{align*}
\bl \pp_x K_1(t)f\bl_{L^2_{xy}}
&\lesssim
\langle t\rangle^{-\frac34}\| \langle \nabla\rangle^{2+}  f\|_{L^1_{xy}}.
\end{align*}
\end{proof}

\vskip .5in
\section{Local theory and energy estimates}

This section presents the local (in time) existence and uniqueness result for \eqref{e2.1},
which can be deduced from the work of Kawashima \cite{Kawa, Kawa2}. More importantly we obtain an
energy inequality needed in the proof of Theorem \ref{th1}.

\begin{prop} \label{prop_local}
Let $1+ \lambda>0$. Assume that $(n_0,\vec u_0,\nabla\psi_0)\in H^\sigma(\R^2)$ with
$\sigma\ge 3$. Assume $$\overline{\rho}\ge\rho_0=1+n_0\ge\underline{\rho}$$
for two constants $\overline{\rho}>\underline{\rho}>0$. Then there exists
$T_0=T_0(\|(n_0,\vec u_0,$ $\nabla\psi_0)\|_{H^\sigma})>0$, and a unique smooth local
solution $(n,\vec u,\nabla\psi)\in C_t^0([0,T_0];H^\sigma(\R^2))$ to \eqref{e2.1} with
\begin{equation}\label{rhobd}
\overline{\rho}\ge\rho=1+n\ge\underline{\rho}.
\end{equation}
Furthermore, for $\epsilon$ and $M$ given in the definition of $X$ in (\ref{X}),
\begin{equation}\label{1sttermbd}
\langle t\rangle^{-\epsilon} \, \big\|\la\na\ra^M (n,\vec u,\nabla\psi)(t)\big\|_2\le \big\|\la\na\ra^M (n_0,\vec u_0,\nabla\psi_0)\big\|_2 \,+\, Q(\|U\|_X),
\end{equation}
where $Q(r)$ represents a polynomial of $r$ with the lowest order at least quadratic.
\end{prop}

\vskip .1in
\begin{proof}[Proof of Proposition \ref{prop_local}]
The existence of the local solution $(n,\vec u,\nabla\psi)$ with the property (\ref{rhobd})
follows from \cite{Kawa, Kawa2}. The main effort of this proof is devoted to the bound in \eqref{1sttermbd}.

\vskip .1in
For any $\sigma \ge 0$, we have,
from $\partial_t n+\nabla\cdot\vec u +\nabla\cdot (n \vec u)= 0$,
\begin{align*}
\frac12\partial_t\int\big|\la\na\ra^\sigma n\big|^2\,dx
=&-\int\big(\la\na\ra^\sigma\nabla\cdot\vec u\big)\la\na\ra^\sigma n\,dx
-\int\big(\la\na\ra^\sigma\nabla\cdot(n\vec u)\big)\la\na\ra^\sigma n\,dx.
\end{align*}
Using the commutator notation
$$
[\la\na\ra^\sigma\nabla\cdot, \,\vec u]n \equiv  \la\na\ra^\sigma\nabla\cdot(n\vec u)-\vec u\cdot \nabla\la\na\ra^\sigma n
$$
and integrating by parts, we have
\begin{align*}
\int\big(\la\na\ra^\sigma\nabla\cdot(n\vec u)\big)\la\na\ra^\sigma n\,dx
 =\int\big([\la\na\ra^\sigma\nabla\cdot, \,\vec u]n\big)\la\na\ra^\sigma n\,dx
 - \frac12\int(\nabla\cdot \vec u)\big(\la\na\ra^\sigma n\big)^2\,dx.
\end{align*}
It then follows from the commutator estimate
$$
\|[\la\na\ra^\sigma\nabla\cdot, \,\vec u]n\|_2 \lesssim
\|n\|_\infty\, \big\|\la\na\ra^\sigma \nabla \vec u\big\|_2
+ \|\nabla \vec u\big\|_\infty\, \|\la\na\ra^\sigma n\|_2
$$
and H\"{o}lder's inequality that
\begin{align}
\frac12\partial_t \left(\|\la\na\ra^\sigma n\|_2^2\right)
\le & -\int\big(\la\na\ra^\sigma\nabla\cdot\vec u\big)\la\na\ra^\sigma n\,dx
+C\|n\|_\infty\,\big\|\la\na\ra^\sigma n\big\|_2 \big\|\la\na\ra^\sigma \nabla \vec u\big\|_2
+ \|\nabla \vec u\big\|_\infty\, \|\la\na\ra^\sigma n\|_2^2\notag\\
\le & -\int\big(\la\na\ra^\sigma\nabla\cdot\vec u\big)\la\na\ra^\sigma n\,dx\notag\\
&\quad  +C\Big(\|n\|_\infty^2+\|\nabla \vec u\big\|_\infty\Big)\,\big\|\la\na\ra^\sigma n\big\|_2^2
+\frac{\bar{\lambda}}{5\underline{\rho}}
\|\langle \na \rangle^\sigma \nabla\vec u\|_2^{2}\label{E1} .
\end{align}
where  $C$ is a large constant which may depend on $\underline{\rho}, \overline{\rho} $ and may
vary from line to line, and $\bar{\lambda}=1+\lambda$.
Applying $\la\na\ra^\sigma$ to (\ref{simpleueq}) and then dotting with $\la\na\ra^\sigma \vec u$, we have,
after integration by parts,
\begin{align*}
\frac12\pp_t\int \big|\la\na\ra^\sigma \vec u\big|^2\,dx =J_1 + J_2 + J_3 +J_4+J_5,
\end{align*}
where
\begin{align*}
J_1 =&-\int \langle \na \rangle^\sigma\big(\vec u\cdot \nabla \vec u\big)\cdot \langle \na \rangle^\sigma \vec u\,dx,\\
J_2=&\int \langle \na \rangle^\sigma n\>  \langle \na \rangle^\sigma\big( \nabla \cdot\vec u\big)\,dx -\int \langle \na \rangle^\sigma\big(n \nabla n\big)\cdot \langle \na \rangle^\sigma \vec u\,dx,\\
J_3=&\int \langle \na \rangle^\sigma(\rho^{-1} \Delta \vec u) \cdot \langle \na \rangle^\sigma \vec u\,dx,\\
J_4=&\lambda\int \langle \na \rangle^\sigma(\rho^{-1} \nabla (\nabla\cdot \vec u)) \cdot \langle \na \rangle^\sigma \vec u\,dx,\\
J_5=&-\int \langle \na \rangle^\sigma(\rho^{-1}\nabla \phi \Delta \phi)\cdot \langle \na \rangle^\sigma \vec u\,dx.
\end{align*}
We bound the terms on the right-hand side. Writing
$$
J_1 = -\int [\langle \na \rangle^\sigma, \vec u\cdot\nabla] \vec u \cdot \langle \na \rangle^\sigma \vec u\,dx
+ \frac12\,\int(\nabla \cdot \vec u)|\langle \na \rangle^\sigma\vec u|^2\,dx,
$$
we have, by a commutator estimate,
$$
|J_1| \lesssim \|\nabla \vec u\|_\infty\big\|\la\na\ra^\sigma \vec u\big\|_2^2.
$$
After integration by parts,
\begin{align*}
\left|\int \langle \na \rangle^\sigma\big(n \nabla n\big)\cdot \langle \na \rangle^\sigma \vec u\,dx\right|
\lesssim &\,
\|\langle \na \rangle^\sigma (n^2)\|_2 \, \|\langle \na \rangle^\sigma\big( \nabla \cdot\vec u\big)\|_{2} \\
\lesssim &\, \|n\|_\infty\, \|\langle \na \rangle^\sigma n\|_2 \,\|\langle \na \rangle^\sigma\big( \nabla \cdot\vec u\big)\|_{2}.
\end{align*}
That is ,
\begin{align*}
J_2  \le &
\int \langle \na \rangle^\sigma n\>  \langle \na \rangle^\sigma\big( \nabla \cdot\vec u\big)\,dx
 + \|n\|_\infty\, \|\la\na\ra^\sigma n\|_2\, \big\|\la\na\ra^\sigma \nabla \vec u\big\|_2\\
\le &
\int \langle \na \rangle^\sigma n\>  \langle \na \rangle^\sigma\big( \nabla \cdot\vec u\big)\,dx
 + \|n\|_\infty^2\, \|\la\na\ra^\sigma n\|_2^2+\frac{\bar{\lambda}}{5\underline{\rho}}
\|\langle \na \rangle^\sigma \nabla\vec u\|_2^{2} .
\end{align*}
To estimate $J_3$, we write
\begin{align*}
J_3 & =-\int \langle \na \rangle^\sigma(\rho^{-1} \nabla\vec u)\cdot \langle \na \rangle^\sigma \nabla\vec u\,dx
+ \int \langle \na \rangle^\sigma(\nabla(\rho^{-1})\cdot\nabla\vec u)\cdot \langle \na \rangle^\sigma \vec u\,dx\\
&\qquad =-\int [\langle \na \rangle^\sigma, \rho^{-1}] \nabla\vec u \cdot \langle \na \rangle^\sigma \nabla\vec u\,dx
- \int  \rho^{-1} \,|\langle \na \rangle^\sigma \nabla\vec u|^2\,dx\\
&\qquad\quad + \int \langle \na \rangle^\sigma(\nabla(\rho^{-1})\cdot\nabla\vec u)\cdot \langle \na \rangle^\sigma \vec u\,dx.
\end{align*}
Therefore, by a commutator estimate,
\begin{align*}
& J_3 \le  - \int  \rho^{-1} \,|\langle \na \rangle^\sigma \nabla\vec u|^2\,dx\\
&\qquad + C \|\nabla \vec u\|_{\infty}\, \|\langle \na \rangle^\sigma(\rho^{-1})\|_2\, \|\langle \na \rangle^\sigma \nabla\vec u\|_2 + C\|\nabla(\rho^{-1})\|_\infty\, \|\langle \na \rangle^\sigma\vec u\|_2\,\|\langle \na \rangle^\sigma
\nabla\vec u\|_2.
\end{align*}
Invoking (\ref{rhobd}) and using Sobolev's inequality,
\begin{align}\label{Sob}
&\|\langle \na \rangle^\sigma(\rho^{-1})\|_2  \le C(\underline{\rho}, \overline{\rho})
\|\langle \na \rangle^\sigma n\|_2, \notag \\
&\|\nabla(\rho^{-1})\|_\infty \le C(\underline{\rho})\,\|\nabla n\|_\infty, 
\end{align}
we obtain, by Young's inequality,
\begin{align*}
J_3  \le & - \int  \rho^{-1} \,|\langle \na \rangle^\sigma \nabla\vec u|^2\,dx+ C\|\nabla \vec u\|_{\infty}\,\|\langle \na \rangle^\sigma n\|_2\|\langle \na \rangle^\sigma \nabla\vec u\|_2\\
&
\quad+ C\| \na  n\|_\infty \, \|\langle \na \rangle^\sigma\vec u\|_2\,
\|\langle \na \rangle^\sigma \nabla\vec u\|_2\\
\le &- \int  \rho^{-1} \,|\langle \na \rangle^\sigma \nabla\vec u|^2\,dx
+ C\|\nabla \vec u\|_{\infty}^2\,\|\langle \na \rangle^\sigma n\|_2^2+C\|\nabla n\|_\infty^2\|\langle \na \rangle^\sigma\vec u\|_2^2
+\frac{\bar{\lambda}}{10\underline{\rho}}
\|\langle \na \rangle^\sigma \nabla\vec u\|_2^{2}.
\end{align*}
$J_4$ is similarly estimated as $J_3$. Writing $J_4$ as
\begin{align*}
J_4 =& -\lambda\int \rho^{-1}\, |\langle \na \rangle^\sigma\nabla\cdot \vec u|^2\,dx -\lambda \int
[\langle \na \rangle^\sigma, \rho^{-1}] \nabla\cdot\vec u \cdot \langle \na \rangle^\sigma \nabla\vec u\,dx\\
& -\lambda \int \langle \na \rangle^\sigma(\nabla(\rho^{-1})\,\nabla\cdot\vec u)\cdot \langle \na \rangle^\sigma \vec u\,dx,
\end{align*}
we obtain, after similar estimates as for $J_3$,
\begin{align*}
J_4  \le  & -\lambda\int \rho^{-1}\, |\langle \na \rangle^\sigma\nabla\cdot \vec u|^2\,dx
+ C\|\nabla \vec u\|_{\infty}^2\,\|\langle \na \rangle^\sigma n\|_2^2+C\|\nabla n\|_\infty^2\|\langle \na \rangle^\sigma\vec u\|_2^2
+\frac{\bar{\lambda}}{10\underline{\rho}}
\|\langle \na \rangle^\sigma \nabla\vec u\|_2^{2}.
\end{align*}
Since $|\langle \na \rangle^\sigma\nabla\cdot \vec u|^2\le |\langle \na \rangle^\sigma\nabla \vec u|^2$, for any $\lambda+1>0$, we have
\begin{align*}
J_3+J_4  \le  & - \int  \rho^{-1} \,|\langle \na \rangle^\sigma \nabla\vec u|^2\,dx-\lambda\int \rho^{-1}\, |\langle \na \rangle^\sigma\nabla\cdot \vec u|^2\,dx\\
&
\quad+ C\|\nabla \vec u\|_{\infty}^2\,\|\langle \na \rangle^\sigma n\|_2^2+C\|\nabla n\|_\infty^2\|\langle \na \rangle^\sigma\vec u\|_2^2
+\frac{\bar{\lambda}}{5\underline{\rho}}
\|\langle \na \rangle^\sigma \nabla\vec u\|_2^{2}\\
\le  & -  \frac{\bar{\lambda}}{\underline{\rho}}\int  \,|\langle \na \rangle^\sigma \nabla\vec u|^2\,dx\\
&
\quad+ C\|\nabla \vec u\|_{\infty}^2\,\|\langle \na \rangle^\sigma n\|_2^2+C\|\nabla n\|_\infty^2\|\langle \na \rangle^\sigma\vec u\|_2^2
+\frac{\bar{\lambda}}{5\underline{\rho}}
\|\langle \na \rangle^\sigma \nabla\vec u\|_2^{2}
.
\end{align*}
To bound $J_5$, after integration by parts and using $\phi=\psi+y$, we write
\begin{align*}
J_5  =& \int \langle \na \rangle^\sigma \nabla \psi\cdot \langle \na \rangle^\sigma \nabla v\,dx
- \int \rho^{-1}\langle \na \rangle^\sigma\big(\nabla \psi\Delta \psi\big)\cdot \langle \na \rangle^\sigma \vec u\,dx\\
&- \int [\langle \na \rangle^\sigma, \rho^{-1}](\nabla \psi\Delta \psi)
\cdot \langle \na \rangle^\sigma \vec u\,dx-\int \langle \na \rangle^\sigma (\frac{n}{\rho}\Delta \psi)\, \langle \na \rangle^\sigma v\,dx .
\end{align*}
Since $\nabla \psi\Delta \psi=\nabla\cdot (\nabla\psi\nabla\psi)-\frac12\nabla(|\nabla\psi|^2)$, we have
\begin{align*}
\Big|\int \rho^{-1}\langle \na \rangle^\sigma\big(\nabla \psi\Delta \psi\big)\cdot \langle \na \rangle^\sigma \vec u\,dx\Big|
\le &\,
C(\underline{\rho})\|\nabla \psi\|_\infty\big\|\la\na\ra^\sigma \nabla\vec u\big\|_2\big\|\la\na\ra^\sigma \nabla\psi\big\|_2\\
\le &\,
C\|\nabla\psi\|_\infty^2\,\|\la\na\ra^\sigma \nabla\psi\|_2^2+\frac{\bar{\lambda}}{15\underline{\rho}}
\|\langle \na \rangle^\sigma \nabla\vec u\|_2^{2}.
\end{align*}

By a commutator estimate and (\ref{Sob}),
\begin{align*}
&\left|\int [\langle \na \rangle^\sigma, \rho^{-1}](\nabla \psi\Delta \psi)
\cdot \langle \na \rangle^\sigma \vec u\,dx \right| \\
\le & C\|\nabla(\rho^{-1})\|_\infty\, \|\nabla\psi\|_\infty\, \|\langle \na \rangle^{\sigma}\nabla\psi\|_2\, \|\la\na\ra^\sigma \nabla\vec u\|_2
 +C\|\la\na\ra\nabla\psi\|_\infty^2\,\|\la\na\ra^\sigma (\rho^{-1})\|_2\, \|\la\na\ra^\sigma \nabla\vec u\|_2\\
\le &
C\big(\|\nabla n\|_\infty^2+\|\nabla\psi\|_\infty^2\big)\, \|\langle \na \rangle^{\sigma}\nabla\psi\|_2^2
 +C\|\la\na\ra\nabla\psi\|_\infty^2\,\|\la\na\ra^\sigma n\|_2^2+\frac{\bar{\lambda}}{15\underline{\rho}}
\|\langle \na \rangle^\sigma \nabla\vec u\|_2^{2}.
\end{align*}

Since $\langle \na \rangle^\sigma v=\langle \na \rangle^\sigma P_{\le 1} v+\langle \na \rangle^\sigma P_{\ge 1} v$,
\begin{align*}
&\left|\int \langle \na \rangle^\sigma (\frac{n}{\rho}\Delta \psi)\, \langle \na \rangle^\sigma v\,dx\right|\\
\le &
\left|\int \langle \na \rangle^\sigma (\frac{n}{\rho}\Delta \psi)\, \langle \na \rangle^\sigma P_{\le 1} v\,dx\right|
+\left|\int \langle \na \rangle^\sigma (\frac{n}{\rho}\Delta \psi)\, \langle \na \rangle^\sigma P_{\ge 1} v\,dx\right|\\
\le &
C\|v\|_\infty\| n\|_2 \|\Delta\psi\|_2\\
&+
C\|\la\na\ra n\|_\infty \big\|\langle \na \rangle^{\sigma}\nabla\psi\big\|_2\big\|\langle \na \rangle^\sigma \nabla v\big\|_2
 +C\big\|\la\na\ra^\sigma n\big\|_2\,\|\la\na\ra \nabla\psi\big\|_\infty\|\langle \na \rangle^\sigma P_{\ge 1} v\big\|_2\\
 \le &
C\Big(\|v\|_\infty+\|\la\na\ra  n\|_\infty^2+\|\la\na\ra \nabla\psi\big\|_\infty^2\Big)\>\big(\|\la\na\ra^\sigma n\|_2^2+ \|\langle \na \rangle^{\sigma}\nabla\psi\|_2^2\big)+\frac{\bar{\lambda}}{15\underline{\rho}}
\|\langle \na \rangle^\sigma \nabla\vec u\|_2^{2}.
\end{align*}

Therefore,
\begin{align*}
J_5 \le &
C\Big(\|v\|_\infty+\|\la\na\ra  n\|_\infty^2+\|\la\na\ra \nabla\psi\big\|_\infty^2\Big)\>\big(\|\la\na\ra^\sigma n\|_2^2+ \|\langle \na \rangle^{\sigma}\nabla\psi\|_2^2\big)
+\frac{\bar{\lambda}}{5\underline{\rho}}
\|\langle \na \rangle^\sigma \nabla\vec u\|_2^{2}.
\end{align*}
From $\pp_t\psi+v+\vec u\cdot \nabla\psi=0$, we have
\begin{align*}
&\frac12\pp_t\int \big|\la\na\ra^\sigma \nabla\psi \big|^2\,dx\\
=&
-\int \langle \na \rangle^\sigma \nabla v\cdot \langle \na \rangle^\sigma \nabla \psi\,dx
-\int \langle \na \rangle^\sigma \nabla\big(\vec u\cdot \nabla \psi\big)\cdot\langle \na \rangle^\sigma \nabla \psi\,dx\\
=&-\int \langle \na \rangle^\sigma \nabla v\cdot \langle \na \rangle^\sigma \nabla \psi\,dx
-\int \langle \na \rangle^\sigma \big(\nabla\vec u\cdot \nabla \psi\big)\cdot\langle \na \rangle^\sigma \nabla \psi\,dx\\
&+\frac12\int \big(\nabla\cdot\vec u\big) \big|\langle \na \rangle^\sigma  \nabla \psi\big|^2\,dx
-\int \big[\langle \na \rangle^\sigma,\vec u\cdot \nabla\big]\nabla \psi\cdot\langle \na \rangle^\sigma \nabla \psi\,dx.
\end{align*}
Similar as above, we obtain
\begin{align*}
\frac12\pp_t\int \big|\la\na\ra^\sigma \nabla\psi \big|^2\,dx
\le  & -\int \langle \na \rangle^\sigma \nabla v\cdot \langle \na \rangle^\sigma \nabla \psi\,dx\\
&
+\,\big(\|\nabla \vec u\|_\infty+ \|\nabla \psi\|_\infty^2\big)\|\la\na\ra^\sigma\nabla\psi\|_2^2\,+\frac{\bar{\lambda}}{5\underline{\rho}}
\|\langle \na \rangle^\sigma \nabla\vec u\|_2^{2}.
\end{align*}

Adding the estimates above, and taking $\sigma =M$ as in the definition of $X$ in (\ref{X}), we have
\begin{align*}
\big\|\la\na\ra^M (n,\vec u,\nabla\psi)(t)\big\|_2
\lesssim &
\big\|\la\na\ra^M (n_0,\vec u_0,\nabla\psi_0)\big\|_2\\
&+
\int_0^t \Big(\|\la \nabla\ra\vec u\|_\infty+\big\|\la\na\ra (n,\nabla\psi,\vec u)\big\|_\infty^2\Big)\> \big\|\la\na\ra^M (n,\nabla\psi,\vec u)\big\|_2\,ds\\
\lesssim &
\big\|\la\na\ra^M (n_0,\vec u_0,\nabla\psi_0)\big\|_2^2+
\int_0^t\langle s\rangle^{-1+\epsilon}\,ds Q(\|U\|_X)\\
\lesssim &
\big\|\la\na\ra^M (n_0,\vec u_0,\nabla\psi_0)\big\|_2+
\langle t\rangle^{\epsilon} Q(\|U\|_X),
\end{align*}
which gives (\ref{1sttermbd}). This completes the proof of Proposition \ref{prop_local}.
\end{proof}

\vskip .4in

\section{The estimates on $n$}
\label{sect:n}

In this section, we shall prove that there exists some small constant $\epsilon_0>0$, such that
\begin{align}
\|\langle \nabla\rangle^3n(t)\|_{L^2_{xy}}&\lesssim \langle t\rangle^{-\frac14}\big(\|U_0\|_{X_0}+\epsilon_0\|U\|_X+Q(\|U\|_{X})\big);\label{est:n-L2}\\
\|\langle \nabla\rangle^{\frac32} n(t)\|_{L^\infty_{xy}}&\lesssim \langle t\rangle^{-\frac12}\big(\|U_0\|_{X_0}+\epsilon_0\|U\|_X+Q(\|U\|_{X})\big);\label{est:n-Linfty}\\
\|\langle \nabla\rangle \pp_x n(t)\|_{L^2_{xy}}&\lesssim \langle t\rangle^{-\frac34}\big(\|U_0\|_{X_0}+\epsilon_0\|U\|_X+Q(\|U\|_{X})\big).\label{est:n-xL2}
\end{align}
Moreover, it follows from Nash's inequality, that
$$
\|f\|_{L^\infty_{xy}}\lesssim \|\pp_xf\|_{L^2_{xy}}^{\frac12}\|\langle \nabla\rangle^{1+}f\|_{L^2_{xy}}^{\frac12}.
$$
Using this inequality,
$$
\|\langle \nabla\rangle^{\frac32} n(t)\|_{L^\infty_{xy}}\lesssim \|\langle \nabla\rangle^3n(t)\|_{L^2_{xy}}^{\frac12}\|\langle \nabla\rangle\pp_xn(t)\|_{L^2_{xy}}^{\frac12}.
$$
Thus, we only need to show \eqref{est:n-L2} and \eqref{est:n-xL2}. The rest of this section is divided
into three subsections. The first subsection recasts the integral representation given in Lemma
\ref{fiii}, which we call ``reexpression" of $n$. This new representation further reveals
the structure of the model and is suitable for the desired estimates. The second subsection estimates the linear parts
while the third section bounds the nonlinear part.

\subsection{The reexpression of $n$}\label{sec:exp-n}
This subsection recasts the representation given in Lemma \ref{fiii}.
More precisely,  we prove the following proposition.
\begin{prop} \label{nprop1}
The unknown function $n$ obeys the formula,
\begin{align}\label{exp-n}
n(t,x,y)=
(L_n+B_n)(t;n_0,\vec u_0, \vec b_0)+\mathcal{N}_n(t;n,\vec u, \psi),
\end{align}
where $(L_n+B_n)$ is given by
\begin{align}
&L_n(t;n_0,\vec u_0, \vec b_0)+B_n(t;n_0,\vec u_0, \vec b_0)\notag\\
=&-K(t)\big[\partial_y \Delta^2\psi_0\big]+\partial_tK(t)\big[\Delta\partial_y \psi_0\big]-(\partial_{tt}-\Delta\partial_t-\Delta)K(t)\big[\pp_xu_0\big]\notag\\
&-(\partial_{tt}-\Delta\partial_t)K(t)\big[\pp_yv_0\big]-\frac12\Delta\sqrt{\Delta\partial_{yy}} K(t)n_0+K_1(t)n_0\label{eq:LBn}
\end{align}
and $\mathcal{N}_n(t;n,\vec u, \psi)$ is defined as
\begin{align}
\mathcal{N}_n(t;n,&\vec u, \psi)
=\int_0^t\partial_t(\partial_{tt}-\Delta\partial_t-\Delta)K(t-s)N_0(s)\,ds\notag\\
&-\int_0^t \Delta(\partial_{tt}-\Delta\partial_t-\Delta)K(t-s)N_0(s)\,ds-\int_0^t\partial_x(\partial_{tt}-\Delta\partial_t-\Delta)K(t-s)N_1(s)\,ds\notag\\
&
+\int_0^t\partial_y(\partial_{tt}-\Delta\partial_t)K(t-s)N_2(s)\,ds+\int_0^t\Delta\partial_y(\partial_t-\Delta)K(t-s)N_3(s)\,ds\notag\\
&
-\lambda\int_0^t(\partial_{tt}-\Delta\partial_t-\partial_{xx})K(t-s)(\partial_x\Delta u+\partial_y\Delta v)\,ds.\label{n-N}
\end{align}
\end{prop}

Now we begin to prove this proposition.
According to \eqref{eq:Formula},  \eqref{L-1}--\eqref{L-5} and \eqref{F0-r},
\begin{align*}
n(t,x,y)=&L_n(t;n_0,\vec u_0, \vec b_0)+\int_0^t K(t-s)F_0(s)\,ds,
\end{align*}
where
\begin{align}
&L_n(t;n_0,\vec u_0, \vec b_0)\notag\\
=&
K(t)\big[(\pp_{tt}-\Delta\pp_t-\Delta)\pp_{t}n(0)\big]\label{Ln-1}\\
&+(\partial_{tt}-\Delta\partial_t-\Delta)K(t)\big[\partial_tn(0)\big]\label{Ln-2}\\
&-\Delta K(t)\big[(\pp_{tt}-\Delta\pp_t-\Delta)n(0)\big]
+\partial_t K(t)\big[(\pp_{tt}-\Delta\pp_t-\Delta)n(0)\big]\label{Ln-3}\\
&-\frac12\Delta\sqrt{\Delta\partial_{yy}} K(t)[n_0]+K_1(t)[n_0]\label{Ln-4}
\end{align}
and
\begin{align*}
F_0(s) =&\big(\partial_s-\Delta\big)(\partial_{ss}-\Delta\partial_s-\Delta)N_0
-\partial_x(\partial_{ss}-\Delta\partial_s-\Delta)N_1\\
&-\partial_y(\partial_{ss}-\Delta\partial_s)N_2+\Delta\partial_y(\partial_s-\Delta)N_3
-\lambda(\partial_{ss}-\Delta\partial_s-\partial_{xx})(\partial_x\Delta u+\partial_y\Delta v).
\end{align*}
To prove Proposition \ref{nprop1}, we essentially replace $\pp_{t}n(0)$ and $\pp_{tt}n(0)$ in $L_n$
by the terms in the equation of $\pp_{t}n$. The terms in $B_n(t;n_0,\vec u_0, \vec b_0)$ are the boundary terms
that come from
integration by parts in the time integral $\int_0^t K(t-s)F_0(s)\,ds$.
The details for deriving the expressions of $L_n$ and $B_n$ are given in the following two sub-subsections.

\vskip .1in
\subsubsection{$L_n(t;n_0,\vec u_0, \vec b_0)$}
Since $\pp_t n(0)=-\pp_x u(0)-\pp_y v(0)+N_0(0)$, by \eqref{e2.4} and \eqref{e2.6} we have
\begin{align*}
\eqref{Ln-1}=&-K(t)(\partial_{tt}-\Delta\partial_t-\Delta)(\pp_x u(0)+\pp_y v(0))
+K(t)(\partial_{tt}-\Delta\partial_t-\Delta)N_0(0)\notag\\
=&-K(t)\big(-\partial_{xyy}u(0)+\partial_{xxy}v(0)+\pp_t \pp_xN_1(0)-\pp_{xx}N_0(0)+\lambda\pp_t(\pp_{xxx} u(0)+\pp_{xxy} v(0))\notag\\
&-K(t)\big(\pp_{yyy} v+\partial_{xyy}u-\pp_{yy}N_0+\partial_t\pp_yN_2-\Delta \pp_yN_3+\lambda\pp_t(\pp_{xyy} u+\pp_{yyy} v)\big)\notag\\
&+K(t)(\partial_{tt}-\Delta\partial_t-\Delta)N_0(0)\notag\\
=&K(t)\big[\Delta N_0(0)-\pp_t \pp_xN_1(0)-\partial_t\pp_yN_2(0)+\pp_y\Delta N_3(0)\big]+K(t)\big[(\partial_{tt}-\Delta\partial_t-\Delta)N_0(0)\big]\notag\\
&-K(t)\big[\partial_{y}\Delta v(0)+\lambda\pp_t\pp_{x}\Delta u(0)+\lambda\pp_t\pp_{y}\Delta v(0)\big].
\end{align*}
Moreover,
\begin{align*}
\eqref{Ln-2}
=&(\partial_{tt}-\Delta\partial_t-\Delta)K(t)\big[\partial_tn(0)\big]\\
=&(\partial_{tt}-\Delta\partial_t-\Delta)K(t)\big[-\pp_x u(0)-\pp_y v(0)+N_0(0)\big]\\
=&-\pp_x(\partial_{tt}-\Delta\partial_t-\Delta)K(t)u(0)-\pp_y(\partial_{tt}-\Delta\partial_t-\Delta)K(t)v(0)\\
&\quad +(\partial_{tt}-\Delta\partial_t-\Delta)K(t)\big[N_0(0)\big].
\end{align*}
Further, since by \eqref{e2.2},
\begin{align*}
&\Delta K(t)\big[(\partial_{tt}-\Delta\partial_t-\Delta)n(0)\big]\\
=&
\Delta K(t)\big[\Delta\partial_y \phi(0)+\partial_tN_0(0)-\Delta N_0(0)-\partial_xN_1(0)-\partial_yN_2(0)-\lambda(\partial_x\Delta u(0)+\partial_y\Delta v(0))\big];\\
&\partial_tK(t)\big[(\partial_{tt}-\Delta\partial_t-\Delta)n(0)\big]\\
=&
\partial_tK(t)\big(\Delta\partial_y \phi(0)+\partial_tN_0(0)-\Delta N_0(0)-\partial_xN_1(0)-\partial_yN_2(0)-\lambda(\partial_x\Delta u(0)+\partial_y\Delta v(0))\big),
\end{align*}
we have
\begin{align*}
\eqref{Ln-3}
=&\partial_tK(t)\big[(\partial_{tt}-\Delta\partial_t-\Delta)n(0)\big]-\Delta K(t)\big[(\partial_{tt}-\Delta\partial_t-\Delta)n(0)\big]\notag\\
=&\big(\partial_t-\Delta)K(t)\big[\partial_tN_0(0)-\Delta N_0(0)-\partial_xN_1(0)-\partial_yN_2(0)\big]\notag\\
&\quad+\big(\partial_t-\Delta)K(t)\big[\Delta\partial_y \phi(0)-\lambda(\partial_x\Delta u(0)+\partial_y\Delta v(0))\big].
\end{align*}
Collecting the equalities above, we have
\begin{align}
&L_n(t;n_0,\vec u_0, \vec b_0)\notag\\
=&K(t)\big[(\partial_{tt}-\Delta\partial_t-\Delta)N_0(0)\big]-(\pp_t-\Delta)\Delta K(t)N_0(0)\\
&-K(t)\big[\partial_t\Delta N_0(0)+\pp_t \pp_xN_1(0)+\partial_t\pp_yN_2(0)\big]\notag\\
&+K(t)\big[\Delta N_0(0)+\partial_x\Delta N_1(0)+\partial_y\Delta N_2(0)+ \pp_y\Delta N_3(0)\big]\notag\\
&+\partial_tK(t)\big[\partial_tN_0(0)-\partial_xN_1(0)-\partial_yN_2(0)\big] +(\partial_{tt}-\Delta\partial_t-\Delta)K(t)\big[N_0(0)\big]\notag\\
&-K(t)\big[\partial_{y}\Delta v(0)+\partial_y \Delta^2\phi(0)\big]+\partial_tK(t)\Delta\partial_y \phi(0)-(\partial_{tt}-\Delta\partial_t-\Delta)K(t)\big[\pp_xu(0)+\pp_yv(0)\big]\notag\\
&-\lambda K(t)\big[\partial_x\partial_t\Delta u(0)+\partial_y\partial_t\Delta v(0)-\partial_x\Delta^2 u(0)-\partial_y\Delta^2 v(0)\big]\notag\\
&-\lambda\partial_tK(t)\big[\partial_x\Delta u(0)+\partial_y\Delta v(0)\big]\notag\\
&-\frac12\Delta\sqrt{\Delta\partial_{yy}} K(t)[n_0]+K_1(t)[n_0].\label{eq:Ln}
\end{align}

\subsubsection{$B_n(t;n_0,\vec u_0,\vec b_0)$}\label{sec:Bn}
Now we consider the boundary term $B_n(t;n_0,\vec u_0,\vec b_0)$.
Since $K(0)=\partial_tK(0)=\partial_{tt}K(0)=0$, we have
\begin{align*}
\int_0^t K(t-s)&\partial_s(\partial_{ss}-\Delta\partial_s-\Delta)N_0(s)\,ds\\
=&K(t-s)\Big[(\partial_{ss}-\Delta\partial_s-\Delta)N_0(s)\Big|_0^t\Big]-\int_0^t \partial_sK(t-s)(\partial_{ss}-\Delta\partial_s-\Delta)N_0(s)\,ds\\
=&-K(t)\big[(\partial_{tt}-\Delta\partial_t-\Delta)N_0(0)\big]+\int_0^t \partial_tK(t-s)(\partial_{ss}-\Delta\partial_s-\Delta)N_0(s)\,ds\\
=&-K(t)\big[(\partial_{tt}-\Delta\partial_t-\Delta)N_0(0)\big]-\partial_tK(t)\big[(\partial_{t}-\Delta)N_0(0)\big]\\
&\quad
-\int_0^t \partial_tK(t-s)\Delta N_0(s)\,ds
+\int_0^t \partial_{tt}K(t-s)(\partial_{s}-\Delta)N_0(s)\,ds\\
=&-K(t)\big[(\partial_{tt}-\Delta\partial_t-\Delta)N_0(0)\big]-\partial_tK(t)\big[(\partial_{t}-\Delta)N_0(0)\big]-\partial_{tt}K(t)\big[ N_0(0)\big]\\
&\quad
+\int_0^t \partial_t(\partial_{tt}-\Delta\partial_t-\Delta)K(t-s)N_0(s)\,ds.
\end{align*}
By a similar treatment, we have
\begin{align*}
-\int_0^t K(t-s)&\Delta(\partial_{ss}-\Delta\partial_s-\Delta)N_0(s)\,ds\\
=& K(t)\big[(\partial_{t}-\Delta)\Delta N_0(0)\big]+\pp_tK(t)\big[\Delta N_0(0)\big]
-\int_0^t \Delta(\partial_{tt}-\Delta\partial_t-\Delta)K(t-s)N_0(s)\,ds;\\
-\int_0^t K(t-s)&\partial_x(\partial_{ss}-\Delta\partial_s-\Delta)N_1(s)\,ds\\
=&K(t)\big[(\partial_{t}-\Delta)\partial_x N_1(0)\big]+\partial_tK(t)\big[\partial_{x}N_1(0)\big]-\int_0^t\partial_x(\partial_{tt}-\Delta\partial_t-\Delta) K(t-s)N_1(s)\,ds;\\
-\int_0^t K(t-s)&\partial_y(\partial_{ss}-\Delta\partial_s)N_2(s)\,ds\\
=&K(t)\big[(\partial_{t}-\Delta)\partial_y N_2(0)\big]+\partial_tK(t)\big[\partial_{y}N_2(0)\big]-\int_0^t\partial_y(\partial_{tt}-\Delta\partial_t)K(t-s)N_2(s);\\
\int_0^t K(t-s)&\Delta\partial_y(\partial_s-\Delta)N_3(s)\,ds\\
=&-K(t)\big[\Delta\partial_yN_3(0)\big]+\int_0^t \Delta\partial_y\big(\partial_t-\Delta\big)K(t-s)N_3(s)\,ds;\\
-\lambda\int_0^tK(t-s)&(\partial_{ss}-\Delta\partial_s-\partial_{xx})(\partial_x\Delta u+\partial_y\Delta v)\,ds\\
=&\lambda K(t)\big[(\partial_{t}-\Delta)(\partial_x\Delta u(0)+\partial_y\Delta v(0))\big]
+\lambda \partial_{t}K(t)\big[\partial_x\Delta u(0)+\partial_y\Delta v(0)\big]\\
&-\lambda\int_0^t(\partial_{tt}-\Delta\partial_t-\partial_{xx})K(t-s)(\partial_x\Delta u(s)+\partial_y\Delta v(s))\,ds.
\end{align*}
Collecting the estimates, we obtain that
\begin{align}
&B_n(t;n_0,\vec u_0, \vec b_0)\notag\\
=&-K(t)\big[(\partial_{tt}-\Delta\partial_t-\Delta)N_0(0)\big]+K(t)\big[(\partial_{t}-\Delta)\Delta N_0(0)\big]\notag\\
&+\pp_tK(t)\big[\Delta N_0(0)\big]+K(t)\big[\pp_t \pp_xN_1(0)+\partial_t\pp_yN_2(0)\big]\notag\\
&-K(t)\big[\partial_x\Delta N_1(0)+\partial_y\Delta N_2(0)+ \pp_y\Delta N_3(0)\big]\notag\\
&-\partial_tK(t)\big[\partial_tN_0(0)-\partial_xN_1(0)-\partial_yN_2(0)-\Delta N_0(0)\big]-\partial_{tt}K(t)\big[ N_0(0)\big]\notag\\
&+\lambda K(t)\big[(\partial_{t}-\Delta)(\partial_x\Delta u(0)+\partial_y\Delta v(0))\big]
+\lambda \partial_{t}K(t)\big[\partial_x\Delta u(0)+\partial_y\Delta v(0)\big].\label{eq:Bn}
\end{align}
The recasting above allowed us to cancel some of the troubling terms from \eqref{eq:Ln} and \eqref{eq:Bn} to get \eqref{eq:LBn}.
This property is also valid for $\vec u$ and $\psi$, which is important in analysis.
 The rest of this section is split into  two subsections to estimate the linear and nonlinear parts, respectively.

\subsection{Estimates on the linear parts $L_n+B_n$}
\label{sec:LB-n}
In this subsection, we prove that
\begin{lemma}\label{lem:n-linear}
\begin{align}
\big\|\langle \nabla\rangle^3(L_n+B_n)(t;n_0,\vec u_0, \vec b_0 )\big\|_{L^2}
\lesssim&
\langle t\rangle^{-\frac14}\|U_0\|_{X_0};\label{est:n-linear}\\
\big\|\langle \nabla\rangle\pp_x(L_n+B_n)(t;n_0,\vec u_0, \vec b_0 )\big\|_{L^2}
\lesssim&
\langle t\rangle^{-\frac34}\|U_0\|_{X_0}.\label{est:n-linear-x2}
\end{align}
\end{lemma}
Now we estimate the terms in \eqref{eq:LBn}.
First of all, we consider $K(t)\big[\partial_y \Delta^2\psi_0\big]$. By  Proposition \ref{lem:Kn1-L},  we have
\begin{align}
&\big\|\langle \nabla\rangle^{3}K(t)\big[\partial_y \Delta^2\psi_0\big]\big\|_{L^2_{xy}}
=\big\|\Delta^2K(t)\big[\partial_y \langle \nabla\rangle^{3}\psi_0\big]\big\|_{L^2_{xy}}
\lesssim
\langle t\rangle^{-\frac14}\big\|\langle \nabla\rangle^{4+}\nabla\psi_0\big\|_{L^1_{xy}}.\label{LBn-1-1}
\end{align}
Moreover, almost the same as \eqref{LBn-1-1}, and using Proposition \ref{lem:Ku1-L} (i) ($\beta'=3$) instead, we have
\begin{align}
\big\|\langle \nabla\rangle\pp_xK(t)\big[\partial_y \Delta^2\psi_0\big]\big\|_{L^2_{xy}}
= \big\|\Delta\nabla\pp_{xy}K(t)\cdot \big[\langle \nabla\rangle\nabla\psi_0\big]\big\|_{L^2_{xy}}
\lesssim
\langle t\rangle^{-\frac34}\big\|\langle \nabla\rangle^{3+}\nabla\psi_0\big\|_{L^1_{xy}}.\label{LBn-1-3}
\end{align}

Now we consider $\partial_tK(t)[\Delta\partial_y \psi_0]$. From Proposition \ref{lem:Kn2-L} (i) ($\beta=1$), we have
\begin{align}
\big\|\langle \nabla\rangle^{3}\partial_tK(t)\big[\Delta\partial_y \psi_0\big]\big\|_{L^2_{xy}}\
= \big\|\nabla\partial_y\partial_tK(t)\cdot \big[ \langle \nabla\rangle^{3}\nabla\psi_0\big]\big\|_{L^2_{xy}}
\lesssim
\langle t\rangle^{-\frac12}\big\|\langle \nabla\rangle^{2+}\nabla \psi_0\big\|_{L^1_{xy}}.\label{LBn-2-1}
\end{align}
Similarly, by Proposition \ref{lem:Kn2-L} (i) again ($\beta=2$),
\begin{align}
\big\|\langle \nabla\rangle\pp_x\partial_tK(t)\big[\Delta\partial_y \psi_0\big]\big\|_{L^2_{xy}}
=\big\|\Delta\partial_y\partial_tK(t)\big[\langle \nabla\rangle\pp_x \psi_0\big]\big\|_{L^2_{xy}}
\lesssim
\langle t\rangle^{-\frac34}\big\|\langle \nabla\rangle^{1+}\nabla \psi_0\big\|_{L^1_{xy}}.\label{LBn-2-3}
\end{align}

Now we consider the term $(\partial_{tt}-\Delta\partial_t-\Delta)K(t)\big[\pp_xu_0\big]$. From Proposition \ref{lem:Kn3-L} (i) ($p=2$) and Proposition \ref{lem:Kn8-L} (i), we have
\begin{align}
\big\|\langle \nabla\rangle^{3}(\partial_{tt}-\Delta\partial_t-\Delta)K(t)\big[\pp_xu_0\big]\big\|_{L^2_{xy}}
\lesssim &
\langle t\rangle^{-\frac12}\big\|\langle \nabla\rangle^{3+}u_0\big\|_{L^1_{xy}};\label{LBn-3-1}\\
\big\|\langle \nabla\rangle\pp_x(\partial_{tt}-\Delta\partial_t-\Delta)K(t)\big[\pp_xu_0\big]\big\|_{L^2_{xy}}
\lesssim &
\langle t\rangle^{-1}\big\|\langle \nabla\rangle^{2+}u_0\big\|_{L^1_{xy}}.\label{LBn-3-3}
\end{align}

Now we consider the term $(\partial_{tt}-\Delta\partial_t)K(t)\big[\pp_yv_0\big]$.
From  Proposition \ref{lem:Kn4-L} (i) ($\beta=0,1$),
\begin{align}
\big\|\langle \nabla\rangle^{3}(\partial_{tt}-\Delta\partial_t)K(t)\big[\pp_yv_0\big]\big\|_{L^2_{xy}}
\lesssim &
\langle t\rangle^{-\frac12}\big\|\langle \nabla\rangle^{3+}v_0\big\|_{L^1_{xy}};\label{LBn-4-1}\\
\big\|\langle \nabla\rangle\pp_x(\partial_{tt}-\Delta\partial_t)K(t)\big[\pp_yv_0\big]\big\|_{L^2_{xy}}
\lesssim &
\langle t\rangle^{-1}\big\|\langle \nabla\rangle^{2+}v_0\big\|_{L^1_{xy}}.\label{LBn-4-3}
\end{align}

At last, we consider the terms \eqref{Ln-4}. For the first term $-\frac12\Delta\sqrt{\Delta\partial_{yy}} K(t)[n_0]$,  from Proposition  \ref{lem:Kn1-L},   we have
\begin{align}
\big\|\langle \nabla\rangle^{3}\Delta\sqrt{\Delta\partial_{yy}} K(t)[n_0]\big\|_{L^2_{xy}}
\lesssim
\langle t\rangle^{-\frac{1}{4}}\big\|\langle \nabla\rangle^{4+} n_0\big\|_{L^1_{xy}}.\label{LBn-5-1}
\end{align}
Similarly, by Proposition \ref{lem:Ku1-L} (i) ($\beta'=3$),
\begin{align}
\big\|\langle \nabla\rangle\pp_x\Delta\sqrt{\Delta\partial_{yy}} K(t)[n_0]\big\|_{L^2_{xy}}
\lesssim
\langle t\rangle^{-\frac{3}{4}}\big\|\langle \nabla\rangle^{3+} n_0\big\|_{L^1_{xy}}.\label{LBn-5-3}
\end{align}

For the second term $K_1(t)[n_0]$,
by Proposition \ref{lem:K1-L},
\begin{align}
\big\|\langle \nabla\rangle^{3}K_1(t)[n_0]\big\|_{L^2_{xy}}
\lesssim
\langle t\rangle^{-\frac{1}{4}}\big\|\langle \nabla\rangle^{\frac72+} n_0\big\|_{L^1_{xy}};\label{LBn-6-1}\\
\big\|\langle \nabla\rangle\pp_xK_1(t)[n_0]\big\|_{L^2_{xy}}
\lesssim
\langle t\rangle^{-\frac{3}{4}}\big\|\langle \nabla\rangle^{\frac52+} n_0\big\|_{L^1_{xy}}.\label{LBn-6-3}
\end{align}

By the estimates \eqref{LBn-1-1}, \eqref{LBn-2-1}, \eqref{LBn-3-1}, \eqref{LBn-4-1}, \eqref{LBn-5-1},  and \eqref{LBn-6-1}, we have
\begin{align*}
&\big\|\langle \nabla\rangle^3(L_n+B_n)(t;n_0,\vec u_0, \vec b_0 )\big\|_{L^2}\\
\lesssim &
\langle t\rangle^{-\frac14}\Big(\big\|\langle \nabla\rangle^{4+}\nabla\psi_0\big\|_{L^1_{xy}}+
\big\|\langle \nabla\rangle^{3+}\vec u_0\big\|_{L^1_{xy}}+
\big\|\langle \nabla\rangle^{4+} n_0\big\|_{L^1_{xy}}\Big)\notag\\
\lesssim &
\langle t\rangle^{-\frac14}\|U_0\|_{X_0}.
\end{align*}

Similarly,
by the estimates \eqref{LBn-1-3}, \eqref{LBn-2-3}, \eqref{LBn-3-3}, \eqref{LBn-4-3}, \eqref{LBn-5-3}, and \eqref{LBn-6-3}, we have
\begin{align*}
\big\|\langle \nabla\rangle\pp_x(L_n+B_n)(t;n_0,\vec u_0, \vec b_0 &)\big\|_{L^2_{xy}}
\lesssim
\langle t\rangle^{-\frac34}\|U_0\|_{X_0}.
\end{align*}

\subsection{Estimates on the nonlinear parts $\mathcal{N}_n$} \label{sec:n-L2}
By Proposition \ref{nprop1} and Lemma \ref{lem:n-linear}, we are left to prove that
\begin{align}
\big\|\langle \nabla\rangle^3\mathcal{N}_n(t;n,\vec u, \psi)\big\|_{L^2}
& \lesssim
\langle t\rangle^{-\frac14}\Big(\epsilon_0\|U\|_X+Q(\|U\|_{X})\Big);\label{est:n-nonlinear}\\
\big\|\langle \nabla\rangle\pp_x\mathcal{N}_n(t;n,\vec u, \psi)\big\|_{L^2_{xy}}
&\lesssim
\langle t\rangle^{-\frac34}\Big(\epsilon_0\|U\|_X+Q(\|U\|_{X})\Big).\label{est:n-nonlinear-xL2}
\end{align}

First of all, we split each of the nonlinearities $N_j, j=0,1,2,3$ (which were defined in Section \ref{sec:Linearization}) into the low frequency part and high frequency part, which reads as
\begin{align}\label{dec:high-low}
N_j(n,u,v,\psi)=N_j^l(n,u,v,\psi)+N_j^h(n,u,v,\psi),
\end{align}
where
\begin{align*}
N_j^l(n,u,v,\psi)(t)=N_j\big(n_{\le \langle t\rangle^{0.01}},u_{\le \langle t\rangle^{0.01}},v_{\le \langle t\rangle^{0.01}},\psi_{\le \langle t\rangle^{0.01}}\big)(t).
\end{align*}
Here we use the notation $f_{\le N}=P_{\le N}f$.
That is, each term in $N_j^h$ contains at least one high frequency part and the terms in $N_j^l$ involve only
low frequencies. Then we have
\begin{lemma}
\begin{align}
\big\|\langle \nabla\rangle^5 N_j^h(n,u,v,\psi)(t)\big\|_{L^1_{xy}}&\lesssim \langle t\rangle^{-1.03}\|U\|_X^2\quad j=0,3;\label{high-03}\\
\big\|\langle \nabla\rangle^4 N_j^h(n,u,v,\psi)(t)\big\|_{L^1_{xy}}&\lesssim \langle t\rangle^{-1.03}\frac{\|U\|_X^2}{1-\|U\|_X}\quad j=1,2.\label{high-12}
\end{align}
\end{lemma}
\begin{proof}
We only give the corresponding estimate on $N_0$, since the others can be proved by the same standard way.
By the definition of  $N_0$ in \eqref{N0}, and choosing $M$ large enough and $\epsilon$ small enough, we have, for any $2\le p\le +\infty$,
\begin{align*}
\big\|\langle \nabla\rangle^5 N_0^h(n,u,v,\psi)(t)\big\|_{L^1_{xy}}
\lesssim &
\big\|\langle \nabla\rangle^6 P_{\ge \langle t\rangle^{0.01}}n(t)\big\|_{L^2_{xy}}\Big(\big\|u(t)\big\|_{L^2_{xy}}+\big\|v(t)\big\|_{L^2_{xy}}\Big)\\
&\quad
+\big\|n(t)\big\|_{L^2_{xy}}\Big(\big\|\langle \nabla\rangle^6 P_{\ge \langle t\rangle^{0.01}}u(t)\big\|_{L^2_{xy}}+\big\|\langle \nabla\rangle^6 P_{\ge \langle t\rangle^{0.01}}v(t)\big\|_{L^2_{xy}}\Big)\\
\lesssim &
 \langle t\rangle^{-0.01(M-6)}\Big[\big\|\langle \nabla\rangle^{M}n(t)\big\|_{L^2_{xy}}\Big(\big\|u(t)\big\|_{L^2_{xy}}+\big\|v(t)\big\|_{L^2_{xy}}\Big)\\
&\quad
+\big\|n(t)\big\|_{L^2_{xy}}\Big(\big\|\langle \nabla\rangle^{M}u(t)\big\|_{L^2_{xy}}+\big\|\langle \nabla\rangle^{M}v(t)\big\|_{L^2_{xy}}\Big)\Big]\\
\lesssim &
 \langle t\rangle^{-0.8}\Big[\langle t\rangle^{\epsilon}\|U\|_X\>\langle t\rangle^{-0.5}\|U\|_X
+\langle t\rangle^{-0.25}\|U\|_X\>\langle t\rangle^{\epsilon}\|U\|_X\Big]\\
\lesssim &
\langle t\rangle^{-1.03}\|U\|_X^2.
\end{align*}
This proves the estimate on $N_0$.
\end{proof}

This lemma provides explicit decay rates for the $L^1$-norm of the high frequency parts in the nonlinearities. Since the decay rates in the estimates are smaller than $-1$ ($-1.03$ indeed), the corresponding nonlinearities are time integrable in the light of kernel estimates in Section \ref{KernelProperty}. See Section \ref{sec:n-N0-1} as an example.

\subsubsection{$\int_0^t\partial_t(\partial_{tt}-\Delta\partial_t-\Delta)K(t-s)N_0(s)\,ds$}\label{sec:n-N0-1}
In this subsubsection, we will prove that
\begin{align}
\Big\|\langle\nabla\rangle^3\int_0^t\partial_t(\partial_{tt}-\Delta\partial_t-\Delta)K(t-s)N_0(s)\,ds\Big\|_{L^2_{xy}}
\lesssim
\langle t \rangle^{-\frac14}\|U\|_X^2;\label{est:n-N0-1-1}\\
\Big\|\langle \nabla\rangle\pp_x\int_0^t\partial_t(\partial_{tt}-\Delta\partial_t-\Delta)K(t-s)N_0(s)\,ds\Big\|_{L^2_{xy}}
\lesssim
\langle t \rangle^{-\frac34}\|U\|_X^2.\label{est:n-N0-1-3}
\end{align}
First, we consider the piece of $N_0^h$. By Proposition \ref{lem:Kn5-L} (i) ($\beta=0$) and \eqref{high-03},  we have
\begin{align*}
&\Big\|\langle\nabla\rangle^3\int_0^t\partial_t(\partial_{tt}-\Delta\partial_t-\Delta)K(t-s)N_0^h(s)\,ds\Big\|_{L^2_{xy}}\\
\lesssim
&\int_0^t\Big\|\partial_t(\partial_{tt}-\Delta\partial_t-\Delta)K(t-s)\langle\nabla\rangle^3N_0^h(s)\Big\|_{L^2_{xy}}\,ds\\
\lesssim &
\int_0^t \langle t-s \rangle^{-\frac12}\big\|\langle\nabla\rangle^3N_0^h(s)\big\|_{L^1_{xy}}\,ds\\
\lesssim &
\int_0^t \langle t-s \rangle^{-\frac12}\langle s\rangle^{-1.03}\,ds\|U\|_X^2\\
\lesssim &
\langle t \rangle^{-\frac12}\|U\|_X^2.
\end{align*}
By the same way, from Proposition \ref{lem:Kn5-L} (i) ($\beta=1$) again, we have
\begin{align*}
&\Big\|\langle \nabla\rangle\pp_x\int_0^t\partial_t(\partial_{tt}-\Delta\partial_t-\Delta)K(t-s)N_0^h(s)\,ds\Big\|_{L^2_{xy}}\\
\lesssim
&\int_0^t\Big\|\nabla\partial_t(\partial_{tt}-\Delta\partial_t-\Delta)K(t-s)\langle\nabla\rangle N_0^h(s)\Big\|_{L^2_{xy}}\,ds\\
\lesssim &
\int_0^t \langle t-s \rangle^{-1}\big\|\langle\nabla\rangle^{1+} N_0^h(s)\big\|_{L^1_{xy}}\,ds\\
\lesssim &
\int_0^t \langle t-s \rangle^{-1}\langle s\rangle^{-1.03}\,ds\|U\|_X^2\\
\lesssim &
\langle t \rangle^{-1}\|U\|_X^2.
\end{align*}

\begin{rem}
In the following context, we focus our attention on the low-frequency parts most of the time, since the high-frequency parts can be treated standardly as $N_0^h$ above, by using
the energy estimates. To simplify the notation, we only write $n,u,v,\psi$ for $n_{\le \langle s\rangle^{0.01}},u_{\le \langle s\rangle^{0.01}},v_{\le \langle s\rangle^{0.01}},\psi_{\le \langle s\rangle^{0.01}}$ in the low-frequency parts, if there is no confusion.
\end{rem}
Now we consider $N_0^l$ piece. By Lemma \ref{lem:Kn5},
\begin{align*}
&\Big\|\langle\nabla\rangle^3\int_0^t\partial_t(\partial_{tt}-\Delta\partial_t-\Delta)K(t-s)N_0^l(s)\,ds\Big\|_{L^2_{xy}}\\
\lesssim
&\int_0^t\Big\|\nabla\partial_t(\partial_{tt}-\Delta\partial_t-\Delta)K(t-s)\cdot\langle\nabla\rangle^3P_{\lesssim \langle s\rangle^{0.01}}(n\vec u)(s)\Big\|_{L^2_{xy}}\,ds\\
\lesssim &
\int_0^t  \langle t-s \rangle^{-1} \big\|\langle\nabla\rangle^3P_{\lesssim \langle s\rangle^{0.01}}(n\vec u)(s)\big\|_{L^1_{xy}} \,ds\\
\lesssim &
\int_0^t  \langle t-s \rangle^{-1}  \langle s\rangle^{0.04}\big\|n\big\|_{L^2_{xy}}\big\|\vec u\big\|_{L^2_{xy}}  \,ds\\
\lesssim &
\int_0^t \langle t-s \rangle^{-1}\langle s\rangle^{0.04-0.25-0.5}\,ds\|U\|_X^2\\
\lesssim &
\langle t \rangle^{-\frac1{4}}\|U\|_X^2.
\end{align*}
Also, almost the same, and using
\begin{align}\label{1301}
\|\pp_x(n\vec u)\|_{L^1_{xy}}\le \Big(\|\pp_xn\|_{L^2_{xy}}\|\vec u\|_{L^2_{xy}}
+\|n\|_{L^2_{xy}}\|\nabla\cdot\vec u\|_{L^2_{xy}}\Big)\lesssim \langle s\rangle^{-1.25}\|U\|_X^2,
\end{align}
we have
\begin{align*}
&\Big\|\langle \nabla\rangle\pp_x\int_0^t\partial_t(\partial_{tt}-\Delta\partial_t-\Delta)K(t-s)N_0^l(s)\,ds\Big\|_{L^2_{xy}}\\
\lesssim &
\int_0^t\Big\|\nabla\partial_t(\partial_{tt}-\Delta\partial_t-\Delta)K(t-s)\cdot P_{\lesssim \langle s\rangle^{0.01}}\langle \nabla\rangle\pp_x(n\vec u)(s)\Big\|_{L^2_{xy}}\,ds\\
\lesssim &
\int_0^t \langle t-s \rangle^{-1}\langle s\rangle^{0.03}\|\pp_x(n\vec u)\|_{L^1_{xy}}\\
\lesssim &
\int_0^t \langle t-s \rangle^{-1}\langle s\rangle^{0.03-1.25}\,ds\|U\|_X^2\\
\lesssim &
\langle t \rangle^{-\frac34}\|U\|_X^2.
\end{align*}
Therefore, we give the estimate \eqref{est:n-N0-1-1} and \eqref{est:n-N0-1-3}.

\subsubsection{$\int_0^t\Delta(\partial_{tt}-\Delta\partial_t-\Delta)K(t-s)N_0(s)\,ds$}\label{sec:n-N0-2}
This is one of most tricky terms. Indeed, regardless of the loss of regularity, the best estimate one may get is
\begin{align*}
\big\|N_0(s)\big\|_{L^1_{x}L^2_{y}}\lesssim \langle s \rangle^{-1}\|U\|_X^2,
\end{align*}
which is critical for integrable (and it is not integrable). We need some new argument to prove the estimates followed,
\begin{align}
\Big\|\langle\nabla\rangle^3\int_0^t\Delta(\partial_{tt}-\Delta\partial_t-\Delta)K(t-s)N_0(s)\,ds\Big\|_{L^2_{xy}}
\lesssim
\langle t \rangle^{-\frac14}\|U\|_X^2;\label{est:n-N0-2-1}\\
\Big\|\langle \nabla\rangle\pp_x\int_0^t\Delta(\partial_{tt}-\Delta\partial_t-\Delta)K(t-s)N_0(s)\,ds\Big\|_{L^2_{xy}}
\lesssim
\langle t \rangle^{-\frac34}\|U\|_X^2.\label{est:n-N0-2-3}
\end{align}
According to the frequency,  and the definition \eqref{N0}: $N_0=\pp_x(nu)+\pp_y(nv)$, we have
\begin{align*}
N_0 =&N_0^h +N_0^l \\
=&N_0^h + \pp_x(n_{\le \langle s\rangle^{0.01}}\>u_{\le \langle s\rangle^{0.01}}) +\pp_y(n_{\le \langle s\rangle^{0.01}}\>u_{\le \langle s\rangle^{0.01}}) \\
=&N_0^h + \pp_x(n_{\le \langle s\rangle^{0.01}}\>u_{\le \langle s\rangle^{0.01}}) +\pp_y(n_{\le \langle s\rangle^{0.01}}\>v_{\le \langle s\rangle^{-10}}) \,ds
\\
&\quad+\pp_y(n_{\le \langle s\rangle^{0.01}}\>v_{\langle s\rangle^{-0.05}\le \cdot \le \langle s\rangle^{0.01}}) +\pp_y(n_{\le \langle s\rangle^{0.01}}\>v_{\langle s\rangle^{-10}\le\cdot \le \langle s\rangle^{-0.05}}).
\end{align*}
Similar as before, we denote $f_{\ge N}=P_{\ge N}f$ and  $f_{N\le \cdot \le M}=P_{N\le \cdot \le M}f$ here.
Then by the fourth equation in \eqref{e2.1}: $v=-\pp_t\psi+N_3$, and the product rule, we further have
\begin{align}
N_0 =&N_0^h + \pp_x(n_{\le \langle s\rangle^{0.01}}\>u_{\le \langle s\rangle^{0.01}}) +\pp_y(n_{\le \langle s\rangle^{0.01}}\>v_{\le \langle s\rangle^{-10}})
+\pp_y(n_{\le \langle s\rangle^{0.01}}\>v_{\langle s\rangle^{-0.05}\le \cdot \le \langle s\rangle^{0.01}})\notag \\
&\quad+\pp_y\big(n_{\le \langle s\rangle^{0.01}}\>P_{\langle s\rangle^{-10}\le\cdot \le \langle s\rangle^{-0.05}}N_3\big) +\pp_y\big(n_{\le \langle s\rangle^{0.01}}\>P_{\langle s\rangle^{-10}\le\cdot \le \langle s\rangle^{-0.05}}\pp_s \psi\big) \notag\\
=&N_0^h + \pp_x(n_{\le \langle s\rangle^{0.01}}\>u_{\le \langle s\rangle^{0.01}}) +\pp_y(n_{\le \langle s\rangle^{0.01}}\>v_{\le \langle s\rangle^{-10}})
+\pp_y(n_{\le \langle s\rangle^{0.01}}\>v_{\langle s\rangle^{-0.05}\le \cdot \le \langle s\rangle^{0.01}})\notag \\
&\quad+\pp_y(n_{\le \langle s\rangle^{0.01}}\>P_{\langle s\rangle^{-10}\le\cdot \le \langle s\rangle^{-0.05}}N_3) +\pp_y\pp_s \big(n_{\le \langle s\rangle^{0.01}}\>P_{\langle s\rangle^{-10}\le\cdot \le \langle s\rangle^{-0.05}}\psi\big) \notag\\
&\quad
-\pp_y\big(\pp_s n_{\le \langle s\rangle^{0.01}}\>P_{\langle s\rangle^{-10}\le\cdot \le \langle s\rangle^{-0.05}}\psi\big) + \mbox{\emph{other parts}},\label{re-N0}
\end{align}
where  the \emph{other parts} include the terms which $\pp_s$ hits $P_{\le \langle s\rangle^{0.01}}$ or $P_{\langle s\rangle^{-10}\le\cdot \le \langle s\rangle^{-0.05}}$.
By \eqref{re-N0}, we split the term into several parts as follows,
\begin{align}
&\int_0^t\Delta(\partial_{tt}-\Delta\partial_t-\Delta)K(t-s)N_0(s)\,ds\notag\\
=&
\int_0^t\Delta(\partial_{tt}-\Delta\partial_t-\Delta)K(t-s)N_0^h(s)\,ds\label{n-N0-2-h}\\
&+\int_0^t\Delta(\partial_{tt}-\Delta\partial_t-\Delta)K(t-s)\pp_x(n_{\le \langle s\rangle^{0.01}}\>u_{\le \langle s\rangle^{0.01}})(s)\,ds\label{n-N0-2-l-1}\\
&+\int_0^t\Delta(\partial_{tt}-\Delta\partial_t-\Delta)K(t-s)\pp_y(n_{\le \langle s\rangle^{0.01}}\>v_{\le \langle s\rangle^{-10}})(s)\,ds\label{nv-easy1}\\
&+
\int_0^t\Delta(\partial_{tt}-\Delta\partial_t-\Delta)K(t-s)\pp_y(n_{\le \langle s\rangle^{0.01}}\>v_{\ge \langle s\rangle^{-0.05}})(s)\,ds\label{nv-easy2}\\
&+
\int_0^t\Delta\pp_y(\partial_{tt}-\Delta\partial_t-\Delta)K(t-s)\big(n_{\le \langle s\rangle^{0.01}}\>P_{\langle s\rangle^{-10}\le\cdot \le \langle s\rangle^{-0.05}}N_3\big)(s)\,ds\label{nv-hard-2}\\
&+
\int_0^t\Delta\pp_y(\partial_{tt}-\Delta\partial_t-\Delta)K(t-s)\pp_s \big(n_{\le \langle s\rangle^{0.01}}\>P_{\langle s\rangle^{-10}\le\cdot \le \langle s\rangle^{-0.05}}\psi\big)(s)\,ds\label{nv-hard-11}\\
&
-\int_0^t\Delta\pp_y(\partial_{tt}-\Delta\partial_t-\Delta)K(t-s)\big(\pp_s n_{\le \langle s\rangle^{0.01}}\>P_{\langle s\rangle^{-10}\le\cdot \le \langle s\rangle^{-0.05}}\psi\big) (s)\,ds\label{nv-hard-12}\\
&\quad+ \mbox{\emph{other easy terms}},\label{nv-hard-13}
\end{align}
where the \emph{other easy terms} is the corresponding terms from \emph{other parts} in \eqref{re-N0}.

The high frequency piece \eqref{n-N0-2-h} can be treated standardly as in Section \ref{sec:n-N0-1}, and thus we obtain (the details of the proof are omitted)
\begin{align*}
\Big\|\langle\nabla\rangle^3\eqref{n-N0-2-h}\Big\|_{L^2_{xy}}
&\lesssim
\langle t \rangle^{-\frac1{4}}\|U\|_X^2;\\
\Big\|\langle \nabla\rangle\pp_x\eqref{n-N0-2-h}\Big\|_{L^2_{xy}}
&\lesssim
\langle t \rangle^{-\frac3{4}}\|U\|_X^2.
\end{align*}

Now we consider the part \eqref{n-N0-2-l-1}. By Proposition \ref{lem:Kn3-L} (ii) we have
\begin{align*}
\Big\|\langle\nabla\rangle^3\eqref{n-N0-2-l-1}\Big\|_{L^2_{xy}}
\lesssim &
\int_0^t\big\|P_{\lesssim \langle s\rangle^{0.01}}\nabla \pp_x(\partial_{tt}-\Delta\partial_t-\Delta)K(t-s)\cdot \nabla \langle\nabla\rangle^3 (nu)(s)\big\|_{L^2_{xy}}\,ds\\
\lesssim &
\int_0^t\langle t-s\rangle^{-\frac34}  \|P_{\lesssim \langle s\rangle^{0.01}}\langle \nabla \rangle^{4+}nu\|_{L^1_{xy}}\,ds\\
\lesssim &
\int_0^t \langle t-s\rangle^{-\frac34} \langle s\rangle^{0.05}
\langle s\rangle^{-0.75}\,ds\>\|U\|_X^2\\
\lesssim &
\langle t \rangle^{-\frac14}\|U\|_X^2.
\end{align*}
By the same way,  and using \eqref{1301} we have
\begin{align*}
\Big\|\langle \nabla\rangle\pp_x\eqref{n-N0-2-l-1}\Big\|_{L^2_{xy}}
\lesssim &
\int_0^t\langle s\rangle^{0.03} \langle t-s\rangle^{-\frac34}
\|\pp_x(nu)\|_{L^1_{xy}}\,ds\\
\lesssim &
\int_0^t\langle s\rangle^{0.03-1.25} \langle t-s\rangle^{-\frac34}
\,ds\>\|U\|_X^2
\lesssim
\langle t \rangle^{-\frac34}\|U\|_X^2.
\end{align*}

Next, we consider \eqref{nv-easy1} and  \eqref{nv-easy2}.
First, by Bernstein's inquality, we have
$$
\big\|n_{\le \langle s\rangle^{0.01}}\>v_{\le \langle s\rangle^{-10}}\big\|_{L^1_x L^2_y}
\lesssim \big\|n\big\|_{L^2_{xy}}\>\big\|v_{\le \langle s\rangle^{-10}}\big\|_{L^2_x L^\infty_y}
\lesssim
\langle s \rangle^{-4.9} \big\|n\big\|_{L^2_{xy}}\>\big\|v\big\|_{L^2_{xy}}\lesssim \langle s \rangle^{-5} \|U\|_X^2.
$$
Then by Proposition \ref{lem:Kn6-L},
\begin{align*}
&\Big\|\langle\nabla\rangle^3\eqref{nv-easy1}\Big\|_{L^2_{xy}}\\
\lesssim &
\int_0^t\big\|\Delta(\partial_{tt}-\Delta\partial_t-\Delta)K(t-s)\pp_y(n_{\le \langle s\rangle^{0.01}}\>v_{\le \langle s\rangle^{-10}})(s)\|_{L^2_{xy}}\,ds\\
\lesssim &
\int_0^t \langle t-s \rangle^{-\frac14}  \big\| \langle \nabla \rangle^5\big(n_{\le \langle s\rangle^{0.01}}\>v_{\le \langle s\rangle^{-10}}\big)\big\|_{L^1_x L^2_y}\,ds\\
\lesssim &
\int_0^t\langle t-s \rangle^{-\frac14}\langle s \rangle^{0.05} \langle s \rangle^{-5}\,ds \|U\|_X^2\\
\lesssim &
\langle t \rangle^{-\frac14}\|U\|_X^2.
\end{align*}
On the other hand, since
$$
\big\|n\big\|_{L^2_{xy}}\big\|v_{\ge \langle s\rangle^{-0.05}}\big\|_{L^2_{xy}}\lesssim \langle s \rangle^{0.05}\big\|n\big\|_{L^2_{xy}}\|\nabla v\|_{L^2_{xy}}
\lesssim\langle s \rangle^{-1.2} \|U\|_X^2,
$$
by the same treatment we also have
\begin{align*}
\Big\|\langle\nabla\rangle^3\eqref{nv-easy2}\Big\|_{L^2_{xy}}
\lesssim
\langle t \rangle^{-\frac14}\|U\|_X^2.
\end{align*}
Therefore, we get
\begin{align}
\Big\|\langle\nabla\rangle^3\big(\eqref{nv-easy1}+\eqref{nv-easy2}\big)\Big\|_{L^2_{xy}}
\lesssim
\langle t \rangle^{-\frac1{4}}\|U\|_X^2.\label{1533-1}
\end{align}
By using Proposition \ref{lem:Kn3-L} (ii) instead, we also get
\begin{align}
\Big\|\langle \nabla\rangle\pp_x\big(\eqref{nv-easy1}+\eqref{nv-easy2}\big)\Big\|_{L^2_{xy}}
\lesssim
\langle t \rangle^{-\frac3{4}}\|U\|_X^2.\label{1533-3}
\end{align}

Now we turn to consider the part \eqref{nv-hard-2}, by Proposition \ref{lem:Kn6-L},
\begin{align}
\|\langle\nabla\rangle^3\eqref{nv-hard-2}\|_{L^2_{xy}}
\lesssim &
\int_0^t\big\|\Delta(\partial_{tt}-\Delta\partial_t-\Delta)K(t-s)\pp_y\big(n_{\le \langle s\rangle^{0.01}}\>P_{\langle s\rangle^{-10}\le\cdot \le \langle s\rangle^{-0.05}}N_3\big)(s)\big\|_{L^2_{xy}}\,ds\notag\\
\lesssim &
\int_0^t \langle t-s \rangle^{-\frac14}\langle s \rangle^{0.05} \big\|n_{\le \langle s\rangle^{0.01}}\>P_{\langle s\rangle^{-10}\le\cdot \le \langle s\rangle^{-0.05}}N_3\big\|_{L^1_{xy}}\,ds\notag\\
\lesssim &
\int_0^t\langle s \rangle^{-0.05}\langle t-s \rangle^{-\frac14}\big\|n\big\|_{L^2_{xy}}\big\|\nabla\psi\big\|_{L^2_{xy}}\>\big\|\vec u\big\|_{L^\infty_{xy}}\,ds\notag\\
\lesssim &
\int_0^t\langle s \rangle^{0.05-1.5}\langle t-s \rangle^{-\frac14}\,ds\|U\|_X^3\notag\\
\lesssim &
\langle t \rangle^{-\frac14}\|U\|_X^3.\label{140211}
\end{align}
Also, using Proposition \ref{lem:Kn3-L} (ii) instead, we obtain
\begin{align}
\Big\|\langle \nabla\rangle\pp_x\eqref{nv-hard-2}\Big\|_{L^2_{xy}}
\lesssim
\langle t \rangle^{-\frac3{4}}\|U\|_X^2.\label{140211-3}
\end{align}

To prove \eqref{nv-hard-11},  we integrate by parts  to get
\begin{align}
\eqref{nv-hard-11}=&-\Delta\pp_y(\partial_{tt}-\Delta\partial_t-\Delta)K(t) \big(P_{\lesssim 1}n_0\>P_{\sim 1}\psi_0\big)\label{nv-hard-11-b}\\
&+\int_0^t\Delta\pp_y(\partial_{tt}-\Delta\partial_t-\Delta)\pp_tK(t-s) \big(n_{\le \langle s\rangle^{0.01}}\>P_{\langle s\rangle^{-10}\le\cdot \le \langle s\rangle^{-0.05}}\psi\big)(s)\,ds.\label{nv-hard-111}
\end{align}
By Proposition \ref{lem:Kn6-L} and Proposition \ref{lem:Kn3-L} (ii), the boundary term \eqref{nv-hard-11-b} can be controlled as following,
\begin{align*}
\|\langle\nabla\rangle^3\eqref{nv-hard-11-b}\|_{L^2_{xy}}
\lesssim &
\langle t \rangle^{-\frac14}
\|\langle\nabla\rangle^3 P_{\lesssim 1}n_0\>P_{\sim 1}\psi_0\|_{L^1_{xy}}
\lesssim
\langle t \rangle^{-\frac14}\|n_0\|_{L^2_{xy}}\|\nabla\psi_0\|_{L^2_{xy}};\\
\|\pp_x\eqref{nv-hard-11-b}\|_{L^2_{xy}}
\lesssim &
\langle t \rangle^{-\frac3{4}}
\|\langle\nabla\rangle^3 P_{\lesssim 1}n_0\>P_{\sim 1}\psi_0\|_{L^1_{xy}}
\lesssim
\langle t \rangle^{-\frac3{4}}\|n_0\|_{L^2_{xy}}\|\nabla\psi_0\|_{L^2_{xy}}
\end{align*}

To prove \eqref{nv-hard-111}, we need the following estimate, for any $\epsilon>0$,
\begin{align}
\|P_{\langle s\rangle^{-10}\le\cdot \le \langle s\rangle^{-0.05}}\psi\|_{L^\infty_{xy}}
\lesssim& \big\||\nabla|^{1-\epsilon}P_{\langle s\rangle^{-10}\le\cdot \le \langle s\rangle^{-0.05}}\psi\big\|_{L^2_{xy}}\notag\\
\lesssim&\langle s \rangle^{-0.02} \big\||\nabla|^{\frac12+\epsilon}\psi\big\|_{L^2_{xy}}\lesssim \langle s \rangle^{-0.27} \|U\|_X.\label{psi-infty}
\end{align}
Then using \eqref{est:Kn6-2}, \eqref{est:Kn8-2} and \eqref{psi-infty}, we have
\begin{align*}
\|\langle\nabla\rangle^3\eqref{nv-hard-111}\|_{L^2_{xy}}
\lesssim &
\int_0^t\big\|A^2\pp_t (\partial_{tt}+A^2\partial_t+A^2)\widehat K(t-s,\xi,\eta)\big\|_{L^\infty_{\xi\eta}}\\
&\qquad \cdot\big\|\langle\nabla\rangle^3\big(n_{\le \langle s\rangle^{0.01}}\>P_{\langle s\rangle^{-10}\le\cdot \le \langle s\rangle^{-0.05}}\psi\big)\big\|_{L^2_{xy}}\,ds\\
\lesssim &
\int_0^t\langle t-s\rangle^{-1}\langle s \rangle^{0.03}\big\|n\big\|_{L^2_{xy}}\big\|P_{\langle s\rangle^{-10}\le\cdot \le \langle s\rangle^{-0.05}}\psi\big\|_{L^\infty_{xy}}\,ds\\
\lesssim &
\int_0^t\langle s \rangle^{-0.27+0.03-0.25}\langle t-s\rangle^{-1}\,ds\|U\|_X^2\\
\lesssim &
\langle t \rangle^{-\frac14}\|U\|_X^2.
\end{align*}

Similarly,
\begin{align*}
\|\langle \nabla\rangle\pp_x\eqref{nv-hard-111}\|_{L^2_{xy}}
\lesssim &
\int_0^t\big\|A^2\pp_t (\partial_{tt}+A^2\partial_t+A^2)\widehat K(t-s,\xi,\eta)\big\|_{L^\infty_{\xi\eta}}\\
&\qquad\cdot
\big\|\langle \nabla\rangle\pp_x\big(n_{\le \langle s\rangle^{0.01}}\>P_{\langle s\rangle^{-10}\le\cdot \le \langle s\rangle^{-0.05}}\psi\big)\big\|_{L^2_{xy}}\,ds\\
\lesssim &
\int_0^t\langle t-s\rangle^{-1}\langle s \rangle^{0.03}\Big(\big\|\pp_xn\big\|_{L^2_{xy}}\big\|P_{\langle s\rangle^{-10}\le\cdot \le \langle s\rangle^{-0.05}}\psi\big\|_{L^\infty_{xy}}+\big\|n\big\|_{L^\infty_{xy}}\big\|\pp_x\psi\big\|_{L^2_{xy}}\Big)\,ds\\
\lesssim &
\int_0^t\langle s \rangle^{0.04-1}\langle t-s\rangle^{-1}\,ds\|U\|_X^2\\
\lesssim &
\langle t \rangle^{-\frac3{4}}\|U\|_X^2.
\end{align*}

For the term \eqref{nv-hard-12},
since $\pp_t n=-\pp_xu-\pp_yv+N_0$, we have
\begin{align}
\eqref{nv-hard-12}
=&
-\int_0^t\Delta\pp_y(\partial_{tt}-\Delta\partial_t-\Delta)K(t-s)\big(P_{\le \langle s\rangle^{0.01}}\pp_xu\>P_{\langle s\rangle^{-10}\le\cdot \le \langle s\rangle^{-0.05}}\psi\big)(s)\,ds\label{nv-hard-121}\\
&\quad
-\int_0^t\Delta\pp_y(\partial_{tt}-\Delta\partial_t-\Delta)K(t-s)\big(P_{\le \langle s\rangle^{0.01}}\pp_yv\>P_{\langle s\rangle^{-10}\le\cdot \le \langle s\rangle^{-0.05}}\psi\big)(s)\,ds\label{nv-hard-122}\\
&\quad
+\int_0^t\Delta\pp_y(\partial_{tt}-\Delta\partial_t-\Delta)K(t-s)\big(P_{\le \langle s\rangle^{0.01}}N_0\>P_{\langle s\rangle^{-10}\le\cdot \le \langle s\rangle^{-0.05}}\psi\big)(s)\,ds.\label{nv-hard-123}
\end{align}
These three terms can be treated by the similar way. Now we consider the term  \eqref{nv-hard-121}.  By Proposition \ref{lem:Kn6-L}, Sobolev's and Beinstein's inequalities,  we have
\begin{align*}
\|\langle\nabla\rangle^3\eqref{nv-hard-121}\|_{L^2_{xy}}
\lesssim &
\int_0^t\big\|\Delta(\partial_{tt}-\Delta\partial_t-\Delta)K(t-s)\pp_y\big(P_{\le \langle s\rangle^{0.01}}\pp_xu\>P_{\langle s\rangle^{-10}\le\cdot \le \langle s\rangle^{-0.05}}\psi\big)(s)\big\|_{L^2_{xy}}\notag\\
\lesssim &
\int_0^t\langle t-s \rangle^{-\frac14}\langle s \rangle^{0.01(3+\frac52)}\big\|P_{\le \langle s\rangle^{0.01}}\pp_x u\>P_{\langle s\rangle^{-10}\le\cdot \le \langle s\rangle^{-0.05}}\psi\big\|_{L^1_xL^2_{y}}\,ds\notag\\
\lesssim &
\int_0^t\langle t-s \rangle^{-\frac14}\langle s \rangle^{0.06}\big\|\pp_x u\big\|_{L^2_{xy}}
\>\big\|P_{\langle s\rangle^{-10}\le\cdot \le \langle s\rangle^{-0.05}}\psi\big\|_{L^2_{x}L^\infty_{y}}\,ds\notag\\
\lesssim &
\int_0^t\langle t-s \rangle^{-\frac14}\langle s \rangle^{0.06}\big\|\pp_x u\big\|_{L^2_{xy}}\>\big\||\nabla|^{\frac12-}\psi\big\|_{L^2_{xy}}\,ds\notag\\
\lesssim &
\int_0^t\langle t-s \rangle^{-\frac14}\langle s \rangle^{0.07-1-0.25}\,ds\|U\|_X^2\notag\\
\lesssim &
\langle t \rangle^{-\frac14}\|U\|_X^2.
\end{align*}
Almost the same, we have the estimate on  the term \eqref{nv-hard-122},
\begin{align*}
\|\langle\nabla\rangle^3\eqref{nv-hard-122}\|_{L^2_{xy}}
\lesssim &
\int_0^t\langle t-s \rangle^{-\frac14}\langle s \rangle^{0.06}\big\|\pp_yv\big\|_{L^2_{xy}}\>\big\||\nabla|^{\frac12-}\psi\big\|_{L^2_{xy}}\,ds\notag\\
\lesssim &
\int_0^t\langle t-s \rangle^{-\frac14}\langle s \rangle^{0.07-1-0.25}\,ds\|U\|_X^2\notag\\
\lesssim &
\langle t \rangle^{-\frac14}\|U\|_X^2.
\end{align*}
Also, for the term \eqref{nv-hard-123}, we have
\begin{align*}
\|\langle\nabla\rangle^3\eqref{nv-hard-123}\|_{L^2_{xy}}
\lesssim &
\int_0^t\langle s \rangle^{0.06}\langle t-s \rangle^{-\frac14}\big\|n\big\|_{L^2_{xy}}\big\|\vec u\big\|_{L^2_{x}L^\infty_y}\>\big\|P_{\langle s\rangle^{-10}\le\cdot \le \langle s\rangle^{-0.05}}\psi\big\|_{L^\infty_{xy}}\,ds\notag\\
\lesssim &
\int_0^t\langle s \rangle^{0.07}\langle t-s \rangle^{-\frac14}\big\|n\big\|_{L^2_{xy}}\big\|\vec u\big\|_{L^2_{x}L^\infty_y}\>\big\|\nabla\psi\big\|_{L^2_{xy}}\,ds\notag\\
\lesssim &
\int_0^t\langle s \rangle^{0.08-0.25-0.75-0.25}\langle t-s \rangle^{-\frac14}\,ds\|U\|_X^3\notag\\
\lesssim &
\langle t \rangle^{-\frac14}\|U\|_X^3.
\end{align*}
Combining these three estimates,  we obtain
\begin{align}
\|\langle\nabla\rangle^3\eqref{nv-hard-12}\|_{L^2_{xy}}
\lesssim
\langle t \rangle^{-\frac1{4}}\big(\|U\|_X^2+\|U\|_X^3\big).\label{140213-1}
\end{align}
Replacing the operator $\Delta (\partial_{tt}-\Delta\partial_t-\Delta) K(t-s)$ by  $\Delta\pp_x (\partial_{tt}-\Delta\partial_t-\Delta) K(t-s)$,
and using Proposition \ref{lem:Kn3-L} (ii) instead, we also get
\begin{align}
\|\langle \nabla\rangle\pp_x\eqref{nv-hard-12}\|_{L^2_{xy}}
\lesssim
\langle t \rangle^{-\frac3{4}}\big(\|U\|_X^2+\|U\|_X^3\big).\label{140213-3}
\end{align}

At last, we estimate \eqref{nv-hard-13}. Since for any $\beta\in \R$, $1\le q\le \infty$,
\begin{align}\label{17.17}
\|\pp_sP_{\le \langle s\rangle^\beta}f\|_{L^q}\lesssim \langle s\rangle^{-1}\|P_{\sim \langle s\rangle^\beta}f\|_{L^q},
\end{align}
(see Appendix \eqref{app4} for its proof) it is easy to prove that
\begin{align}
\|\langle\nabla\rangle^3\eqref{nv-hard-13}\|_{L^2_{xy}}
\lesssim
\langle t \rangle^{-\frac14}\|U\|_X^2;\quad \|\pp_x\eqref{nv-hard-13}\|_{L^2_{xy}}
\lesssim
\langle t \rangle^{-\frac34}\|U\|_X^2.
\label{est:otheresayterms}
\end{align}

Collecting the estimates above, we establish \eqref{est:n-N0-2-1} and \eqref{est:n-N0-2-3}.

\subsubsection{$\int_0^t\partial_x(\partial_{tt}-\Delta\partial_t-\Delta)K(t-s)N_1(s)\,ds$.}\label{sec:n-N1}
We shall prove that
\begin{align}
\Big\|\langle\nabla\rangle^3\int_0^t\partial_x(\partial_{tt}-\Delta\partial_t-\Delta)K(t-s)N_1(s)\,ds\Big\|_{L^2_{xy}}
\lesssim &
\langle t \rangle^{-\frac1{4}}Q\big(\|U\|_X\big);\label{est:n-N1-L2}\\
\Big\|\int_0^t\partial_x^2(\partial_{tt}-\Delta\partial_t-\Delta)K(t-s)N_1(s)\,ds\Big\|_{L^2_{xy}}
\lesssim &
\langle t \rangle^{-\frac3{4}}Q\big(\|U\|_X\big).\label{est:n-N1-xL2}
\end{align}
As before, the estimate for the high frequency part is standard and can be obtained by
Lemma \ref{lem:Kn3} and \eqref{high-12}.
Thus we only consider the low frequency piece $N_1^l$. We first split $N_1$ it into two parts,
$$
N_1=N_{11}+N_{12},
$$
where
\begin{align*}
N_{11}&=-(u\pp_x u+v\pp_y u) -\frac{-\nabla n\cdot\nabla u+n\lambda(\partial_{xx} u+\partial_{xy}v)}{\rho}+\frac{n\>\pp_x\psi\Delta\psi}{\rho};\\
N_{12}&= -\pp_x\psi\Delta\psi- n\pp_x n-\nabla\cdot(n\nabla u),
\end{align*}
and write $N_1^l=N_{11}^l+N_{12}^l$ respectively.
Then
\begin{align}
&\int_0^t\partial_x(\partial_{tt}-\Delta\partial_t-\Delta)K(t-s)N_1^l(s)\,ds\notag\\
= &
\int_0^t\partial_x(\partial_{tt}-\Delta\partial_t-\Delta)K(t-s)
N_{11}^l(s)\,ds\label{1114}\\
&
+\int_0^t\partial_x(\partial_{tt}-\Delta\partial_t-\Delta)K(t-s)
N_{12}^l(s)\,ds.\label{1115}
\end{align}
First, we have
\begin{lemma}\label{lem:N11-l}
\begin{align*}
\|N_{11}^l\|_{L^1_{xy}}\lesssim \langle t\rangle^{-1}Q(\|U\|_X).
\end{align*}
\end{lemma}
\begin{proof} By H\"older's inequality and Bernstein' inequality, we have
\begin{align*}
\|N_{11}^l\|_{L^1_{xy}}
\lesssim &\|\nabla u \|_{L^2_{xy}}\|\vec u\|_{L^2_{xy}}
+ \frac{ \|\langle\nabla\rangle n\|_{L^2_{xy}}\Big(\|\nabla u\|_{L^2_{xy}}+
\|\lambda P_{\le \langle s\rangle^{0.01}}(\partial_{xx} u+\partial_{xy}v)\|_{L^2_{xy}}\Big)}{1-\|n\|_{L^\infty_{xy}}}\\
&\quad +\frac{ \|n\|_{L^2_{xy}}\|\pp_x\psi\|_{L^2_{xy}}\|P_{\le \langle s\rangle^{0.01}}\Delta\psi\|_{L^2_{xy}}}{1-\|n\|_{L^\infty_{xy}}}\\
\lesssim &\|\nabla u\|_{L^2_{xy}}\|\vec u\|_{L^2_{xy}}
+  \frac{\|\langle\nabla\rangle n\|_{L^2_{xy}}\big(\|\nabla u\|_{L^2_{xy}}+\langle s\rangle^{0.01}\|\nabla\cdot\vec u\|_{L^2_{xy}}+ \langle s\rangle^{0.01}\|\pp_x\psi\|_{L^2_{xy}}\|\nabla\psi\|_{L^2_{xy}}\big)}{1-\|n\|_{L^\infty_{xy}}}\\
\lesssim & \langle t\rangle^{-0.25-0.75}\frac{\|U\|_X^2+\|U\|_X^3}{1-\|U\|_X}\\
\lesssim & \langle t\rangle^{-1}Q(\|U\|_X).
\end{align*}
\end{proof}

For \eqref{1114}, by Proposition \ref{lem:Kn3-L} (i) and Lemma \ref{lem:N11-l},
\begin{align}
\Big\|\langle\nabla\rangle^3\eqref{1114} \Big\|_{L^2_{xt}}
\lesssim &
\int_0^t\Big\|\partial_x(\partial_{tt}-\Delta\partial_t-\Delta)K(t-s)
\langle\nabla\rangle^3N_{11}^l(s)\Big\|_{L^2_{xt}}\,ds\notag\\
\lesssim &
\int_0^t\langle t-s\rangle^{-\frac12}\langle s\rangle^{0.04}\|N_{11}^l(s)\|_{L^1_{xy}}\,ds\notag\\
\lesssim &
\int_0^t\langle t-s\rangle^{-\frac12}
\langle s\rangle^{0.04-1}Q(\|U\|_X)\notag\\
\lesssim &
\langle t \rangle^{-\frac14}Q(\|U\|_X).\label{est:n-N1-1}
\end{align}
Replacing $\partial_x(\partial_{tt}-\Delta\partial_t-\Delta)K(t-s)$ by $\partial_{xx}(\partial_{tt}-\Delta\partial_t-\Delta)K(t-s)$, and using Proposition \ref{lem:Kn8-L} (i) instead,
we get
\begin{align}
\Big\|\langle \nabla\rangle\pp_x\eqref{1114} \Big\|_{L^2_{xt}}
\lesssim &
\int_0^t\langle t-s\rangle^{-1}
\langle s\rangle^{0.04-1}Q(\|U\|_X)
\lesssim
\langle t \rangle^{-\frac34}Q(\|U\|_X).\label{est:n-N1-3}
\end{align}

Now we consider the part \eqref{1115}. Since
$$
N_{12}=-\nabla\cdot(\pp_x\psi\nabla\psi+n\nabla u)+\frac12\pp_x\big(|\nabla \psi|^2\big)-\frac12\pp_x(n^2),
$$
we have,
\begin{align}
\eqref{1115}
=&
-\int_0^t\partial_x(\partial_{tt}-\Delta\partial_t-\Delta)K(t-s)
\nabla\cdot(\pp_x\psi\nabla\psi+n\nabla u)(s)\,ds\label{1115-1}\\
&+\frac12\int_0^t\partial_x(\partial_{tt}-\Delta\partial_t-\Delta)K(t-s)
\pp_x\big(|\nabla \psi|^2-n^2)(s)\,ds\label{1115-2},
\end{align}
where $\psi$ and $n$ lie in the low frequency.
For \eqref{1115-1}, by Proposition \ref{lem:Kn3-L} (ii), we have
\begin{align}
\Big\|\langle\nabla\rangle^3\eqref{1115-1} \Big\|_{L^2_{xt}}
\lesssim &
\int_0^t\big\|\nabla\partial_x(\partial_{tt}-\Delta\partial_t-\Delta)K(t-s)
\cdot P_{\lesssim \langle s\rangle^{0.01}}\langle\nabla\rangle^3(\pp_x\psi\nabla\psi+n\nabla u)(s)\big\|_{L^2_{xy}}\,ds\notag\\
\lesssim &
\int_0^t \langle t-s\rangle^{-\frac34}
\langle s\rangle^{0.05}\big(\|\pp_x\psi\nabla\psi\|_{L^1_{xy}}+\|n\,\nabla u_{\le \langle s\rangle^{0.01}}\|_{L^1_{xy}}\big)\,ds\notag\\
\lesssim &
\int_0^t\langle t-s\rangle^{-\frac34}
\langle s\rangle^{0.06}\big(\|\pp_x\psi\|_{L^2_{xy}}\|\nabla\psi \|_{L^2_{xy}}+\|n\|_{L^2_{xy}}\| u\|_{L^2_{xy}}\big)\,ds\notag\\
\lesssim &
\int_0^t\langle t-s\rangle^{-\frac34}
\langle s\rangle^{0.06-0.5-0.25}\,ds\|U\|_X^2\notag\\
\lesssim &
\langle t \rangle^{-\frac14}\|U\|_X^2.\label{1552-1}
\end{align}
Moreover, by Proposition \ref{lem:Kn8-L} (ii) ($\beta=1$),
\begin{align}
\Big\|\langle\nabla\rangle\pp_x\eqref{1115-1} \Big\|_{L^2_{xt}}
\lesssim &
\int_0^t\big\|\nabla\partial_{xx}(\partial_{tt}-\Delta\partial_t-\Delta)K(t-s)
\cdot P_{\lesssim \langle s\rangle^{0.01}}\langle\nabla\rangle (\pp_x\psi\nabla\psi+n\nabla u)(s)\big\|_{L^2_{xy}}\,ds\notag\\
\lesssim &
\int_0^t\langle t-s\rangle^{-1}
\langle s\rangle^{0.03} \big(\|\pp_x\psi\nabla\psi\|_{L^2_{xy}}+\|n\,\nabla u_{\le \langle s\rangle^{0.01}}\|_{L^2_{xy}}\big) \,ds\notag\\
\lesssim &
\int_0^t\langle t-s\rangle^{-1}
\langle s\rangle^{0.04}\big(\|\pp_x\psi\|_{L^2_{xy}}\|\nabla\psi \|_{L^\infty_{xy}}+\|n\|_{L^2_{xy}}\| u \|_{L^\infty_{xy}}\big)\,ds\notag\\
\lesssim &
\int_0^t\langle t-s\rangle^{-1}
\langle s\rangle^{0.04-0.5-0.5}\,ds\|U\|_X^2\notag\\
\lesssim &
\langle t \rangle^{-\frac34}\|U\|_X^2.\label{1552-2}
\end{align}

For \eqref{1115-2}, by Proposition \ref{lem:Kn8-L} (i),  we have
\begin{align}
\Big\|\langle\nabla\rangle^3\eqref{1115-2} \Big\|_{L^2_{xt}}
\lesssim &
\int_0^t\big\|\partial_{xx}(\partial_{tt}-\Delta\partial_t-\Delta)K(t-s)\,\,
P_{\lesssim \langle s\rangle^{0.01}}\langle\nabla\rangle^3\big(|\nabla \psi|^2-n^2)(s)\big\|_{L^2_{xy}}\,ds\notag\\
\lesssim &
\int_0^t\langle t-s\rangle^{-1}\langle s\rangle^{0.05}
\big\||\nabla\psi|^2-n^2\big\|_{L^1_{xy}}\,ds\notag\\
\lesssim &
\int_0^t\langle t-s\rangle^{-1}\langle s\rangle^{0.05}
\big(\|\nabla\psi\|_{L^2_{xy}}\>\|\nabla\psi\|_{L^2_{xy}}+\|n\|_{L^2_{xy}}\>\|n\|_{L^2_{xy}}\big)\,ds\notag\\
\lesssim &
\int_0^t\langle t-s\rangle^{-1}\langle s\rangle^{0.05}
\langle s \rangle^{-\frac12}\,ds\|U\|_X^2\notag\\
\lesssim &
\langle t \rangle^{-\frac14}\|U\|_X^2.\label{1552-3}
\end{align}
Moreover,  Proposition \ref{lem:Kn8-L} (ii) ($\beta=1$) again,
\begin{align}
\Big\|\langle\nabla\rangle\pp_x\eqref{1115-2} \Big\|_{L^2_{xt}}
\lesssim &
\int_0^t\big\|\nabla \partial_{xx}(\partial_{tt}-\Delta\partial_t-\Delta)K(t-s)\,\,
P_{\lesssim \langle s\rangle^{0.01}}\langle\nabla\rangle \big(|\nabla \psi|^2-n^2)(s)\big\|_{L^2_{xy}}\,ds\notag\\
\lesssim &
\int_0^t\langle t-s\rangle^{-1}
\langle s\rangle^{0.03}\Big(\big\|\nabla\psi_{\le \langle s\rangle^{0.01}}\cdot\nabla\pp_x\psi_{\le \langle s\rangle^{0.01}}\big\|_{L^2_{xy}}+\big\|\pp_xn n\big\|_{L^2_{xy}}\Big)\,ds\notag\\
\lesssim &
\int_0^t\langle t-s\rangle^{-1}
\langle s\rangle^{0.03}\Big(\|\pp_x\psi\|_{L^2_{xy}}\|\nabla\psi \|_{L^\infty_{xy}}+\|\pp_x n\|_{L^2_{xy}}\|n\|_{L^\infty_{xy}}\Big)\,ds\notag\\
\lesssim &
\int_0^t\langle t-s\rangle^{-1}
\langle s\rangle^{0.03-0.5-0.5}\,ds\|U\|_X^2\notag\\
\lesssim &
\langle t \rangle^{-\frac34}\|U\|_X^2.\label{1552-4}
\end{align}
Collecting the estimates in \eqref{1552-1}--\eqref{1552-4}, we have
\begin{align}
\Big\|\langle\nabla\rangle^3\eqref{1115} \Big\|_{L^p_{xt}}
\lesssim
\langle t \rangle^{-\frac14}Q\big(\|U\|_X\big);\label{est:n-N1-2}\\
\Big\|\langle\nabla\rangle\pp_x\eqref{1115} \Big\|_{L^2_{xt}}
\lesssim
\langle t \rangle^{-\frac34}Q\big(\|U\|_X\big).\label{est:n-N1-21}
\end{align}
Therefore, we finish the estimates in this subsubsection  and thus give \eqref{est:n-N1-L2} and \eqref{est:n-N1-xL2}.

\subsubsection{$\int_0^t\partial_y(\partial_{tt}-\Delta\partial_t)K(t-s)N_2(s)\,ds$.}
This term is at the same level as $\partial_x(\partial_{tt}-\Delta\partial_t-\Delta)K(t-s)N_1(s)\,ds$. Indeed,
the two symbols have the relation
$$
\eta(\partial_{tt}+A^2\partial_t)\sim \frac{\xi}{A}\cdot \xi(\partial_{tt}+A^2\partial_t+A^2).
$$
Therefore, by the same way as Section \ref{sec:n-N1} (details are omitted here), we have
\begin{align}
\Big\|\langle\nabla\rangle^3\int_0^t\partial_y(\partial_{tt}-\Delta\partial_t)K(t-s)N_2(s)\,ds\Big\|_{L^2_{xy}}
\lesssim &
\langle t \rangle^{-\frac1{4}}Q\big(\|U\|_X\big);\label{est:n-N2-L2}\\
\Big\|\langle\nabla\rangle\pp_x\int_0^t\partial_y(\partial_{tt}-\Delta\partial_t)K(t-s)N_2(s)\,ds\Big\|_{L^2_{xy}}
\lesssim &
\langle t \rangle^{-\frac3{4}}Q\big(\|U\|_X\big).\label{est:n-N2-xL2}
\end{align}

\subsubsection{$\int_0^t\Delta\partial_y(\partial_t-\Delta)K(t-s)N_3(s)\,ds$.}\label{sec:n-N3}
This term is at the same level as $\int_0^t (\pp_t-\Delta)(\partial_{tt}-\Delta\partial_t-\Delta)K(t-s)N_0(s)\,ds$.
Indeed,  the operator $\Delta(\partial_t-\Delta)K(t)$ has the same decay estimates as $(\pp_t-\Delta)(\partial_{tt}-\Delta\partial_t-\Delta)K(t)$.
Moreover, we shall show that the nonlinearity $\pp_yN_3$ has the same estimates as  $N_0$. Indeed, as the part $\pp_x(nu)$ in $N_0$, the part $\pp_y(u\pp_x\psi)$
in $N_3$ is subcritical in integration (decay rate larger than $-1$), so it is easy to treat. For the other term
$$\pp_y(v\pp_y\psi)=\pp_yv\>\pp_y\psi+v\pp_{yy}\psi,$$
the first piece is easy since it is subcritical in integration again. So one may find that the trouble is from the piece $v\pp_{yy}\psi$. However,
by the third equation in \eqref{e2.1}, one may find that
$$
\pp_{yy}\psi+\pp_yn=-\pp_{xx}\psi-\pp_tv+\Delta v+\lambda(\pp_{xy}u+\pp_{yy}v)+N_2.
$$
Since the right-hand side in the identity is decaying faster than the each piece of the left-hand side, we roughly obtain that
$$
\pp_{yy}\psi\sim-\pp_yn.
$$
Thus, the estimates on $v\pp_{yy}\psi$ can be reduced to the ones on $v\pp_{y}n$, which the most trouble piece in $N_0$ and we have dealt with it as before.
So by the same way, we establish that
\begin{align}
\Big\|\langle\nabla\rangle^3\int_0^t\Delta\partial_y(\partial_t-\Delta)K(t-s)N_3(s)\,ds\Big\|_{L^2_{xy}}
\lesssim &
\langle t \rangle^{-\frac1{4}}Q\big(\|U\|_X\big);\label{est:n-N3-L2}\\
\Big\|\langle\nabla\rangle\pp_x\int_0^t\Delta\partial_y(\partial_t-\Delta)K(t-s)N_3(s)\,ds\Big\|_{L^2_{xy}}
\lesssim &
\langle t \rangle^{-\frac3{4}}Q\big(\|U\|_X\big).\label{est:n-N3-xL2}
\end{align}

\subsubsection{$\lambda\int_0^t(\partial_{tt}-\Delta\partial_t-\partial_{xx})K(t-s)(\partial_x\Delta u(s)+\partial_y\Delta v(s))\,ds$.}\label{sec:n-lambda}
This indeed is a linear term, to use the continuity argument, we need some small bound. To do this, we set $T_0=\epsilon_0^{-6}$, and
assume that $t\ge 2T_0$, otherwise it is concluded in the local theory.
Then from  Proposition \ref{lem:Kn11-L} (ii) ($\beta_2=2$), we get
\begin{align}
&\Big\|\langle\nabla\rangle^3\lambda\int_0^t(\partial_{tt}-\Delta\partial_t-\partial_{xx})K(t-s)(\partial_x\Delta u(s)+\partial_y\Delta v(s))\,ds\Big\|_{L^2_{xy}}\notag\\
\le &|\lambda|\int_0^t\Big\|\Delta\nabla(\partial_{tt}-\Delta\partial_t-\partial_{xx})K(t-s)\cdot\langle\nabla\rangle^3\vec u(s)\Big\|_{L^2_{xy}}\,ds\notag\\
\lesssim &
\int_0^t\langle t-s\rangle^{-1}
\big\|\langle\nabla\rangle^4\vec u(s)\big\|_{L^2_{xy}}\,ds\notag\\
\lesssim &
\int_0^t\langle t-s\rangle^{-1}\langle s\rangle^{-\frac1{3}}\,ds\|U\|_X\notag\\
\lesssim &
\langle t \rangle^{-\frac1{2}}\|U\|_X\lesssim T_0^{-\frac14}\langle t \rangle^{-\frac14}\|U\|_X\notag\\
= & \epsilon_0\langle t \rangle^{-\frac14}\|U\|_X,\label{est:n-lambda-L2}
\end{align}
where we have used  (by choosing $M$ large enough)
$$
\big\|\langle\nabla\rangle^3\vec u(s)\big\|_{L^2_{xy}}\lesssim \big\|\vec u(s)\big\|_{L^2_{xy}}^{1-\frac4M}
\big\|\langle\nabla\rangle^M\vec u(s)\big\|_{L^2_{xy}}^{\frac4M}
\lesssim \langle s\rangle^{-\frac1{3}}\|U\|_X.
$$
Similarly,
\begin{align}
&\Big\|\langle\nabla\rangle\pp_x\int_0^t(\partial_{tt}-\Delta\partial_t-\partial_{xx})K(t-s)(\partial_x\Delta u(s)+\partial_y\Delta v(s))\,ds\Big\|_{L^2_{xt}}\notag\\
\lesssim &
\int_0^t\langle t-s\rangle^{-1}
\big\|\langle\nabla\rangle^2\pp_x\vec u(s)\big\|_{L^2_{xy}}\,ds
\lesssim
\epsilon_0\langle t \rangle^{-\frac34}\|U\|_X.\label{est:n-lambda-xL2}
\end{align}

Now collecting the estimates obtained in Section \ref{sec:n-N0-1}--Section \ref{sec:n-lambda}, we obtain \eqref{est:n-nonlinear}.
Combination with \eqref{est:n-linear}, gives us \eqref{est:n-L2}.

\vskip .4in
\section{The estimates on $u$}\label{sec:n}

In this section, we shall prove that
\begin{align}
\|u(t)\|_{L^2_{xy}}&\lesssim \langle t\rangle^{-\frac12}\big(\|U_0\|_{X_0}+|\lambda|\>\|U\|_X+Q(\|U\|_{X})\big);\label{est:u-L2}\\
\|\la\na\ra u(t)\|_{L^\infty_{xy}}&\lesssim \langle t\rangle^{-1}\big(\|U_0\|_{X_0}+|\lambda|\>\|U\|_X+Q(\|U\|_{X})\big);\label{est:u-Linfty}\\
\||\nabla|^\gamma u(t)\|_{L^2_{xy}}&\lesssim \langle t\rangle^{-\frac34}\big(\|U_0\|_{X_0}+|\lambda|\>\|U\|_X+Q(\|U\|_{X})\big);\label{est:u-gaL2}\\
\|\langle \nabla\rangle\pp_xu(t)\|_{L^2_{xy}}&\lesssim \langle t\rangle^{-1}\big(\|U_0\|_{X_0}+|\lambda|\>\|U\|_X+Q(\|U\|_{X})\big).\label{est:u-xL2}
\end{align}

\subsection{The reexpression of $u$}
Similar as  the expression $n$ in Section \ref{sec:exp-n}, we give the reexpression of $u$ and prove  that
\begin{prop}
The unknown function $u$ obeys the formula,
\begin{align}\label{exp-u}
u(t,x,y)=
(L_u+B_u)(t;n_0,\vec u_0, \vec b_0)+\mathcal{N}_u(t;n,\vec u, \psi),
\end{align}
where $(L_u+B_u)$ is given by
\begin{align}
&L_u(t;n_0,\vec u_0, \vec b_0)+B_u(t;n_0,\vec u_0, \vec b_0)\notag\\
=&
-K(t)\big[\partial_{xy}\Delta\psi_0\big]-\pp_tK(t)[\pp_{yy}u_0-\partial_{xy}v_0]\notag\\
&-(\partial_{tt}-\Delta\partial_t-\Delta)K(t)
\big[\partial_x n_0\big]+\frac12\Delta (\partial_{tt}-\Delta\partial_t)K(t)\big[u_0\big]\notag\\
&\quad-\frac12\Delta\big(\Delta+\sqrt{\Delta\pp_{yy}}\big)K(t)\big[u_0\big]+\pp_t (\partial_{tt}-\Delta\partial_t-\Delta)K(t)\big[u_0\big],\label{eq:LBu-r}
\end{align}
and $\mathcal{N}_u(t;n,\vec u, \psi)$ is given by
\begin{align}
\mathcal{N}_u(t;n,\vec u, \psi)=&
-\int_0^t\partial_x(\pp_{tt}-\Delta\pp_t-\Delta)K(t-s)N_0(s)\,ds\notag\\
&\quad
+\int_0^t\partial_t(\pp_{tt}-\Delta\pp_t-\Delta-\partial_{yy})K(t-s)N_1(s)\,ds\notag\\
&\quad+\int_0^t\partial_{xy}\partial_t K(t-s)N_2(s)\,ds
-\int_0^t\partial_{xy}\Delta K(t-s)N_3(s)\,ds\notag\\
&\quad
+\lambda\int_0^t(\pp_{tt}-\Delta\pp_t-\Delta)\partial_t K(t-s)\big(\partial_{xx}u(s)+\partial_{xy}v(s)\big)\,ds.
\label{u-N}
\end{align}
\end{prop}

The proof of this proposition will occupy the rest of this subsection.
First, according to \eqref{eq:Formula}, using \eqref{F1-r}  and integration by parts (see details in Section \ref{sec:Bu}), we have
\begin{align*}
u(t,x,y)=&L_u(t;n_0,\vec u_0, \vec b_0)+\int_0^t K(t-s)F_1(s)\,ds\\
=& L_u(t;n_0,\vec u_0, \vec b_0)+\int_0^t K(t-s)\Big[ -\partial_x(\pp_{ss}-\Delta\pp_s-\Delta)N_0+\partial_s(\pp_{ss}-\Delta\pp_s-\Delta-\partial_{yy})N_1\notag\\
&\quad + \partial_{xy}\partial_sN_2 -\partial_{xy}\Delta N_3+\lambda(\pp_{ss}-\Delta\pp_s-\Delta)\partial_s\partial_x(\partial_xu+\partial_yv)\Big]\,ds\\
=&
L_u(t;n_0,\vec u_0, \vec b_0)+B_u(t;n_0,\vec u_0, \vec b_0)
+\mathcal{N}_u(t;n,\vec u, \psi).
\end{align*}
Here
\begin{align}
&L_u(t;n_0,\vec u_0, \vec b_0)\notag\\
=&
K(t)\big[(\pp_{tt}-\Delta\pp_t-\Delta)\pp_{t}u(0)\big]\label{Lu-1}\\
&+(\partial_{tt}-\Delta\partial_t-\Delta)K(t)\big[\partial_tu(0)\big]\label{Lu-2}\\
&-\Delta K(t)\big[(\pp_{tt}-\Delta\pp_t-\Delta)u(0)\big]
+\partial_t K(t)\big[(\pp_{tt}-\Delta\pp_t-\Delta)u(0)\big]\label{Lu-3}\\
&-\frac12\Delta\sqrt{\Delta\partial_{yy}} K(t)u_0+K_1(t)u_0,\label{Lu-5}
\end{align}
and $B_u(t;n_0,\vec u_0, \vec b_0)$ is the boundary term given below. Now we give the
explicit expressions of $L_u$ and $B_u$ respectively.

\subsubsection{$L_u(t;n_0,\vec u_0, \vec b_0)$}\label{sec:Lu}
By the equations \eqref{e2.1} and \eqref{e2.2} at $t=0$, we have
\begin{align*}
(\pp_{tt}-\Delta&\pp_t-\Delta)\pp_{t}u(0)
=\big[(\pp_{tt}-\Delta\pp_t-\Delta)\big(\Delta u(0)+\lambda(\pp_{xx}u(0)+\pp_{xy}v(0))-\pp_xn(0)+N_1(0)\big)\big]\\
=&(\pp_{tt}-\Delta\pp_t-\Delta)\Delta u(0)+(\pp_{tt}-\Delta\pp_t-\Delta)N_1(0)\\
&-\partial_x\big[\pp_y\Delta\psi(0)+\pp_tN_0(0)-\Delta N_0(0)-\pp_xN_1(0)-\pp_yN_2(0)-\lambda(\pp_x\Delta u(0)+\pp_y\Delta v(0))\big]\\
&+\lambda(\pp_{tt}-\Delta\pp_t-\Delta)\big[\pp_{xx}u(0)+\pp_{xy}v(0)\big]\\
=&(\pp_{tt}-\Delta\pp_t-\Delta)\Delta u(0)+(\pp_{tt}-\Delta\pp_t-\Delta)N_1(0)-\partial_{xy}\Delta\psi(0)\\
&+\big[-\pp_t\partial_xN_0+\Delta\partial_x N_0(0)+\pp_{xx}N_1(0)+\pp_{xy}N_2(0)\big]\\
&+\lambda(\pp_{tt}-\Delta\pp_t)\big[\pp_{xx}u(0)+\pp_{xy}v(0)\big].
\end{align*}
Thus,
\begin{align*}
\eqref{Lu-1}=&
K(t)\big[(\pp_{tt}-\Delta\pp_t-\Delta)\Delta u(0)\big]+K(t)\big[(\pp_{tt}-\Delta\pp_t-\Delta)N_1(0)\big]-K(t)\big[\partial_{xy}\Delta\psi(0)\big]\\
&+K(t)\big[-\pp_t\partial_xN_0(0)+\pp_x\Delta N_0(0)+\pp_{xx}N_1(0)+\pp_{xy}N_2(0)\big]\\
&+\lambda K(t)\big[(\pp_{tt}-\Delta\pp_t)(\pp_{xx}u(0)+\pp_{xy}v(0))\big].
\end{align*}
Similarly, by the equations \eqref{e2.1} and \eqref{e2.4} at $t=0$,
\begin{align*}
\eqref{Lu-2}=&
(\partial_{tt}-\Delta\partial_t-\Delta)K(t)
\big[\Delta u(0)+\lambda(\partial_{xx} u(0)+\partial_{xy}v(0))-\partial_x n(0)+N_1(0)\big]\\
=&(\partial_{tt}-\Delta\partial_t-\Delta)K(t)
\big[\Delta u(0)-\partial_x n(0)\big]+(\partial_{tt}-\Delta\partial_t-\Delta)K(t)\big[N_1(0)\big]\\
&+\lambda(\partial_{tt}-\Delta\partial_t-\Delta)K(t)\big[\partial_{xx} u(0)+\partial_{xy}v(0)\big];\\
\eqref{Lu-3}=&
-\Delta K(t)\big[(\pp_{tt}-\Delta\pp_t-\Delta)u(0)\big]
+\partial_t K(t)\big[(\pp_{tt}-\Delta\pp_t-\Delta)u(0)\big]\\
=&-\Delta K(t)\big[(\pp_{tt}-\Delta\pp_t-\Delta)u(0)\big]-\pp_{yy}\pp_tK(t)[u(0)]\\
&+\pp_tK(t)\big[\partial_{xy}v(0)+\pp_t N_1(0)-\pp_xN_0(0)+\lambda\pp_t(\pp_{xx} u(0)+\pp_{xy} v(0))\big].
\end{align*}
Then collecting the estimates above, we have
\begin{align*}
&L_u(t;n_0,\vec u_0, \vec b_0)\notag\\
=&
K(t)\big[(\pp_{tt}-\Delta\pp_t-\Delta)N_1(0)\big]
+K(t)\big[-\pp_t\partial_xN_0(0)+\pp_x\Delta N_0(0)+\pp_{xx}N_1(0)+\pp_{xy}N_2(0)\big]\\
&
+(\partial_{tt}-\Delta\partial_t-\Delta)K(t)\big[N_1(0)\big]+\pp_tK(t)\big[\pp_t N_1(0)-\pp_xN_0(0)\big]\\
&
-K(t)\big[\partial_{xy}\Delta\psi(0)\big]-\pp_tK(t)[\pp_{yy}u(0)-\partial_{xy}v(0)]\\
&+(\partial_{tt}-\Delta\partial_t-\Delta)K(t)
\big[\Delta u(0)-\partial_x n(0)\big]
+\lambda K(t)\big[(\pp_{tt}-\Delta\pp_t)(\pp_{xx}u(0)+\pp_{xy}v(0))\big]\\
&
+\lambda(\partial_{tt}-\Delta\partial_t-\Delta)K(t)\big[\partial_{xx} u(0)+\partial_{xy}v(0)\big]+\lambda\pp_tK(t)\big[\pp_t(\pp_{xx} u(0)+\pp_{xy} v(0))\big]\\
&-\frac12\Delta\sqrt{\Delta\partial_{yy}} K(t)[u_0]+K_1(t)[u_0].
\end{align*}

\subsubsection{$B_u(t;n_0,\vec u_0, \vec b_0)$}\label{sec:Bu}
 Now we consider the boundary term $B_u(t;n_0,\vec u_0, \vec b_0)$. By integration by parts, and arguing similarly as in Section \ref{sec:Bn} we have
\begin{align*}
 -\int_0^t K(t-s)&\big[\partial_x(\pp_{ss}-\Delta\pp_s-\Delta)N_0(s)\big]\,ds\\
 = & K(t)\big[(\pp_t-\Delta)\pp_x N_0(0)\big]+\pp_t K(t)\big[\pp_x N_0(0)\big]\\
 &\quad-\int_0^t \partial_x(\pp_{ss}-\Delta\pp_s-\Delta)K(t-s)N_0(s)\,ds;
\\
\int_0^t K(t-s)&\big[\partial_s(\pp_{ss}-\Delta\pp_s-\Delta-\partial_{yy})N_1(s)\big]\,ds\\
= &
-K(t)\big[(\pp_{tt}-\Delta\pp_t-\Delta-\partial_{yy})N_1(0)\big]-\pp_t K(t)\big[(\pp_t-\Delta)N_1(0) \big]\\
&\quad-\pp_{tt}K(t)\big[N_1(0)\big]+\int_0^t \partial_t(\pp_{tt}-\Delta\pp_t-\Delta-\partial_{yy})K(t-s)N_1(s)\,ds;\\
\int_0^t K(t-s)&\big[ \partial_{xy}\partial_sN_2(s)\big]\,ds\\
=&-K(t)\big[\partial_{xy}N_2(0)\big]+\int_0^t \partial_{xy}\partial_tK(t-s)N_2(s)\,ds;\\
\lambda\int_0^t K(t-s)&\big[(\pp_{ss}-\Delta\pp_s-\Delta)\partial_s(\partial_{xx}u(s)+\partial_{xy}v(s))\big]\,ds\\
=&-\lambda K(t)\big[(\pp_{tt}-\Delta\pp_t-\Delta)(\partial_{xx}u(0)+\partial_{xy}v(0))\big]\\
&-\lambda \pp_tK(t)\big[(\pp_{t}-\Delta)(\partial_{xx}u(0)+\partial_{xy}v(0))\big]-\lambda \pp_{tt}K(t)\big[\partial_{xx}u(0)+\partial_{xy}v(0)\big]\\
&\quad+\lambda\int_0^t(\pp_{tt}-\Delta\pp_t-\Delta)\partial_t K(t-s)\big(\partial_{xx}u(s)+\partial_{xy}v(s)\big)\,ds.
\end{align*}
Therefore, we obtain the boundary term $B_u(t;n_0,\vec u_0, \vec b_0)$ as
\begin{align*}
&B_u(t;n_0,\vec u_0, \vec b_0)\\
=&-K(t)\big[(\pp_{tt}-\Delta\pp_t-\Delta-\partial_{yy})N_1(0)\big]
+K(t)\big[\pp_t\pp_x N_0(0)-\pp_x\Delta N_0(0)-\partial_{xy}N_2(0)\big]\\
&+\pp_t K(t)\big[\pp_x N_0(0)-\pp_tN_1(0)+\Delta N_1(0)\big]-\pp_{tt}K(t)\big[N_1(0)\big]\\
&-\lambda K(t)\big[(\pp_{tt}-\Delta\pp_t-\Delta)(\partial_{xx}u(0)+\partial_{xy}v(0))\big]\\
&-\lambda \pp_tK(t)\big[(\pp_{t}-\Delta)(\partial_{xx}u(0)+\partial_{xy}v(0))\big]-\lambda \pp_{tt}K(t)\big[\partial_{xx}u(0)+\partial_{xy}v(0)\big].
\end{align*}

Together with the result obtained in Section \ref{sec:Lu}, we have
\begin{align}
&L_u(t;n_0,\vec u_0, \vec b_0)+B_u(t;n_0,\vec u_0, \vec b_0)\notag\\
=&
-K(t)\big[\partial_{xy}\Delta\psi(0)\big]-\pp_tK(t)[\pp_{yy}u(0)-\partial_{xy}v(0)]\notag\\
&+(\partial_{tt}-\Delta\partial_t-\Delta)K(t)
\big[\Delta u(0)-\partial_x n(0)\big]
-\frac12\Delta\sqrt{\Delta\partial_{yy}} K(t)[u_0]+K_1(t)[u_0]. \label{eq:LBu}
\end{align}
The terms can be further simplified.
Indeed, by \eqref{K-ttt2t2-K1},
\begin{align}
&(\partial_{tt}-\Delta\partial_t-\Delta)K(t)\big[\Delta u(0)\big]
-\frac12\Delta\sqrt{\Delta\partial_{yy}} K(t)u_0+K_1(t)u_0\notag\\
=&\Big[\frac12\Delta(\partial_{tt}-\Delta\partial_t-\Delta)K(t)-\frac12\Delta\sqrt{\Delta\partial_{yy}} K(t)\Big]\big[ u(0)\big]\notag\\
&\qquad
+\Big[\frac12\Delta (\partial_{tt}-\Delta\partial_t-\Delta)K(t)+K_1(t)\Big]\big[u(0)\big]\notag\\
=&\frac12\Delta (\partial_{tt}-\Delta\partial_t)K(t)\big[u(0)\big]-\frac12\Delta\big(\Delta+\sqrt{\Delta\pp_{yy}}\big)K(t)\big[u(0)\big]\notag\\
&\quad+\pp_t (\partial_{tt}-\Delta\partial_t-\Delta)K(t)\big[u(0)\big].\label{eq:cx1}
\end{align}
Using \eqref{eq:cx1} and \eqref{eq:LBu}, we have \eqref{eq:LBu-r}.
Now we split into the following two subsection to consider the linear parts and nonlinear parts separately.

\subsection{Estimates on the linear parts $L_u+B_u$}\label{sec:LBu}In this subsection, we prove that
\begin{lemma}
\begin{align*}
\big\|L_u(t;n_0,\vec u_0, \vec b_0)+B_u(t;n_0,\vec u_0, \vec b_0)\big\|_{L^2_{xy}}
\lesssim &
\langle t\rangle^{-\frac12}\|U_0\|_{X_0};\\
\big\|\langle \nabla\rangle L_u(t;n_0,\vec u_0, \vec b_0)+B_u(t;n_0,\vec u_0, \vec b_0)\big\|_{L^\infty_{xy}}
\lesssim &
\langle t\rangle^{-1}\|U_0\|_{X_0}; \\
\big\||\nabla|^{\gamma}L_u(t;n_0,\vec u_0, \vec b_0)+B_u(t;n_0,\vec u_0, \vec b_0)\big\|_{L^2_{xy}}
\lesssim &
\langle t\rangle^{-\frac34}\|U_0\|_{X_0}; \\
\big\|\langle\nabla\rangle\pp_xL_u(t;n_0,\vec u_0, \vec b_0)+B_u(t;n_0,\vec u_0, \vec b_0)\big\|_{L^2_{xy}}
\lesssim &
\langle t\rangle^{-1}\|U_0\|_{X_0}.
\end{align*}
\end{lemma}

To prove this lemma, we will show that each term in \eqref{eq:LBu-r} obeys the estimates claimed in the lemma. Since it is quite direct by the same agrument in Section \ref{sec:LB-n}, we only give the sketch of proof.
For example,  the first term $K(t)\big[\partial_{xy}\Delta\psi_0\big]$ can be proved by using  Proposition \ref{lem:Ku1-L} and Proposition \ref{lem:Ku1''-L} as follows,
\begin{align}
\big\|K(t)\big[\partial_{xy}\Delta\psi_0\big]\big\|_{L^2_{xy}}
=&\big\|\nabla \partial_{xy}K(t)\cdot\big[\nabla \psi_0\big]\big\|_{L^2_{xy}}
\lesssim
\langle t\rangle^{-\frac12}\big\|\langle \nabla\rangle^{0+}\nabla \psi_0\big\|_{L^1_{xy}};\label{LBu-1-1}\\
\big\|\langle \nabla\rangle K(t)\big[\partial_{xy}\Delta\psi_0\big]\big\|_{L^\infty_{xy}}
=&
\big\|\nabla\partial_{xy} K(t)\cdot\big[ \langle \nabla\rangle\nabla\psi_0\big]\big\|_{L^\infty_{xy}}
\lesssim
\langle t\rangle^{-1}\big\|\langle \nabla\rangle^{2+}\nabla \psi_0\big\|_{L^1_{xy}};\label{LBu-1-2}\\
\big\||\nabla|^{\gamma}K(t)\big[\partial_{xy}\Delta\psi_0\big]\big\|_{L^2_{xy}}
=&
\big\||\nabla|^{\gamma}\nabla\partial_{xy}K(t)\cdot \big[\nabla\psi_0\big]\big\|_{L^2_{xy}}
\lesssim
\langle t\rangle^{-\frac34}\big\|\langle \nabla\rangle^{\gamma+}\nabla \psi_0\big\|_{L^1_{xy}};\label{LBu-1-3}\\
\big\|\langle \nabla\rangle\pp_xK(t)\big[\partial_{xy}\Delta\psi_0\big]\big\|_{L^2_{xy}}
=&
\big\|\nabla\partial_{xx}\pp_y K(t)\cdot \big[\langle \nabla\rangle\nabla \psi_0\big]\big\|_{L^2_{xy}}
\lesssim
\langle t\rangle^{-1}\big\|\langle \nabla\rangle^{2+}\nabla \psi_0\big\|_{L^1_{xy}}.\label{LBu-1-4}
\end{align}
Similarly, the estimates on the  term  $\pp_tK(t)[\pp_{yy}u_0-\partial_{xy}v_0]$ follow from Proposition \ref{lem:Kn2-L}.
The estimates on the term $(\partial_{tt}-\Delta\partial_t-\Delta)K(t)
\big[\partial_x n_0\big]$ can be obtained by using Propositions \ref{lem:Kn3-L} and \ref{lem:Kn8-L}.
The estimates on the term $\Delta (\partial_{tt}-\Delta\partial_t)K(t)\big[u_0\big]$ follow from Proposition \ref{lem:Ku2-L}.
The estimates on the term $\Delta\big(\Delta+\sqrt{\Delta\pp_{yy}}\big)K(t)\big[u_0\big]$ can be proved by Proposition \ref{lem:Ku3-L}.
At last, The estimates on the term $\pp_t (\partial_{tt}-\Delta\partial_t-\Delta)K(t)\big[u_0\big]$ can be shown by Proposition \ref{lem:Kn5-L}.

\subsection{The estimates on  nonlinear parts $\mathcal N_u$}
In this subsection, we establish that
\begin{align}
\big\|\mathcal N_u(t;n,\vec u, \psi)\big\|_{L^2_{xy}}
\lesssim &
\langle t\rangle^{-\frac12}\big(\|U_0\|_{X_0}+|\lambda|\>\|U\|_X+Q(\|U\|_{X})\big);\label{nu-1}\\
\big\|\langle \nabla\rangle\mathcal N_u(t;n,\vec u, \psi)\big\|_{L^\infty_{xy}}
\lesssim &
\langle t\rangle^{-1}(\|U_0\|_{X_0}+|\lambda|\>\|U\|_X+Q(\|U\|_{X})\big);\label{nu-2}\\
\big\||\nabla|^{\gamma}\mathcal N_u(t;n,\vec u, \psi)\big\|_{L^2_{xy}}
\lesssim &
\langle t\rangle^{-\frac34}(\|U_0\|_{X_0}+|\lambda|\>\|U\|_X+Q(\|U\|_{X})\big);\label{nu-3}\\
\big\|\langle\nabla\rangle\pp_x\mathcal N_u(t;n,\vec u, \psi)\big\|_{L^2_{xy}}
\lesssim &
\langle t\rangle^{-1}(\|U_0\|_{X_0}+|\lambda|\>\|U\|_X+Q(\|U\|_{X})\big).\label{nu-4}
\end{align}

By the definition of $\mathcal N_u(t;n,\vec u, \psi)$, we estimate it terms by terms in \eqref{u-N}.

\subsubsection{$\int_0^t\partial_x(\pp_{tt}-\Delta\pp_t-\Delta)K(t-s)N_0(s)\,ds$}\label{sec:u-N0-1}
In this subsubsection, and we prove that
\begin{align*}
\big\|\int_0^t\partial_x(\pp_{tt}-\Delta\pp_t-\Delta)K(t-s)N_0(s)\,ds\big\|_{L^2_{xy}}
\lesssim &
\langle t\rangle^{-\frac12}Q(\|U\|_X);\\
\big\|\langle\nabla\rangle\int_0^t\partial_x(\pp_{tt}-\Delta\pp_t-\Delta)K(t-s)N_0(s)\,ds\big\|_{L^\infty_{xy}}
\lesssim &
\langle t\rangle^{-1}Q(\|U\|_X);\\
\big\||\nabla|^{\gamma}\int_0^t\partial_x(\pp_{tt}-\Delta\pp_t-\Delta)K(t-s)N_0(s)\,ds\big\|_{L^2_{xy}}
\lesssim &
\langle t\rangle^{-\frac34}Q(\|U\|_X);\\
\big\|\langle\nabla\rangle\pp_x\int_0^t\partial_x(\pp_{tt}-\Delta\pp_t-\Delta)K(t-s)N_0(s)\,ds\big\|_{L^2_{xy}}
\lesssim &
\langle t\rangle^{-1}Q(\|U\|_X).
\end{align*}

The treatment is similar as what in Section \ref{sec:n-N0-2}.  First, we rewrite
\begin{align*}
&\int_0^t\partial_x(\partial_{tt}-\Delta\partial_t-\Delta)K(t-s)N_0(s)\,ds\\
=&
\int_0^t\partial_x(\partial_{tt}-\Delta\partial_t-\Delta)K(t-s)N_0^h(s)\,ds\\
&+\int_0^t\partial_{xx}(\partial_{tt}-\Delta\partial_t-\Delta)K(t-s)(n_{\le \langle s\rangle^{0.01}}\>u_{\le \langle s\rangle^{0.01}})(s)\,ds\notag\\
&+\int_0^t\partial_{xy}(\partial_{tt}-\Delta\partial_t-\Delta)K(t-s)P_{\le \langle s\rangle^{-0.04}}(n_{\le \langle s\rangle^{0.01}}\>v_{\le \langle s\rangle^{0.01}})(s)\,ds\\
&+\int_0^t\partial_{xy}(\partial_{tt}-\Delta\partial_t-\Delta)K(t-s)P_{\ge \langle s\rangle^{-0.04}}(n_{\le \langle s\rangle^{0.01}}\>v_{\le \langle s\rangle^{0.01}})(s)\,ds.
\end{align*}
Furthermore, by \eqref{re-N0}, we also split it into several parts as follows,
\begin{align}
&\int_0^t\partial_x(\partial_{tt}-\Delta\partial_t-\Delta)K(t-s)N_0(s)\,ds\notag\\
=&\int_0^t\partial_x(\partial_{tt}-\Delta\partial_t-\Delta)K(t-s)N_0^h(s)\,ds\label{u-N0-2-h}\\
&+\int_0^t\partial_{xx}(\partial_{tt}-\Delta\partial_t-\Delta)K(t-s)(n_{\le \langle s\rangle^{0.01}}\>u_{\le \langle s\rangle^{0.01}})(s)\,ds\label{u-N0-2-l-1}\\
&+\int_0^t\partial_{xy}(\partial_{tt}-\Delta\partial_t-\Delta)K(t-s)P_{\le \langle s\rangle^{-0.04}}(n_{\le \langle s\rangle^{0.01}}\>v_{\le \langle s\rangle^{0.01}})(s)\,ds\label{uv-lh}\\
&+\int_0^t\partial_{xy}(\partial_{tt}-\Delta\partial_t-\Delta)K(t-s)P_{\ge \langle s\rangle^{-0.04}}(n_{\le \langle s\rangle^{0.01}}\>v_{\le \langle s\rangle^{-10}})(s)\,ds\label{uv-easy1}\\
&+
\int_0^t\partial_{xy}(\partial_{tt}-\Delta\partial_t-\Delta)K(t-s)P_{\ge \langle s\rangle^{-0.04}}(n_{\le \langle s\rangle^{0.01}}\>v_{\ge \langle s\rangle^{-0.05}})(s)\,ds\label{uv-easy2}\\
&+
\int_0^t\partial_{xy}(\partial_{tt}-\Delta\partial_t-\Delta)K(t-s)P_{\ge \langle s\rangle^{-0.04}}\big(n_{\le \langle s\rangle^{0.01}}\>P_{\langle s\rangle^{-10}\le\cdot \le \langle s\rangle^{-0.05}}N_3\big)(s)\,ds.\label{uv-hard-2}\\
&
-\int_0^t\partial_{xy}(\partial_{tt}-\Delta\partial_t-\Delta)K(t-s)P_{\ge \langle s\rangle^{-0.04}}\big(\pp_s n_{\le \langle s\rangle^{0.01}}\>P_{\langle s\rangle^{-10}\le\cdot \le \langle s\rangle^{-0.05}}\psi\big) (s)\,ds\label{uv-hard-12}\\
&+
\int_0^t\partial_{xy}(\partial_{tt}-\Delta\partial_t-\Delta)K(t-s)\pp_s\Big( P_{\ge \langle s\rangle^{-0.04}}\big(n_{\le \langle s\rangle^{0.01}}\>P_{\langle s\rangle^{-10}\le\cdot \le \langle s\rangle^{-0.05}}\psi\big)(s)\Big)\,ds\label{uv-hard-11}\\
&\quad+ \mbox{\emph{other easy terms}}.\label{uv-hard-13}
\end{align}
Again, the \emph{other easy terms} is the corresponding terms from \emph{other parts} in \eqref{re-N0}.

As before, the high frequency piece is standard, and thus the estimates on \eqref{u-N0-2-h} are omitted here. Now
we consider the term \eqref{u-N0-2-l-1},  by Proposition \ref{lem:Kn8-L} (i) and interpolation estimates,
\begin{align*}
\big\|\eqref{u-N0-2-l-1}\big\|_{L^2_{xy}}
\lesssim &
\int_0^t\Big\|\partial_{xx}(\partial_{tt}-\Delta\partial_t-\Delta)K(t-s)(n_{\le \langle s\rangle^{0.01}}\>u_{\le \langle s\rangle^{0.01}})(s)\Big\|_{L^2_{xy}}\,ds\\
\lesssim &
\int_0^t\langle t-s\rangle^{-\frac34}\big\|(n_{\le \langle s\rangle^{0.01}}\>u_{\le \langle s\rangle^{0.01}})(s)\big\|_{L^1_{x}L^2_y}\,ds\\
\lesssim &
\int_0^t\langle t-s\rangle^{-\frac34}\|n(s)\|_{L^2_{xy}}\big\|u(s)\big\|_{L^2_{x}L^\infty_{y}}
\,ds\\
\lesssim &
\int_0^t\langle t-s\rangle^{-\frac34}\langle s\rangle^{0.01}\|n(s)\|_{L^2_{xy}}\big\|u(s)\big\|_{L^2_{xy}}^{0+}\big\||\nabla|^{\frac12+}u(s)\big\|_{L^2_{xy}}^{1-}
\,ds\\
\lesssim &
\int_0^t\langle t-s\rangle^{-\frac34}\langle s\rangle^{0.02-0.25-0.75}
\,ds\|U\|_X^2\\
\lesssim &
\langle t \rangle^{-\frac1{2}}\|U\|_X^2.
\end{align*}
Similarly, by Proposition \ref{lem:Kn8-L} (iii) instead,
\begin{align*}
\big\|\langle\nabla\rangle\eqref{u-N0-2-l-1}\big\|_{L^\infty_{xy}}
\lesssim &
\int_0^t\Big\|\partial_{xx}(\partial_{tt}-\Delta\partial_t-\Delta)K(t-s)(n_{\le \langle s\rangle^{0.01}}\>u_{\le \langle s\rangle^{0.01}})(s)\Big\|_{L^2_{xy}}\,ds\\
\lesssim &
\int_0^t\langle t-s\rangle^{-1}\langle s\rangle^{0.03}\|(n_{\le \langle s\rangle^{0.01}}\>u_{\le \langle s\rangle^{0.01}})(s)\|_{L^2_{xy}}\,ds\\
\lesssim &
\int_0^t\langle t-s\rangle^{-1}\langle s\rangle^{0.03}\|n(s)\|_{L^2_{xy}}\big\|u(s)\big\|_{L^\infty_{xy}}\,ds\\
\lesssim &
\int_0^t\langle t-s\rangle^{-1}\langle s\rangle^{0.03-0.25-1}
\,ds\|U\|_X^2\\
\lesssim &
\langle t \rangle^{-1}\|U\|_X^2.
\end{align*}
and by Proposition \ref{lem:Kn8-L} (ii), for any $\beta\in [\frac12,1]$,
\begin{align*}
\big\|\langle\nabla\rangle|\nabla|^\beta\eqref{u-N0-2-l-1}\big\|_{L^2_{xy}}
\lesssim &
\int_0^t\Big\||\nabla|^\beta\partial_{xx}(\partial_{tt}-\Delta\partial_t-\Delta)K(t-s)\,\,\langle\nabla\rangle(n_{\le \langle s\rangle^{0.01}}\>u_{\le \langle s\rangle^{0.01}})(s)\Big\|_{L^2_{xy}}\,ds\\
\lesssim &
\int_0^t\langle t-s\rangle^{-\frac{1+\beta}2}\langle s\rangle^{0.02}\big\|(n_{\le \langle s\rangle^{0.01}}\>u_{\le \langle s\rangle^{0.01}})(s)\big\|_{L^2_{xy}}\,ds\\
\lesssim &
\int_0^t\langle t-s\rangle^{-\frac{1+\beta}2}\langle s\rangle^{0.02-0.25-1}
\,ds\|U\|_X^2\\
\lesssim &
\langle t \rangle^{-\frac{1+\beta}2}\|U\|_X^2,
\end{align*}
where $\beta=\gamma$ and $\beta=1$ (which turns to $\pp_x\eqref{u-N0-2-l-1}$) are the estimates we want.

Now we consider \eqref{uv-lh}.  By Beinstein's inequality, and  Proposition \ref{lem:Kn3-L} (ii), we have
\begin{align*}
\big\||\nabla|^\beta\eqref{uv-lh}\big\|_{L^2_{xy}}
\lesssim &
\int_0^t \big\||\nabla|^\beta\partial_{xy}(\partial_{tt}-\Delta\partial_t-\Delta)K(t-s)\,\,P_{\le \langle s\rangle^{-0.04}}(n_{\le \langle s\rangle^{0.01}}\>v_{\le \langle s\rangle^{0.01}})(s)\big\|_{L^2_{xy}}\,ds\\
\lesssim &
\int_0^t \langle s\rangle^{-0.04\beta}\big\|\nabla\partial_{x}(\partial_{tt}-\Delta\partial_t-\Delta)K(t-s)\,\,P_{\le \langle s\rangle^{-0.04}}(n_{\le \langle s\rangle^{0.01}}\>v_{\le \langle s\rangle^{0.01}})(s)\big\|_{L^2_{xy}}\,ds\\
\lesssim &
\int_0^t\langle t-s\rangle^{-\frac34}\langle s\rangle^{-0.04\beta}\big\|(n_{\le \langle s\rangle^{0.01}}\>v_{\le \langle s\rangle^{0.01}})(s)\big\|_{L^1_{x}L^2_{y}}\,ds\\
\lesssim &
\int_0^t\langle t-s\rangle^{-\frac34}\langle s\rangle^{-0.04\beta}\|n(s)\|_{L^2_{xy}}\big\|v(s)\big\|_{L^2_xL^\infty_{y}}\,ds\\
\lesssim &
\int_0^t\langle t-s\rangle^{-\frac34}\langle s\rangle^{-0.03\beta-1}
\,ds\|U\|_X^2,
\end{align*}
if $\beta=0$, we obtain
\begin{align*}
\big\|\eqref{uv-lh}\big\|_{L^2_{xy}}
\lesssim \langle t\rangle^{-\frac12}\|U\|_X^2,
\end{align*}
if $\beta=\gamma$, we obtain
\begin{align*}
\big\||\nabla|^\gamma\eqref{uv-lh}\big\|_{L^2_{xy}}
\lesssim \langle t\rangle^{-\frac34}\|U\|_X^2.
\end{align*}

By using Proposition \ref{lem:Kn3-L} (iii) instead, we have
\begin{align*}
\big\|\langle\nabla\rangle\eqref{uv-lh}\big\|_{L^\infty_{xy}}
\lesssim &
\int_0^t \big\|\nabla\partial_{x}(\partial_{tt}-\Delta\partial_t-\Delta)K(t-s)\,\,P_{\le \langle s\rangle^{-0.04}}(n_{\le \langle s\rangle^{0.01}}\>v_{\le \langle s\rangle^{0.01}})(s)\big\|_{L^\infty_{xy}}\,ds\\
\lesssim &
\int_0^t \langle s\rangle^{-0.02}\big\||\nabla|^\frac12\partial_{x}(\partial_{tt}-\Delta\partial_t-\Delta)K(t-s)\,\,P_{\le \langle s\rangle^{-0.04}}(n_{\le \langle s\rangle^{0.01}}\>v_{\le \langle s\rangle^{0.01}})(s)\big\|_{L^\infty_{xy}}\,ds\\
\lesssim &
\int_0^t\langle t-s\rangle^{-1}\langle s\rangle^{-0.02}\big\|(n_{\le \langle s\rangle^{0.01}}\>v_{\le \langle s\rangle^{0.01}})(s)\big\|_{L^1_{x}L^2_{y}}\,ds\\
\lesssim &
\int_0^t\langle t-s\rangle^{-1}\langle s\rangle^{-0.02}\|n(s)\|_{L^2_{xy}}\big\|v(s)\big\|_{L^2_xL^\infty_{y}}\,ds\\
\lesssim &
\int_0^t\langle t-s\rangle^{-1}\langle s\rangle^{-0.01-1}
\,ds\|U\|_X^2\\
\lesssim &\langle t\rangle^{-1}\|U\|_X^2.
\end{align*}

By using Proposition \ref{lem:Kn8-L} (ii) ($\beta=1$) instead,
\begin{align*}
\big\|\langle\nabla\rangle\pp_x\eqref{uv-lh}\big\|_{L^2_{xy}}
\lesssim &
\int_0^t \big\|\nabla\partial_{xx}(\partial_{tt}-\Delta\partial_t-\Delta)K(t-s)\,\,P_{\le \langle s\rangle^{-0.04}}(n_{\le \langle s\rangle^{0.01}}\>v_{\le \langle s\rangle^{0.01}})(s)\big\|_{L^2_{xy}}\,ds\\
\lesssim &
\int_0^t\langle t-s\rangle^{-1} \|n_{\le \langle s\rangle^{0.01}}\>v_{\le \langle s\rangle^{0.01}}\|_{L^2_{xy}}\,ds\\
\lesssim &
\int_0^t\langle t-s\rangle^{-1} \|n(s)\|_{L^2_{xy}}\|v(s)\|_{L^\infty_{xy}}\,ds\\
\lesssim &
\int_0^t\langle t-s\rangle^{-1}\langle s\rangle^{-0.25-1}
\,ds\|U\|_X^2\\
\lesssim &\langle t\rangle^{-1}\|U\|_X^2.
\end{align*}

We consider the terms  \eqref{uv-easy1}--\eqref{uv-hard-12} together.
First, we show that the nonlinearities in these terms satisfy the following same estimates. Let
\begin{align*}
\Pi(s)\triangleq &\big|(n_{\le \langle s\rangle^{0.01}}\>v_{\le \langle s\rangle^{-10}})(s)\big|+\big|(n_{\le \langle s\rangle^{0.01}}\>v_{\ge \langle s\rangle^{-0.05}})(s)
\big|\\
&+\big|\big(n_{\le \langle s\rangle^{0.01}}\>P_{\langle s\rangle^{-10}\le\cdot \le \langle s\rangle^{-0.05}}N_3\big)(s)\big|
+\big|\big(\pp_s n_{\le \langle s\rangle^{0.01}}\>P_{\langle s\rangle^{-10}\le\cdot \le \langle s\rangle^{-0.05}}\psi\big) (s)\big|.
\end{align*}
Then,
\begin{align}
&
\|\Pi(s)\big\|_{L^1_xL^2_y}
\lesssim \langle s\rangle^{-1.01}Q(\|U\|_X).\label{2129}
\end{align}
Indeed, by  Sobolev' and  Beinstein's inequalities, we have
\begin{align*}
\|(n_{\le \langle s\rangle^{0.01}}\>v_{\le \langle s\rangle^{-10}})\|_{L^1_xL^2_y}
\lesssim& \|n\|_{L^2_{xy}}\|v_{\le \langle s\rangle^{-10}}\|_{L^2_xL^\infty_y}\\
\lesssim& \langle s\rangle^{-4}\|n\|_{L^2_{xy}}\|v\|_{L^2_{xy}}
\lesssim \langle s\rangle^{-4}\|U\|_X^2;\\
\|(n_{\le \langle s\rangle^{0.01}}\>v_{\ge \langle s\rangle^{-0.05}})(s)\|_{L^1_xL^2_y}
\lesssim& \|n\|_{L^2_{xy}}\|v_{\ge \langle s\rangle^{-0.05}}\|_{L^2_xL^\infty_y}\\
\lesssim&\langle s\rangle^{0.05} \|n\|_{L^2_{xy}}\|\nabla v\|_{L^2_{xy}}
\lesssim\langle s\rangle^{-1.2}\|U\|_X^2;
\end{align*}
also, recall $N_3 =\vec u \cdot \nabla \psi$,
\begin{align*}
\big\|\big(n_{\le \langle s\rangle^{0.01}}\>P_{\langle s\rangle^{-10}\le\cdot \le \langle s\rangle^{-0.05}}N_3\big)(s)\big\|_{L^1_xL^2_y}
\lesssim& \|n\|_{L^2_{xy}}\|\vec u\|_{L^2_xL^\infty_y}\|\nabla\psi\|_{L^\infty_{xy}}\\
\lesssim&\langle s\rangle^{-0.25-0.75-0.25+} \|U\|_X^3
\lesssim\langle s\rangle^{-1.2}\|U\|_X^3.
\end{align*}
Since $\pp_t n=-\pp_xu-\pp_yv+N_0$, we have
\begin{align*}
\big\|\pp_s n_{\le \langle s\rangle^{0.01}} &\>P_{\langle s\rangle^{-10}\le\cdot \le \langle s\rangle^{-0.05}}\psi\big\|_{L^1_xL^2_y}
\lesssim
\big\|P_ {\le \langle s\rangle^{0.01}}\pp_s n\big\|_{L^2_{xy}}\big\|P_{\langle s\rangle^{-10}\le\cdot \le \langle s\rangle^{-0.05}}\psi\big\|_{L^2_xL^\infty_y}\\
\lesssim &
\langle s\rangle^{0.01} \big(\|\nabla\cdot\vec u\|_{L^2_{xy}}+\|P_{\le \langle s\rangle^{0.01}}N_0\|_{L^2_{xy}}\big)\big\||\nabla|^{\frac12+}\psi\big\|_{L^2_{xy}}\\
\lesssim &
\langle s\rangle^{0.02} \big(\|\nabla\cdot\vec u\|_{L^2_{xy}}+\|n\|_{L^2_{xy}}\|\vec u\|_{L^\infty_{xy}}\big)\big\||\nabla|^{\frac12+}\psi\big\|_{L^2_{xy}}
\lesssim
\langle s\rangle^{-1.2}(\|U\|_X^2+\|U\|_X^3).
\end{align*}
Thus we have \eqref{2129}. Now from Proposition \ref{lem:Kn3-L} (ii), we have
\begin{align*}
&\big\|\langle\nabla\rangle\big(\eqref{uv-easy1}+\cdots+\eqref{uv-hard-12}\big)\big\|_{L^2_{xy}}\\
\lesssim &
\int_0^t\big\|\nabla\partial_{x}(\partial_{tt}-\Delta\partial_t-\Delta)K(t-s)\,\,P_{\lesssim \langle s\rangle^{0.01}}\langle\nabla\rangle\Pi(s)\|_{L^2_{xy}}\,ds\\
\lesssim &\int_0^t\langle t-s\rangle^{-\frac34}\langle s\rangle^{0.02}\|\Pi(s)\|_{L^1_{x}L^2_y}\,ds\\
\lesssim &
\int_0^t\langle t-s\rangle^{-\frac34}\langle s\rangle^{0.02-1.2}
\,dsQ(\|U\|_X)\\
\lesssim &
\langle t \rangle^{-\frac34}Q(\|U\|_X).
\end{align*}
This estimate concludes  the estimates what we want on $\big\|\eqref{uv-easy1}+\cdots+\eqref{uv-hard-12}\big\|_{L^2_{xy}}$ and
$\big\||\nabla|^\gamma\big(\eqref{uv-easy1}+\cdots+\eqref{uv-hard-12}\big)\big\|_{L^2_{xy}}$.
Similarly, by Proposition \ref{lem:Kn3-L} (iii), we have
\begin{align*}
&\big\|\langle\nabla\rangle\big(\eqref{uv-easy1}+\cdots+\eqref{uv-hard-12}\big)\big\|_{L^\infty_{xy}}\\
\lesssim &
\int_0^t\big\|\partial_{xy}(\partial_{tt}-\Delta\partial_t-\Delta)K(t-s)\,\,P_{\lesssim \langle s\rangle^{0.01}}\langle\nabla\rangle\Pi(s)\|_{L^\infty_{xy}}\,ds\\
\lesssim &\int_0^t\langle t-s\rangle^{-1}\langle s\rangle^{0.03}\|\Pi(s)\|_{L^1_{x}L^2_y}\,ds\\
\lesssim &
\int_0^t\langle t-s\rangle^{-1}\langle s\rangle^{0.03-1.2}
\,dsQ(\|U\|_X)\\
\lesssim &
\langle t \rangle^{-1}Q(\|U\|_X).
\end{align*}
Now by Proposition \ref{lem:Kn8-L} (ii) ($\beta=1$) instead, and by Beinstein's inequality we have
\begin{align*}
&\big\|\langle\nabla\rangle\pp_x\big(\eqref{uv-easy1}+\cdots+\eqref{uv-hard-12}\big)\big\|_{L^2_{xy}}\\
\lesssim &
\int_0^t\big\|\nabla\partial_{xx}(\partial_{tt}-\Delta\partial_t-\Delta)K(t-s)\,\,P_{\lesssim \langle s\rangle^{0.01}}\langle\nabla\rangle\Pi(s)\|_{L^2_{xy}}\,ds\\
\lesssim &\int_0^t\langle t-s\rangle^{-1}\langle s\rangle^{0.03}\|P_{\lesssim \langle s\rangle^{0.01}}\Pi(s)\|_{L^2_{xy}}\,ds\\
\lesssim &\int_0^t\langle t-s\rangle^{-1}\langle s\rangle^{0.04}\|\Pi(s)\|_{L^1_{x}L^2_y}\,ds\\
\lesssim &
\int_0^t\langle t-s\rangle^{-1}\langle s\rangle^{0.04-1.2}
\,dsQ(\|U\|_X)\\
\lesssim &
\langle t \rangle^{-1}Q(\|U\|_X).
\end{align*}

Similar as \eqref{est:otheresayterms}, it is easy to treat the term \eqref{uv-hard-13}, so we omit here.
At last, we consider the term \eqref{uv-hard-11}. Integration by parts, we find
\begin{align}
&\quad\eqref{uv-hard-11}=\pp_{xy}(\partial_{tt}-\Delta\partial_t-\Delta)K(t) P_{\sim 1}\big(P_{\lesssim 1}n_0\>P_{\sim 1}\psi_0\big)\label{uv-hard-11-b}\\
&-\int_0^t\pp_{y}(\partial_{tt}-\Delta\partial_t-\Delta)\pp_tK(t-s) P_{\ge \langle s\rangle^{-0.04}}\pp_x\big(n_{\le \langle s\rangle^{0.01}}\>P_{\langle s\rangle^{-10}\le\cdot \le \langle s\rangle^{-0.05}}\psi\big)(s)\,ds.\label{uv-hard-111}
\end{align}
By Proposition \ref{lem:Kn3-L} (ii) and (iii) and Proposition \ref{lem:Kn8-L} (i), the boundary term \eqref{uv-hard-11-b} can be controlled as following.
For $2\le p\le \infty$,
\begin{align*}
\|\langle\nabla\rangle\eqref{uv-hard-11-b}\|_{L^p_{xy}}
\lesssim &
\langle t \rangle^{-1+\frac1{2p}}
\|\langle\nabla\rangle \big(P_{\lesssim 1}n_0\>P_{\sim 1}\psi_0\big)\|_{L^1_{xy}}
\lesssim
\langle t \rangle^{-1+\frac1{2p}}\|n_0\|_{L^2_{xy}}\|\nabla\psi_0\|_{L^2_{xy}};\\
\|\langle\nabla\rangle\pp_x\eqref{uv-hard-11-b}\|_{L^2_{xy}}
\lesssim &
\langle t \rangle^{-1}
\|\langle\nabla\rangle P_{\lesssim 1}n_0\>P_{\sim 1}\psi_0\|_{L^1_{xy}}
\lesssim
\langle t \rangle^{-1}\|n_0\|_{L^2_{xy}}\|\nabla\psi_0\|_{L^2_{xy}}.
\end{align*}
The first estimate gives the desirable estimates on $\|\eqref{uv-hard-11-b}\|_{L^2_{xy}}, \||\nabla|^\gamma\eqref{uv-hard-11-b}\|_{L^2_{xy}}$ and $\|\eqref{uv-hard-11-b}\|_{L^\infty_{xy}}$.

Now we consider \eqref{uv-hard-111}.
By Proposition \ref{lem:Kn5-L} (i) ($\beta=2$) and \eqref{psi-infty}, we have
\begin{align*}
\big\|\eqref{uv-hard-111}\big\|_{L^2_{xy}}
\lesssim &
\int_0^t\big\|\Delta\pp_t\big(\partial_{tt}-\Delta\partial_t-\Delta)K(t-s) \big(n_{\le \langle s\rangle^{0.01}}\>P_{\langle s\rangle^{-10}\le\cdot \le \langle s\rangle^{-0.05}}\psi\big)(s)\|_{L^2_{xy}}\,ds\\
\lesssim &
\int_0^t\langle t-s\rangle^{-1}\langle s\rangle^{0.01}\|n_{\le \langle s\rangle^{0.01}}\>P_{\langle s\rangle^{-10}\le\cdot \le \langle s\rangle^{-0.05}}\psi\|_{L^2_{xy}}\,ds\\
\lesssim &
\int_0^t\langle t-s\rangle^{-1}\langle s\rangle^{0.01}\|n\|_{L^2_{xy}}\|P_{\langle s\rangle^{-10}\le\cdot \le \langle s\rangle^{-0.05}}\psi\|_{L^\infty_{xy}}
\,ds\\
\lesssim &
\int_0^t\langle t-s\rangle^{-1}\langle s\rangle^{0.01-0.25-0.27}\|U\|_X^2\\
\lesssim &
\langle t \rangle^{-\frac12}\|U\|_X^2.
\end{align*}
Moreover, since $\gamma>\frac12$, by Proposition \ref{lem:Ku6-L} and \eqref{psi-infty}, we have
\begin{align*}
\big\||\nabla|^\gamma\eqref{uv-hard-111}\big\|_{L^2_{xy}}
\lesssim &
\int_0^t\big\||\nabla|^{\gamma+1}\pp_t\big(\partial_{tt}-\Delta\partial_t-\Delta)K(t-s)  \pp_x\big(n_{\le \langle s\rangle^{0.01}}\>P_{\langle s\rangle^{-10}\le\cdot \le \langle s\rangle^{-0.05}}\psi\big)(s)\|_{L^2_{xy}}\,ds\\
\lesssim &
\int_0^t\langle t-s\rangle^{-\frac{1+\gamma}2}\|\pp_x\big(n_{\le \langle s\rangle^{0.01}}\>P_{\langle s\rangle^{-10}\le\cdot \le \langle s\rangle^{-0.05}}\psi\big)\|_{L^2_{xy}}\,ds\\
\lesssim &
\int_0^t\langle t-s\rangle^{-\frac{1+\gamma}2}\big(\|\pp_xn\|_{L^2_{xy}}\|P_{\langle s\rangle^{-10}\le\cdot \le \langle s\rangle^{-0.05}}\psi\|_{L^\infty_{xy}}+\|n\|_{L^\infty_{xy}}\|\pp_x\psi\|_{L^2_{xy}}\big)
\,ds\\
\lesssim &
\int_0^t\langle t-s\rangle^{-\frac{1+\gamma}2}\langle s\rangle^{-1}\,ds\|U\|_X^2\\
\lesssim &
\langle t \rangle^{-\frac34}\|U\|_X^2.
\end{align*}
For $\big\|\langle\nabla\rangle\eqref{uv-hard-111}\big\|_{L^\infty_{xy}}$,  we need some special treatment (which will be used frequently below). By Beinstein's inequality and Proposition \ref{lem:Kn5-L} (iv), we have
\begin{align*}
\big\|\langle\nabla\rangle\eqref{uv-hard-111}\big\|_{L^\infty_{xy}}
\lesssim &
\int_0^t\big\|P_{\ge \langle s\rangle^{-0.04}}\pp_y \pp_t\big(\partial_{tt}-\Delta\partial_t-\Delta)K(t-s)\,\,  \pp_x\langle\nabla\rangle\big(n_{\le \langle s\rangle^{0.01}}\>P_{\langle s\rangle^{-10}\le\cdot \le \langle s\rangle^{-0.05}}\psi\big)(s)\|_{L^\infty_{xy}}\,ds\\
\lesssim &
\int_0^t\big\|\Delta \pp_t\big(\partial_{tt}-\Delta\partial_t-\Delta)K(t-s)\,\, P_{\ge \langle s\rangle^{-0.04}} \frac{\pp_y\langle\nabla\rangle}{-\Delta}\pp_x\big(n_{\le \langle s\rangle^{0.01}}\>P_{\langle s\rangle^{-10}\le\cdot \le \langle s\rangle^{-0.05}}\psi\big)(s)\|_{L^\infty_{xy}}\,ds\\
\lesssim &
\int_0^t\langle t-s\rangle^{-1}\Big\|P_{\ge \langle s\rangle^{-0.04}} \frac{\pp_y\langle\nabla\rangle}{-\Delta}\pp_x\big(n_{\langle s\rangle^{-0.04}\lesssim \cdot\le \langle s\rangle^{0.01}}
  \>P_{\langle s\rangle^{-10}\le\cdot \le \langle s\rangle^{-0.05}}\psi\big)\Big\|_{L^\infty_{xy}}\,ds.
\end{align*}
Now we claim that
\begin{align}
\Big\|P_{\ge \langle s\rangle^{-0.04}} \frac{\pp_y\langle\nabla\rangle}{-\Delta}\pp_x\big(n_{\langle s\rangle^{-0.04}\lesssim \cdot\le \langle s\rangle^{0.01}}
  \>P_{\langle s\rangle^{-10}\le\cdot \le \langle s\rangle^{-0.05}}\psi\big)\Big\|_{L^\infty_{xy}}
\lesssim
\langle t\rangle^{-1.01}\|U\|_X^2.\label{90}
\end{align}
Indeed, by Littlewood-Paley's decomposition and \eqref{psi-infty},
\begin{align*}
&\Big\|P_{\ge \langle s\rangle^{-0.04}} \frac{\pp_y\langle\nabla\rangle}{-\Delta}\pp_x\big(n_{\langle s\rangle^{-0.04}\lesssim \cdot\le \langle s\rangle^{0.01}}
  \>P_{\langle s\rangle^{-10}\le\cdot \le \langle s\rangle^{-0.05}}\psi\big)\Big\|_{L^\infty_{xy}}\\
\lesssim &
\sum\limits_{\langle s\rangle^{-0.04}\le N\le1} N^{-1}\big\|\pp_x\big(n_N
  \>P_{\langle s\rangle^{-10}\le\cdot \le \langle s\rangle^{-0.05}}\psi\big)\big\|_{L^\infty_{xy}}
+\big\|P_{\ge1}\langle \nabla\rangle^{0+}\pp_x\big(n_{1\le \cdot\le \langle s\rangle^{0.01}}
  \>P_{\langle s\rangle^{-10}\le\cdot \le \langle s\rangle^{-0.05}}\psi\big)\big\|_{L^\infty_{xy}}\\
\lesssim &
\sum\limits_{\langle s\rangle^{-0.04}\le N\le1} N^{-1}\big(\|\partial_xn_N\|_{L^\infty_{xy}}\|P_{\langle s\rangle^{-10}\le\cdot \le \langle s\rangle^{-0.05}}\psi\|_{L^\infty_{xy}}+\|n_N\|_{L^\infty_{xy}}\|P_{\langle s\rangle^{-10}\le\cdot \le \langle s\rangle^{-0.05}}\partial_x\psi\|_{L^\infty_{xy}}\big)\\
&+\big(\|\langle \nabla\rangle^{0+}\partial_xn_{\le \langle s\rangle^{0.01}}\|_{L^\infty_{xy}}\|P_{\langle s\rangle^{-10}\le\cdot \le \langle s\rangle^{-0.05}}\psi\|_{L^\infty_{xy}}+\|\langle \nabla\rangle^{0+}n_{\le \langle s\rangle^{0.01}}\|_{L^\infty_{xy}}\|P_{\langle s\rangle^{-10}\le\cdot \le \langle s\rangle^{-0.05}}\partial_x\psi\|_{L^\infty_{xy}}\big)\\
\lesssim &
\sum\limits_{\langle s\rangle^{-0.04}\le N\le1} N^{0-}\|\partial_xn\|_{L^2_{xy}}\|P_{\langle s\rangle^{-10}\le\cdot \le \langle s\rangle^{-0.05}}\psi\|_{L^\infty_{xy}}+N^{-1}\|n\|_{L^\infty_{xy}}\|\partial_x\psi\|_{L^2_{xy}}^{\frac12+}\|\nabla \partial_x\psi\|_{L^2_{xy}}^{\frac12-}\big)\\
&+\langle s\rangle^{0+}\big(\|\langle \nabla\rangle\partial_xn\|_{L^2_{xy}}\|P_{\langle s\rangle^{-10}\le\cdot \le \langle s\rangle^{-0.05}}\psi\|_{L^\infty_{xy}}+\langle s\rangle^{-0.02}\|n\|_{L^\infty_{xy}}\|\partial_x\psi\|_{L^2_{xy}}\big)\\
\lesssim &
\langle t\rangle^{-1.01}\|U\|_X^2.
\end{align*}
This gives \eqref{90}. Using \eqref{90}, we further obtain that
\begin{align*}
\big\|\langle\nabla\rangle\eqref{uv-hard-111}\big\|_{L^\infty_{xy}}
\lesssim &
\int_0^t\langle t-s\rangle^{-1}\langle s\rangle^{-1.01}\,ds\|U\|_X^2\\
\lesssim &
\langle t \rangle^{-1}\|U\|_X^2.
\end{align*}

For $\big\|\langle\nabla\rangle\pp_x\eqref{uv-hard-111}\big\|_{L^2_{xy}}$, similarly, by  Proposition \ref{lem:Ku6-L} instead and the same treatment to $\big\|\langle\nabla\rangle\eqref{uv-hard-111}\big\|_{L^\infty_{xy}}$, we also  have
\begin{align*}
\big\|\langle\nabla\rangle\pp_x\eqref{uv-hard-111}\big\|_{L^2_{xy}}
\lesssim &
\langle t \rangle^{-1}\|U\|_X^2.
\end{align*}
Collecting the estimates obtained above yields the claimed estimates in this subsubsection.

\subsubsection{$\int_0^t\partial_t(\pp_{tt}-\Delta\pp_t-\Delta-\pp_{yy})K(t-s)N_1(s)\,ds$}\label{sec:Nu-N1}
This is the most troubled term in  this section.
Indeed, we take one of the terms in $N_1$, $\pp_x\psi\Delta\psi$ for example. We have the estimate
$$
\|\pp_x\psi\Delta\psi\|_{L^1_{xy}}\le \|\pp_x\psi\|_{L^2_{xy}}\|\Delta\psi\|_{L^2_{xy}}\lesssim \langle s\rangle^{-\frac34}\|U\|_X^2.
$$
But this is far from integrable, even weaker than the term in \ref{sec:n-N0-2}. Therefore, some special techniques which heavily depends on the structure and the Fourier analysis, shall be employed to deal with it.
The main estimates we state in  this subsubsection are the following
\begin{align*}
\big\|\int_0^t\partial_t(\pp_{tt}-\Delta\pp_t-\Delta-\pp_{yy})K(t-s)N_1(s)\,ds\big\|_{L^2_{xy}}
\lesssim &
\langle t\rangle^{-\frac12}Q(\|U\|_X);\\
\big\|\langle\nabla\rangle\int_0^t\partial_t(\pp_{tt}-\Delta\pp_t-\Delta-\pp_{yy})K(t-s)N_1(s)\,ds\big\|_{L^\infty_{xy}}
\lesssim &
\langle t\rangle^{-1}Q(\|U\|_X);\\
\big\||\nabla|^{\gamma}\int_0^t\partial_t(\pp_{tt}-\Delta\pp_t-\Delta-\pp_{yy})K(t-s)N_1(s)\,ds\big\|_{L^2_{xy}}
\lesssim &
\langle t\rangle^{-\frac34}Q(\|U\|_X);\\
\big\|\langle\nabla\rangle\pp_x\int_0^t\partial_t(\pp_{tt}-\Delta\pp_t-\Delta-\pp_{yy})K(t-s)N_1(s)\,ds\big\|_{L^2_{xy}}
\lesssim &
\langle t\rangle^{-1}Q(\|U\|_X).
\end{align*}

According to the frequency, we also split $N_1$ into $N_1^l$ and $N_2^h$ pieces.
The estimates on the high frequency piece $N_1^h$ is standard and we omit the details.
Now we continue to estimate the low frequency piece $N_1^l$. However, one may find that direct estimate fails. To overcome the difficulties,
firstly, by using the first equation in \eqref{e2.1}, we rewrite $N_1$ as
\begin{align*}
N_1=&- (u\pp_x u+v\pp_y u)-n\big[\Delta u+\lambda(\partial_{xx} u+\partial_{xy}v)-\pp_xn\big]- 2n\pp_x n-\pp_x\psi\Delta\psi\\
 &\quad+\frac{n^2\Delta u+n^2\lambda(\partial_{xx} u+\partial_{xy}v)}{\rho} +\frac{n\pp_x\psi\Delta\psi}{\rho}\\
=&-n\big(\pp_tu-N_1\big)-\nabla\cdot(\pp_x\psi\nabla\psi)+\pp_x\big(\frac12|\nabla\psi|^2-n^2\big)\\
 &\quad- (u\pp_x u+v\pp_y u)+\frac{n^2\Delta u+n^2\lambda(\partial_{xx} u+\partial_{xy}v)}{\rho} +\frac{n\pp_x\psi\Delta\psi}{\rho}\\
=&-\pp_t(nu)+\pp_tn\>u-\nabla\cdot(\pp_x\psi\nabla\psi)+\pp_x\big(\frac12|\nabla\psi|^2-n^2\big)\\
 &\quad+nN_1- (u\pp_x u+v\pp_y u)+\frac{n^2\Delta u+n^2\lambda(\partial_{xx} u+\partial_{xy}v)}{\rho} +\frac{n\pp_x\psi\Delta\psi}{\rho},
\end{align*}
Let
\begin{equation}\label{dec:N1}
\aligned
N_{13}=&\pp_tn\>u+nN_1- (u\pp_x u+v\pp_y u)+\frac{n^2\Delta u+n^2\lambda(\partial_{xx} u+\partial_{xy}v)}{\rho} +\frac{n\pp_x\psi\Delta\psi}{\rho};\\
N_{14}=&-\pp_t(nu);\quad N_{15}=-\nabla\cdot(\pp_x\psi\nabla\psi);\quad N_{16}=\pp_x\big(\frac12|\nabla\psi|^2-n^2\big).
\endaligned\end{equation}
and write $N_1^l=N_{13}^l+\cdots+N_{16}^l$.
Similar as Lemma \ref{lem:N11-l}, we have
\begin{align}
\|N_{13}^l\|_{L^1_{xy}}\lesssim \langle s\rangle^{-1.2}Q(\|U\|_X).\label{1044}
\end{align}
\begin{rem}
Note that
$$
\pp_{tt}-\Delta\pp_t-\Delta-\pp_{yy}=\big(\pp_{tt}-\Delta\pp_t-\Delta\big)-\pp_{yy},
$$
and $\pp_{tt}+A^2\pp_t+A^2$ has almost the same estimates as $A|\eta|$, so one may regard $\pp_{tt}-\Delta\pp_t-\Delta-\pp_{yy}$ as $\pp_{tt}-\Delta\pp_t-\Delta$
with no differences. Thus the linear estimates concerned $\pp_{tt}-\Delta\pp_t-\Delta$ also hold for $\pp_{tt}-\Delta\pp_t-\Delta-\pp_{yy}$.  Accordingly, we will not mession the difference below.
\end{rem}
Hence by Proposition \ref{lem:Kn5-L} (i) ($\beta=0$) and (ii), and \eqref{1044}, for $p=2,\infty$ we have
\begin{align*}
&\big\|\langle\nabla\rangle\int_0^t\partial_t(\pp_{tt}-\Delta\pp_t-\Delta-\pp_{yy})K(t-s)N_{13}^l(s)\,ds\big\|_{L^p_{xy}}\\
\lesssim &
\int_0^t\big\|\partial_t(\pp_{tt}-\Delta\pp_t-\Delta-\pp_{yy})K(t-s)\langle\nabla\rangle N_{13}^l(s)\big\|_{L^p_{xy}}\,ds\\
\lesssim &
\int_0^t\langle t-s\rangle^{-1+\frac1{p}}\|\langle\nabla\rangle N_{13}^l\|_{L^1_{xy}}\,ds\\
\lesssim &
\int_0^t\langle t-s\rangle^{-1+\frac1{p}}\langle s\rangle^{-1.2+0.01}\,dsQ(\|U\|_X)\\
\lesssim &
\langle t \rangle^{-1+\frac1{p}}Q(\|U\|_X).
\end{align*}
By using Proposition \ref{lem:Kn5-L} (i) ($\beta=\gamma, 1$)  instead, we also get the following estimates,
\begin{align*}
\big\||\nabla|^\gamma\int_0^t\partial_t(\pp_{tt}-\Delta\pp_t-\Delta-\pp_{yy})K(t-s)N_{13}^l(s)\,ds\big\|_{L^2_{xy}}
\lesssim &
\langle t \rangle^{-\frac34}Q(\|U\|_X);\\
\big\|\langle\nabla\rangle\pp_x\int_0^t\partial_t(\pp_{tt}-\Delta\pp_t-\Delta-\pp_{yy})K(t-s)N_{13}^l(s)\,ds\big\|_{L^2_{xy}}
\lesssim &
\langle t \rangle^{-1}Q(\|U\|_X).
\end{align*}

For the piece $N_{14}$, we integrate by parts,  to get
\begin{align}
&-\int_0^t\partial_t(\pp_{tt}-\Delta\pp_t-\Delta-\pp_{yy})K(t-s)\pp_s (nu)(s)\,ds\notag\\
=&
\partial_t(\pp_{tt}-\Delta\pp_t-\Delta-\pp_{yy})K(t)\big(n_0u_0\big)\label{0323-1}\\
&\quad -\int_0^t\partial_{tt}(\pp_{tt}-\Delta\pp_t-\Delta-\pp_{yy})K(t-s)(nu)(s)\,ds.\label{0323-2}
\end{align}
For \eqref{0323-1},  as in Proposition \ref{lem:Ku3-L}, we have
\begin{align*}
\big\|\eqref{0323-1}\big\|_{L^2_{xy}}
\lesssim
\langle t \rangle^{-\frac12}Q(\|U_0\|_{X_0});&\quad
\big\|\langle\nabla\rangle\eqref{0323-1}\big\|_{L^\infty_{xy}}
\lesssim
\langle t \rangle^{-1}Q(\|U_0\|_{X_0});\\
\big\||\nabla|^\gamma\eqref{0323-1}\big\|_{L^2_{xy}}
\lesssim
\langle t \rangle^{-\frac34}Q(\|U_0\|_{X_0});&\quad
\big\|\langle\nabla\rangle\pp_x\eqref{0323-1}\big\|_{L^2_{xy}}
\lesssim
\langle t \rangle^{-1}Q(\|U_0\|_{X_0}).
\end{align*}
For \eqref{0323-2}, we find
\begin{align}
\|(nu)(s)\|_{L^2_{xy}}\lesssim \|n(s)\|_{L^2_{xy}}\|u(s)\|_{L^\infty_{xy}}
\lesssim \langle s\rangle^{-1.25}Q(\|U\|_X).\label{1045}
\end{align}
Now by \eqref{1045} and Proposition \ref{lem:Kn9-L} (i), for $\beta=0,\gamma,1$, we have
\begin{align*}
&\big\|\langle\nabla\rangle|\na|^\beta\eqref{0323-2}\big\|_{L^2_{xy}}\\
\lesssim &
\int_0^t\big\|\langle \nabla \rangle |\nabla|^\beta\partial_{tt}(\pp_{tt}-\Delta\pp_t-\Delta-\pp_{yy})K(t-s)(nu)(s)\|_{L^2_{xy}}\,ds\\
\lesssim &
\int_0^t\langle t-s\rangle^{-\frac12-\frac\beta2}\|(nu)(s)\|_{L^2_{xy}}\,ds\\
\lesssim &
\langle t \rangle^{-\frac12-\frac\beta2}Q(\|U\|_X).
\end{align*}
This gives the desired estimates of $\big\|\eqref{0323-2}\big\|_{L^2_{xy}}, \big\||\na|^\gamma\eqref{0323-2}\big\|_{L^2_{xy}},\big\|\langle\nabla\rangle\pp_x\eqref{0323-2}\big\|_{L^2_{xy}}$. Furthermore, from  Proposition \ref{lem:Kn9-L} (ii),
\begin{align*}
\big\|\langle\nabla\rangle\eqref{0323-2}\big\|_{L^\infty_{xy}}
\lesssim &
\int_0^t\big\|\partial_{tt}(\pp_{tt}-\Delta\pp_t-\Delta-\pp_{yy})K(t-s)\langle\nabla\rangle(nu)(s)\|_{L^2_{xy}}\,ds\\
\lesssim &
\int_0^t\langle t-s\rangle^{-1} \big\|P_{\lesssim \langle s \rangle^{0.01}}\langle\nabla\rangle (nu)(s)\big\|_{L^2_{xy}}\,ds\\
\lesssim &
\int_0^t\langle t-s\rangle^{-1}\langle s\rangle^{-1.25+0.02}\,dsQ(\|U\|_X)\\
\lesssim &
\langle t \rangle^{-1}Q(\|U\|_X).
\end{align*}
Thus we finish the estimates on the piece $N_{14}$.

Now we continue to consider the piece $N_{15}$.
First we write
\begin{align*}
\pp_x\psi\nabla\psi=&\pp_x\psi_{\ge \langle s\rangle^{-0.05}}\>\nabla\psi+\pp_x\psi_{\le \langle s\rangle^{-0.05}}\>\nabla \psi_{\le \langle s\rangle^{-0.04}}\\
&+P_{\gtrsim \langle s\rangle^{-0.04}}\big(\pp_x\psi_{\le \langle s\rangle^{-0.05}}\>\nabla \psi_{\ge \langle s\rangle^{-0.04}}\big),
\end{align*}
and thus
\begin{align}
&\int_0^t\partial_t(\pp_{tt}-\Delta\pp_t-\Delta-\pp_{yy})K(t-s)N_{15}(s)\,ds\notag\\
=&
\int_0^t\partial_t\nabla(\pp_{tt}-\Delta\pp_t-\Delta-\pp_{yy})K(t-s)\cdot\Big(\pp_x\psi_{\ge \langle s\rangle^{-0.05}}\>\nabla\psi+\pp_x\psi_{\le \langle s\rangle^{-0.05}}\>\nabla \psi_{\le \langle s\rangle^{-0.04}}\Big)\,ds\label{1510-1}\\
& +\int_0^t\partial_t\nabla(\pp_{tt}-\Delta\pp_t-\Delta-\pp_{yy})K(t-s)\cdot P_{\gtrsim \langle s\rangle^{-0.04}}\big(\pp_x\psi_{\le \langle s\rangle^{-0.05}}\>\nabla \psi_{\ge \langle s\rangle^{-0.04}}\big)(s)\,ds.\label{1510-2}
\end{align}
Since the first two parts have the estimates
\begin{align*}
&\big\|\pp_x\psi_{\ge \langle s\rangle^{-0.05}}\>\nabla\psi\big\|_{L^2_{xy}}\\
\lesssim &
\big\|\pp_x\psi_{\ge \langle s\rangle^{-0.05}}\big\|_{L^2_{xy}}\big\|\nabla\psi\big\|_{L^\infty_{xy}}
\lesssim
\langle s \rangle^{0.05}\big\|\pp_x\nabla \psi\big\|_{L^2_{xy}}\big\|\nabla\psi\big\|_{L^\infty_{xy}}\\
\lesssim &
\langle s \rangle^{0.05-0.75-0.5}\|U\|_X^2
\lesssim
\langle s \rangle^{-1.2}\|U\|_X^2;\\
&\big\|\pp_x\psi_{\le \langle s\rangle^{-0.05}}\>\nabla \psi_{\le \langle s\rangle^{-0.04}}\big\|_{L^2_{xy}}\\
\lesssim &
\big\|\pp_x\psi\big\|_{L^2_{xy}}\big\|\nabla \psi_{\le \langle s\rangle^{-0.04}}\big\|_{L^\infty_{xy}}
\lesssim
\langle s \rangle^{-0.03+}\big\|\pp_x \psi\big\|_{L^2_{xy}}\big\||\nabla|^{\frac14+}\psi\big\|_{L^\infty_{xy}}\\
\lesssim &
\langle s \rangle^{-0.03-0.5-0.5+}\|U\|_X^2
\lesssim
\langle s \rangle^{-1.02}\|U\|_X^2,
\end{align*}
we can treat them together. Therefore, by Proposition \ref{lem:Kn5-L} (i) ($\beta=1$) and (iii), for $p=2, \infty$, we have
\begin{align*}
\big\|\langle\nabla\rangle \eqref{1510-1}\big\|_{L^p_{xy}}
\lesssim &
\int_0^t\Big\|\partial_t\nabla (\pp_{tt}-\Delta\pp_t-\Delta-\pp_{yy})K(t-s)\\
&\qquad\cdot\langle\nabla\rangle\Big(\pp_x\psi_{\ge \langle s\rangle^{-0.05}}\>\nabla\psi+\pp_x\psi_{\le \langle s\rangle^{-0.05}}\>\nabla \psi_{\le \langle s\rangle^{-0.04}}\Big)(s)\Big\|_{L^p_{xy}}\,ds\\
\lesssim &
\int_0^t\langle t-s\rangle^{-1+\frac1p}\Big\|\langle\nabla\rangle^{1+}\big(\pp_x\psi_{\ge \langle s\rangle^{-0.05}}\>\nabla\psi+\pp_x\psi_{\le \langle s\rangle^{-0.05}}\>\nabla \psi_{\le \langle s\rangle^{-0.04}}\big)\Big\|_{L^2_{xy}}\,ds\\
\lesssim &
\int_0^t\langle t-s\rangle^{-1+\frac1p}\langle s \rangle^{0.01-1.02+}\|U\|_X^2
\,ds\\
\lesssim &
\langle t\rangle^{-1+\frac1p}\|U\|_X^2.
\end{align*}
By the same way, and using Proposition \ref{lem:Kn5-L} (i) ($\beta=1+\gamma,2$) instead, we also get
\begin{align*}
\big\||\nabla|^\gamma\eqref{1510-1}\big\|_{L^2_{xy}}
\lesssim &
\langle t\rangle^{-\frac34}\|U\|_X^2;\quad
\big\|\langle\nabla\rangle\pp_x\eqref{1510-1}\big\|_{L^2_{xy}}
\lesssim
\langle t\rangle^{-1}\|U\|_X^2.
\end{align*}
For \eqref{1510-2}, we use the argument to treat $\big\|\langle\nabla\rangle \eqref{uv-hard-111}\big\|_{L^\infty_{xy}}$ before. Indeed, by Proposition \ref{lem:Kn5-L} (i) ($\beta=2$) and (iv),
\begin{align*}
\big\|\langle\nabla\rangle\eqref{1510-2}\big\|_{L^p_{xy}}
\lesssim &
\int_0^t\big\|\Delta\partial_t(\pp_{tt}-\Delta\pp_t-\Delta-\pp_{yy})K(t-s)\\
&\qquad\qquad P_{\gtrsim \langle s\rangle^{-0.04}}\frac{\langle \nabla\rangle\nabla}{-\Delta}\cdot \big(\pp_x\psi_{\le \langle s\rangle^{-0.05}}\>\nabla \psi_{\ge \langle s\rangle^{-0.04}}\big)(s)\big\|_{L^p_{xy}}\\
\lesssim &
\int_0^t\langle t-s\rangle^{-1}\big\|P_{\gtrsim \langle s\rangle^{-0.04}}\frac{\langle \nabla\rangle\nabla}{-\Delta}\cdot \big(\pp_x\psi_{\le \langle s\rangle^{-0.05}}\>\nabla \psi_{\ge \langle s\rangle^{-0.04}}\big)(s)\big\|_{L^p_{xy}}.
\end{align*}
Similar as \eqref{90}, we have
\begin{align*}
\big\|P_{\gtrsim \langle s\rangle^{-0.04}}\frac{\langle \nabla\rangle\nabla}{-\Delta}\cdot \big(\pp_x\psi_{\le \langle s\rangle^{-0.05}}\>\nabla \psi_{\ge \langle s\rangle^{-0.04}}\big)(s)\big\|_{L^p_{xy}}\,ds
\lesssim &
\langle s\rangle^{0.05}
\big\|\big(\pp_x\psi_{\le \langle s\rangle^{-0.05}}\>\nabla \psi_{\ge \langle s\rangle^{-0.04}}\big)(s)\big\|_{L^p_{xy}}\,ds,
\end{align*}
and thus it is less than $ \langle s\rangle^{-0.95}\|U\|_X^2$ when $p=2$; less than  $ \langle s\rangle^{-1.05}\|U\|_X^2$ when $p=\infty$. Hence, we obtain that
$$
\big\|\eqref{1510-2}\big\|_{L^2_{xy}}\lesssim \langle t\rangle^{-\frac12}\|U\|_X^2;\quad
\big\||\nabla|^\gamma\eqref{1510-2}\big\|_{L^2_{xy}}\lesssim \langle t\rangle^{-\frac34}\|U\|_X^2;\quad
\big\|\langle\nabla\rangle\eqref{1510-2}\big\|_{L^\infty_{xy}}\lesssim \langle t\rangle^{-1}\|U\|_X^2.
$$
For $\big\|\langle\nabla\rangle\pp_x\eqref{1510-2}\big\|_{L^2_{xy}}$, we use Proposition \ref{lem:Ku6-L} instead, to get
$$
\big\|\langle\nabla\rangle\pp_x\eqref{1510-2}\big\|_{L^2_{xy}}\lesssim \langle t\rangle^{-1}\|U\|_X^2.
$$
Thus we finish the estimates on the piece $N_{15}$.

Now we consider the estimates on the piece $N_{16}$. To do this, we swap the places of the operators $\pp_t$ and $\pp_x$ via integration by parts, and get
\begin{align}
&\int_0^t\partial_t(\pp_{tt}-\Delta\pp_t-\Delta-\pp_{yy})K(t-s)\pp_x\big(\frac12|\nabla\psi|^2-n^2\big)(s)\,ds\notag\\
=&
-\pp_x(\pp_{tt}-\Delta\pp_t-\Delta-\pp_{yy})K(t)\big(\frac12|\nabla\psi_0|^2-n_0^2\big)\label{1627-1}\\
&\quad +\int_0^t\partial_{x}(\pp_{tt}-\Delta\pp_t-\Delta-\pp_{yy})K(t-s)\pp_s\big(\frac12|\nabla\psi|^2-n^2\big)(s)\,ds.\label{1627-2}
\end{align}
For \eqref{1627-1}, as in Proposition \ref{lem:Ku3-L}, we have
\begin{align*}
\big\|\eqref{1627-1}\big\|_{L^2_{xy}}
\lesssim
\langle t \rangle^{-\frac12}Q(\|U_0\|_{X_0});&\quad
\big\|\eqref{1627-1}\big\|_{L^\infty_{xy}}
\lesssim
\langle t \rangle^{-1}Q(\|U_0\|_{X_0});\\
\big\||\nabla|^\gamma\eqref{1627-1}\big\|_{L^2_{xy}}
\lesssim
\langle t \rangle^{-\frac34}Q(\|U_0\|_{X_0});&\quad
\big\|\pp_x\eqref{1627-1}\big\|_{L^2_{xy}}
\lesssim
\langle t \rangle^{-1}Q(\|U_0\|_{X_0}).
\end{align*}
For \eqref{1627-2}, by the first and fourth equations in \eqref{e2.1}, we have
\begin{align*}
\pp_t\big(\frac12|\nabla\psi|^2-n^2\big)
= &
\nabla\psi\cdot (-\nabla v+\nabla N_3)-2n(-\nabla\cdot \vec u+N_0)\\
= &
-\nabla\psi\cdot\nabla v+2n\nabla\cdot \vec u+\nabla\psi\cdot \nabla N_3-2nN_0.
\end{align*}
Thus similar as Lemma \ref{lem:N11-l}, and using interpolation, we have
\begin{align}
\|\langle \nabla\rangle^3\pp_s\big(\frac12|\nabla\psi|^2-n^2\big)\|_{L^1_{xy}}\lesssim \langle s\rangle^{-1.1}Q(\|U\|_X).\label{1658}
\end{align}
Then by Proposition \ref{lem:Kn3-L} (i), we have
\begin{align*}
\big\|\langle\nabla\rangle\eqref{1627-2}\big\|_{L^p_{xy}}
\lesssim &
\int_0^t\Big\|\partial_{x}(\pp_{tt}-\Delta\pp_t-\Delta-\pp_{yy})K(t-s)\,\,\langle\nabla\rangle\pp_s\big(\frac12|\nabla\psi|^2-n^2\big)(s)\Big\|_{L^p_{xy}}\,ds\\
\lesssim &
\int_0^t\langle t-s\rangle^{-1+\frac1p}\big\| \langle \nabla\rangle^3\pp_s\big(\frac12|\nabla\psi|^2-n^2\big)\big\|_{L^1_{xy}}\,ds\\
\lesssim &
\int_0^t\langle t-s\rangle^{-1+\frac1p}\langle s\rangle^{-1.1}
\,ds\|U\|_X^2\\
\lesssim &
\langle t \rangle^{-1+\frac1p}\|U\|_X^2.
\end{align*}
By using Proposition \ref{lem:Kn3-L} (iii) ($\beta=\gamma$) instead, we obtain
\begin{align*}
\big\||\na|^\gamma\eqref{1627-2}\big\|_{L^2_{xy}}
\lesssim
\langle t \rangle^{-\frac34}\|U\|_X^2.
\end{align*}
While by using Proposition \ref{lem:Kn8-L} (i) instead, we obtain
\begin{align*}
\big\|\langle\nabla\rangle\pp_x\eqref{1627-2}\big\|_{L^2_{xy}}
\lesssim
\langle t \rangle^{-1}\|U\|_X^2.
\end{align*}
Thus we finish the estimates on the piece $N_{16}$.
Now collecting the estimates obtained above, we finish the proof of the claimed estimates in this subsubsection.

\subsubsection{$\int_0^t\partial_{xy}\pp_t K(t-s)N_2(s)\,ds$}\label{sec:Nu-N2}
This term is in the same level as $\int_0^t\partial_t(\pp_{tt}-\Delta\pp_t-\Delta-\pp_{yy})K(t-s)N_1(s)\,ds$,
so we just give the sketch of proof. Indeed, similar as above, we rewrite $N_2$ as
\begin{align*}
N_2=&- (u\pp_x v+v\pp_y v)-n\big[\Delta v+\lambda(\partial_{xy} u+\partial_{yy}v)-\Delta\psi-\pp_yn\big]- 2n\pp_y n-\pp_y\psi\Delta\psi\\
 &\quad+\frac{n^2\Delta v+n^2\lambda(\partial_{xy} u+\partial_{yy}v)-n^2\Delta\psi}{\rho} +\frac{n\pp_y\psi\Delta\psi}{\rho}\\
=&-n\big(\pp_tv-N_2\big)-\nabla\cdot(\pp_y\psi\nabla\psi)+\pp_y\big(\frac12|\nabla\psi|^2-n^2\big)\\
 &\quad- (u\pp_x v+v\pp_y v)+\frac{n^2\Delta v+n^2\lambda(\partial_{xy} u+\partial_{yy}v)-n^2\Delta\psi}{\rho} +\frac{n\pp_y\psi\Delta\psi}{\rho}\\
=&-\pp_t(nv)+\pp_tn\>v-\nabla\cdot(\pp_y\psi\nabla\psi)+\pp_y\big(\frac12|\nabla\psi|^2-n^2\big)\\
 &\quad+nN_2- (u\pp_x v+v\pp_y v)+\frac{n^2\Delta v+n^2\lambda(\partial_{xy} u+\partial_{yy}v)-n^2\Delta\psi}{\rho} +\frac{n\pp_y\psi\Delta\psi}{\rho},
\end{align*}
Let
\begin{equation}\label{dec:N2}
\aligned
N_{23}=&\pp_tn\>v+nN_2- (u\pp_x v+v\pp_y v)+\frac{n^2\Delta v+n^2\lambda(\partial_{xy} u+\partial_{yy}v)-n^2\Delta\psi}{\rho} +\frac{n\pp_y\psi\Delta\psi}{\rho};\\
N_{24}=&-\pp_t(nv);\quad N_{25}=-\nabla\cdot(\pp_y\psi\nabla\psi)+\pp_y\big(\frac12|\nabla\psi|^2-n^2\big).
\endaligned\end{equation}
By the same arguments to treat the piece $N_{13}$ and $N_{14}$ respectively as above, we obtain the disable estimates on $N_{23}$ and $N_{24}$.
Moreover, the piece $N_{25}$ is similar as $N_{16}$. Indeed, we change the places between $\pp_x$ and $\nabla$, then
\begin{align*}
&\int_0^t\partial_{xy}\pp_t K(t-s)N_{25}(s)\,ds\\
=&
\int_0^t\partial_{y}\nabla\cdot\pp_t K(t-s)\pp_x(\pp_y\psi\nabla\psi)(s)\,ds
+\int_0^t\partial_{yy}\pp_t K(t-s)\pp_x\big(\frac12|\nabla\psi|^2-n^2\big)(s)\,ds.
\end{align*}
Since the operator $\partial_{xy}\nabla K(t)$ obeys the same decaying estimates as $\partial_{x}(\pp_t-\Delta\pp_t-\Delta-\pp_{yy})\nabla K(t)$,
we can obtain the disable estimates by the way to treat  $N_{16}$. Therefore, we get that
\begin{align*}
\big\|\int_0^t\partial_{xy}\pp_t K(t-s)N_2(s)\,ds\big\|_{L^2_{xy}}
\lesssim &
\langle t\rangle^{-\frac12}Q(\|U\|_X);\\
\big\|\langle\nabla\rangle\int_0^t\partial_{xy}\pp_t K(t-s)N_2(s)\,ds\big\|_{L^\infty_{xy}}
\lesssim &
\langle t\rangle^{-1}Q(\|U\|_X);\\
\big\||\nabla|^{\gamma}\int_0^t\partial_{xy}\pp_t K(t-s)N_2(s)\,ds\big\|_{L^2_{xy}}
\lesssim &
\langle t\rangle^{-\frac34}Q(\|U\|_X);\\
\big\|\langle\nabla\rangle\pp_x\int_0^t\partial_{xy}\pp_t K(t-s)N_2(s)\,ds\big\|_{L^2_{xy}}
\lesssim &
\langle t\rangle^{-1}Q(\|U\|_X).
\end{align*}

\subsubsection{$\int_0^t\partial_{xy}\Delta K(t-s)N_3(s)\,ds$} The same reason as in Section \ref{sec:n-N3}, we find
$$
\partial_{xy}\nabla\sim \pp_x(\pp_t-\Delta\pp_t-\Delta),\quad \mbox{ and } \nabla N_3\sim N_0.
$$
Therefore, the  term $\int_0^t\partial_{xy}\Delta K(t-s)N_3(s)\,ds$ can be treated by the similar way as
$\int_0^t \pp_x(\partial_{tt}-\Delta\partial_t-\Delta)K(t-s)N_0(s)\,ds$, and thus we get the estimates (the details are omitted here)
\begin{align*}
\big\|\int_0^t\partial_{xy}\Delta K(t-s)N_3(s)\,ds\big\|_{L^2_{xy}}
\lesssim &
\langle t\rangle^{-\frac12}Q(\|U\|_X);\\
\big\|\langle\nabla\rangle\int_0^t\partial_{xy}\Delta K(t-s)N_3(s)\,ds\big\|_{L^\infty_{xy}}
\lesssim &
\langle t\rangle^{-1}Q(\|U\|_X);\\
\big\||\nabla|^{\gamma}\int_0^t\partial_{xy}\Delta K(t-s)N_3(s)\,ds\big\|_{L^2_{xy}}
\lesssim &
\langle t\rangle^{-\frac34}Q(\|U\|_X);\\
\big\|\langle\nabla\rangle\pp_x\int_0^t\partial_{xy}\Delta K(t-s)N_3(s)\,ds\big\|_{L^2_{xy}}
\lesssim &
\langle t\rangle^{-1}Q(\|U\|_X).
\end{align*}

\subsubsection{$\lambda\int_0^t(\pp_{tt}-\Delta\pp_t-\Delta)\partial_t K(t-s)\big(\partial_{xx}u(s)+\partial_{xy}v(s)\big)\,ds$}
\label{sec:u-lambda}
We will prove in this subsubsection that
\begin{align*}
\big\|\lambda\int_0^t(\pp_{tt}-\Delta\pp_t-\Delta)\partial_t K(t-s)\big(\partial_{xx}u(s)+\partial_{xy}v(s)\big)\,ds\big\|_{L^2_{xy}}
\lesssim &
\langle t\rangle^{-\frac12}|\lambda|\>\|U\|_X;\\
\big\|\langle\nabla\rangle\lambda\int_0^t(\pp_{tt}-\Delta\pp_t-\Delta)\partial_t K(t-s)\big(\partial_{xx}u(s)+\partial_{xy}v(s)\big)\,ds\big\|_{L^\infty_{xy}}
\lesssim &
\langle t\rangle^{-1}|\lambda|\>\|U\|_X;\\
\big\||\nabla|^{\gamma}\lambda\int_0^t(\pp_{tt}-\Delta\pp_t-\Delta)\partial_t K(t-s)\big(\partial_{xx}u(s)+\partial_{xy}v(s)\big)\,ds\big\|_{L^2_{xy}}
\lesssim &
\langle t\rangle^{-\frac34}|\lambda|\>\|U\|_X;\\
\big\|\langle\nabla\rangle\pp_x\lambda\int_0^t(\pp_{tt}-\Delta\pp_t-\Delta)\partial_t K(t-s)\big(\partial_{xx}u(s)+\partial_{xy}v(s)\big)\,ds\big\|_{L^2_{xy}}
\lesssim &
\langle t\rangle^{-1}|\lambda|\>\|U\|_X.
\end{align*}
To do this, we first write
\begin{align}
&\lambda\int_0^t(\pp_{tt}-\Delta\pp_t-\Delta)\partial_t K(t-s)\big(\partial_{xx}u(s)+\partial_{xy}v(s)\big)\,ds\notag\\
=&\lambda\int_0^{t/2}(\pp_{tt}-\Delta\pp_t-\Delta)\partial_t K(t-s)\big(\partial_{xx}u(s)+\partial_{xy}v(s)\big)\,ds\label{1709-1}\\
&+\lambda\int_{t/2}^t(\pp_{tt}-\Delta\pp_t-\Delta)\partial_t K(t-s)\big(\partial_{xx}u(s)+\partial_{xy}v(s)\big)\,ds.\label{1709-2}
\end{align}
By Proposition \ref{lem:Kn5-L} (i) ($\beta=2$), we have
\begin{align*}
\big\|\eqref{1709-1}\big\|_{L^2_{xy}}
\lesssim &
|\lambda|\int_0^{t/2}\big\|\Delta (\pp_{tt}-\Delta\pp_t-\Delta)\partial_t K(t-s) \vec u(s)\big\|_{L^2_{xy}}\,ds\\
\lesssim &
|\lambda|\int_0^{t/2}\langle t-s\rangle^{-1}
\big\|\vec u(s)\big\|_{L^2_{xy}}\,ds
\lesssim
|\lambda|\int_0^{t/2}\langle t-s\rangle^{-1}\langle s\rangle^{-\frac12}\,ds\|U\|_X\\
\lesssim &
|\lambda|\langle t \rangle^{-\frac12}\|U\|_X.
\end{align*}
By using Proposition \ref{lem:Kn5-L} (i) ($\beta=2$), (iv) and (iv) respectively, we also get
\begin{align*}
\big\||\nabla|^\gamma \eqref{1709-1}\big\|_{L^2_{xy}}&\lesssim |\lambda|\langle t \rangle^{-\frac34}\|U\|_X;\\
\big\|\langle\nabla\rangle\eqref{1709-1}\big\|_{L^\infty_{xy}}&\lesssim |\lambda|\langle t \rangle^{-1}\|U\|_X;\\
\big\|\langle\nabla\rangle\pp_x \eqref{1709-1}\big\|_{L^2_{xy}}&\lesssim |\lambda|\langle t \rangle^{-1}\|U\|_X.
\end{align*}
While by Proposition \ref{lem:Kn5-L} (i) ($\beta=1$),
\begin{align*}
\big\|\eqref{1709-2}\big\|_{L^2_{xy}}
\lesssim &
|\lambda|\int_{t/2}^t\big\|\nabla (\pp_{tt}-\Delta\pp_t-\Delta)\partial_t K(t-s)\,\, \nabla\cdot \vec u(s)\big\|_{L^2_{xy}}\,ds\\
\lesssim &
|\lambda|\int_{t/2}^t\langle t-s\rangle^{-\frac12}\langle s\rangle^{-1}\,ds\|U\|_X\\
\lesssim &
|\lambda|\langle t \rangle^{-\frac12}\|U\|_X.
\end{align*}
Thus, we give that
\begin{align*}
|\lambda|\big\|\int_0^t(\pp_{tt}-\Delta\pp_t-\Delta)\partial_t K(t-s)\big(\partial_{xx}u(s)+\partial_{xy}v(s)\big)\,ds\big\|_{L^2_{xy}}
\lesssim
|\lambda|\langle t \rangle^{-\frac12}\|U\|_X.
\end{align*}
Similar argument, we obtain that
\begin{align*}
|\lambda|\big\||\nabla|^\gamma\int_0^t(\pp_{tt}-\Delta\pp_t-\Delta)\partial_t K(t-s)\big(\partial_{xx}u(s)+\partial_{xy}v(s)\big)\,ds\big\|_{L^2_{xy}}
\lesssim
|\lambda|\langle t \rangle^{-\frac34}\|U\|_X.
\end{align*}
However, to estimate $\big\|\langle\nabla\rangle\eqref{1709-2}\big\|_{L^\infty_{xy}}$ and $\big\|\pp_x\eqref{1709-2}\big\|_{L^2_{xy}}$, they are much difficult. The reason is that
we do not have
\begin{align}\label{eq:10.28}
\int_{t/2}^t \langle t-s\rangle^{-1}\langle s\rangle^{-1}\,ds\|U\|_X\leq C \langle t\rangle^{-1}\|U\|_X,
\end{align}
due to the unboundedness of the integral
on the left-hand side as $t \to \infty$.  It is also worthing to note that
$ \int_{t/2}^t\langle t-s\rangle^{-1-\delta}\langle s\rangle^{-1}\,ds\|U\|_X$
is bounded by the right-hand side for any positive $\delta$. However, we have such a $\delta$-loss in \eqref{eq:10.28}.
To overcome the difficulties, we split the operator $K(t)$ into two parts,
$$
K(t)=K_a(t)+K_b(t),
$$
where $K_a, K_b$ are defined by
\begin{align*}
\widehat{K_a}(t, \xi, \eta)=\chi_{|\xi|\le  A^2} \widehat{K}(t, \xi, \eta);\quad
\widehat{K_b}(t, \xi, \eta)=\chi_{|\xi|\ge A^2} \widehat{K}(t, \xi, \eta).
\end{align*}
Therefore,
\begin{align}
\eqref{1709-2}
=&\lambda\int_{t/2}^t(\pp_{tt}-\Delta\pp_t-\Delta)\partial_t K_a(t-s)\big(\partial_{xx}u(s)+\partial_{xy}v(s)\big)\,ds\label{1709-21}\\
&+\lambda\int_{t/2}^t(\pp_{tt}-\Delta\pp_t-\Delta)\partial_t K_b(t-s)\big(\partial_{xx}u(s)+\partial_{xy}v(s)\big)\,ds.\label{1709-22}
\end{align}
In the following, we only consider $\big\|\langle\nabla\rangle\pp_x\eqref{1709-2}\big\|_{L^2_{xy}}$. The estimate $\big\|\langle\nabla\rangle\eqref{1709-2}\big\|_{L^\infty_{xy}}$ can be treated
by the same way, or one may get it by interpolation directly.

By \eqref{est:K-ttt2t2}, we have
\begin{align*}
&|A\xi\pp_t\big(\pp_{tt}+A^2\pp_t+A^2\big)\widehat{K_a}(t, \xi, \eta)|\\
\lesssim & \chi_{|\xi|\le A^2 }\big( \chi_{A\ge 1} e^{-ct}+A|\xi|\chi_{A\le 1}e^{-\frac14 A^2t}
+\frac{\xi^3}{A^3}e^{-\frac{\xi^2}{2A^2}t}\big)\\
\lesssim &  \chi_{A\ge 1} \frac1{A} e^{-ct}+A^3\chi_{A\le 1}e^{-\frac14 A^2t}
+\frac{\xi^3}{A^3}e^{-\frac{\xi^2}{2A^2}t}\\
\lesssim & \langle t\rangle^{-\frac32}.
\end{align*}
Therefore,
\begin{align}
\big\|\langle\nabla\rangle\pp_x\eqref{1709-21}\big\|_{L^2_{xy}}
\lesssim &
|\lambda|\int_{t/2}^t\big\|A\xi\partial_t\big(\pp_{tt}+A^2\pp_t+A^2\big)\widehat K_a(t-s,\xi,\eta)\big\|_{L^\infty_{\xi\eta}}\big\|\langle\nabla\rangle\nabla\cdot\vec u(s)\big\|_{L^2_{xy}}\,ds\notag\\
\lesssim &
|\lambda|\int_{t/2}^t\langle t-s\rangle^{-\frac32}\langle s\rangle^{-1}\,ds\|U\|_X\notag\\
\lesssim &
|\lambda|\langle t \rangle^{-1}\|U\|_X.\label{10.53}
\end{align}
Now we consider the term \eqref{1709-22}. To this end, we give some analysis on $K_b(t)$ first. By \eqref{K-ttt2t2-K1}, we have
\begin{align}
&\pp_t\big(\pp_{tt}+A^2\pp_t+A^2\big)\widehat K_b(t,\xi,\eta)\notag\\
=&-\frac12A^2\big(\pp_{tt}+A^2\pp_t+A^2\big)\widehat K_b(t,\xi,\eta)
+\chi_{|\xi|\ge A^2}\widehat K_1(t,\xi,\eta),
\end{align}
where $K_1$ is defined in \eqref{K1} at the beginning of Section \ref{sec:LF}.
Therefore, we have
\begin{align}
\eqref{1709-22}
=&\frac12\lambda\int_{t/2}^t\Delta(\pp_{tt}-\Delta\pp_t-\Delta)\partial_t K_b(t-s)\big(\partial_{xx}u(s)+\partial_{xy}v(s)\big)\,ds\label{1709-221}\\
&+\lambda\int_{t/2}^tK_{1,b}(t-s)\big(\partial_{xx}u(s)+\partial_{xy}v(s)\big)\,ds,\label{1709-222}
\end{align}
where $\widehat{K_{1,b}}(t, \xi, \eta)=\chi_{|\xi|\ge A^2} \widehat{K_1}(t, \xi, \eta).$
By \eqref{est:K-tt2t2}, we have
\begin{align*}
|\xi A^3\big(\pp_{tt}+A^2\pp_t+A^2\big)\widehat K_b(t,\xi,\eta)|
\lesssim  A^3\chi_{A\le 1}e^{-\frac14 A^2t}
\lesssim  \langle t\rangle^{-\frac32}.
\end{align*}
It obeys  the same estimate as $A\xi\pp_t\big(\pp_{tt}+A^2\pp_t+A^2\big)\widehat{K_a}(t, \xi, \eta)$.
Thus, we have the same estimates on \eqref{1709-221} as \eqref{10.53}. By the definition of $K_{1,b}$, we have
\begin{align*}
\widehat{K_{1,b}}(t,\xi,\eta)=&\frac14\chi_{|\xi|\ge A^2} e^{(-\frac12A^2+\sqrt{\frac14A^4-A^2+A|\eta|})t}
+\frac14\chi_{|\xi|\ge A^2} e^{(-\frac12A^2-\sqrt{\frac14A^4-A^2+A|\eta|})t}\notag\\
&\qquad+\frac14\chi_{|\xi|\ge A^2} e^{(-\frac12A^2+\sqrt{\frac14A^4-A^2-A|\eta|})t}
+\frac14\chi_{|\xi|\ge A^2} e^{(-\frac12A^2-\sqrt{\frac14A^4-A^2-A|\eta|})t}.
\end{align*}
It includes four parts, and each part has the bound of $e^{-\frac14A^2t}$. So we may only consider one of them.
In particular, let the operator $K_c(t)$ be defined by
$$
\widehat{K_c}(t,\xi,\eta)=\frac14\chi_{|\xi|\ge A^2} e^{(-\frac12A^2+\mu_1\sqrt{\frac14A^4-A^2+\mu_2A|\eta|})t},
$$
for $\mu_1,\mu_2=\pm1$.
Then we have
$$
K_c(t-s)=K_c(t)K_c(-s).
$$
Since $K_c$ is bounded from $L^2$ to $L^2$, we have
\begin{align*}
&\big\|\langle\nabla\rangle\pp_x\lambda\int_{t/2}^tK_c(t-s)\big(\partial_{xx}u(s)+\partial_{xy}v(s)\big)\,ds\big\|_{L^2_{xy}}\\
= &
|\lambda|\big\|\langle\nabla\rangle\pp_x\int_{t/2}^tK_c(t)K_c(-s)\big(\partial_{xx}u(s)+\partial_{xy}v(s)\big)\,ds\big\|_{L^2_{xy}}\\
\lesssim &
|\lambda|\int_{t/2}^t\big\|K_c(t)\pp_{xx}K_c(-s)\big(\langle\nabla\rangle\nabla \cdot \vec u(s)\big)\big\|_{L^2_{xy}}\,ds\\
\lesssim &
|\lambda|\int_{t/2}^t\big\|\pp_{xx}K_c(-s)\langle\nabla\rangle\big(\nabla \cdot u(s)\big)\big\|_{L^2_{xy}}\,ds\\
\lesssim &
|\lambda|\int_{t/2}^t\big\|A^2\widehat{K_c}(-s,\xi,\eta)\big\|_{L^\infty_{\xi\eta}}\big\|\langle\nabla\rangle\big(\nabla \cdot u(s)\big)\big\|_{L^2_{xy}}\,ds\\
\lesssim &
|\lambda|\int_{t/2}^t\langle s\rangle^{-2}\,ds\|U\|_X
\lesssim
|\lambda|\langle t \rangle^{-1}\|U\|_X.
\end{align*}
Therefore, we prove that
$$
\big\|\langle\nabla\rangle\pp_x\eqref{1709-222}\big\|_{L^2_{xy}}\lesssim |\lambda|\langle t \rangle^{-1}\|U\|_X.
$$
Now collecting the estimates obtained in Section \ref{sec:u-N0-1}--Section \ref{sec:u-lambda}, we obtain \eqref{nu-1}--\eqref{nu-4}. Combining with the estimates in Section \ref{sec:LBu} gives \eqref{est:u-L2}--\eqref{est:u-xL2}.

\vskip .4in

\section{The estimates on $v$}
\label{sec:v}

In this section, we shall prove that
\begin{align}
\|v(t)\|_{L^2_{xy}}&\lesssim \langle t\rangle^{-\frac12}\big(\|U_0\|_{X_0}+|\lambda|\>\|U\|_X+Q(\|U\|_{X})\big);\label{est:v-L2}\\
\|\langle \nabla\rangle v(t)\|_{L^\infty_{xy}}&\lesssim \langle t\rangle^{-1}\big(\|U_0\|_{X_0}+|\lambda|\>\|U\|_X+Q(\|U\|_{X})\big);\label{est:v-Linfty}\\
\|v(t)\|_{L^2_{x}L^\infty_y}&\lesssim \langle t\rangle^{-\frac34}\big(\|U_0\|_{X_0}+|\lambda|\>\|U\|_X+Q(\|U\|_{X})\big);\label{est:v-xL2}\\
\||\nabla|^\gamma v(t)\|_{L^2_{xy}}&\lesssim \langle t\rangle^{-\frac34}\big(\|U_0\|_{X_0}+|\lambda|\>\|U\|_X+Q(\|U\|_{X})\big);\label{est:v-gaL2}\\
\|\langle \nabla\rangle\nabla v(t)\|_{L^2_{xy}}&\lesssim \langle t\rangle^{-1}\big(\|U_0\|_{X_0}+|\lambda|\>\|U\|_X+Q(\|U\|_{X})\big).\label{est:v-naL2}
\end{align}
Note that by Nash's inequality, and Sobolev's inequality, we have
\begin{align*}
\| v(t)\|_{L^2_{x}L^\infty_y}\lesssim & \|v(t)\|_{L^2_{xy}}^\frac12 \|\pp_yv(t)\|_{L^2_{xy}}^\frac12\lesssim \|v(t)\|_{L^2_{xy}}^\frac12\|\langle \nabla\rangle\nabla v(t)\|_{L^2_{xy}}^\frac12;\\
\big\||\nabla|^\gamma v(t)\big\|_{L^2_{xy}}\lesssim &\|v(t)\|_{L^2_{xy}}^\frac12\|\langle \nabla\rangle\nabla v(t)\|_{L^2_{xy}}^\frac12.
\end{align*}
Thus \eqref{est:v-xL2} and \eqref{est:v-gaL2} can be established by \eqref{est:v-L2} and \eqref{est:v-naL2}, and we only need to prove \eqref{est:v-L2},
\eqref{est:v-Linfty} and \eqref{est:v-naL2}.

\subsection{The reexpression of $v$}
Now we give the reexpression of $v$ as
\begin{prop}
The unknown function $v$ obeys the formula,
\begin{align}\label{exp-v}
v(t,x,y)=
(L_v+B_v)(t;n_0,\vec u_0, \vec b_0)+\mathcal{N}_v(t;n,\vec u, \psi),
\end{align}
where $(L_v+B_v)$ is given by
\begin{align}
&L_v(t;n_0,\vec u_0, \vec b_0)+B_v(t;n_0,\vec u_0, \vec b_0)\notag\\
=&\pp_tK(t)[\pp_{xy}u_0+\pp_{yy}v_0]-(\partial_{tt}-\Delta\partial_t-\pp_{xx})K(t)
\big[\Delta\psi_0\big]+(\partial_{tt}-\Delta\partial_t-\Delta)K(t)[\Delta v(0)]\notag\\
&-(\partial_{tt}-\Delta\partial_t)K(t)
\big[\partial_y n_0\big]-\frac12\Delta\big(\Delta+\sqrt{\Delta\pp_{yy}}\big)K(t)\big[v_0\big]+\pp_t (\partial_{tt}-\Delta\partial_t-\pp_{xx})K(t)\big[v_0\big], \label{eq:LBv}
\end{align}
and $\mathcal{N}_v(t;n,\vec u, \psi)$ is defined as
\begin{align}
\mathcal{N}_v(t;n,\vec u, \psi)
=&
-\int_0^t\partial_y(\pp_{tt}-\Delta\pp_t)K(t-s)N_0(s)\,ds\notag\\
&\quad+\int_0^t\partial_t\partial_{xy}K(t-s)N_1(s)\,ds+\int_0^t\partial_t(\pp_{tt}-\Delta\pp_t-\partial_{xx}) K(t-s)N_2(s)\,ds\notag\\
&\quad
-\int_0^t\Delta (\partial_{tt}-\Delta\partial_t-\partial_{xx})K(t-s)N_3(s)\,ds\notag\\
&\quad
+\lambda\int_0^t \partial_{t}(\pp_{tt}-\Delta\partial_{t})K(t-s)\big[\partial_{xy}u(s)+\partial_{yy}v(s)\big]\,ds.\label{v-N}
\end{align}
\end{prop}

Now we begin to prove this proposition.
Again, using \eqref{F2-r}  and integration by parts (see below for the details), we have
\begin{align*}
v(t,x,y)=&L_v(t;n_0,\vec u_0, \vec b_0)+\int_0^t K(t-s)F_2(s)\,ds\\
=& L_v(t;n_0,\vec u_0, \vec b_0)+\int_0^t K(t-s)\Big[ -\partial_y(\pp_{ss}-\Delta\pp_s)N_0(s)+\partial_s\partial_{xy}N_1(s)\notag\\
&\quad+\partial_s(\pp_{ss}-\Delta\pp_s-\partial_{xx})N_2(s) -\Delta (\partial_{ss}-\Delta\partial_s-\partial_{xx}) N_3(s)\\ &\quad+\lambda(\pp_{ss}-\Delta\partial_{s})\partial_{s}\big(\partial_{xy}u(s)+\partial_{yy}v(s)\big)\Big]\,ds\\
=&
L_v(t;n_0,\vec u_0, \vec b_0)+B_v(t;n_0,\vec u_0, \vec b_0)+\mathcal{N}_v(t;n,\vec u, \psi).
\end{align*}
Here according to \eqref{eq:Formula}, we denote
\begin{align}
&L_v(t;n_0,\vec u_0, \vec b_0)\notag\\
=&
K(t)\big[(\pp_{tt}-\Delta\pp_t-\Delta)\pp_{t}v(0)\big]\label{Lv-1}\\
&+(\partial_{tt}-\Delta\partial_t-\Delta)K(t)\big[\partial_tv(0)\big]\label{Lv-2}\\
&-\Delta K(t)\big[(\pp_{tt}-\Delta\pp_t-\Delta)v(0)\big]
+\partial_t K(t)\big[(\pp_{tt}-\Delta\pp_t-\Delta)v(0)\big]\label{Lv-3}\\
&-\frac12\Delta\sqrt{\Delta\partial_{yy}} K(t)[v_0]+K_1(t)[v_0]\label{Lv-4},
\end{align}
and $B_v(t;n_0,\vec u_0, \vec b_0)$ is the boundary term given below.

Next we will simply $L_v+B_v$. To do this, we give the
explicit expressions of $L_v$ and $B_v$ respectively.

\subsubsection{$L_v(t;n_0,\vec u_0, \vec b_0)$}\label{sec:Lv}

By the equations \eqref{e2.1}, \eqref{e2.2} and \eqref{e2.8} at $t=0$, we have
\begin{align*}
&(\pp_{tt}-\Delta\pp_t-\Delta)\pp_{t}v(0)\\
=&(\pp_{tt}-\Delta\pp_t-\Delta)\big(\Delta v(0)+\lambda(\partial_{xy} u(0)+\partial_{yy}v(0))-\partial_y n(0)-\Delta \psi(0)+N_2(0)\big)\\
=&(\pp_{tt}-\Delta\pp_t-\Delta)\Delta v(0)+(\pp_{tt}-\Delta\pp_t-\Delta)N_2(0)\\
&-\partial_y\big[\pp_y\Delta\psi(0)+\pp_tN_0(0)-\Delta N_0(0)-\pp_xN_1(0)-\pp_yN_2(0)-\lambda(\pp_x\Delta u(0)+\pp_y\Delta v(0))\big]\\
&-\Delta\big[\partial_{y}n(0)-N_2(0)+(\partial_t-\Delta) N_3(0)-\lambda(\pp_{xy} u(0)+\pp_{yy} v(0))\big]\\
&+\lambda(\pp_{tt}-\Delta\pp_t-\Delta)\big[\partial_{xy} u(0)+\partial_{yy}v(0)\big]\\
=&(\pp_{tt}-\Delta\pp_t-\Delta)\Delta v(0)+(\pp_{tt}-\Delta\pp_t-\Delta)N_2(0)-\partial_{yy}\Delta\psi(0)-\pp_y\Delta n(0)\\
&+\big[-\pp_t\partial_yN_0+\Delta\pp_y N_0(0)+\pp_{xy}N_1(0)+(\Delta+\pp_{yy}) N_2(0)-(\partial_t-\Delta)\Delta N_3(0)\big]\\
&+\lambda(\pp_{tt}-\Delta\pp_t+\Delta)\big[\partial_{xy} u(0)+\partial_{yy}v(0)\big].
\end{align*}
Thus,
\begin{align*}
\eqref{Lv-1}=&
K(t)\big[(\pp_{tt}-\Delta\pp_t-\Delta)\Delta v(0)\big]+K(t)\big[(\pp_{tt}-\Delta\pp_t-\Delta)N_2(0)\big]\\
&-K(t)\big[\partial_{yy}\Delta\psi(0)+\pp_y\Delta n(0)\big]\\
&+K(t)\big[-\pp_t\partial_yN_0+\partial_y\Delta N_0(0)+\pp_{xy}N_1(0)+(\Delta+\pp_{yy}) N_2(0)-(\partial_t-\Delta)\Delta N_3(0)\big]\\
&+\lambda K(t)\big[(\pp_{tt}-\Delta\pp_t+\Delta)(\partial_{xy} u(0)+\partial_{yy}v(0))\big].
\end{align*}
Now we consider \eqref{Lv-2} and \eqref{Lv-3}. Similarly,
\begin{align*}
\eqref{Lv-2}=&
(\partial_{tt}-\Delta\partial_t-\Delta)K(t)
\big[\Delta v(0)+\lambda(\partial_{xy} u(0)+\partial_{yy}v(0))-\partial_y n(0)-\Delta \psi(0)+N_2(0)\big]\\
=&(\partial_{tt}-\Delta\partial_t-\Delta)K(t)
\big[\Delta v(0)-\partial_y n(0)-\Delta \psi(0)\big]+(\partial_{tt}-\Delta\partial_t-\Delta)K(t)\big[N_2(0)\big]\\
&+\lambda(\partial_{tt}-\Delta\partial_t-\Delta)K(t)\big[\partial_{xy} u(0)+\partial_{yy}v(0)\big];\\
\eqref{Lv-3}=&
-\Delta K(t)\big[(\pp_{tt}-\Delta\pp_t-\Delta)v(0)\big]
+\partial_t K(t)\big[(\pp_{tt}-\Delta\pp_t-\Delta)v(0)\big]\\
=&-\Delta K(t)\big[(\pp_{tt}-\Delta\pp_t-\Delta)v(0)\big]+\pp_{yy}\pp_tK(t)[v(0)]\\
&+\pp_tK(t)\big[\partial_{xy}u(0)-\pp_yN_0(0)+\partial_tN_2(0)-\Delta N_3(0)+\lambda\pp_t(\pp_{xy} u(0)+\pp_{yy} v(0))\big].
\end{align*}
Then collecting the estimates above, we have
\begin{align*}
&L_v(t;n_0,\vec u_0, \vec b_0)\notag\\
=&
K(t)\big[(\pp_{tt}-\Delta\pp_t-\Delta)N_2(0)\big]+(\partial_{tt}-\Delta\partial_t-\Delta)K(t)\big[N_2(0)\big]\\
&
+K(t)\big[-\pp_t\partial_yN_0+\partial_y\Delta N_0(0)+\pp_{xy}N_1(0)+(\Delta+\pp_{yy}) N_2(0)-(\partial_t-\Delta)\Delta N_3(0)\big]\\
&+\pp_tK(t)\big[-\pp_yN_0(0)+\partial_tN_2(0)-\Delta N_3(0)\big]+(\partial_{tt}-\Delta\partial_t-\Delta)K(t)[\Delta v(0)]\\
&
+\pp_{xy}\pp_tK(t)[u_0]-(\partial_{tt}-\Delta\partial_t)K(t)
\big[\partial_y n_0\big]-(\partial_{tt}-\Delta\partial_t-\pp_{xx})K(t)
\big[\Delta\psi_0\big]\\%
%
&+\pp_{yy}\pp_tK(t)[v_0]+\lambda K(t)\big[(\pp_{tt}-\Delta\pp_t)(\pp_{xy}u(0)+\pp_{yy}v(0))\big]\\
&
+\lambda(\partial_{tt}-\Delta\partial_t)K(t)\big[\partial_{xy} u(0)+\partial_{yy}v(0)\big]+\lambda\pp_tK(t)\big[\pp_t(\pp_{xy} u(0)+\pp_{yy} v(0))\big]\\
&-\frac12\Delta\sqrt{\Delta\partial_{yy}} K(t)[v_0]+K_1(t)[v_0].
\end{align*}

\subsubsection{$B_v(t;n_0,\vec u_0, \vec b_0)$} Now we consider the boundary term $B_v(t;n_0,\vec u_0, \vec b_0)$. By integration by parts, and arguing similarly as in Section \ref{sec:Bn} we have
\begin{align*}
-\int_0^t K(t-s)&\big[ \partial_y(\pp_{ss}-\Delta\pp_s)N_0(s)\big]\,ds\notag\\
=&K(t)\big[\partial_y(\pp_{t}-\Delta)N_0(0)\big]+\pp_t K(t)\big[\partial_yN_0(0)\big]\\
&\quad -\int_0^t \partial_y(\pp_{tt}-\Delta\pp_t)K(t-s)N_0(s)\,ds;\\
\int_0^t K(t-s)&\big[\partial_s\partial_{xy}N_1(s)\big]\,ds\notag\\
=&-K(t)\big[\partial_{xy}N_1(0)\big]+\int_0^t \partial_t\partial_{xy}K(t-s)N_1(s);\\
\int_0^t K(t-s)&\big[ \partial_s(\pp_{ss}-\Delta\pp_s-\partial_{xx})N_2(s)\big]\,ds\notag\\
=&-K(t)\big[(\pp_{tt}-\Delta\pp_t-\partial_{xx})N_2(0)\big]-\pp_tK(t)\big[(\pp_{t}-\Delta)N_2(0)\big]-\pp_{tt}K(t)\big[N_2(0)\big]\\
&\quad+\int_0^t \partial_t(\pp_{tt}-\Delta\pp_t-\partial_{xx})K(t-s) N_2(s)\,ds;\\
-\int_0^t K(t-s)&\big[(\partial_{ss}-\Delta\partial_s-\partial_{xx}) \Delta N_3(s)\big]\,ds\notag\\
=&K(t)\big[(\partial_{t}-\Delta)\Delta N_3(0)\big]+\pp_tK(t)\big[\Delta N_3(0)\big]\\
&\quad-\int_0^t \Delta (\partial_{tt}-\Delta\partial_t-\partial_{xx})K(t-s)N_3(s)\,ds;\notag\\
 \lambda\int_0^t K(t-s)&\big[(\pp_{ss}-\Delta\partial_{s})\partial_{s}\big(\partial_{xy}u(s)+\partial_{yy}v(s)\big)\big]\,ds\notag\\
=&-\lambda K(t)\big[(\pp_{tt}-\Delta\partial_{t})\big(\partial_{xy}u(0)+\partial_{yy}v(0)\big)\big]\\
&\quad-\lambda \partial_{t}K(t)\big[(\pp_{t}-\Delta)\big(\partial_{xy}u(0)+\partial_{yy}v(0)\big)\big]
-\lambda \partial_{tt}K(t)\big[\partial_{xy}u(0)+\partial_{yy}v(0)\big]\\
&\quad+\lambda\int_0^t \partial_{t}(\pp_{tt}-\Delta\partial_{t})K(t-s)\big[\partial_{xy}u(s)+\partial_{yy}v(s)\big]\,ds.
\end{align*}
Therefore, we obtain the boundary term $B_v(t;n_0,\vec u_0, \vec b_0)$ as
\begin{align*}
&B_v(t;n_0,\vec u_0, \vec b_0)\\
=&-K(t)\big[(\pp_{tt}-\Delta\pp_t-\partial_{xx})N_2(0)\big]\\
&+K(t)\big[\pp_t\pp_y N_0(0)-\pp_y\Delta N_0(0)-\partial_{xy}N_1(0)+\pp_t\Delta N_3(0)-\Delta^2N_3(0)\big]\\
&+\pp_t K(t)\big[\pp_y N_0(0)-\pp_tN_2(0)+\Delta N_2(0)+\Delta N_3(0)\big]-\pp_{tt}K(t)\big[N_2(0)\big]\\
&-\lambda K(t)\big[(\pp_{tt}-\Delta\partial_{t})\big(\partial_{xy}u(0)+\partial_{yy}v(0)\big)\big]\\
&-\lambda \partial_{t}K(t)\big[(\pp_{t}-\Delta)\big(\partial_{xy}u(0)+\partial_{yy}v(0)\big)\big]
-\lambda \partial_{tt}K(t)\big[\partial_{xy}u(0)+\partial_{yy}v(0)\big].
\end{align*}

Now combining with the result obtained in Section \ref{sec:Lv}, and by the same argument in \eqref{eq:LBu-r}, we have \eqref{eq:LBv}.

Again,
we split into the following two subsection to consider the linear parts and nonlinear parts separately.
Most of the terms are in the same level as the corresponding ones in Section \ref{sec:n}, and can be treated similarly, so we only give the
sketch of the proof.

\vskip .1in
\subsection{Estimates on the linear parts $L_v+B_v$}
In this subsection, we shall prove that
\begin{lemma}\label{lem:v-linear}
\begin{align}
\big\|(L_v+B_v)(t;n_0,\vec u_0, \vec b_0 )\big\|_{L^2_{xy}}
\lesssim&
\langle t\rangle^{-\frac12}\|U_0\|_{X_0};\label{est:v-linear-L2}\\
\big\|\langle \nabla\rangle(L_v+B_v)(t;n_0,\vec u_0, \vec b_0 )\big\|_{L^\infty_{xy}}
\lesssim&
\langle t\rangle^{-1}\|U_0\|_{X_0};\label{est:v-linear-Linfty}\\
\big\|\langle \nabla\rangle\nabla(L_v+B_v)(t;n_0,\vec u_0, \vec b_0 )\big\|_{L^2_{xy}}
\lesssim&
\langle t\rangle^{-1}\|U_0\|_{X_0}.\label{est:v-linear-nax2}
\end{align}
\end{lemma}

Now we show that the each in  \eqref{eq:LBv} obeys  the estimates in \eqref{est:v-linear-L2}--\eqref{est:v-linear-nax2}.

The estimates on the term $\pp_tK(t)[\pp_{xy}u_0+\pp_{yy}v_0]$ can be obtained
by Proposition \ref{lem:Kn2-L}. Also, for the term $(\partial_{tt}-\Delta\partial_t)K(t)\big[\partial_y n_0\big]$, we use Proposition \ref{lem:Ku2-L} to obtain the desirable estimates.
For the term $(\partial_{tt}-\Delta\partial_t-\pp_{xx})K(t)\big[\Delta\psi_0\big]$, we use Proposition \ref{lem:Kn11-L} to obtain  the desirable estimates.
The estimates on the terms $\frac12\Delta\big(\Delta+\sqrt{\Delta\pp_{yy}}\big)K(t)\big[v_0\big]$
and $\pp_t (\partial_{tt}-\Delta\partial_t-\Delta)K(t)\big[v_0\big]$ are obtained by Propositions \ref{lem:Ku3-L} and  \ref{lem:Kn5-L} (i) (ii) respectively.

\subsection{Estimates on the nonlinear parts $\mathcal{N}_v$}
In this subsection, we shall prove that
\begin{lemma}\label{lem:v-nonlinear}
\begin{align}
\big\|\mathcal N_v(t;n_0,\vec u_0, \vec b_0 )\big\|_{L^2_{xy}}
\lesssim&
\langle t\rangle^{-\frac12}\|U_0\|_{X_0};\label{est:v-nonlinear-L2}\\
\big\|\langle \nabla\rangle\mathcal N_v(t;n_0,\vec u_0, \vec b_0 )\big\|_{L^\infty_{xy}}
\lesssim&
\langle t\rangle^{-1}\|U_0\|_{X_0};\label{est:v-nonlinear-Linfty}\\
\big\|\langle \nabla\rangle\nabla(\mathcal N_v)(t;n_0,\vec u_0, \vec b_0 )\big\|_{L^2_{xy}}
\lesssim&
\langle t\rangle^{-1}\|U_0\|_{X_0}.\label{est:v-nonlinear-nax2}
\end{align}
\end{lemma}
We estimate $\mathcal N_v(t;n,\vec u, \psi)$ in \eqref{v-N} terms by terms.

\subsubsection{$\int_0^t\pp_y(\pp_{tt}-\Delta\pp_t)K(t-s)N_0(s)\,ds$}\label{sec:Nv-N0}
The high frequency piece $N_0^h$ can be treated standard, so we only treat $N_0^l$.  Recall that $N_0=\nabla\cdot (n\vec u)$. First, we have
\begin{align}\label{91}
\big\|(n\vec u)(s)\big\|_{L^2_{xy}}\lesssim \|n(s)\|_{L^2_{xy}}\|\vec u(s)\|_{L^\infty_{xy}}\lesssim \langle s\rangle^{-1.25}\|U\|_X^2.
\end{align}
Then by Proposition \ref{lem:Kn4-L} (ii), for $\beta=0,1$, and \eqref{91} we have
\begin{align*}
&\Big\||\nabla|^\beta\int_0^t\pp_y(\pp_{tt}-\Delta\pp_t)K(t-s)N_0^l(s)\,ds\Big\|_{L^2_{xy}}\\
\lesssim &
\int_0^t\Big\||\nabla|^\beta\nabla\pp_y(\pp_{tt}-\Delta\pp_t)K(t-s)\cdot (n\vec u)(s)\,ds\Big\|_{L^2_{xy}}\\
\lesssim &
\int_0^t\langle t-s\rangle^{-\frac{1+\beta}2}\big\|P_{\lesssim \langle s \rangle^{0.01}}\langle \nabla \rangle^{\beta}(n\vec u)(s)\big\|_{L^2_{xy}}\,ds\\
\lesssim &
\int_0^t\langle t-s\rangle^{-\frac{1+\beta}2} \langle s\rangle^{-1.25+0.01}\|U\|_X^2\,ds\\
\lesssim &
\langle t \rangle^{-\frac{1+\beta}2}\|U\|_X^2.
\end{align*}
Similarly, using Proposition \ref{lem:Kn4-L} (iii) instead, we have
\begin{align*}
&\Big\|\int_0^t\pp_y(\pp_{tt}-\Delta\pp_t)K(t-s)N_0^l(s)\,ds\Big\|_{L^\infty_{xy}}\\
\lesssim &
\int_0^t\langle t-s\rangle^{-1} \big\|P_{\lesssim \langle s \rangle^{0.01}}\langle \nabla \rangle^{1+}(n\vec u)(s)\big\|_{L^2_{xy}}\,ds\\
\lesssim &
\int_0^t\langle t-s\rangle^{-1} \langle s\rangle^{-1.25+0.02}\|U\|_X^2\,ds\\
\lesssim &
\langle t \rangle^{-1}\|U\|_X^2.
\end{align*}
Thus, we obtain the claimed results.

\subsubsection{$\int_0^t\partial_t\partial_{xy}K(t-s)N_1(s)\,ds$}
We use the same argument as in Section \ref{sec:Nu-N1}, and decompose $N_1$ into four parts as \eqref{dec:N1}. Then all the desirable estimates
on the terms $N_{13}, N_{14}, N_{15}, N_{16}$ can be obtained by the corresponding ways. More precisely, since
$$
\pp_t\pp_x\pp_y\sim \pp_t(\pp_{tt}-\Delta\pp_t-\Delta-\pp_{yy});\quad \nabla\pp_t\pp_x\pp_y\sim \pp_x\pp_t(\pp_{tt}-\Delta\pp_t-\Delta-\pp_{yy}),
$$
here $\sim $ means that the two operator obey the similar decaying estimates presented in Section \ref{KernelProperty},
the details are just mimicked there and so are omitted here.

\subsubsection{$\int_0^t\partial_t(\pp_{tt}-\Delta\pp_t-\partial_{xx}) K(t-s)N_2(s)\,ds$}
Also, using the argument in Section \ref{sec:Nu-N2}, and noting that
$$
 \pp_t(\pp_{tt}-\Delta\pp_t-\pp_{xx})\sim \pp_t\pp_x\pp_y;\quad \nabla \pp_t(\pp_{tt}-\Delta\pp_t-\pp_{xx})\sim \pp_x\pp_t\pp_x\pp_y
$$
(indeed, the former behaviors slightly better than the latter),
we can obtain the desirable estimates here. The details are omitted again.

\subsubsection{$\int_0^t\Delta (\partial_{tt}-\Delta\partial_t-\partial_{xx})K(t-s)N_3(s)\,ds$}
By Lemma \ref{lem:Kn4} and Lemma \ref{lem:Kn11}, we note that
$$
\nabla (\partial_{tt}-\Delta\partial_t-\partial_{xx})\sim \pp_y(\pp_{tt}-\Delta\pp_t).
$$
Moreover,
$$
\|\nabla\psi\cdot\vec u\|_{L^2_{xy}}\lesssim \|\nabla\psi\|_{L^2_{xy}}\|\vec u\|_{L^\infty_{xy}}\lesssim \langle s\rangle^{-1.25}\|U\|_X^2.
$$
Therefore, by the same way as in Section \ref{sec:Nv-N0}, we obtain the desirable estimates.

\subsubsection{$\lambda\int_0^t\partial_{t}(\pp_{tt}-\Delta\partial_{t})K(t-s)\big[\big(\partial_{xy}u(s)+\partial_{yy}v(s)\big)\big]\,ds$}
\label{sec:Nv-lambda}
To prove this term, we use the similar precess as in Section \ref{sec:u-lambda}. Indeed, we rewrite
$$
\partial_{t}(\pp_{tt}-\Delta\partial_{t})K_b(t)=\partial_{t}(\pp_{tt}-\Delta\partial_{t}-\Delta)K_b(t)+\pp_t\Delta K_b(t).
$$
Moreover, by \eqref{Kt-2}, we have
$$
\pp_tA^2\widehat{K_b}(t)=-\frac{A^4}2\widehat{ K_b}(t)+\frac{A}{|\eta|}\widehat{K_{1,b}}(t).
$$
Therefore, we have the similar structure as the first part. Then by the same way as in Section \ref{sec:Nv-N0}, we get the claimed estimates.

Collecting all the estimates in Section \ref{sec:Nv-N0}--Section \ref{sec:Nv-lambda}, we obtain Lemma \ref{lem:v-nonlinear}. Together with
Lemma \ref{lem:v-linear}, we establish \eqref{est:v-L2}--\eqref{est:v-naL2}.

\vskip .4in
\section{The estimates on $\psi$}
\label{sec:psi}

In this section, we shall prove that there exists some small constant $\epsilon>0$, such that
\begin{align}
\|\langle \nabla\rangle^4|\nabla|^\gamma\psi(t)\|_{L^2_{xy}}&\lesssim \langle t\rangle^{-\frac14}\big(\|U_0\|_{X_0}+\epsilon_0\|U\|_X+Q(\|U\|_{X})\big);\label{est:psi-L2}\\
\big\||\nabla|^{\bar \gamma}\langle \nabla\rangle \psi(t)\big\|_{L^\infty_{xy}}&\lesssim \langle t\rangle^{-\frac12}\big(\|U_0\|_{X_0}+\epsilon_0\|U\|_X+Q(\|U\|_{X})\big);\label{est:psi-Linfty}\\
\|\pp_x\psi(t)\|_{L^2_{xy}}&\lesssim \langle t\rangle^{-\frac12}\big(\|U_0\|_{X_0}+\epsilon_0\|U\|_X+Q(\|U\|_{X})\big);\label{est:psi-xL2}\\
\|\pp_x\nabla\psi(t)\|_{L^2_{xy}}&\lesssim \langle t\rangle^{-\frac34}\big(\|U_0\|_{X_0}+\epsilon_0\|U\|_X+Q(\|U\|_{X})\big).\label{est:psi-naxL2}
\end{align}
Again, it follows from Nash's inequality that for any $\frac\gamma2<\bar \gamma<1+\frac\gamma2$,
\begin{align}\label{Nash}
\big\||\nabla|^{\bar\gamma}\langle \nabla\rangle \psi(t)\big\|_{L^\infty_{xy}}\lesssim \|\langle \nabla\rangle^4|\nabla|^\gamma\psi(t)\|_{L^2_{xy}}^{\frac12}\|\nabla\pp_x\psi(t)\|_{L^2_{xy}}^{\frac12},
\end{align}
(see its proof in Appendix \ref{app5}). Thus, we only need to show \eqref{est:psi-L2}, \eqref{est:psi-xL2} and \eqref{est:psi-naxL2}.

\subsection{The reexpression of $\psi$}
Now we give the reexpression of $\psi$ and obtain
\begin{prop}
The unknown function $\psi$ obeys the formula,
\begin{align}\label{exp-psi}
\psi(t,x,y)=
(L_\psi+B_\psi)(t;n_0,\vec u_0, \vec b_0)+\mathcal{N}_\psi(t;n,\vec u, \psi),
\end{align}
where $(L_\psi+B_\psi)$ is given by
\begin{equation}
\aligned
L_\psi(t;n_0,\vec u_0, \vec b_0)&+B_\psi(t;n_0,\vec u_0, \vec b_0)
=-K(t)\big[\partial_{xy}u_0+\partial_{y}\Delta n_0\big]+\pp_{y}\pp_tK(t)[n_0]\\
&-(\partial_{tt}-\Delta\partial_t-\pp_{xx})K(t)\big[v_0\big]
-\frac12\Delta\sqrt{\Delta\partial_{yy}} K(t)[\psi_0]+K_1(t)[\psi_0], \label{eq:LBpsi}
\endaligned
\end{equation}
and
\begin{align}
\mathcal{N}_\psi(t;n,\vec u, \psi)=
&\int_0^t\partial_y\partial_tK(t-s)N_0(s)\,ds-\int_0^t\partial_y\Delta K(t-s)N_0(s)\,ds\notag\\
&\quad-\int_0^t\partial_{xy}K(t-s)N_1(s)\,ds-\int_0^t(\pp_{tt}-\Delta\pp_t-\partial_{xx})K(t-s)N_2(s)\,ds\notag\\
&\quad
-\int_0^t(\partial_{t}-\Delta) (\partial_{tt}-\Delta\partial_t-\Delta) K(t-s)N_3(s)\,ds\notag\\
&\quad
-\lambda\int_0^t(\pp_{tt}-\Delta\pp_t)K(t-s)\big(\partial_{xy}u(s)+\partial_{yy}v(s)\big)\,ds.\label{psi-N}
\end{align}
\end{prop}

The proof of this proposition is similar as the one in the previous three sections.
Again, using \eqref{F3-r}  and integration by parts (see below for details), we have
\begin{align*}
\psi(t,x,y)=&L_\psi(t;n_0,\vec u_0, \vec b_0)+\int_0^t K(t-s)F_3(s)\,ds\\
=& L_\psi(t;n_0,\vec u_0, \vec b_0)+\int_0^t K(t-s)\Big[\partial_y\partial_sN_0(s)-\partial_y\Delta N_0(s)-\partial_{xy}N_1(s)\\
&\quad-(\pp_{ss}-\Delta\pp_s-\partial_{xx})N_2(s) +(\partial_{s}-\Delta) (\partial_{ss}-\Delta\partial_s-\Delta) N_3(s)\\
&\quad -\lambda(\pp_{ss}-\Delta\pp_s)(\partial_{xy}u(s)+\partial_{yy}v(s))\Big]\,ds\\
=&
L_\psi(t;n_0,\vec u_0, \vec b_0)+B_\psi(t;n_0,\vec u_0, \vec b_0)+\mathcal{N}_\psi(t;n,\vec u, \psi).
\end{align*}
Here according to \eqref{eq:Formula},
\begin{align}
&L_\psi(t;n_0,\vec u_0, \vec b_0)\notag\\
=&
K(t)\big[(\pp_{tt}-\Delta\pp_t-\Delta)\pp_{t}\psi(0)\big]\label{Lpsi-1}\\
&+(\partial_{tt}-\Delta\partial_t-\Delta)K(t)\big[\partial_t\psi(0)\big]\label{Lpsi-2}\\
&-\Delta K(t)\big[(\pp_{tt}-\Delta\pp_t-\Delta)\psi(0)\big]
+\partial_t K(t)\big[(\pp_{tt}-\Delta\pp_t-\Delta)\psi(0)\big]\label{Lpsi-3}\\
&-\frac12\Delta\sqrt{\Delta\partial_{yy}} K(t)[\psi_0]+K_1(t)[\psi_0]\label{Lpsi-4},
\end{align}
and $B_\psi(t;n_0,\vec u_0, \vec b_0)$ is the boundary term given below.

Next we will simply $L_\psi+B_\psi$.
To this end, we give the
explicit expressions of $L_\psi$ and $B_\psi$ respectively.

\subsubsection{$L_\psi(t;n_0,\vec u_0, \vec b_0)$}\label{sec:Lpsi}

By the equations \eqref{e2.1} and \eqref{e2.6} at $t=0$, we have
\begin{align*}
(\pp_{tt}-\Delta\pp_t-\Delta)&\pp_{t}\psi(0)
=\big[(\pp_{tt}-\Delta\pp_t-\Delta)(-v(0)+N_3(0))\big]\\
=&-\big(\pp_{yy} v(0)+\partial_{xy}u(0)-\pp_yN_0(0)+\partial_tN_2(0)-\Delta N_3(0)\\
&+\lambda\pp_t(\pp_{xy} u(0)+\pp_{yy} v(0))\big)+(\pp_{tt}-\Delta\pp_t-\Delta)N_3(0).
\end{align*}
Thus,
\begin{align*}
\eqref{Lpsi-1}=&
-K(t)\big[\pp_{yy} v(0)+\partial_{xy}u(0)-\pp_yN_0(0)+\partial_tN_2(0)-\Delta N_3(0)\\
&\quad +\lambda\pp_t(\pp_{xy} u(0)+\pp_{yy} v(0))\big] + K(t)\big[(\pp_{tt}-\Delta\pp_t-\Delta)N_3(0)\big].
\end{align*}
Now we consider \eqref{Lpsi-2} and \eqref{Lpsi-3}, similarly,
\begin{align*}
\eqref{Lpsi-2}=
&(\partial_{tt}-\Delta\partial_t-\Delta)K(t)
\big[-v(0)+N_3(0)\big]\\
=&-(\partial_{tt}-\Delta\partial_t-\Delta)K(t)\big[v(0)\big]+
(\partial_{tt}-\Delta\partial_t-\Delta)K(t)\big[N_3(0)\big];\\
\eqref{Lpsi-3}=&
\big(\partial_t -\Delta\big) K(t)\big[(\pp_{tt}-\Delta\pp_t-\Delta)\psi(0)\big]\\
=&\big(\partial_t -\Delta\big)  K(t)\big[\partial_{y}n(0)-N_2(0)+(\partial_t-\Delta) N_3(0)-\lambda(\pp_{xy} u(0)+\pp_{yy} v(0))\big].
\end{align*}
Then collecting the estimates above, we have
\begin{align*}
&L_\psi(t;n_0,\vec u_0, \vec b_0)\notag\\
=&
K(t)\big[(\pp_{tt}-\Delta\pp_t-\Delta)N_3(0)\big]+(\partial_{tt}-\Delta\partial_t-\Delta)K(t)\big[N_3(0)\big]\\
&
+K(t)\big[\pp_yN_0(0)-\partial_tN_2(0)+\Delta N_2(0)+\Delta N_3(0)-(\partial_t-\Delta)\Delta N_3(0)\big]\\
&+\pp_tK(t)\big[-N_2(0)+(\partial_t-\Delta) N_3(0)\big]\\
&
-K(t)\big[\partial_{xy}u(0)+\partial_{y}\Delta n(0)\big]+\pp_{y}\pp_tK(t)[n(0)]\\
&-(\partial_{tt}-\Delta\partial_t-\pp_{xx})K(t)\big[v(0)\big]
-\lambda K(t)\big[(\pp_{t}-\Delta)(\pp_{xy}u(0)+\pp_{yy}v(0))\big]\\
&
-\lambda\pp_tK(t)\big[\pp_{xy} u(0)+\pp_{yy} v(0)\big]+\eqref{Lpsi-4}.
\end{align*}

\subsubsection{$B_\psi(t;n_0,\vec u_0, \vec b_0)$} Now we consider the boundary term $B_\psi(t;n_0,\vec u_0, \vec b_0)$. By integration by parts, and arguing similarly as in Section \ref{sec:Bn} we have
\begin{align*}
\int_0^t K(t-s)&\big[\partial_y\partial_sN_0(s)\big]\,ds\\
=&-K(t)\big[\partial_yN_0(0)\big]+\int_0^t \partial_y\partial_tK(t-s)N_0(s)\,ds;\\
-\int_0^t K(t-s)&\big[(\pp_{ss}-\Delta\pp_s-\partial_{xx})N_2(s)\big]\,ds\\
=&K(t)\big[(\pp_{t}-\Delta)N_2(0)\big]+\pp_{t}K(t)\big[N_2(0)\big]\\
&\quad-\int_0^t (\pp_{tt}-\Delta\pp_t-\partial_{xx})K(t-s)N_2(s)\,ds;\\
\int_0^t K(t-s)&\big[(\partial_{s}-\Delta) (\partial_{ss}-\Delta\partial_s-\Delta) N_3(s)\big]\,ds\\
=&-K(t)\big[(\partial_{tt}-\Delta\partial_t-\Delta) N_3(0)\big]-\partial_tK(t)\big[(\partial_{t}-\Delta) N_3(0)\big]-\partial_{tt}K(t)\big[N_3(0)\big]\\
&\quad +K(t)\big[(\partial_{t}-\Delta)\Delta  N_3(0)\big]+\Delta \partial_{t}K(t)\big[N_3(0)\big]\\
&\quad+(\partial_{t}-\Delta) (\partial_{tt}-\Delta\partial_t-\Delta) \int_0^t K(t-s)N_3(s)\,ds;\\
-\lambda\int_0^t K(t-s)&\big[(\pp_{tt}-\Delta\pp_t)(\partial_{xy}u(s)+\partial_{yy}v(s))\big]\,ds\\
=&\lambda K(t)\big[(\pp_{t}-\Delta)(\partial_{xy}u(0)+\partial_{yy}v(0))\big]+\lambda \pp_{t}K(t)\big[\partial_{xy}u(0)+\partial_{yy}v(0)\big]\\
&\quad-\lambda\int_0^t (\pp_{tt}-\Delta\pp_t)K(t-s)(\partial_{xy}u(s)+\partial_{yy}v(s))\,ds.
\end{align*}
Therefore, we obtain the boundary term $B_\psi(t;n_0,\vec u_0, \vec b_0)$ as
\begin{align*}
&B_\psi(t;n_0,\vec u_0, \vec b_0)\\
=&-K(t)\big[(\partial_{tt}-\Delta\partial_t-\Delta) N_3(0)\big]\\
&+K(t)\big[-\pp_y N_0(0)+(\pp_{t}-\Delta)N_2(0)+(\partial_{t}-\Delta)\Delta  N_3(0)\big]\\
&+\pp_t K(t)\big[ N_2(0)-(\pp_t-\Delta) N_3(0)\big]+\Delta\pp_t  K(t)\big[N_3(0)\big]-\pp_{tt}K(t)\big[N_3(0)\big]\\
&+\lambda K(t)\big[(\pp_{t}-\Delta)(\partial_{xy}u(0)+\partial_{yy}v(0))\big]+\lambda \pp_{t}K(t)\big[\partial_{xy}u(0)+\partial_{yy}v(0)\big].
\end{align*}
Combining with the result obtained in Section \ref{sec:Lpsi}, we have \eqref{eq:LBpsi}.

Again, we split into two subsections to consider the linear parts and nonlinear parts separately.
\subsection{Estimates on the linear parts $L_\psi+B_\psi$}
In this subsection, we prove that
\begin{lemma}\label{lem:psi-linear}
For $\gamma>\frac12$,
\begin{align}
\big\|\langle \nabla\rangle^4|\nabla|^\gamma(L_\psi+B_\psi)(t;n_0,\vec u_0, \vec b_0 )\big\|_{L^2}
\lesssim&
\langle t\rangle^{-\frac14}\|U_0\|_{X_0};\label{est:psi-linear}\\
\big\|\pp_x(L_\psi+B_\psi)(t;n_0,\vec u_0, \vec b_0 )\big\|_{L^2}
\lesssim&
\langle t\rangle^{-\frac12}\|U_0\|_{X_0}.\label{est:psi-linear-x2}\\
\big\|\nabla\pp_x(L_\psi+B_\psi)(t;n_0,\vec u_0, \vec b_0 )\big\|_{L^2}
\lesssim&
\langle t\rangle^{-\frac34}\|U_0\|_{X_0}.\label{est:psi-linear-x3}
\end{align}
\end{lemma}

Now we estimate the terms in \eqref{eq:LBpsi}. For the term $\pp_{xy}K(t)[u_0]$, by Proposition \ref{lem:Ku1-L}  (i), we have
\begin{align*}
\big\|\langle \nabla\rangle^4|\nabla|^\gamma\pp_{xy}K(t)[u_0]\big\|_{L^2_{xy}}
= &
\big\||\nabla|^\gamma \pp_{xy}K(t)\langle \nabla\rangle^4[u_0]\big\|_{L^2_{xy}}
\lesssim
\langle t\rangle^{-\frac14}\big\|\langle \nabla\rangle^4u_0\big\|_{L^1_{xy}};\\
\big\|\pp_x\pp_{xy}K(t)[u_0]\big\|_{L^2_{xy}}
\lesssim &
\big\|\nabla\pp_{xy}K(t)[u_0]\big\|_{L^2_{xy}}
\lesssim
\langle t\rangle^{-\frac12}\big\|\langle \nabla\rangle^{0+}u_0\big\|_{L^1_{xy}};\\
\big\|\nabla\pp_x\pp_{xy}K(t)[u_0]\big\|_{L^2_{xy}}
\lesssim &
\big\||\nabla|^2\pp_{xy}K(t)[u_0]\big\|_{L^2_{xy}}
\lesssim
\langle t\rangle^{-\frac34}\big\|\langle \nabla\rangle^{1+}u_0\big\|_{L^1_{xy}}.
\end{align*}
Now we consider the term $K(t)\big[\partial_{y}\Delta n_0\big]$. By Proposition \ref{lem:Kn1-L} (i) and Proposition \ref{lem:Ku1-L} (i), for $\beta=0$ or $1$, we have
(since $\gamma>\frac12$),
\begin{align*}
\big\|\langle \nabla\rangle^4|\nabla|^\gamma K(t)\big[\partial_{y}\Delta n_0\big]\big\|_{L^2_{xy}}
\lesssim &
\big\||\nabla|^{3+\gamma} K(t)\big[\langle \nabla\rangle^4 n_0\big]\big\|_{L^2_{xy}}
\lesssim
\langle t\rangle^{-\frac14}\big\|\langle \na\rangle^5 n_0\big\|_{L^1_{xy}};\\
\big\||\nabla|^\beta\pp_x K(t)\big[\partial_{y}\Delta n_0\big]\big\|_{L^2_{xy}}
\lesssim &
\big\||\nabla|^{2+\beta}\pp_{xy} K(t)\big[ n_0\big]\big\|_{L^2_{xy}}
\lesssim
\langle t\rangle^{-\frac34}\big\|\langle \nabla\rangle^{2+}n_0\big\|_{L^1_{xy}}.
\end{align*}

Next, we consider the term $\pp_{y}\pp_tK(t)[n_0]$. Similarly, by Proposition \ref{lem:Kn2-L} (i), for $\beta=\gamma,1,2$,
\begin{align*}
\big\|\langle \nabla\rangle^4|\nabla|^\beta\pp_t  K(t)\big[\partial_{y}n_0\big]\big\|_{L^2_{xy}}
\lesssim &
\big\||\nabla|^\beta\partial_{y}\pp_t  K(t)\big[\langle \nabla\rangle^4n_0\big]\big\|_{L^2_{xy}}
\lesssim
\langle t\rangle^{-\frac\beta2}\big\|\langle \na\rangle^{4+} n_0\big\|_{L^1_{xy}}.
\end{align*}
Choosing $\beta=\gamma,1,2$, we have the claimed estimates.

Now we consider the term $(\partial_{tt}-\Delta\partial_t-\pp_{xx})K(t)\big[v_0\big]$. From Proposition \ref{lem:Kn11-L} (i),
\begin{align*}
\big\|\langle \nabla\rangle^4|\nabla|^\beta(\partial_{tt}-\Delta\partial_t-\pp_{xx})K(t)\big[v_0\big]\big\|_{L^2_{xy}}
\lesssim &
\langle t\rangle^{-\frac\beta2}\big\|\langle \na\rangle^{4+} v_0\big\|_{L^1_{xy}}.
\end{align*}
Again, letting $\beta=\gamma,1,2$, we have the claimed estimates.

At last, we consider the terms in \eqref{Lpsi-4}. The same as
the term $K(t)\big[\partial_{y}\Delta n_0\big]$,  for $\beta=0$ or $1$, we have
\begin{align*}
\big\|\langle \nabla\rangle^4|\nabla|^\gamma \Delta\sqrt{\Delta\partial_{yy}} K(t)[\psi_0]\big\|_{L^2_{xy}}
\lesssim &
\langle t\rangle^{-\frac14}\big\|\langle \na\rangle^4 \nabla\psi_0\big\|_{L^1_{xy}};\\
\big\||\nabla|^\beta\pp_x \Delta\sqrt{\Delta\partial_{yy}} K(t)[\psi_0]\big\|_{L^2_{xy}}
\lesssim &
\langle t\rangle^{-\frac34}\big\|\langle \nabla\rangle^{2}\nabla\psi_0\big\|_{L^1_{xy}}.
\end{align*}

Similarly,  using Proposition \ref{lem:K1-L}, we have
\begin{align*}
\big\|\langle \nabla\rangle^4|\nabla|^\gamma K_1(t)[\psi_0]\big\|_{L^2_{xy}}
\lesssim &
\langle t\rangle^{-\frac14}\big\|\langle \na\rangle^5 \nabla\psi_0\big\|_{L^1_{xy}};\\
\big\||\nabla|^\beta\pp_x K_1(t)[\psi_0]\big\|_{L^2_{xy}}
\lesssim &
\langle t\rangle^{-\frac12-\frac\beta4}\big\|\langle \nabla\rangle^{2}\nabla \psi_0\big\|_{L^1_{xy}}.
\end{align*}

Collecting the estimates above, we give the proof of Lemma \ref{est:psi-linear}.

\subsection{Estimates on the nonlinear parts $\mathcal {N}_\psi$}

In this subsection, we shall prove that
\begin{lemma}\label{lem:psi-nonlinear}
\begin{align}
\big\|\la\na\ra^4|\nabla|^\gamma\mathcal N_\psi(t;n_0,\vec u_0, \vec b_0 )\big\|_{L^2_{xy}}
\lesssim&
\langle t\rangle^{-\frac12}\big(\|U_0\|_{X_0}+\epsilon_0\|U\|_X+Q(\|U\|_{X})\big);\label{est:psi-nonlinear-L2}\\
\big\|\pp_x\mathcal N_\psi(t;n_0,\vec u_0, \vec b_0 )\big\|_{L^2_{xy}}
\lesssim&
\langle t\rangle^{-1}\big(\|U_0\|_{X_0}+\epsilon_0\|U\|_X+Q(\|U\|_{X})\big);\label{est:psi-nonlinear-Linfty}\\
\big\|\pp_x\nabla(\mathcal N_\psi)(t;n_0,\vec u_0, \vec b_0 )\big\|_{L^2_{xy}}
\lesssim&
\langle t\rangle^{-1}\big(\|U_0\|_{X_0}+\epsilon_0\|U\|_X+Q(\|U\|_{X})\big).\label{est:psi-nonlinear-nax2}
\end{align}
\end{lemma}

We estimate $\mathcal N_\psi(t;n,\vec u, \psi)$ in \eqref{psi-N} terms by terms.

\subsubsection{$\int_0^t\pp_y\pp_tK(t-s)N_0(s)\,ds$}\label{sec:Npsi-N0-1}
The estimates on this term are much similar as the ones in Section \ref{sec:n-N0-1}.
Since the estimates on $N_0^h$ are standard, we only consider $N_0^l$. 
Then by Proposition \ref{lem:Kn2-L} (iii), for $\beta=\gamma, 1,2$ we have
\begin{align*}
&\Big\|\la\na\ra^4|\nabla|^\beta\int_0^t\pp_y\pp_tK(t-s)N_0^l(s)\,ds\Big\|_{L^2_{xy}}\\
\lesssim &\int_0^t\Big\||\nabla|^\beta\nabla\pp_y\pp_tK(t-s)\cdot \la\na\ra^4(n_{\le \langle s\rangle^{0.01}}\vec u_{\le \langle s\rangle^{0.01}})(s)\Big\|_{L^2_{xy}}\,ds\\
\lesssim &
\int_0^t \langle s\rangle^{0.04} \langle t- s\rangle^{-\frac\beta2}\|n(s)\|_{L^2_{xy}}\|\vec u(s)\|_{L^\infty_{xy}}\,ds\\
\lesssim &
\int_0^t \langle t- s\rangle^{-\frac\beta2}\langle s\rangle^{-1.21}\|U\|_X^2\\
\lesssim &
\langle t\rangle^{-\frac\beta2}\|U\|_X^2.
\end{align*}
Letting $\beta=\gamma,1,2$, we have the claimed estimates.

\subsubsection{$\int_0^t\partial_y\Delta K(t-s)N_0(s)\,ds$}\label{sec:Npsi-N0-2}
Note that the operator $\pp_y\nabla\Delta K(t)$ has the similar estimates  with $(\pp_{tt}-\Delta\pp_t-\Delta)\Delta K(t)$.
One just repeats the precess in Section \ref{sec:n-N0-2}, to get the following estimates,
\begin{align*}
\big\|\langle\nabla\rangle^4 |\nabla|^\gamma \int_0^t\partial_y\Delta K(t-s)N_0(s)\,ds\big\|_{L^2_{xy}}
\lesssim&
\langle t\rangle^{-\frac12}\big(\|U_0\|_{X_0}+Q(\|U\|_{X})\big);\\
\big\|\pp_x\int_0^t\partial_y\Delta K(t-s)N_0(s)\,ds\big\|_{L^2_{xy}}
\lesssim&
\langle t\rangle^{-1}\big(\|U_0\|_{X_0}+Q(\|U\|_{X})\big);\\
\big\|\nabla\pp_x\int_0^t\partial_y\Delta K(t-s)N_0(s)\,ds\big\|_{L^2_{xy}}
\lesssim&
\langle t\rangle^{-1}\big(\|U_0\|_{X_0}+Q(\|U\|_{X})\big).
\end{align*}

\subsubsection{$\int_0^t\partial_{xy}K(t-s)N_1(s)\,ds$}\label{sec:Npsi-N1}
 Using \eqref{dec:N1}, we have
$$
N_1=N_{13}+N_{14}+N_{15}+N_{16}.
$$
Again, we only consider $N_{1j}^l, j=3,\cdots,6$.
$N_{13}^l$ and $N_{14}^l$ behave well, we just give the sketch of the estimation. By \eqref{1044}, $\|N_{13}^l\|_{L^1_{xy}}$ decays faster than $\langle s\rangle^{-1}$, it is easy.
By using the corresponding way in Section \ref{sec:Nu-N1}, we can get the desirable estimates.
For $N_{14}^l$, similar as \eqref{0323-1} and \eqref{0323-2}, and by the corresponding ways there,
we also easily get the  desirable estimates and thus the details are omitted here. Now we focus our attentions on the terms $N_{15}^l$ and $N_{16}^l$.

For $N_{15}^l$, by Proposition \ref{lem:Ku1-L} (ii) ($\beta'=\gamma$) we have
\begin{align*}
\Big\|\langle\nabla\rangle^4|\nabla|^\gamma\int_0^t\partial_{xy}K(t-s)N_{15}^l(s)\,ds\Big\|_{L^2_{xy}}
\lesssim&
\int_0^t\Big\||\nabla|^\gamma\nabla\partial_{xy} K(t-s)\cdot P_{\lesssim  \langle s \rangle^{0.01}}\langle\nabla\rangle^4(\pp_x \psi \nabla\psi)(s)\,ds\Big\|_{L^2_{xy}}\,ds\\
\lesssim &
\int_0^t\langle t-s\rangle^{-\frac\gamma2} \big\|\langle\nabla\rangle^4P_{\lesssim  \langle s \rangle^{0.01}}(\pp_x \psi \nabla\psi)(s)\big\|_{L^2_{xy}}\,ds\\
\lesssim &
\int_0^t\langle t-s\rangle^{-\frac\gamma2} \langle s\rangle^{0.04}\|\pp_x \psi\|_{L^2_{xy}} \|\nabla\psi\|_{L^\infty_{xy}}\,ds\\
\lesssim &
\int_0^t\langle t-s\rangle^{-\frac\gamma2} \langle s\rangle^{0.04-1}\|U\|_X^2\,ds\\
\lesssim &
\langle t \rangle^{-\frac14}\|U\|_X^2.
\end{align*}
Moreover, from  Proposition \ref{lem:Ku1''-L} (i),
\begin{align*}
\Big\|\pp_x\int_0^t\partial_{xy}K(t-s)N_{15}^l(s)\,ds\Big\|_{L^2_{xy}}
\lesssim&
\int_0^t\Big\|\partial_{xxy}\nabla K(t-s)\cdot(\pp_x \psi \nabla\psi)(s)\Big\|_{L^2_{xy}}\,ds\\
\lesssim &
\int_0^t \langle t-s\rangle^{-\frac34}\big\|\langle \nabla\rangle^{\frac12+}(\pp_x \psi \nabla\psi)(s)\big\|_{L^{\frac43}_{xy}}\,ds\\
\lesssim &
\int_0^t\langle t-s\rangle^{-\frac34}\|\langle\nabla\rangle\pp_x \psi\|_{L^2_{xy}} \|\nabla\psi\|_{L^2_{xy}}^{\frac12} \|\langle\nabla\rangle^4\nabla\psi\|_{L^\infty_{xy}}^{\frac12}\,ds\\
\lesssim &
\int_0^t\langle t-s\rangle^{-\frac34} \langle s\rangle^{-\frac12-\frac18-\frac14}\|U\|_X^2\,ds\\
\lesssim &
\langle t \rangle^{-\frac12}\|U\|_X^2.
\end{align*}
Similarly, by Proposition \ref{lem:Ku1''-L} (ii) instead,
\begin{align*}
\Big\|\nabla\pp_x\int_0^t\partial_{xy}K(t-s)N_{15}^l(s)\,ds\Big\|_{L^2_{xy}}
\lesssim &
\int_0^t\Big\|\partial_{xxy}\nabla\nabla K(t-s)\cdot(\pp_x \psi \nabla\psi)(s)\Big\|_{L^2_{xy}}\,ds\\
\lesssim &
\int_0^t\langle t-s\rangle^{-1} \big\|P_{\lesssim\langle s\rangle^{0.01}}\langle\nabla\rangle (\pp_x \psi \nabla\psi)(s)\big\|_{L^2_{xy}}\,ds\\
\lesssim &
\int_0^t\langle t-s\rangle^{-1} \langle s\rangle^{0.01-1}\|U\|_X^2\,ds\\
\lesssim &
\langle t \rangle^{-\frac34}\|U\|_X^2.
\end{align*}

For $N_{16}^l$, we need some special handling.  Since $\gamma >\frac12$, by Proposition \ref{lem:Ku1-L} (i) and (ii) (in low and high frequence cases respectively), and Bernstein's inequality,
\begin{align*}
&\Big\|\langle\nabla\rangle^4|\nabla|^\gamma\int_0^t\partial_{xy}K(t-s)N_{16}^l(s)\,ds\Big\|_{L^2_{xy}}\\
\lesssim &
\int_0^t\Big\|P_{\le1}|\nabla|^\gamma\partial_{xy} K(t-s)\,\,  \partial_{x}(\frac12|\nabla \psi|^2-n^2)(s)\Big\|_{L^2_{xy}}\,ds\\
&\qquad +\int_0^t\Big\|P_{\ge1}|\nabla|^{\gamma+1}\partial_{xy} K(t-s)\,\,P_{\lesssim\langle s\rangle^{0.01}} \langle\nabla\rangle^3\partial_{x}(\frac12|\nabla \psi|^2-n^2)(s)\Big\|_{L^2_{xy}}\,ds\\
\lesssim &
\int_0^t\langle t-s\rangle^{-\frac\gamma2}\big\|\partial_{x}(\frac12|\nabla \psi|^2-n^2)(s)\big\|_{L^1_{xy}}\\
&\qquad +
\int_0^t\langle t-s\rangle^{-\frac12}
\big\|P_{\lesssim\langle s\rangle^{0.01}}\langle\nabla\rangle^{4+\gamma}\partial_{x}(\frac12|\nabla \psi|^2-n^2)(s)\big\|_{L^2_{xy}}\,ds\\
\lesssim &
\int_0^t\langle t-s\rangle^{-\frac\gamma2}\big(\|\nabla\psi\|_{L^2_{xy}}\|\nabla\pp_x\psi\|_{L^2_{xy}}+\|n\|_{L^2_{xy}}\|\pp_xn\|_{L^2_{xy}}\big)\\
&\qquad +
\int_0^t\langle t-s\rangle^{-\frac12} \langle s\rangle^{0.05} \big(\|\nabla\psi\|_{L^\infty_{xy}}\|\nabla\pp_x\psi\|_{L^2_{xy}}+\|n\|_{L^\infty_{xy}}\|\pp_xn\|_{L^2_{xy}}\big)\,ds\\
\lesssim &
\int_0^t\langle t-s\rangle^{-\frac\gamma2} \langle s\rangle^{-1}\|U\|_X^2\,ds
+\int_0^t\langle t-s\rangle^{-\frac12} \langle s\rangle^{-1.25+0.05}\|U\|_X^2\,ds\\
\lesssim &
\langle t \rangle^{-\frac14}\|U\|_X^2.
\end{align*}
Further, we use the argument in Section \ref{sec:u-lambda}, and write
\begin{align}
&\Big\|\pp_x\int_0^t\partial_{xy}K(t-s)N_{16}(s)\,ds\Big\|_{L^2_{xy}}\notag\\
\le &
\Big\|\int_0^{ t/2}\partial_{xxy}K(t-s)N_{16}(s)\,ds\Big\|_{L^2_{xy}}
+\Big\|\int_{t/2}^t\partial_{xxy}K(t-s)N_{16}(s)\,ds\Big\|_{L^2_{xy}}\notag\\
= &
\Big\|\int_0^{ t/2}\partial_{xxxy}K(t-s)(\frac12|\nabla \psi|^2-n^2)(s)\,ds\Big\|_{L^2_{xy}}\label{1412-1}\\
&+\Big\|\int_{t/2}^t\partial_{xxy}K(t-s)\partial_{x}(\frac12|\nabla \psi|^2-n^2)(s)\,ds\Big\|_{L^2_{xy}}.\label{1412-2}
\end{align}
For \eqref{1412-1},  using Proposition \ref{lem:Ku1''-L} (i),
\begin{align*}
\eqref{1412-1}
\lesssim &
\int_0^{ t/2}\Big\|\nabla\partial_{xxy}K(t-s)(\frac12|\nabla \psi|^2-n^2)(s)\Big\|_{L^2_{xy}}\,ds\\
\lesssim &
\int_0^{ t/2}\langle t-s\rangle^{-1}\big\|\langle\nabla\rangle^{1+}(\frac12|\nabla \psi|^2-n^2)(s)\big\|_{L^1_{xy}}\,ds\\
\lesssim &
\int_0^{ t/2}\langle t-s\rangle^{-1} \big(\|\langle\nabla\rangle^2\nabla\psi\|_{L^2_{xy}}^2+\|\langle\nabla\rangle^2 n\|_{L^2_{xy}}^2\big)\,ds\\
\lesssim &
\langle t\rangle^{-1} \int_0^{ t/2}\langle s\rangle^{-\frac12}\|U\|_X^2\,ds\\
\lesssim &
\langle t \rangle^{-\frac12}\|U\|_X^2.
\end{align*}
While by Beinstein's inequality, and Proposition \ref{lem:Ku1-L} (i) ($\beta=1,\frac32$ in the low and high frequence cases respectively),
\begin{align*}
\eqref{1412-2}
\lesssim &
\int_{ t/2}^t\Big\|P_{\le 1}\nabla\partial_{xy}K(t-s)\,\,\partial_{x}(\frac12|\nabla \psi|^2-n^2)(s)\Big\|_{L^2_{xy}}\,ds\\
&+\int_{ t/2}^t\Big\|P_{\ge 1}|\nabla|^{\frac32}\partial_{xy}K(t-s)\,\,P_{\lesssim \langle s\rangle^{0.01}}\partial_{x}(\frac12|\nabla \psi|^2-n^2)(s)\Big\|_{L^2_{xy}}\,ds\\
\lesssim &
\int_{ t/2}^t\langle t-s\rangle^{-\frac12} \big(\|\nabla\psi\|_{L^2_{xy}}\|\nabla\pp_x\psi\|_{L^2_{xy}}+\|n\|_{L^2_{xy}}\|\pp_xn\|_{L^2_{xy}}\big)\,ds\\
&\quad+
\int_{ t/2}^t\langle t-s\rangle^{-\frac34}\langle s\rangle^{0.01}\big(\|\nabla\psi\|_{L^2_{xy}}\|\nabla\pp_x\psi\|_{L^2_{xy}}+\|n\|_{L^2_{xy}}\|\pp_xn\|_{L^2_{xy}}\big)\,ds\\
\lesssim &
\int_{ t/2}^t\langle t-s\rangle^{-\frac12}\langle s\rangle^{-1}\,ds\|U\|_X^2+
\int_{ t/2}^t\langle t-s\rangle^{-\frac34}\langle s\rangle^{0.01-1}\,ds\|U\|_X^2\\
\lesssim &
\langle t \rangle^{-\frac12}\|U\|_X^2.
\end{align*}
Thus combining the two estimates above, we get that
$$
\Big\|\pp_x\int_0^t\partial_{xy}K(t-s)N_{16}^l(s)\,ds\Big\|_{L^2_{xy}}\lesssim \langle t \rangle^{-\frac12}\|U\|_X^2.
$$
Similarly, by Proposition \ref{lem:Ku1''-L} (i),
\begin{align*}
\Big\|\nabla\pp_x\int_0^t\partial_{xy}K(t-s)N_{16}(s)\,ds\Big\|_{L^2_{xy}}
\lesssim &
\int_0^t\Big\|\nabla\partial_{xxy} K(t-s)\pp_x(\frac12|\nabla \psi|^2-n^2)(s)\Big\|_{L^2_{xy}}\,ds\\
\lesssim &
\int_0^t\langle t-s\rangle^{-1}\big\|P_{\lesssim \langle s\rangle^{0.01}}\langle\nabla\rangle^{1+}\pp_x(\frac12|\nabla \psi|^2-n^2)(s)\big\|_{L^1_{xy}}\,ds\\
\lesssim &
\int_0^t\langle t-s\rangle^{-1} \langle s\rangle^{-1+0.01}\|U\|_X^2\,ds\\
\lesssim &
\langle t \rangle^{-\frac34}\|U\|_X^2.
\end{align*}
Therefore, we complete all the estimates in this subsubsection and give the desirable results.

\subsubsection{$\int_0^t(\pp_{tt}-\Delta\pp_t-\partial_{xx})K(t-s)N_2(s)\,ds$}
Also, by \eqref{dec:N2} in Section \ref{sec:Nu-N2}, we have
$$
N_2=N_{23}+N_{24}+N_{25}.
$$
The treatments on $N_{23},N_{24},N_{25}$ are the same as the treatments on $N_{13},N_{14},N_{16}$ in Section \ref{sec:Npsi-N1},
so repeating the precess, we have the claimed result as follows,
\begin{align*}
\big\|\langle\nabla\rangle^4 |\nabla|^\gamma \int_0^t(\pp_{tt}-\Delta\pp_t-\partial_{xx})K(t-s)N_2(s)\,ds\big\|_{L^2_{xy}}
\lesssim&
\langle t\rangle^{-\frac12}\big(\|U_0\|_{X_0}+Q(\|U\|_{X})\big);\\
\big\|\pp_x\int_0^t(\pp_{tt}-\Delta\pp_t-\partial_{xx})K(t-s)N_2(s)\,ds\big\|_{L^2_{xy}}
\lesssim&
\langle t\rangle^{-1}\big(\|U_0\|_{X_0}+Q(\|U\|_{X})\big);\\
\big\|\nabla\pp_x\int_0^t(\pp_{tt}-\Delta\pp_t-\partial_{xx})K(t-s)N_2(s)\,ds\big\|_{L^2_{xy}}
\lesssim&
\langle t\rangle^{-1}\big(\|U_0\|_{X_0}+Q(\|U\|_{X})\big).
\end{align*}

\subsubsection{$\int_0^t(\partial_{t}-\Delta) (\partial_{tt}-\Delta\partial_t-\Delta) K(t-s)N_3(s)\,ds$}
As the same reason in Section \ref{sec:n-N3}, this term is in the same level as  $\int_0^t\partial_y(\pp_t-\Delta) K(t-s)N_0(s)\,ds$,
which has  been shown in Sections \ref{sec:Npsi-N0-1} and \ref{sec:Npsi-N0-2}.
Thus we have the desirable estimates by the same way and omit the details.

\subsubsection{$\lambda\int_0^t(\pp_{tt}-\Delta\pp_t)K(t-s)\big(\partial_{xy}u(s)+\partial_{yy}v(s)\big)\,ds$}
This term can be treated by the same way as in Section \ref{sec:n-lambda}, thus we omit the cumbersome details again.

Therefore, we finish all the estimates in this subsection, and obtain Lemma \ref{lem:psi-nonlinear}. Combining with Lemma \ref{lem:psi-linear},
we establish the claimed results in \eqref{est:psi-L2}--\eqref{est:psi-naxL2}.

Combining with the estimates obtained in Section \ref{sec:n}--\ref{sec:psi}, we finish the proof of \eqref{conneed}, and thus prove Theorem \ref{th1}.

\vskip .4in

\appendix

\section{}

\subsection{Proof of Lemma \ref{lem:Elem1}}\label{app2}
We denote
$$
I=\frac1c\Bigg\{\frac1{\sqrt{b+c}}\Big[e^{(-a+\sqrt{b+c})t}-e^{(-a-\sqrt{b+c})t}\Big]
-\frac1{\sqrt{b-c}}\Big[e^{(-a+\sqrt{b-c})t}-e^{(-a-\sqrt{b-c})t}\Big]\Bigg\},
$$
then
\begin{align*}
I=&e^{-at}\frac1c\int_{-t}^t \Big(e^{\sqrt{b+c}x}-e^{\sqrt{b-c}x}\Big)\,dx\\
=&e^{-at}\frac1c\int_0^t \Big(e^{\sqrt{b+c}x}-e^{\sqrt{b-c}x}\Big)\,dx+e^{-at}\frac1c\int_{-t}^0 \Big(e^{\sqrt{b+c}x}-e^{\sqrt{b-c}x}\Big)\,dx\\
=&e^{-at}\frac1c\int_0^t \Big(e^{\sqrt{b+c}x}-e^{\sqrt{b-c}x}\Big)\,dx+e^{-at}\frac1c\int_0^t \Big(e^{-\sqrt{b+c}x}-e^{-\sqrt{b-c}x}\Big)\,dx\\
:=&I_1+I_2.
\end{align*}
Note that
$$
I_1=e^{-at}\frac1c\int_0^t e^{\sqrt{b+c}x}\Big[1-e^{(\sqrt{b-c}-\sqrt{b+c})x}\Big]\,dx,
$$
and
$$
I_2=-e^{-at}\frac1c\int_0^t e^{-\sqrt{b-c}x}\Big[1-e^{(\sqrt{b-c}x-\sqrt{b+c})x}\Big]\,dx.
$$
Therefore,
\begin{align}\label{Identity-I}
I=e^{-at}\frac1c\int_0^t e^{\sqrt{b+c}x}\Big[1-e^{-(\sqrt{b-c}+\sqrt{b+c})x}\Big]\Big[1-e^{(\sqrt{b-c}-\sqrt{b+c})x}\Big]\,dx.
\end{align}

First, since $c>0$, we have
$$
\big|1-e^{-(\sqrt{b-c}+\sqrt{b+c})x}\big|, \quad \big|1-e^{(\sqrt{b-c}-\sqrt{b+c})x}\big|\lesssim 1.
$$
Hence we obtain from \eqref{Identity-I} that
if $b+c\ge 0$, then
\begin{align}
|I|\lesssim & e^{-at}\frac1c\int_0^t e^{\sqrt{b+c}x}\,dx
=  e^{-at+\sqrt{b+c}t}\frac1c\int_0^t e^{\sqrt{b+c}(x-t)}\,dx\notag\\
\lesssim & \frac1{c\sqrt{b+c}} e^{-at+\sqrt{b+c}t};\label{10.25}
\end{align}
if $b+c< 0$, then
\begin{align*}
|I|\lesssim &
\frac tc e^{-at}.
\end{align*}

Second,
using the following two inequalities,
\begin{align*}
\Big|1-e^{-(\sqrt{b-c}+\sqrt{b+c})x}\Big|&\lesssim \big|(\sqrt{b-c}+\sqrt{b+c})x\big|; \\
\Big|1-e^{(\sqrt{b-c}-\sqrt{b+c})x}\Big|&\lesssim \big|(\sqrt{b-c}-\sqrt{b+c})x\big|,
\end{align*}
we obtain that if $b+c\ge 0$, then
\begin{align*}
|I|\lesssim &
e^{-at}\frac1c\int_0^t e^{\sqrt{b+c}x}c x^2\,dx
\lesssim
e^{-at+\sqrt{b+c}t}\int_0^t e^{\sqrt{b+c}(x-t)} x^2\,dx
\lesssim
t^3e^{-at+\sqrt{b+c}t};
\end{align*}
if $b+c< 0$, then
$$
|I|\lesssim e^{-at}\frac1c\int_0^t c x^2\,dx\lesssim t^3 e^{-at}.
$$

Third, we particularly consider the case of $|c|\ll |b+c|$, then $|c|\ll |b|$ and
\begin{align*}
\Big|1-e^{-(\sqrt{b-c}+\sqrt{b+c})x}\Big|\lesssim 1; \quad
\Big|1-e^{(\sqrt{b-c}-\sqrt{b+c})x}\Big|\lesssim \frac{cx}{\sqrt{|b|}}.
\end{align*}
If $b+c\ge 0$, then using the inequality above,
\begin{align*}
|I|\lesssim &
e^{-at}\frac1c\int_0^t e^{\sqrt{b+c}x}\frac{cx}{\sqrt{|b|}}\,dx
=
e^{-at+\sqrt{b+c}t}\int_0^t e^{\sqrt{b+c}(x-t)} \frac{x}{\sqrt{|b|}}\,dx\\
\lesssim &
\frac{t}{\sqrt{|b|}\sqrt{b+c}}e^{-at+\sqrt{b+c}t}
\sim
\frac{t}{b+c}e^{-at+\sqrt{b+c}t}.
\end{align*}
Combining with \eqref{10.25}, no matter when  $|c|\ll |b+c|$ or   $|c|\gtrsim |b+c|$, we have
\begin{align*}
|I|\lesssim
\frac{\langle t\rangle}{b+c}e^{-at+\sqrt{b+c}t}.
\end{align*}

In conclusion, we proved  that when $b+c< 0$,
$$
|I|\lesssim  \min\{t^3,\frac tc\} e^{-at};
$$
while when $b+c\ge 0$,
$$
|I|\lesssim  \min\{t^3,\frac1{c\sqrt{b+c}}, \frac{\langle t\rangle}{b+c}\}e^{-at+\sqrt{b+}t}.
$$

\subsection{Proof of \eqref{basic-1}}\label{app3}
We may assume that $x>0,y>0$, otherwise one may replace it by $|x|,|y|$, also we assume that $x\ge y$ by symmetry. Now we prove it
by splitting several cases.

Case 1: $x\ll 1, y\ll 1$, then it follows by Tayor's expansion.

Case 2: $x\gtrsim 1, x\gg y$. Then
$$
\big|\frac{\sin x}{x}-\frac{\sin y}{y}\big|\lesssim \big|\frac{\sin x}{x}\big|+\big|\frac{\sin y}{y}\big|\lesssim 1,
$$
which is less than the right-hand side of \eqref{basic-1}.

Case 3: $x\gtrsim 1, x\sim y$.  Then by the mean value theorem,
$$
\big|\frac{\sin x}{x}-\frac{\sin y}{y}\big|\lesssim  \big|\frac{\bar{x}\cos \bar{x}-\sin \bar{x}}{\bar{x}^2}\big|(x-y)
\lesssim \frac{(x-y)}{x},
$$
where $\bar x$ is middle point satisfying $\bar x\in (y,x)$.

Combining with these three cases, we prove \eqref{basic-1}.

\subsection{Proof of \eqref{17.17}}\label{app4}

For any $\xi\in\R^d$, we have
$$
\widehat{P_{\le \langle s\rangle^\beta}f}(\xi)=\chi_{\le 1}\Big(\frac{\xi}{\langle s\rangle^\beta}\Big) \hat{f}(\xi).
$$
Therefore,
$$
\pp_s\widehat{P_{\le \langle s\rangle^\beta}f}(\xi)=|\xi|\langle s\rangle^{-\beta-1}\chi_{\le 1}'\Big(\frac{\xi}{\langle s\rangle^\beta}\Big)\hat{f}(\xi).
$$
Since supp $\chi_{\le 1}'\in \{|\xi|\sim 1\}$, we have
$$
\pp_sP_{\le \langle s\rangle^\beta}f=\langle s\rangle^{-\beta-1}P_{\sim \langle s\rangle^\beta} |\nabla| f.
$$
Thus by Beinstein's inequality, for any $\beta\in \R$, $1\le q\le \infty$,
\begin{align*}
\|\pp_sP_{\le \langle s\rangle^\beta}f\|_{L^q}
\lesssim
\langle s\rangle^{-\beta-1}\|P_{\sim \langle s\rangle^\beta} |\nabla| f\|_{L^q}
\sim
\langle s\rangle^{-1}\|P_{\sim \langle s\rangle^\beta}f\|_{L^q}.
\end{align*}

\subsection{Proof of \eqref{Nash}}\label{app5}
By the Littlewood-Paley decomposition,
\begin{align*}
\Big\||\nabla|^{\bar{\gamma}} \langle\nabla\rangle\psi\Big\|_{L^\infty_{xy}}
\lesssim &
\sum\limits_{N\le 1} N^{\bar{\gamma}}\big\|P_N\psi\big\|_{L^\infty_{xy}}+\sum\limits_{N\ge 1} N^{\bar{\gamma}+1}\big\|P_N\psi\big\|_{L^\infty_{xy}},
\end{align*}
where $N$ is the dyadic number.
By Nash's inequality,
\begin{align*}
\|P_N\psi\|_{L^\infty_{xy}}
\lesssim &
\|\pp_{xy}P_N\psi\|_{L^2_{xy}}^{\frac14}
\|\pp_{x}P_N\psi\|_{L^2_{xy}}^{\frac14}
\|\pp_{y}P_N\psi\|_{L^2_{xy}}^{\frac14}
\|P_N\psi\|_{L^2_{xy}}^{\frac14}\\
\lesssim &
N^{\frac12}\|\pp_{x}P_N\psi\|_{L^2_{xy}}^{\frac12}
\|P_N\psi\|_{L^2_{xy}}^{\frac12}.
\end{align*}
Therefore, by Beinstein's inequality,
\begin{align*}
\Big\||\nabla|^{\bar{\gamma}} \langle\nabla\rangle\psi\Big\|_{L^\infty_{xy}}
\lesssim &
\sum\limits_{N\le 1} N^{\bar{\gamma}+\frac12}\|\pp_{x}P_N\psi\|_{L^2_{xy}}^{\frac12}
\|P_N\psi\|_{L^2_{xy}}^{\frac12}
+\sum\limits_{N\ge 1} N^{\bar{\gamma}+1+\frac12}\|\pp_{x}P_N\psi\|_{L^2_{xy}}^{\frac12}
\|P_N\psi\|_{L^2_{xy}}^{\frac12}\\
\lesssim &
\sum\limits_{N\le 1} N^{\bar{\gamma}-\frac\gamma2}\| P_N\pp_{x}\nabla\psi\|_{L^2_{xy}}^{\frac12}
\|P_N|\nabla|^\gamma\psi\|_{L^2_{xy}}^{\frac12}\\
&+\sum\limits_{N\ge 1} N^{\bar{\gamma}-1-\frac\gamma2}\|P_N\pp_{x}\nabla\psi\|_{L^2_{xy}}^{\frac12}
\big\|P_N\langle\nabla\rangle^4|\nabla|^\gamma\psi\big\|_{L^2_{xy}}^{\frac12}.
\end{align*}
Hence, when $\frac\gamma2<\bar \gamma<1+\frac\gamma2$, the sums are finite. So we have
\begin{align*}
\big\||\nabla|^{\bar{\gamma}} \langle\nabla\rangle\psi\big\|_{L^\infty_{xy}}
\lesssim \|\pp_{x}\nabla\psi\|_{L^2_{xy}}^{\frac12}
\big\|\langle\nabla\rangle^4|\nabla|^\gamma\psi\big\|_{L^2_{xy}}^{\frac12}.
\end{align*}

\vskip .4in
\section*{Acknowledgements}
J. Wu was partially supported by the NSF grants DMS 1209153 and DMS 1614246 and by the AT \& T Foundation at Oklahoma State University. Y. Wu was partially
supported by the NSFC (No. 11626250, No. 11571118).


\vskip .4in
\begin{thebibliography}{99}

%
%
%
%
%
%
%
%
%
%
%
%

\bibitem{ChenWang} G.-Q. Chen and D. Wang, Global solutions of nonlinear magnetohydrodynamics
with large initial data, {\it J. Differential Equations \bf 182} (2002), 344-376.

\bibitem{HuX} X. Hu, Global existence for two dimensional compressible
magnetohydrodynamic flows with zero magnetic diffusivity, arXiv: 1405.0274v1 [math.AP] 1 May 2014.

\bibitem{HuLin} X. Hu and F. Lin, Global Existence for Two Dimensional Incompressible Magnetohydrodynamic Flows with Zero Magnetic Diffusivity, arXiv: 1405.0082v1 [math.AP] 1 May 2014.

\bibitem{HuWang} X. Hu and D. Wang, Global existence and large-time behavior of solutions to the three-dimensional equations of compressible magnetohydrodynamic flows, {\it Arch. Ration. Mech. Anal. \bf 197} (2010), 203-238.

\bibitem{HuangLi}X. Huang and J. Li, Serrin-type blowup criterion for viscous, compressible, and heat conducting Navier-Stokes and magnetohydrodynamic flows, {\it Comm. Math. Phys. \bf 324} (2013), 147-171.

%
%

\bibitem{Kawa} S. Kawashima, Systems of a hyperbolic-parabolic composite type, with applications to the
equations of magnetohydrodynamics, Ph. D. Thesis, Kyoto University, 1983.

\bibitem{Kawa2} S. Kawashima, Smooth global solutions for two-dimensional equations
of electro-magneto-fluid dynamics, {\it Japan J. Appl. Math. \bf 1} (1984), 207-222.

%
%
%
%
%
%
\bibitem{LinZhang1} F. Lin, L. Xu, and  P. Zhang, Global small solutions to 2-D incompressible MHD system, {\it J. Differential Equations \bf 259} (2015), 5440-5485.

%
%


\bibitem{Zhifei}X. Ren, J. Wu, Z. Xiang and Z. Zhang, Global existence and decay of smooth solution for
the 2-D MHD equations without magnetic diffusion, {\it J. Functional Anal. \bf 267} (2014), 503-541.


\bibitem{WWX} J. Wu, Y. Wu and X. Xu, Global small solution to the 2D MHD system with a velocity damping term, {\it SIAM J. Math. Anal. \bf 47} (2015), 2630-2656.
 

%
%
%
%
%
%
%
%
%

\bibitem{ZhangT} T. Zhang, An elementary proof of the global existence and uniqueness theorem to 2-D incompressible non-resistive MHD system, arXiv:1404.5681v1 [math.AP] 23 Apr 2014.
\end{thebibliography}
\end{document}